\documentclass{amsart}
\usepackage{amsmath}
\usepackage{amsthm}
\usepackage{amsbsy}
\usepackage{amsopn}
\usepackage{amsmath,amscd}
\usepackage{amssymb}
\usepackage{amsfonts}
\usepackage{multirow}
\usepackage[official ]{eurosym}

\newcommand{\cC}{\mathcal{C}}\newcommand{\cD}{\mathcal{D}}
\newcommand{\cE}{\mathcal{E}}\newcommand{\cF}{\mathcal{F}}
\newcommand{\cG}{\mathcal{G}}
\newcommand{\cI}{\mathcal{I}}
\newcommand{\cK}{\mathcal{K}}\newcommand{\cL}{\mathcal{L}}
\newcommand{\cM}{\mathcal{M}}\newcommand{\cN}{\mathcal{N}}
\newcommand{\cO}{\mathcal{O}}

\newcommand{\cW}{\mathcal{W}}
\newcommand{\cZ}{\mathcal{Z}}

\newcommand{\bC}{\mathbb{C}}

\newcommand{\bN}{\mathbb{N}}

\newcommand{\bQ}{\mathbb{Q}}\newcommand{\bR}{\mathbb{R}}

\newcommand{\bZ}{\mathbb{Z}}

\usepackage[dvipsnames,svgnames,x11names,hyperref]{xcolor}
\usepackage{url,graphicx,verbatim,amssymb,enumerate,stmaryrd}
\usepackage[pagebackref,colorlinks,citecolor=Bittersweet,linkcolor=Bittersweet,urlcolor=Bittersweet,filecolor=Bittersweet]{hyperref}
\usepackage{tikz} 
\usetikzlibrary{arrows,decorations.pathmorphing,backgrounds,fit,positioning,shapes.symbols,chains,calc}
\usepackage[all]{xy}
\usepackage[hmargin=3cm,vmargin=3cm]{geometry}
\usepackage{setspace,kantlipsum}

\usepackage{booktabs}

\newtheorem{theorem}{Theorem}[section]
\newtheorem{lemma}[theorem]{Lemma}
\newtheorem{proposition}[theorem]{Proposition}
\newtheorem{corollary}[theorem]{Corollary}
\newtheorem{claim}[theorem]{Claim}
\theoremstyle{definition}
\newtheorem{definition}[theorem]{Definition}
\newtheorem{construction}[theorem]{Construction}

\newtheorem{example}[theorem]{Example}
\newtheorem{inductive lemma}[theorem]{Inductive Lemma}
\newtheorem{remark}[theorem]{Remark}

\newtheorem{warning}[theorem]{Warning}

\begin{document}
\title[]{\large The simplification of singularities of Lagrangian and Legendrian fronts}
\maketitle
\begin{center}
\Large  Daniel  \'Alvarez-Gavela \let\thefootnote\relax\footnote{The author was partially supported by NSF grant DMS-1505910.} \\ $ $ \\ 
\large \itshape Stanford University
\end{center}
\begin{abstract}
We establish a full $h$-principle ($C^0$-close, relative, parametric) for the simplification of singularities of Lagrangian and Legendrian fronts. More precisely, we prove that if there is no homotopy theoretic obstruction to simplifying the singularities of tangency of a Lagrangian or Legendrian submanifold with respect to an ambient foliation by Lagrangian or Legendrian leaves, then the simplification can be achieved by means of a Hamiltonian isotopy. 
 \end{abstract}
\onehalfspacing
\tableofcontents

\section{Introduction and statement of results}\label{introduction and statement of results}

\subsection{Panoramic overview}\label{overview} 

In this paper we establish a general $h$-principle for the simplification of singularities of Lagrangian and Legendrian fronts. The precise formulation is given in Theorem \ref{main result} below. Here is a sample corollary of our results, where $\pi:T^*S^n \to S^n$ denotes the cotangent bundle of the standard $n$-dimensional sphere.

\begin{corollary}\label{sphere1}  Let $S \subset  T^*S^n$ be any embedded Lagrangian sphere. If $n$ is even, then there exists a compactly supported Hamiltonian isotopy $\varphi_t:T^*S^n \to T^*S^n$ such that the singularities of the projection $\pi|_{\varphi_1(S)}: \varphi_1(S) \to S^n$ consist only of folds. An analogous result holds for even-dimensional Legendrian spheres in the $1$-jet space $J^1(S^n, \bR) = T^*S^n \times \bR.$
\end{corollary}

More generally, let $S\subset M$ be any embedded Lagrangian sphere, where $(M^{2n},\omega)$ is a symplectic manifold equipped with a foliation $\cF$ by Lagrangian leaves. Denote by $T\cF$ the distribution of Lagrangian planes tangent to the foliation $\cF$ and let $V$ be the restriction of $T\cF$ to $S$. It is easy to see that a necessary condition for $S$ to be Hamiltonian isotopic to a Lagrangian sphere whose singularities of tangency with respect to $\cF$ consist only of folds is that $V$ is stably trivial as a real vector bundle over the sphere. When $n$ is even, our $h$-principle implies the following converse.

\begin{corollary}\label{sphere2}  Suppose that $V= T \cF|_S$ is stably trivial as a real vector bundle over the sphere. If $n$ is even, then there exists a compactly supported Hamiltonian isotopy $\varphi_t:M \to M$ such that the singularities of tangency of $\varphi_1(S)$ with respect to the foliation $\cF$ consist only of folds. An analogous result holds for even-dimensional Legendrian spheres.
\end{corollary}

\begin{remark}\label{stably trivial assumption}
As we will see, the assumption that $V$ is stably trivial is automatically satisfied for all even $n$ such that $n \not\equiv 2$ mod $8$. The simplest example in which more complicated singularities are necessary occurs when $n=2$ and corresponds to the Hopf bundle on $S^2$, where in addition to the $\Sigma^{10}$ folds we find that a $\Sigma^{110}$ pleat is unavoidable. When $n$ is odd the problem is not as straightforward due to the fact that $\pi_n(U_n) \neq 0$. Nevertheless, we will apply our $h$-principle to give a necessary and sufficient condition for the simplification of singularities to be possible in terms of the homotopy class of the distribution of Lagrangian planes $V$.
\end{remark}

As another application of our $h$-principle, we establish that higher singularities are unnecessary for the homotopy theoretic study of the space of Legendrian knots in the standard contact $\bR^3$. Before we can state our result we need to set some notation.

Recall that the front projection $\bR^3 \to \bR^2$ corresponds to the forgetful map $J^1(\bR,\bR) \to J^0(\bR,\bR)$ where we identify $J^1(\bR,\bR)=\bR^3$ and $J^0(\bR,\bR)=\bR^2$. In coordinates, we have $\bR^3=\bR(q) \times \bR(p) \times \bR(z)$, $\bR^2=\bR(q) \times \bR(z)$, $\xi_{std}=\ker(dz - pdq)$ and the front projection is the map $(q,p,z)\mapsto(q,z)$. The front of a Legendrian knot $f:S^1 \to \bR^3$ is the composition of $f$ with the front projection, which results in a map $S^1 \to \bR^2$. Let $\cL$ be the space of all (parametrized) Legendrian knots  $f:S^1 \to \bR^3$ and let $\cM \subset \cL$ be the subspace consisting of those Legendrian knots whose front only has mild singularities, namely cusps and embryos. A cusp of the front corresponds to a fold type singularity of tangency of $f$ with respect to the foliation given by the fibres of the front projection. An embryo is the instance of birth/death of two cusps and corresponds to the familiar Reidemesiter Type I move. 

The inclusion $\cM \hookrightarrow \cL$ is not a homotopy equivalence. Indeed, it is easy to see that $\pi_2(\cL, \cM) \neq 0$. However, by decorating the mild singularities of the Legendrian knots in $\cM$ we define a space $\cD$, equipped with a map $\cD \to \cM$ which forgets the decoration, such that the composition $\cD \to \cM \hookrightarrow \cL$ is surjective on $\pi_0$ and restricts to a weak homotopy equivalence on each connected component. The precise definition of the space $\cD$ is as follows.

For any $k \geq 0$, consider the unordered configuration space $\text{C}_k(S^1)$ of $k$ distinct points on the circle $S^1=\bR/\bZ$. Define a space $\widetilde{\text{C}}_k(S^1)$ fibered over $\text{C}_k(S^1)$ such that the fibre over the configuration $\{t_1,\ldots ,t_k\} \subset S^1$ consists of all unordered collections of closed intervals $I_1,\ldots ,I_m \subset S^1$ which are disjoint from the points $t_1, \ldots , t_k$ and such that $I_i \cap I_j \neq \varnothing$ implies either $I_i \subset \text{int}(I_j)$ or $I_j \subset \text{int}(I_i)$. In the degenerate case where the endpoints of an interval $I_j$ coincide, the interval consists of a point and this is allowed. The topology is such that an interval $I_j$ which contains no other intervals in its interior can continuously shrink to a point and disappear. Observe therefore that the fibre of the map $\widetilde{\text{C}}_k(S^1) \to \text{C}_k(S^1)$ is contractible. We give $\widetilde{\text{C}}(S^1)= \bigsqcup_k  \widetilde{\text{C}}_k(S^1)$ the disjoint union topology, so that the points $t_i$ are not allowed to collide. We will refer to the elements of $\widetilde{\text{C}}(S^1)$ as decorations.

Let $D=\big( \{ t_i \} , \{ I_j \} \big) \in \widetilde{\text{C}}(S^1)$ be any decoration. We say that a Legendrian knot $f:S^1 \to \bR^3$ is compatible with $D$ if its front has cusp singularities at each of the points $t_j$ and if moreover for each interval $I_j$ the following holds. If $I_j$ is not degenerate, then we demand that the front has cusp singularities at each of the two endpoints of $I_j$ and moreover we require that the two cusps have opposite Maslov co-orientations. If $I_j$ is degenerate and thus consists of a single point, then we demand that the front of $f$ has an embryo singularity at that point. At all other points of $S^1$ we demand that the front is regular. 

Define $\cD$ to be the space of all pairs $(f,D)$ such that $f:S^1 \to \bR^3$ is a Legendrian knot compatible with a decoration $D \in \widetilde{\text{C}}(S^1)$. Note in particular that $f \in \cM$. The composition of the forgetful map $\cD \to \cM$ given by $(f,D) \mapsto f$ with the inclusion $\cM \hookrightarrow \cL$ gives a map $\cD \to \cL$. It is easy to see that the induced map $\pi_0(\cD) \to \pi_0(\cL)$ is surjective but not injective. The parametric version of our $h$-principle implies the following result.
 
\begin{corollary}\label{Reidemeister}
The map $\cD \to \cL$ is a weak homotopy equivalence on each connected component.
\end{corollary}

Given a family of Legendrian knots in $\bR^3$ parametrized by a space of arbitrarily high dimension, Corollary \ref{Reidemeister} allows us to simplify the singularities of the corresponding family of fronts so that we end up having only cusps and embryos. Moreover we have a strong control on the structure of the singularity locus (in the source) given by the family of configurations decorating the mild singularities. Proofs of Corollaries  \ref{sphere1}, \ref{sphere2} and \ref{Reidemeister}, as well as of the claims made in Remark \ref{stably trivial assumption} and elsewhere in the above overview will be given in Section \ref{applications to the simplification of singularities}. 

The singularities of Lagrangian and Legendrian fronts, also known as caustics in the literature, were first extensively studied by Arnold and his collaborators, see \cite{A90} for an introduction to the theory. Today, caustics still play a central role in modern symplectic and contact topology, both rigid and flexible. In many situations it is desirable for a Lagrangian or Legendrian front to have singularities which are as simple as possible. For example, Ekholm's method of Morse flow-trees \cite{E07} for the computation of Legendrian contact homology can only be applied if the caustic of the Legendrian consists only of folds. 

The simplification of singularities of Lagrangian and Legendrian fronts is of course not always possible, since there exists a homotopy theoretic obstruction to removing higher singularities. The main point of this paper is to prove that whenever this formal obstruction vanishes, the simplification can indeed be achieved by means of an ambient Hamiltonian isotopy. Our $h$-principle is \emph{full} in the sense of \cite{EM02} ($C^0$-close, relative parametric). See Section \ref{Main results}, where we state the result precisely, for further details. The key ingredients in the proof are (1) an explicit model for the local wrinkling of Lagrangian and Legendrian submanifolds and (2) our holonomic approximation lemma for $\perp$-holonomic sections from \cite{AG15}, which is a refinement of Eliashberg and Mishachev's holonomic approximation lemma \cite{EM01}.

Our work builds on Entov's paper \cite{E97}, where the first $h$-principle for the simplification of caustics was proved. See Section \ref{Surgery of singularities} for a discussion of his results, which consist of an adaptation of Eliashberg's surgery of singularities \cite{E70}, \cite{E72} to the setting of Lagrangian and Legendrian fronts. Our paper instead follows the strategy employed by Eliashberg and Mishachev in the proof of their wrinkled embeddings theorem \cite{EM09}. The main advantage of the wrinkled approach is the following. The surgery technique can only be applied to $\Sigma^2$-nonsingular fronts, which are fronts whose singularities have the lowest corank possible. This condition is not generic except in low dimensions. By contrast, the wrinkling technique can be applied to any front. By removing the $\Sigma^2$-nonsingularity restriction, we extend considerably the range of application of the $h$-principle.

Given any smooth manifold equipped with a smooth foliation, there is the analogous problem in geometric topology of simplifying of the singularities of tangency of a smooth submanifold with respect to the foliation by means of an ambient smooth isotopy. This problem also abides by an $h$-principle and has been studied by several authors. Gromov's method of continuous sheaves \cite{G69}, \cite{G86}, as well as Eliashberg and Mishachev's holonomic approximation lemma \cite{EM01}, \cite{EM02} can be used to simplify the singularities of tangency when the submanifold is open. Gromov's theory of convex integration \cite{G73}, \cite{G86} also yields the same result. When the submanifold is closed, neither continuous sheaves nor holonomic approximation work, but there are several other methods which do work. We have already mentioned two of them, namely Eliashberg's surgery of singularities \cite{E70}, \cite{E72} and the wrinkling embeddings theorem of Eliashberg and Mishachev \cite{EM09}. Additionally, Spring showed in \cite{S02}, \cite{S05} that convex integration can be applied to the closed case. See also the approach of Rourke and Sanderson \cite{RS01}, \cite{RS03}. 

We should also mention that Corollary \ref{Reidemeister} can be thought of as a Legendrian analogue of Igusa's theorem \cite{I84} which states that higher singularities of smooth functions are unnecessary. The analogy becomes clearer from the viewpoint of generating functions. Closely related is another result of Igusa \cite{I87} on the high connectivity of the space of framed functions and Lurie's improvement in \cite{L09} which sketches a proof of the fact that the space of framed functions is contractible. Eliashberg and Mishachev generalized Igusa's original result in \cite{EM00} and gave a complete proof of the contracibility of framed functions in \cite{EM12}, in both cases using the wrinkling philosophy. There also exists a folklore approach for proving $h$-principles using a categorical delooping technique which was used by Galatius in unpublished work to obtain a different proof of the contractibility of framed functions. The approach of Galatius inspired Kupers' recent paper \cite{K17}, which provides an exposition to the delooping technique and includes yet another proof of the contractibility of framed functions as one of three applications.

\subsection{Singularities of tangency}
Let $g:L \to B$ be any map between smooth manifolds, where we assume $\dim(L) \leq \dim(B)$ for simplicity. A point $q \in L$ is called a singularity of the map $g$ if the differential $dg:T_qL \to T_{g(q)}B$ is not injective. The subset of $L$ consisting of singular points is denoted by $\Sigma(g)$. 
Next, let $\pi:M \to B$ be a fibration of smooth manifolds and let $f:L \to M$ be a smooth embedding. The singularities of the composition $g=\pi \circ f : L \to B$ are precisely the singularities of tangency of the submanifold $f(L) \subset M$ with respect to the foliation $\cF$ of $M$ given by the fibres $\cF_b=\pi^{-1}(b)$, $b \in B$. This latter notion makes sense for arbitrary foliations $\cF$ not necessarily given by a globally defined fibration.

\begin{definition}\label{singularity}
A singularity of tangency of an embedding $f:L \to M$ with respect to a foliation $\cF$ of $M$ is a point $q \in L$ such that $df(T_qL) \cap T_{f(q)}\cF \neq 0$. The subset of $L$ consisting of singular points is denoted by $\Sigma(f, \cF)$.
\end{definition}

We will be interested in the special case in which $(M, \omega)$ is a symplectic $2n$-dimensional manifold and  $\cF$ is a foliation of $M$ by Lagrangian leaves. Such a setup could arise from a Lagrangian fibration $\pi:M \to B$, where $B$ is any $n$-dimensional manifold. A good example to keep in mind is the cotangent bundle $M=T^*B$ with $\pi:T^*B \to B$ the standard projection. 

\begin{figure}[h]
\includegraphics[scale=0.6]{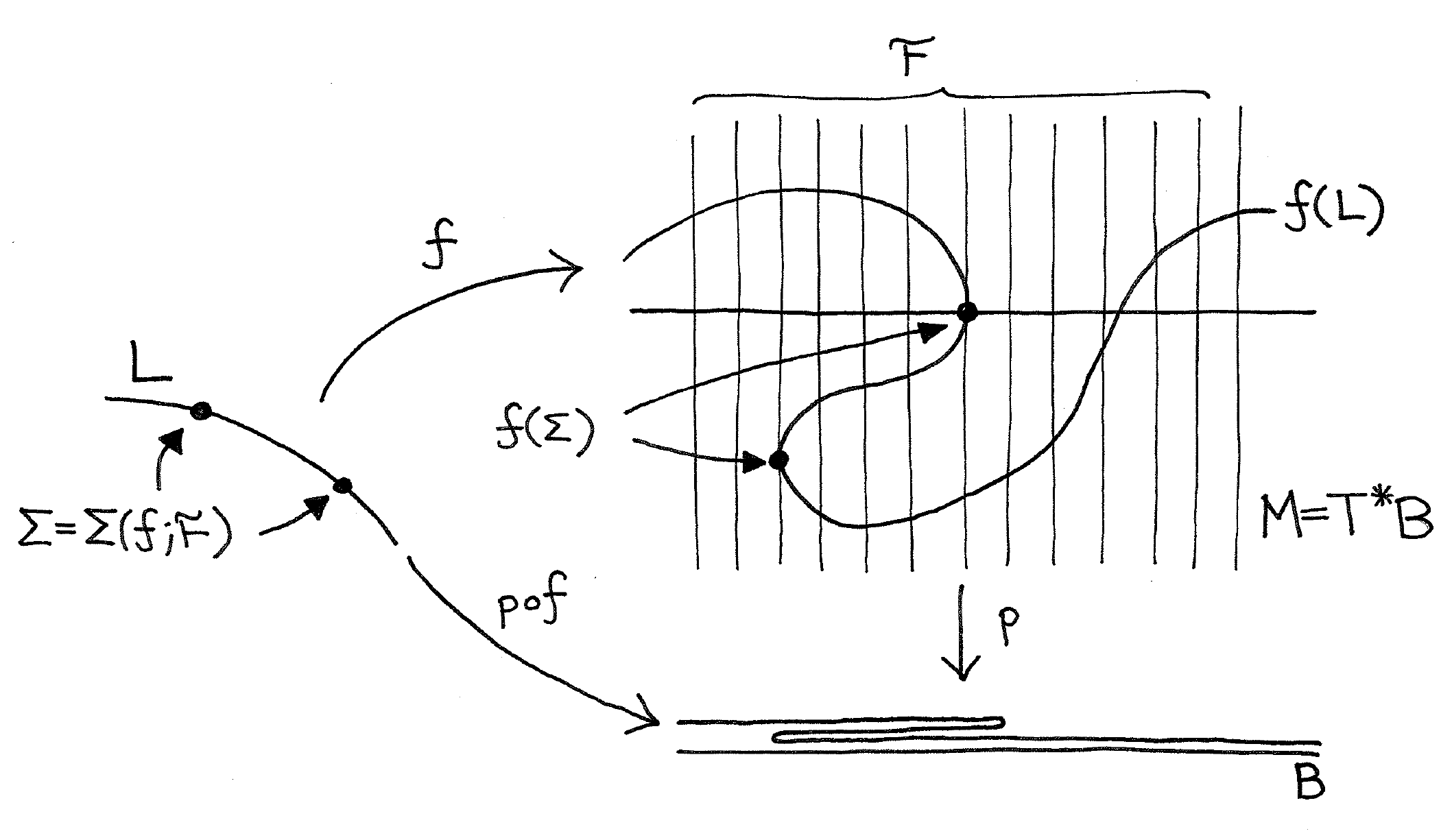}
\caption{The singularities of tangency of a Lagrangian embedding $f:L \to T^*B$.}
\label{singularity}
\end{figure}

We will also consider the analogous notion in contact topology. Here $(M, \xi)$ is a $(2n+1)$-dimensional contact manifold and $\cF$ is a foliation of $M$ by Legendrian leaves. Such a setup could arise from a Legendrian fibration $\pi: M \to B$, where $B$ is an $(n+1)$-dimensional manifold. A good example to keep in mind is the $1$-jet space $M=J^1(E,\bR)$, where $E$ is any $n$-dimensional manifold, $B=J^0(E, \bR)$ and $\pi:J^1(E, \bR) \to J^0(E, \bR)$ is the  forgetful map (which in the literature is usually referred to as the front projection). We remark for future reference that $J^1(E, \bR) = T^*E \times \bR$,  $\, J^0(E, \bR) = E \times \bR$ and that the front projection $T^*E \times \bR \to E \times \bR$ is the product of the cotangent bundle projection $T^*E \to E$ and the identity map $\bR \to \bR$.

Suppose that $\cF$ is induced by a Lagrangian or Legendrian fibration $\pi:M \to B$, so that the singularities of tangency $\Sigma(f, \cF) = \{ q \in L : \, \, df_q(T_qL) \cap T_{f(q)}\cF \neq 0 \} $ coincide with the singularity locus $\Sigma( p\circ f)= \{ q \in L : \, \, \ker\big ( d( p \circ f)_q \big) \neq 0 \}$ of the smooth map $p \circ f:L \to B$. Then the composition $p \circ f$ is called the Lagrangian or Legendrian front associated to $f$. The image of the singularity locups $p \circ f ( \Sigma ) \subset B$ is called the caustic of the front. 

\subsection{The Thom-Boardman hierarchy}\label{Hierarchy of singularities} 
To state our results precisely, we first need to recall some notions from the Thom-Boardman hierarchy of singularities. We do not intend to be thorough and only discuss the basic facts which are necessary to frame our discussion. For a detailed exposition to the theory of singularities we refer the reader to the original papers, including those of Thom \cite{T56}, Boardman \cite{B67} and Morin \cite{MO65}, as well as to the books \cite{AGV88I}, \cite{AGV88II} by Arnold, Gusein-Zade and Varchenko.

Suppose first that $g:L\to B$ is any smooth map between smooth manifolds, where $\dim(L)=n$ and $\dim(B)=m$. The singularity locus $\Sigma=\Sigma(g) \subset L$ of $g$  can be stratified in the following way.
\[ \Sigma=\Sigma^1 \cup \Sigma^2 \cup \cdots \cup \Sigma^n , \qquad \Sigma^k= \{ q \in L : \, \, \, \dim \big( \ker (dg_q) \big) = k \} . \]

The Thom transversality theorem implies that generically $\Sigma^k$ is a smooth submanifold of $L$, whose codimension equals $k\big(m-n+ k \big)$. In fact, to any non-increasing sequence $I$ of non-negative integers $i_1 \geq i_2 \geq \cdots \geq i_k$ we can associate a singularity locus $\Sigma^I \subset L$. Provided that $g$ is generic enough so that its $k$-jet extension $j^k(g)$ satisfies a certain transversality condition, $\Sigma^I$ is a smooth submanifold whose codimension is given by an explicit combinatorial formula. For such $g$, the locus $\Sigma^I$ is determined inductively by $\Sigma^I= \Sigma^{i_k}\big(g|_{\Sigma^{I'}} : \Sigma^{I'} \to B \big) $, where $I'$ denotes the truncated sequence $i_1 \geq i_2 \geq \cdots \geq i_{k-1}$. In particular, $\Sigma^I \subset \Sigma^{I'}$. 

We will mainly be interested in the flag of submanifolds $\Sigma^1 \supset \Sigma^{11} \supset \cdots \supset \Sigma^{1^n}$, where we denote a string of $1'$s of length $k$ by $1^k$. Generically, $\Sigma^{1^k}$ is a smooth submanifold of $L$ with $\dim( \Sigma^{1^k})=n-k$. To understand this flag geometrically it is useful to think of the line field $l= \ker(dg) \subset TL$, which is defined along $\Sigma^1$. Inside $\Sigma^1$ we have the secondary singularity $\Sigma^{11} = \Sigma^1\big( g|_{\Sigma^1} : \Sigma^1\to B \big) $, which consists of the set of points $q \in \Sigma^1$ where $l$ is tangent to $\Sigma^1$. Points in the complement $\Sigma^{10}= \Sigma^{1} \setminus \Sigma^{11}$, where $l$ is transverse to $\Sigma^1$, are called fold points. Similarly, the singularity $\Sigma^{111}$ consists of the set of points $q \in \Sigma^{11}$ where $l$ is tangent to $\Sigma^{11}$. Points in the complement $\Sigma^{110}= \Sigma^{11} \setminus \Sigma^{111}$, where $l$ is transverse to $\Sigma^{11}$ inside $\Sigma^1$, are called pleats. And so on. See Figure \ref{foldingtypesingularities} for an illustration of $\Sigma^{10}$ and $\Sigma^{110}$. Each of the singularities $\Sigma^{1^k0}=\Sigma^{1^k} \setminus \Sigma^{1^{k+1}}$ has a unique local model and is easy to understand explicitly. We call them $\Sigma^1$-type singularities. 

Singularities of type $\Sigma^k$, $k>1$ are much more complicated than $\Sigma^1$-type singularities. In particular, there is no finite list of possible local models for the generic $\Sigma^k$ singularity when $k>1$. The situation is in fact much worse: except in simple cases where the source and target manifolds have low dimension, the generic singularities of smooth maps have moduli. Furthermore, when the dimension is sufficiently high the number of moduli is infinite. Whence the desire to simplify these complicated singularities into singularities which are at least of type $\Sigma^1$ and ideally consisting only of $\Sigma^{10}$ folds. 

\begin{figure}[h]
\includegraphics[scale=0.7]{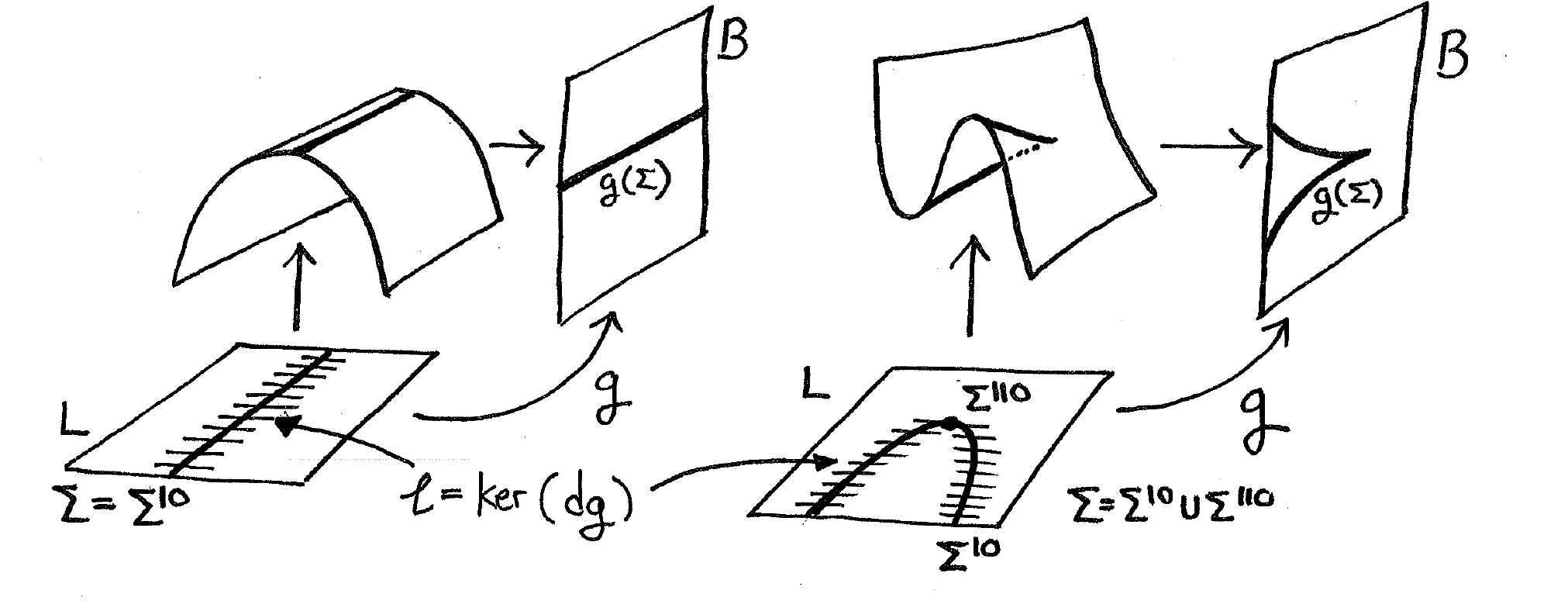}
\caption{The singularities $\Sigma^{10}$ and $\Sigma^{110}$.}
\label{foldingtypesingularities}
\end{figure}

We now return to the setting where $f:L \to M$ is a Lagrangian or Legendrian embedding into a symplectic or contact manifold $M$ equipped with an ambient foliation $\cF$ by Lagrangian or Legendrian leaves. We will assume that the foliation is given by the fibres of a Lagrangian or Legendrian fibration $\pi:M \to B$, indeed there is no harm in doing so since this is always the case locally. Hence the singularities of tangency $\Sigma(f;\cF)$ of $f$ with respect to $\cF$ are the same as the singularities $\Sigma(g)$ of the smooth mapping $g= \pi \circ f$, the Lagrangian or Legendrian front of $f$. Since the map $g$ is constrained by the condition of being a Lagrangian or Legendrian front, its $k$-jet extension $j^k(g)$ cannot always satisfy the transversality condition mentioned in the definition of the loci $\Sigma^I$. For example, the generic codimension of $\Sigma^k(f ; \cF)$ in $L$ is $k(k+1)/2$, which differs from the formula given above for the singularities of smooth maps. This point is better understood from the viewpoint of generating functions, which remove the Lagrangian or Legendrian condition in exchange of increasing the jet order by one. However, we will not pause to discuss this subtlety any further since transversality can be generically achieved at the level of fronts for the singularities that we will be  interested in: the $\Sigma^1$-type singularities. In particular, the generic codimension of $\Sigma^{1^k}(f; \cF)$ in $L$ is $k$, just like in the case of smooth mappings.

Figure \ref{standardfold} illustrates the $\Sigma^{10}$ fold and Figure \ref{standardcusp} illustrates the $\Sigma^{110}$ pleat, both in their Lagrangian and Legendrian realizations. Here and below we use the standard coordinates $(q,p) \in \bR^n \times \bR^n=T^*\bR^n$ and $(q,p,z) \in T^*\bR^n \times \bR = J^1(\bR^n, \bR)$, where the symplectic form on $T^*\bR^n$ is $dp \wedge dq$ and the contact form on $J^1(\bR^n,\bR)$ is $dz-pdq$.

\begin{example}\label{modelforfold} A Lagrangian or Legendrian front has the following unique local model in a neighborhood of any fold point $q \in \Sigma^{10}$.
\begin{itemize}
\item In the symplectic setting where the Lagrangian fibration is $\pi:T^*B \to B$, the front $\pi \circ f : L^n \to B^n$ is locally equivalent near the point $q$ to the map $(q_1 ,q_2  \ldots , q_n) \mapsto (q_1^2, q_2, \ldots ,q_n)$ near the origin. 
\item In the contact setting where the Legendrian fibration is $\pi:J^1(E, \bR) \to J^0(E,\bR)$, the front $\pi\circ f :L^n \to E^n \times \bR$ is locally equivalent near the point $q$ to the map $(q_1,  \ldots , q_n ) \mapsto (q_1^2, q_2, \ldots , q_n, q_1^3)$ near the origin. \end{itemize}
\end{example}

\begin{figure}[h]
\includegraphics[scale=0.55]{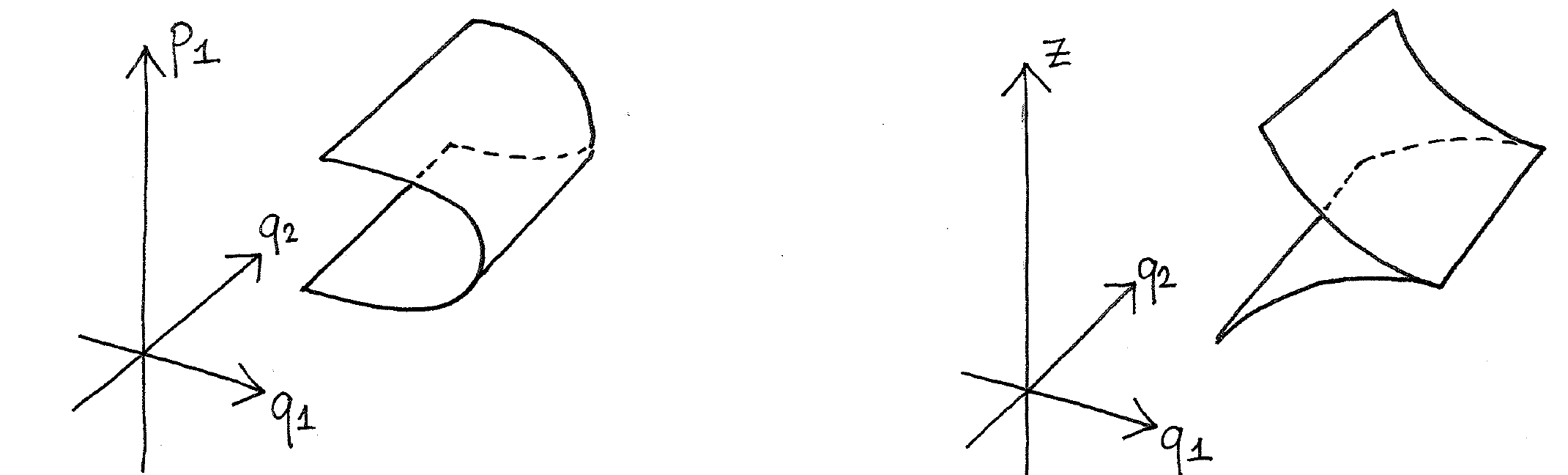}
\caption{The standard $\Sigma^{10}$ fold. The Lagrangian submanifold on the left corresponds to the Legendrian front on the right. The former is the trivial product of a parabola $q_1=p_1^2$ with $\bR^{n-1}$ and the latter is the trivial product of a semi-cubical cusp $q_1^3=z^2$ with $\bR^{n-1}$.}
\label{standardfold}
\end{figure}

\begin{figure}[h]
\includegraphics[scale=0.55]{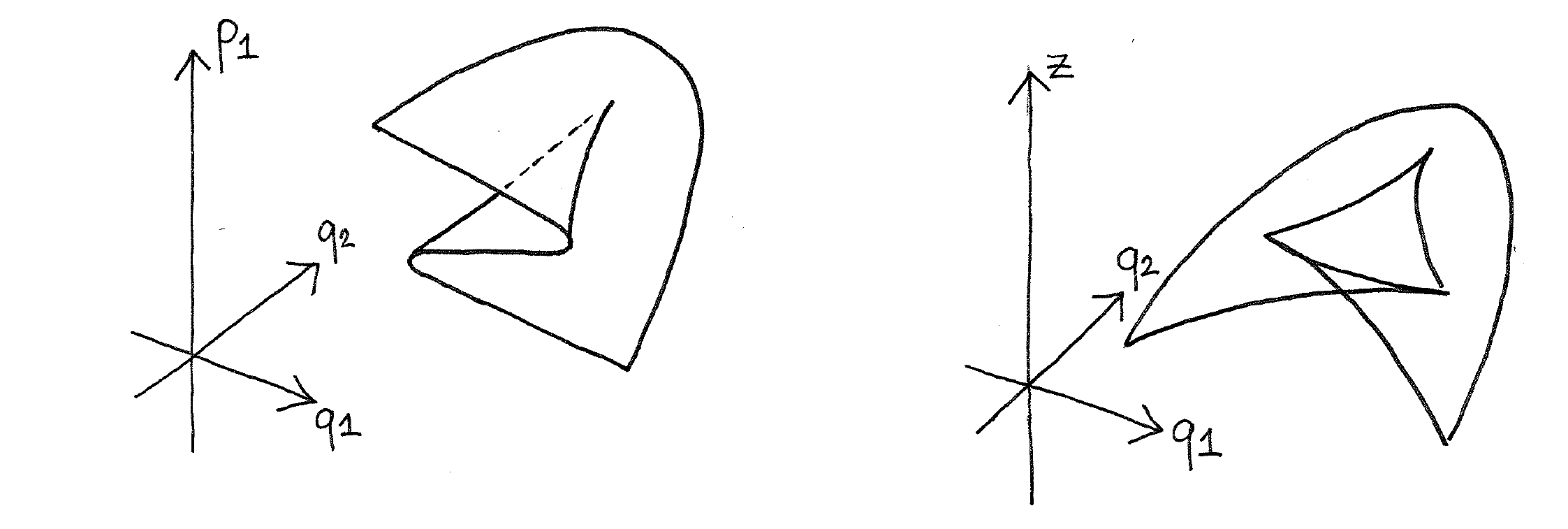}
\caption{The standard $\Sigma^{110}$ pleat. The Lagrangian submanifold on the left corresponds to the Legendrian front on the right. The former is the birth/death of two parabolas and the latter is the birth/death of two semi-cubical cusps.}
\label{standardcusp}
\end{figure}

\subsection{The double fold}

An example of a singularity locus which will be particularly relevant to our discussion is the so-called double fold, which we now describe. For an illustration, see Figure \ref{doublefold} below. Before we give the definition, observe that near a fold point $q \in \Sigma^{10}$, the Lagrangian or Legendrian submanifold $f(L) \subset M$ could be turning in one of two possible directions with respect to $\cF$. This direction can be specified by a co-orientation of the $(n-1)$-dimensional submanifold $\Sigma^{1}$ inside $L$, which is called the Maslov co-orientation and was implicitly introduced in \cite{A67}. Informally, we can view $df(T_qL)$ as a quadratic form over $T_{f(q)}\cF$ whose signature changes by one as $q$ crosses $\Sigma^{10}$ transversely. The Maslov co-orientation specifies the direction in which the signature is increasing. This is the same Maslov co-orientation which appears in Entov's work \cite{E97}. 

\begin{definition} A double fold is a pair of topologically trivial $(n-1)$-spheres $S_1$ and $S_2$ in the fold locus $\Sigma^{10}$ which have opposite Maslov co-orientations and such that $S_1 \cup S_2$ is the boundary of an embedded annulus $A \subset L$.
\end{definition} 

\begin{figure}[h]
\includegraphics[scale=0.5]{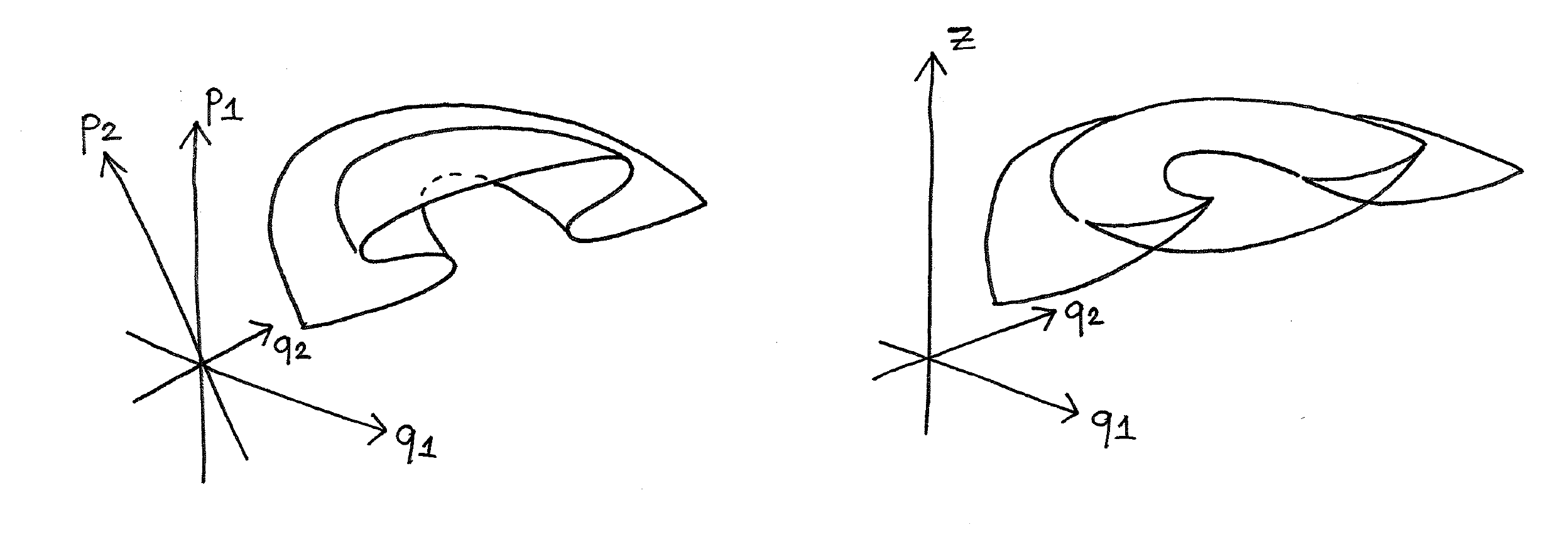}
\caption{One half of a double fold. The Lagrangian submanifold on the left corresponds to the Legendrian front on the right.}
\label{doublefold}
\end{figure}

By a topologically trivial sphere we mean a sphere which bounds an embedded $n$-ball in $L$. We say that a pair of double folds $F=S_1\cup S_2$ and $\widetilde{F}=\widetilde{S}_1 \cup \widetilde{S}_2$ bounding annuli $A$ and $\widetilde{A}$ in $L$ are nested if one annulus is contained inside the other, say $A \subset \widetilde{A}$, and furthermore $A$ bounds an $n$-ball $B \subset L$ which is completely contained in $\widetilde{A}$. See Figure \ref{nesteddoublefold} for an illustration. 

\begin{figure}[h]
\includegraphics[scale=0.5]{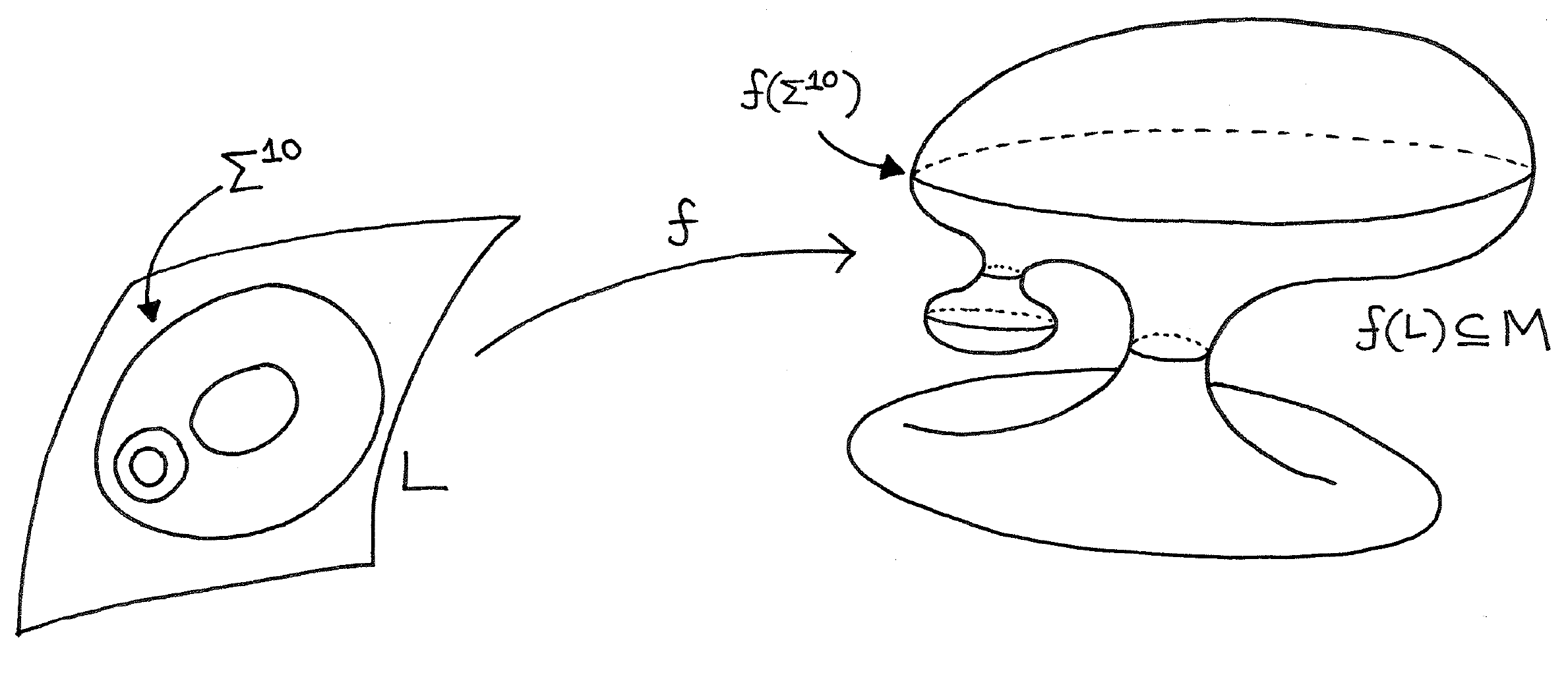}
\caption{A nested double fold.}
\label{nesteddoublefold}
\end{figure}

\subsection{Tangential rotations}\label{Tangential homotopies}
The Lagrangian Grassmannian of a symplectic manifold $(M^{2n}, \omega)$ is a fibre bundle $\Pi:\Lambda_n(M) \to M$ whose fibre $\Pi^{-1}(x)$ over a point $x \in M$ consists of all linear Lagrangian subspaces of the symplectic vector space $(T_xM, \omega_x)$. To each Lagrangian embedding $f:L \to M$ we can associate its Gauss map $G(df):L \to \Lambda_n(M)$, given by $G(df)(q)=  df(T_qL) \subset T_{f(q)}M$. Observe that $\Pi \circ G(df)=f$, in other words, $G(df)$ covers $f$. 

Similarly, given a contact manifold $(M^{2n+1}, \xi)$, where locally $\xi=\ker(\alpha)$ for some $1$-form $\alpha$ for which $d \alpha$ is non-degenerate on $\xi$, the Lagrangian Grassmannian is a fibre bundle $\Pi: \Lambda_n(M) \to M$ whose fibre $\Pi^{-1}(x)$ over a point $x \in M$ consists of all linear Lagrangian subspaces of the symplectic vector space $(\xi_x, d \alpha_x)$. To each Legendrian embedding $f:L \to M$ we associate its Gauss map $G(df):L \to \Lambda_n(M)$, given as before by $G(df)(q)=df(T_qL) \subset \xi_{f(q)}$. 

The formal analogue of the Gauss map is obtained by decoupling a Lagrangian or Legendrian embedding from its tangential information. 

\begin{definition}
A tangential rotation of a Lagrangian or Legendrian embedding $f:L \to M$ is a compactly supported deformation $G_t:L \to \Lambda_n(M)$, $t \in [0,1]$, of $G_0=G(df)$ such that $\Pi \circ G_t=f$.
\end{definition}

\begin{example}
In the previous section we introduced the double fold as an example of a singularity locus. Observe that any double fold is homotopically trivial in the following sense. If $f$ has a double fold on the annulus $A \subset L$, then we can always construct a tangential rotation $G_t$ of $f$ supported in a neighborhood of $A$ such that at time $t=1$ we have $G_1 \pitchfork \cF$ in that same neighborhood. In other words, there is no formal obstruction to removing a double fold. 
\end{example}
%




The formal analogue of the condition $\Sigma^k(f ; \cF)=\varnothing$ is the following. 
\begin{definition}  A map $G:L \to \Lambda_n(M)$ is called $\Sigma^k$-nonsingular with respect to the foliation $\cF$ if $\dim( G(q) \cap T_{g(q)} \cF )<k$ for all $q \in L$. When $k=1$ we simply say that $G$ is nonsingular, or transverse to $\cF$, and write $G \pitchfork \cF$.  

\end{definition}

Accordingly, we say that a Lagrangian or Legendrian embedding $f$ is $\Sigma^k$-nonsingular when $G(df)$ is $\Sigma^k$-nonsingular. It is easy to see that a necessary condition for $f$ to be Hamiltonian isotopic to a $\Sigma^k$-nonsingular embedding is the existence of a tangential rotation $G_t$ such that $G_1$ is $\Sigma^k$-nonsingular. Indeed, if we denote the Hamiltonian isotopy by $\varphi_t$ and we choose a family of symplectic bundle isomorphisms $\Phi_t : TM|_{f(L)} \to TM|_{ \varphi_t \circ f(L)}$ such that $\Phi_0=id$ and such that $\Phi_t\big(T \cF|_{f(L)}\big) =T\cF|_{ \varphi_t \circ f(L)}$, then we can set $G_t= \Phi_t^{-1} \cdot G\big( d ( \varphi_t \circ f) \big)$. The family $\Phi_t$ exists by the homotopy lifting property of a Serre fibration. Note that in the contact case we must replace the symplectic bundle $(TM, \omega)$ by the symplectic bundle $(\xi, d \alpha)$, but the argument is the same.

The results we state in the next section assert that this necessary condition is also sufficient when $k=2$ and is almost sufficient when $k=1$. The `almost' part comes from the necessity of double folds and will be discussed below.

\subsection{Main results}\label{Main results}
We are now ready to state the $h$-principle. Recall that $M$ is a symplectic or contact manifold equipped with a foliation $\cF$ by Lagrangian or Legendrian leaves. By the singularities of a Lagrangian or Legendrian embedding we mean its singularities of tangency with respect to $\cF$.

\newpage 

\begin{theorem}\label{main result} Suppose that there exists a tangential rotation $G_t:L \to \Lambda_n(M)$ of a Lagrangian or Legendrian embedding $f:L \to M$ such that $G_1 \pitchfork \cF$. Then there exists a compactly supported Hamiltonian isotopy $\varphi_t:M \to M$ such that the singularities of $\varphi_1 \circ f$ consist of a union of nested double folds. 
\end{theorem}

\begin{remark}
In particular, $\varphi_1\circ f$ is $\Sigma^2$-nonsingular. Indeed all of its singularities are of the simplest possible type, namely the $\Sigma^{10}$ fold.
\end{remark}

Theorem \ref{main result} is a full $h$-principle in the sense of \cite{EM02}. More precisely, the following $C^0$-close, relative and parametric versions of the statement hold.

\begin{itemize}
\item[($C^0$-close)] We can choose the Hamiltonian isotopy $\varphi_t$ to be arbitrarily $C^0$-close to the identity. Moreover, we can arrange it so that $\varphi_t=id_M$ outside of an arbitrarily small neighborhood of $f(L)$ in $M$.
\item[(relative)] Suppose that $G_t=G(df)$ on $Op(A) \subset L$ for some closed subset $A \subset L$, where here and below we use Gromov's notation $Op(A)$ for an arbitrarily small but unspecified neighborhood of $A$. Then we can arrange it so that $\varphi_t=id_M$ on $Op\big(f(A)\big)\subset M$. 
\item[(parametric)] An analogous result holds for families of Lagrangian or Legendrian embeddings parametrized by a compact manifold of any dimension. The statement also holds relative to a closed subset of the parameter space. For example, it holds for the pair $(D^n,S^{n-1})$ formed by the unit disk and its boundary sphere. For details see Section \ref{applications to the simplification of singularities}. 

\end{itemize}

For singularities of type $\Sigma^2$ we have the following $h$-principle, in which we don't have to worry about the presence of double folds since they are singularities of type $\Sigma^1$.

\begin{theorem}\label{main result 2} Suppose that there exists a tangential rotation $G_t:L \to \Lambda_n(M)$ of a Lagrangian or Legendrian embedding $f:L \to M$ such that $G_1$ is $\Sigma^2$-nonsingular with respect to the foliation $\cF$. Then there exists a compactly supported Hamiltonian isotopy $\varphi_t:M \to M$ such that $\varphi_1 \circ f$ is $\Sigma^2$-nonsingular.
\end{theorem}

In fact, we prove a much stronger version of Theorem \ref{main result 2} which allows for the prescription of any homotopically allowable $\Sigma^1$-type singularity locus. The precise statement is given in Theorem \ref{Entov theorem redux2} below, after we discuss Entov's results on the surgery of Lagrangian and Legendrian singularities.

\subsection{The homotopical obstruction}
Consider the subset $\Sigma(M,\cF) \subset \Lambda_n(M)$ which over each point $x \in M$ consists of all planes $P_x \in \Lambda_n(M)_x$ such that $P_x \cap T_x \cF \neq 0 $. We have a stratification $\Sigma(M, \cF) = \bigcup_k \Sigma^k(M, \cF)$, where $\Sigma^k(M, \cF) = \{ P_x : \,\, \dim(P_x \cap T_x \cF) = k \}$. The formal obstruction to $\Sigma^k$-nonsingularity can be understood as follows: \emph{is it possible to smoothly homotope the map $G(df):L \to \Lambda_n(M)$ through maps $G_t$ covering $f$ so that its image becomes completely disjoint from the subset $\Sigma^k(M,\cF) \subset \Lambda_n(M)$? }  This is a purely topological question.
 
The most obvious cohomological obstruction is given by the higher Maslov classes. To define them, observe that $\Sigma^k(M,\cF)=\{ P_x \in \Lambda_n(M)_x : \, \, \dim(P_x \cap T_x \cF )= k\}$ is a stratified subset of codimension $k(k+1)/2$ inside the Grassmannian $\Lambda_n(M)$, whose boundary $\partial \Sigma^k(M, \cF) = \bigcup_{l>k} \Sigma^l (M, \cF)$ has dimension strictly less than $\dim\big(\Sigma^k(M, \cF)\big) -1$. We can therefore define $\mu_k = G(df)^*m_k \in H^{k(k+1)/2}(L  ; \bZ/2)$, where $m_k \in H^{k(k+1)/2}\big(\Lambda_n(M) ; \bZ/2 \big)$ is Poincar\'e dual to the cycle $\big[\Sigma^k(\cF) \big]$. The class $\mu_k$ is an obstruction to removing the singularity $\Sigma^k$. By an argument which the author learnt from Givental \cite{G17}, the classes $\mu_k$ are defined over $\bZ$ when $k$ is odd, whereas for $k$ even the cycle $\Sigma^k(M, \cF)$ is not co-orientable and therefore $\mu_k$ is only defined over $\bZ/2$. For example, when $k=1$ the integral lift of $\mu_1$ is the familiar Maslov class.
 
More generally, to each multi-index $I = (i_1 \geq i_2 \geq \cdots \geq i_k)$ there exists a cohomology class $\mu_I$ which obstructs the removal of $\Sigma^I$ and which is the pullback of a universal class in the appropriate jet space. In addition to these cohomological obstructions there exist subtler homotopical obstructions to the simplification of singularities.

In certain situations the obstruction to the simplification of singularities can be straightforwardly seen to vanish. In Section \ref{applications to the simplification of singularities} we explore a couple of such cases and are thus able to deduce concrete applications of our $h$-principle. However, in general this homotopical problem can be nontrivial. For instance, consider the setup of the nearby Lagrangian conjecture, so that $f:L \to T^*B$ is an exact Lagrangian embedding of a connected closed manifold $L$ into the cotangent bundle of a connected closed manifold $B$.  Abouzaid and Kragh showed in \cite{K11} that the first Maslov class $\mu_1$ always vanishes. However, to the extent of the author's knowledge it is not known whether the higher Maslov classes $\mu_k$ must also vanish. 

\subsection{Idea of the proof}
The proof of our main result Theorem \ref{main result} is an adaptation to the symplectic and contact setting of the strategy employed in Eliashberg and Mishachev's wrinkled embeddings paper \cite{EM09}. Wrinkled embeddings are topological embeddings of smooth manifolds which are smooth embeddings away from a finite union of spheres of codimension $1$, called wrinkles, where the mapping has cusps (together with their birth/deaths on the equator of each sphere). The rank of the differential falls by one on the wrinkling locus, hence the map fails to be a smooth embedding near the wrinkles. However, there is a well-defined tangent plane at every point of the image and so wrinkled embeddings have Gauss maps just like smooth embeddings. In this paper we define  wrinkled Lagrangian and Legendrian embeddings to be wrinkled embeddings $f$ into a symplectic or contact manifold $M$ whose Gauss map $G(df)$ lands in the Lagrangian Grassmannian.

The advatange of working with wrinkled Lagrangian and Legendrian embeddings over regular Lagrangian and Legendrian embeddings is the following. Any tangential rotation $G_t$ of a regular Lagrangian or Legendrian embedding $f$ can be $C^0$-approximated by the Gauss maps $G(df_t)$ of a homotopy $f_t$ of wrinkled Lagrangian or Legendrian embeddings. Such a statement is false if we demand that the homotopy $f_t$ consist only of regular Lagrangian or Legendrian embeddings. This additional flexibility provided by the wrinkles implies in particular the following result: if there exists a tangential rotation $G_t$ of a regular Lagrangian or Legendrian embedding $f$ such that $G_1$ is transverse to an ambient foliation $\cF$, then there exists a homotopy of wrinkled Lagrangian or Legendrian embeddings $f_t$ such that $f_1$ is transverse to $\cF$. We can then regularize the wrinkles of $f_t$ and obtain a homotopy of regular Lagrangian or Legendrian embeddings $\widetilde{f}_t$. The embedding $\widetilde{f}_1$ is no longer transverse to $\cF$, we must of course pay a price when we pass from $f_1$ to $\widetilde{f}_1$. The regularization process causes $\Sigma^{10}$ folds to appear where the embedding used to be wrinkled, with $\Sigma^{110}$ pleats on the equator of each wrinkle. We can then use the surgery of singularities to get rid of the $\Sigma^{110}$ pleats. The result of the surgery is a union of double folds, as in the conclusion of our $h$-principle. See Figure \ref{ideaofproof} for an illustration of the strategy.

\begin{figure}[h]
\includegraphics[scale=0.6]{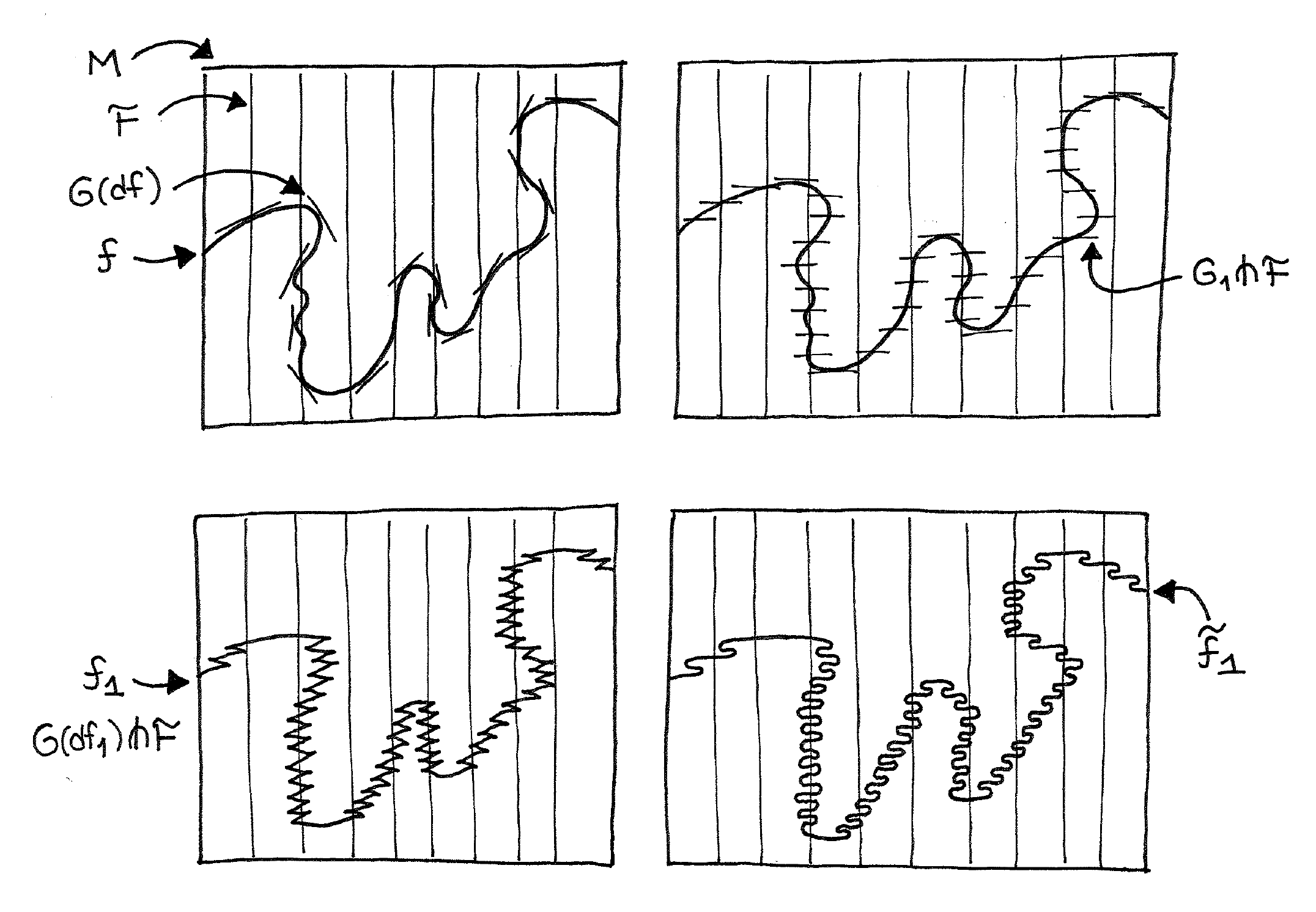}
\caption{The strategy of the proof.}
\label{ideaofproof}
\end{figure}

To prove the $C^0$-approximation result for wrinkled Lagrangian and Legendrian embeddings described in the previous paragraph, two main ingredients are necessary. The first is a technical refinement of the holonomic approximation lemma \cite{EM01} of Eliashberg and Mishachev. We established this refinement in our paper \cite{AG15}. The second ingredient is the observation that Lagrangian and Legendrian embeddings can be locally wrinkled. This observation occupies the bulk of the present paper.

\subsection{Surgery of singularities}\label{Surgery of singularities}
In his thesis \cite{E72}, Eliashberg developed a technique to modify the singularity locus of a $\Sigma^2$-nonsingular map between smooth manifolds by means of a surgery construction, see Figure \ref{surgeryexample} for an example. This technique yields an $h$-principle for the simplification of singularities of $\Sigma^2$-nonsingular smooth maps. Almost thirty years later, Entov adapted this surgery technique to the setting of Lagrangian and Legendrian fronts, also in his thesis \cite{E97}. The main point in Entov's construction is to write down the generating functions that produce Eliashberg's surgery. As a consequence of Entov's results, one obtains an $h$-principle for the simplification of singularities of $\Sigma^2$-nonsingular Lagrangian or Legendrian fronts, which we now briefly discuss.

\begin{figure}[h]
\includegraphics[scale=0.5]{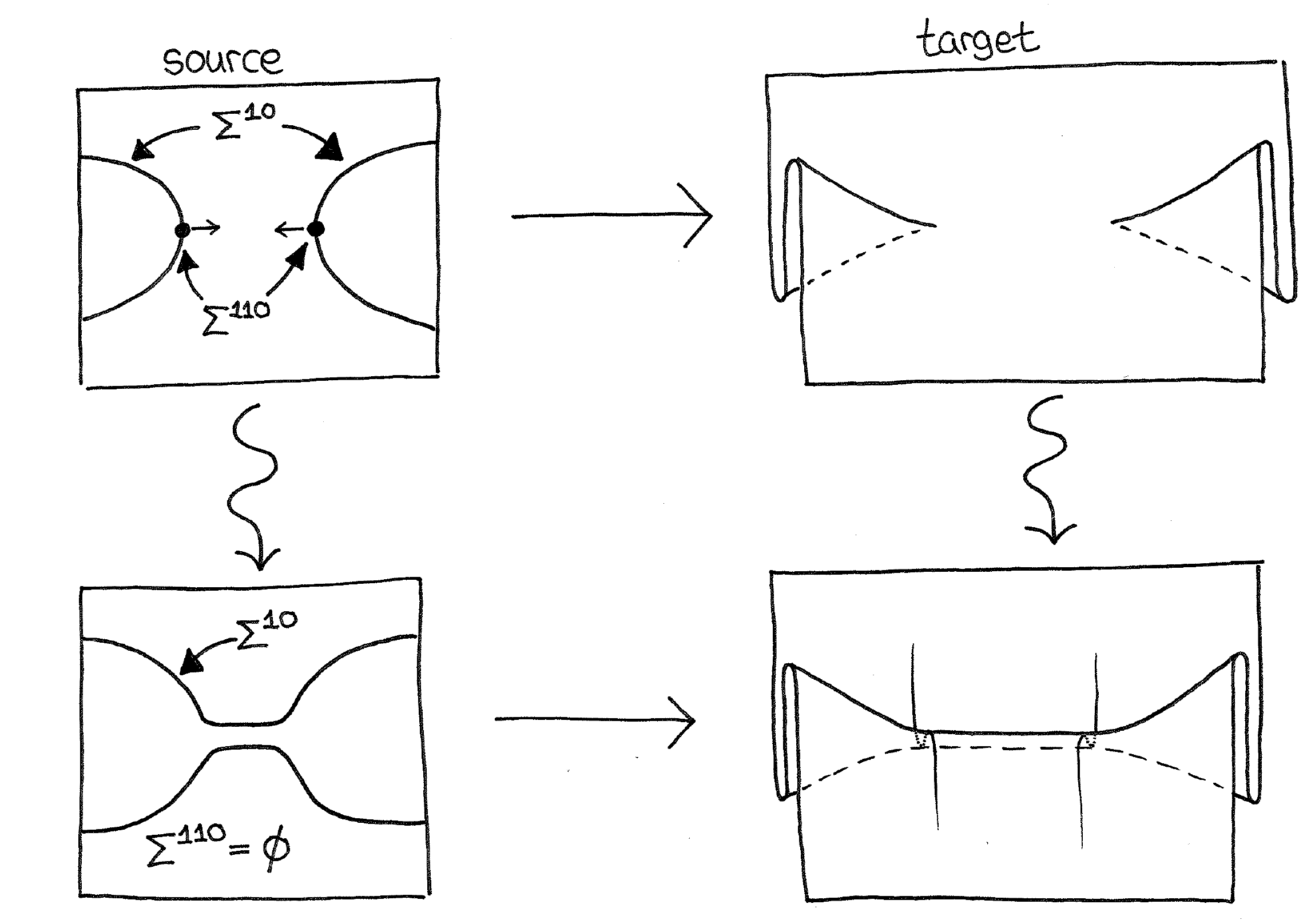}
\caption{An example of the surgery of singularities.}
\label{surgeryexample}
\end{figure}

Suppose that $f:L \to M$ is a $\Sigma^2$-nonsingular Lagrangian or Legendrian embedding into a symplectic or contact manifold $M$ equipped with a foliation $\cF$ by Lagrangian or Legendrian leaves. We recall that $\Sigma^2$-nonsingularity means that $\dim(df(T_qL) \cap T_{f(q)} \cF) <2$ for all $q\in L$, hence $\Sigma^2= \varnothing$. The Thom-Boardman stratification of the singularity locus $\Sigma= \Sigma^1$ therefore consists of a flag of submanifolds $\Sigma^1 \supset \Sigma^{11} \supset \cdots \supset \Sigma^{1^n}$, where $\text{dim}(\Sigma^{1^k})=n-k$. This flag, together with certain co-orientation data which we won't be precise about right now, is called the chain of singularities associated to the embedding $f$ and the foliation $\cF$. More generally, given any Lagrangian distribution $D$ defined along $f(L)$ (not necessarily tangent to an ambient foliation), we say that $D$ is $\Sigma^2$-nonsingular if $\dim(df(T_qL) \cap D_{f(q)} ) <2$ for all $q \in L$. For such Lagrangian distributions $D$ we can similarly define an associated chain of singularities consisting of a flag $\Sigma^1 \supset \Sigma^{11} \supset \cdots \supset \Sigma^{1^n}$ together with certain co-orientation data.

We say that two chains of singularities are equivalent if the flags of submanifolds are isotopic in $L$, with the corresponding co-orientation data also matching up under the isotopy. Entov's main result can be phrased as follows.

\begin{theorem}[Entov]\label{Entov theorem}Let $f:L \to M$ be a $\Sigma^2$-nonsingular Lagrangian or Legendrian embedding into a symplectic or contact manifold $M$ equipped with a foliation $\cF$ by Lagrangian or Legendrian leaves. Let $D_t$ be a homotopy of $\Sigma^2$-nonsingular Lagrangian distributions defined along $f(L)$, fixed outside of a compact subset and such that $D_0=T\cF|_{f(L)}$. We moreover assume that $f \pitchfork \cF$ outside of that compact subset. Then there exists a $C^0$-small compactly supported Hamiltonian isotopy $\varphi_t:M \to M$ such that the chain of singularities of $\varphi_1 \circ f$ with respect to $\cF$ is equivalent to the chain of singularities of $f$ with respect to $D_1$, together with a union of nested double folds.
\end{theorem}

Suppose that $G(df) \pitchfork D_1$. Then the chain of singularities associated to $f$ and $D_1$ is empty and the conclusion of Entov's theorem is the same as the one in our $h$-principle Theorem \ref{main result}. It is no coincidence that both Entov's result and Theorem \ref{main result} only work up to a union of double folds. Although homotopically trivial, one cannot hope to get rid of these double folds in general. The rigidity of Lagrangian and Legendrian folds was first explored by Entov in \cite{E98} and by Ferrand and Pushkar in \cite{FP98} and \cite{FP06}. We note that for singularities of smooth maps as considered by Eliashberg in \cite{E70} and \cite{E72} the situation is slightly better: one can always absorb these double folds into an already existing fold locus with the only condition that this locus is nonempty.

The main limitation of the surgery technique is that it requires $\Sigma^2$-nonsingularity of the initial embedding to even get started. A generic Lagrangian or Legendrian embedding is $\Sigma^2$-nonsingular only when the Lagrangian or Legendrian has dimension $ \leq 2$. This restricts significantly the possible applications of the surgery $h$-principle beyond the case of Lagrangian or Legendrian surfaces. Even in the $2$-dimensional case, $\Sigma^2$-type singularities will generically arise in $1$-parametric families, preventing a satisfactory parametric result from being formulated. 

This limitation is not serious in the smooth version of the problem because one can easily get rid of $\Sigma^2$-type singularities by using a different technique, for example one can use Gromov's convex integration \cite{G86}. Unfortunately, these techniques seem to be inadequate to get rid of the $\Sigma^2$-type singularities of Lagrangian and Legendrian fronts. This paper bypasses this issue by using a different strategy, namely the wrinkling philosophy. Indeed, we will prove in Section \ref{prescription of singularities} the following version of Entov's Theorem \ref{Entov theorem} in which we drop the condition of $\Sigma^2$-nonsingularity.

\begin{theorem}\label{Entov theorem redux2} Let $f:L \to M$ be a Lagrangian or Legendrian embedding into a symplectic or contact manifold $M$ equipped with a foliation $\cF$ by Lagrangian or Legendrian leaves. Let $D_t$ be a homotopy of Lagrangian distributions defined along $f(L)$, fixed outside of a compact subset, such that $D_0=T\cF|_{f(L)}$ and such that $D_1$ is $\Sigma^2$-nonsingular. We moreover assume that $f \pitchfork \cF$ outside of that compact subset. Then there exists a $C^0$-small compactly supported Hamiltonian isotopy $\varphi_t:M \to M$ such that $\varphi_1 \circ f$ is $\Sigma^2$-nonsingular with respect to $\cF$ and moreover such that the chain of singularities of $\varphi_1 \circ f$ with respect to $\cF$ is equivalent to the chain of singularities of $f$ with respect to $D_1$, together with a union of nested double folds.
\end{theorem}

\begin{remark}
Theorem \ref{main result 2}, the $h$-principle for $\Sigma^2$-nonsingular embeddings, is an immediate consequence of Theorem \ref{Entov theorem redux2}.
\end{remark}

\subsection{The wrinkling philosophy}
Many $h$-principles can be proved by interpolating between local Taylor approximations. To achieve this interpolation near a subset of positive codimension, one can use the extra dimension to wiggle the subset in and out, creating extra room. This room ensures that no big derivatives arise when interpolating from one Taylor polynomial to another. This idea has been present throughout the history of the $h$-principle starting with the immersion theory of Smale-Hirsch-Phillips \cite{H59}, \cite{P67}, \cite{S59} and Gromov's method of flexible sheaves \cite{G69}, \cite{G86}. The wiggling strategy was further reformulated into a simple but general statement by Eliashberg and Mishachev in \cite{EM01}, \cite{EM02} with their holonomic approximation lemma. 

 In many cases, however, one wishes to prove a global $h$-principle on the whole manifold (which might be closed) and there is no extra dimension available for wiggling. The wrinkling philosophy provides a strategy for proving $h$-principles in such cases. The idea is to wrinkle the manifold back and forth upon itself. One can then interpolate between local Taylor approximations along the wrinkles. The wrinkling process creates the extra room needed so that this interpolation does not create big derivatives. One pays an unavoidable price, namely the singularities caused by the wrinkles. However, these are very simple singularities which can be explicitly understood.
 
\begin{figure}[h]
\includegraphics[scale=0.6]{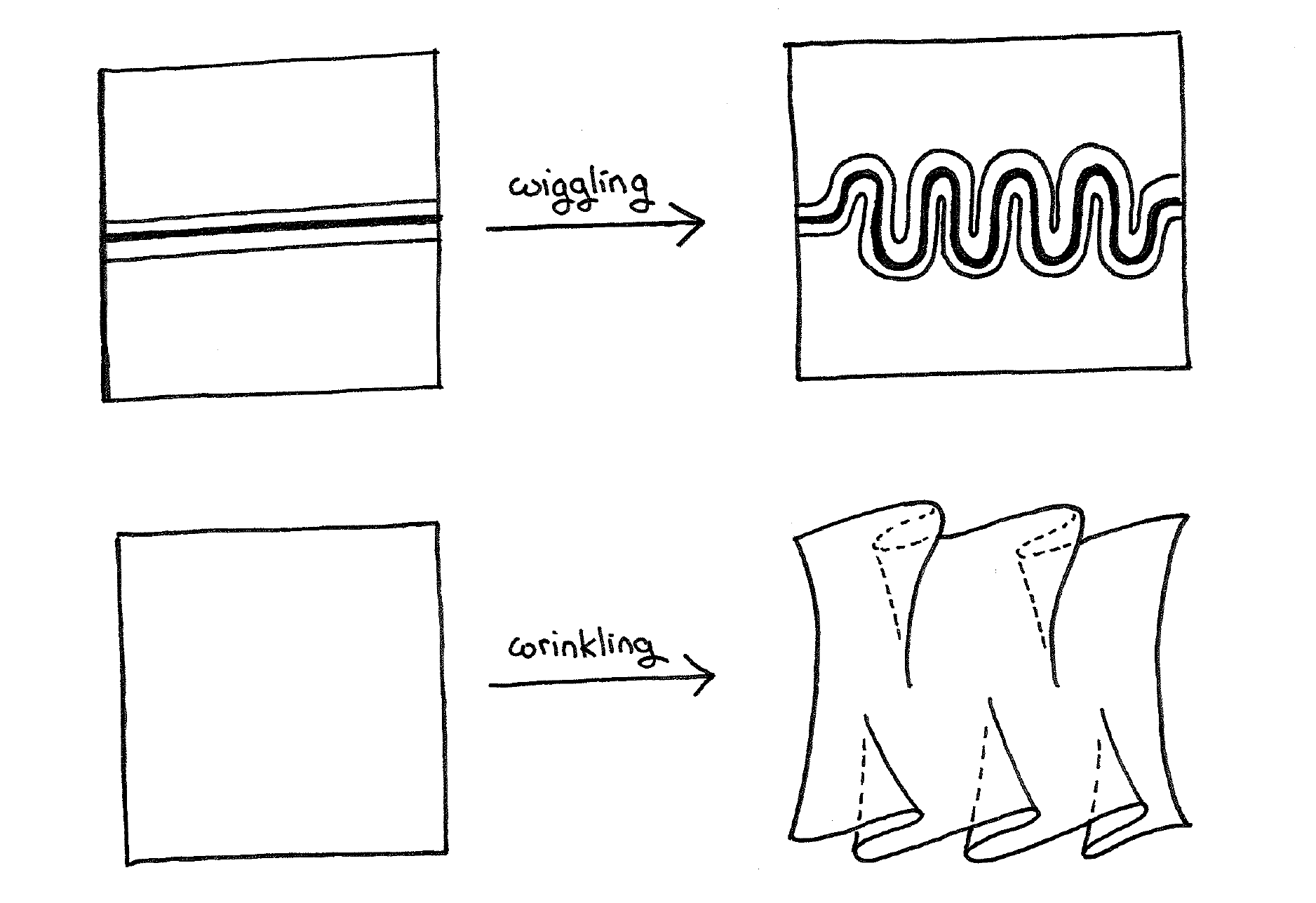}
\caption{The difference between wiggling and wrinkling.}
\label{ideaofproof}
\end{figure}
 
In their papers \cite{EM97}, \cite{EM98}, \cite{EM00}, \cite{EM09}, \cite{EM12}, \cite{EM12(2)}, Eliashberg and Mishachev exploit this wrinkling strategy to prove a number of results in flexible geometric topology. Together with Galatius, they give a further application in \cite{EGM11}. The theorem on wrinkled embeddings from \cite{EM09}, which is particularly relevant for our purposes, has gained greater significance after it was used by Murphy in \cite{M12} to establish the existence of loose Legendrians in high-dimensional contact manifolds. Our paper provides a different application of the wrinkled embeddings theorem to flexible symplectic and contact topology. 

\begin{warning} At this point we should alert the reader that Murphy's wrinkled Legendrians have nothing in common with our wrinkled Lagrangian and Legendrian embeddings. The two notions should not be confused, despite the terribly similar terminology for which the author can only apologize and excuse himself in the desire to be consistent with the existing literature \cite{EM09}.

To be clear: in Murphy's wrinkled Legendrians, the wrinkles occur in the Legendrian front. In the wrinkled Legendrian embeddings under consideration in this paper, the wrinkles occur in the Legendrian submanifold itself.
\end{warning}

\subsection{Outline of the paper}
The plan of the paper is the following. In Section \ref{wrinkled lagrangian and legendrian embeddings} we introduce Lagrangian and Legendrian wrinkles together with all related definitions. In Section \ref{Rotations of Lagrangian and Legendrian planes} we show that we can restrict our attention to a particularly simple class of tangential rotations. In Section \ref{Holonomic approximation with controlled cutoff} we wiggle Lagrangian and Legendrian embeddings using our holonomic approximation lemma for $\perp$-holonomic sections. In Section \ref{Wrinkling embeddings} we wrinkle Lagrangian and Legendrian embeddings using a concrete local model. Finally, in Section \ref{applications to the simplification of singularities} we prove the $h$-principle for the simplification of singularities and give applications.

\subsection{Acknowledgements}
I am very grateful to my advisor Yasha Eliashberg for insightful guidance throughout this project. I would also like to thank Laura Starkston for reading carefully the first draft of this paper and offering numerous remarks and corrections which have greatly improved the exposition. I am indebted to the ANR Microlocal group who held a workshop in January 2017 to dissect an early version of the paper and in particular to Sylvain Courte and Alexandre V\'erine who spotted several mistakes in the proof of the local wrinkling lemma and made useful suggestions for fixing them.  Finally, I would like to express my gratitude to Roger Cassals, Sander Kupers, Emmy Murphy, Oleg Lazarev and Kyler Siegel for many helpful discussions surrounding the general notion of flexibility. 

\section{Lagrangian and Legendrian wrinkles}\label{wrinkled lagrangian and legendrian embeddings}

\subsection{Wrinkled embeddings}\label{Wrinkled embeddings} 
We start by recalling the definition of wrinkled embeddings, from \cite{EM09}. Throughout we denote a point $q \in \bR^n$ by $q=(\hat{q}, q_n)$, where $\hat{q}=(q_1, \ldots , q_{n-1})$.

\begin{definition}
A wrinkled embedding is a topological embedding ${f:L^n\rightarrow X^{n+r}}$ which is a smooth embedding away from a disjoint union of finitely many topologically trivial embedded $(n-1)$-spheres $S \subset L$, with $f$ equivalent (up to diffeomorphism) on $Op(S)$ to the local model $\cW_{n,r}: Op(S^{n-1}) \subset \bR^n \rightarrow \bR^{n+r}$ given by
\[ (q_1, \ldots , q_n) \mapsto \big( q_1, \ldots , q_{n-1} ,\eta, 0, \ldots , 0, h \big),\] 
\[ \text{where}\, \, \, \,   \eta(q)= q_n^3 + 3(|| \hat{q} ||^2-1)q_n  \quad \text{and} \quad h(q)= \int_0^{q_n}(|| \hat{q} ||^2+ u^2 -1)^2 du.\]
\end{definition}

\begin{figure}[h]
\includegraphics[scale=0.6]{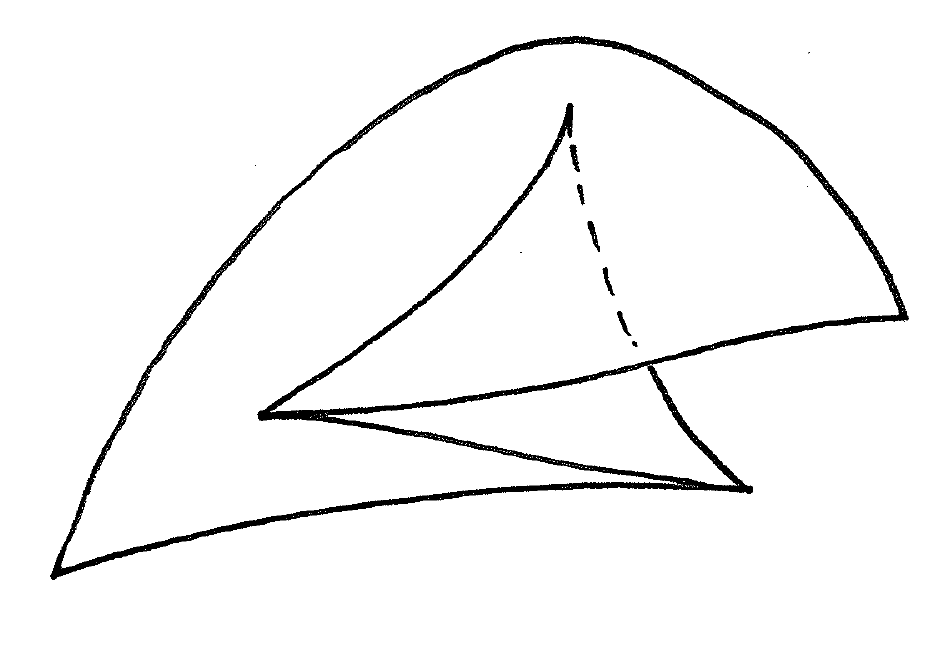}
\caption{One half of a standard wrinkle}
\label{wrinkelagain}
\end{figure}

The mapping $\cW_{n, r}$ has singularities along $S^{n-1}$. On the upper and lower hemispheres $S^{n-1} \cap \{ q_n >0\}$ and $S^{n-1} \cap \{ q_n < 0 \} $, the singularities are semi-cubical cusps. More precisely, near each point of $S^{n-1} \setminus S^{n-2}$,  the model $\cW_{n,r}$ is locally equivalent to the following map near the origin, see Figure \ref{cuspmodel}.
$$ (q_1 , \ldots , q_n ) \mapsto \big(q_1, \ldots , q_{n-1} , q_n^2 , 0 ,\ldots , 0, q_n^3\big). $$

On the equator $S^{n-2}=S^{n-1} \cap \{ q_n=0 \}$, the singularities are the birth/death of semi-cubical zig-zags. More precisely, near each point of $S^{n-2}$, the model $\cW_{n,r}$ is locally equivalent to the following map near the origin, see Figure \ref{birthmodel}.
$$ (q_1 , \ldots , q_n) \mapsto \Big(q_1 , \ldots , q_{n-1} , q_n^3 - 3q_1q_n, 0 , \ldots , 0 , \int_0^{q_n} (u^2-q_1)^2 du \Big). $$

\begin{figure}[h]
\includegraphics[scale=0.6]{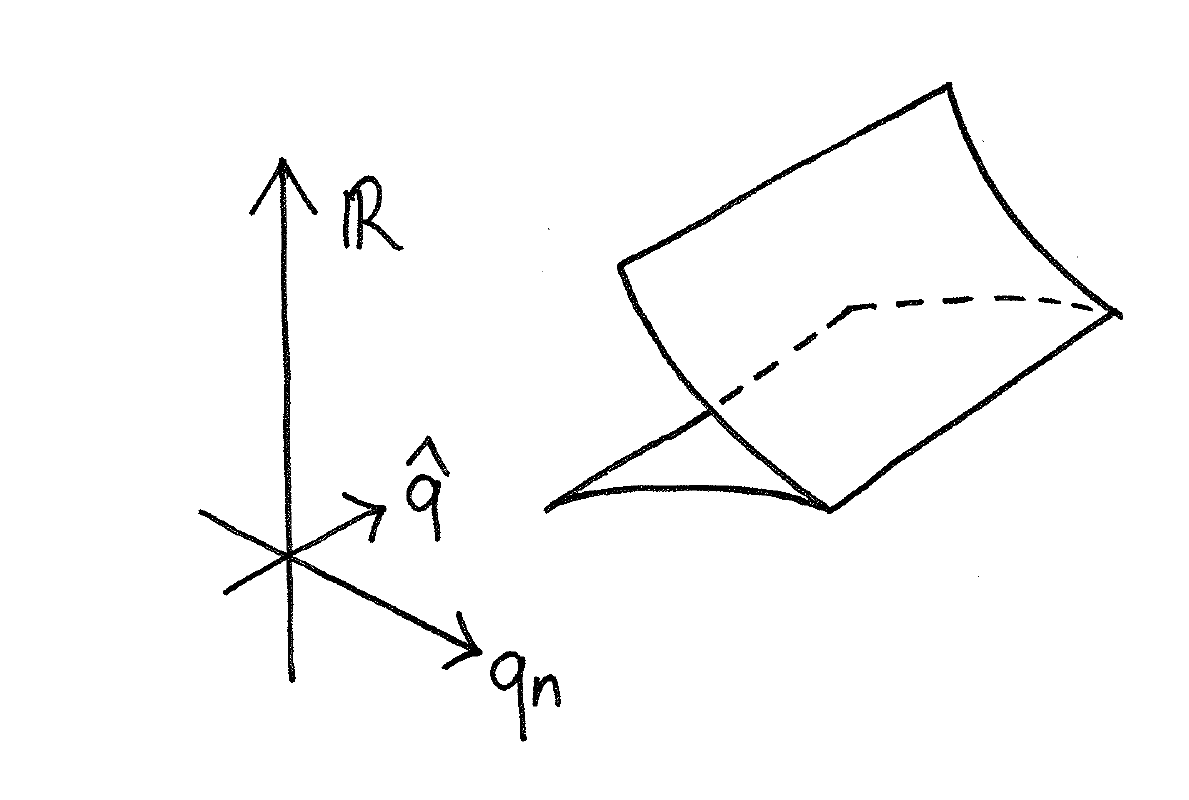}
\caption{A wrinkled embedding has cusps on the complement of the equator of each wrinkle.}
\label{cuspmodel}
\end{figure}

\begin{figure}[h]
\includegraphics[scale=0.6]{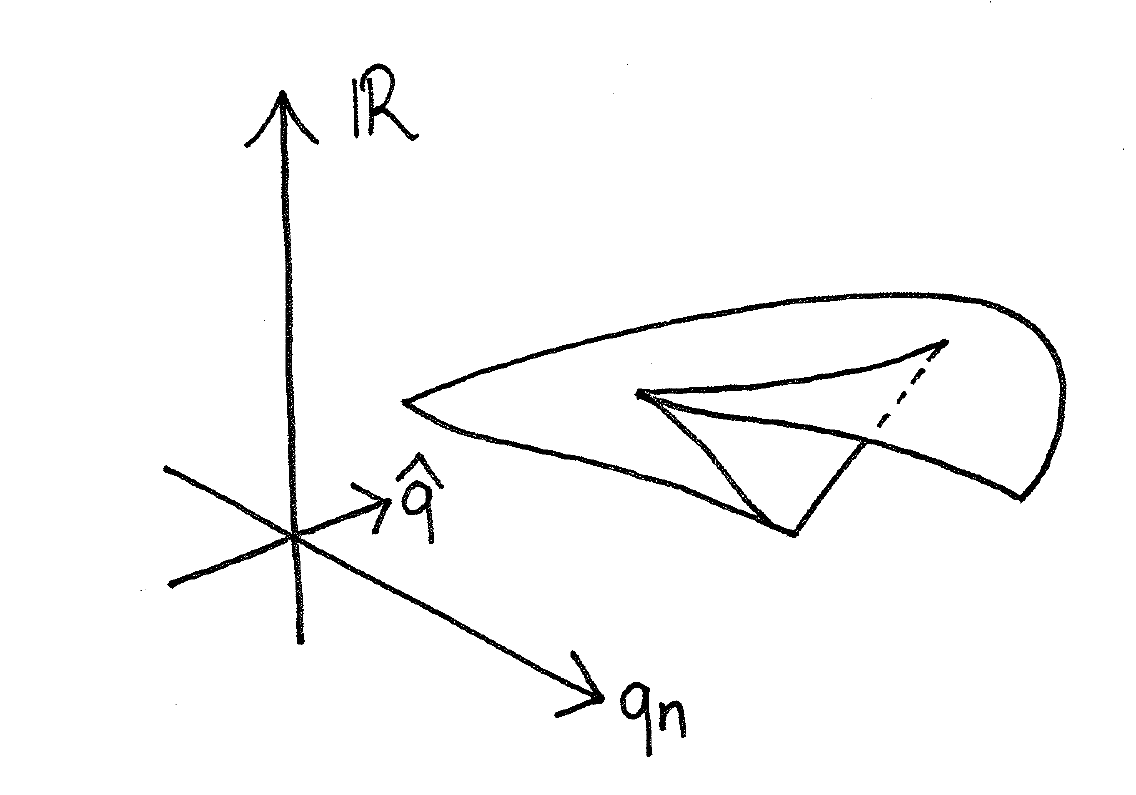}
\caption{A wrinkled embedding has birth/deaths of zig-zags on the equator of each wrinkle.}
\label{birthmodel}
\end{figure}

\begin{warning}
Observe that a wrinkled embedding has singularities along the wrinkles, but these are not singularities of tangency with respect to any foliation. These are (non-generic) singularities of the smooth map, in other words, points in the source where the rank of the differential is strictly less than the possible maximum. Throughout the paper we will be talking about both types of singularities but it should always be clear from the context which type we are referring to in each case. \end{warning}

A wrinkled embedding has a well defined Gauss map $G(df):L \to Gr_n(X)$, where $Gr_n(X)$ is the Grasmannian of $n$-planes in $TX$. For each $q \in L$ there is a unique $n$-dimensional subspace $G(df)(q) \subset T_{f(q)}X$ tangent to $f(L)$ at $f(q)$. At regular points $q \in L$ we have of course $G(df)(q)=df(T_qL)$, but $G(df)(q)$ is defined even at singular points, see Figure \ref{gaussmap}.

\begin{figure}[h]
\includegraphics[scale=0.6]{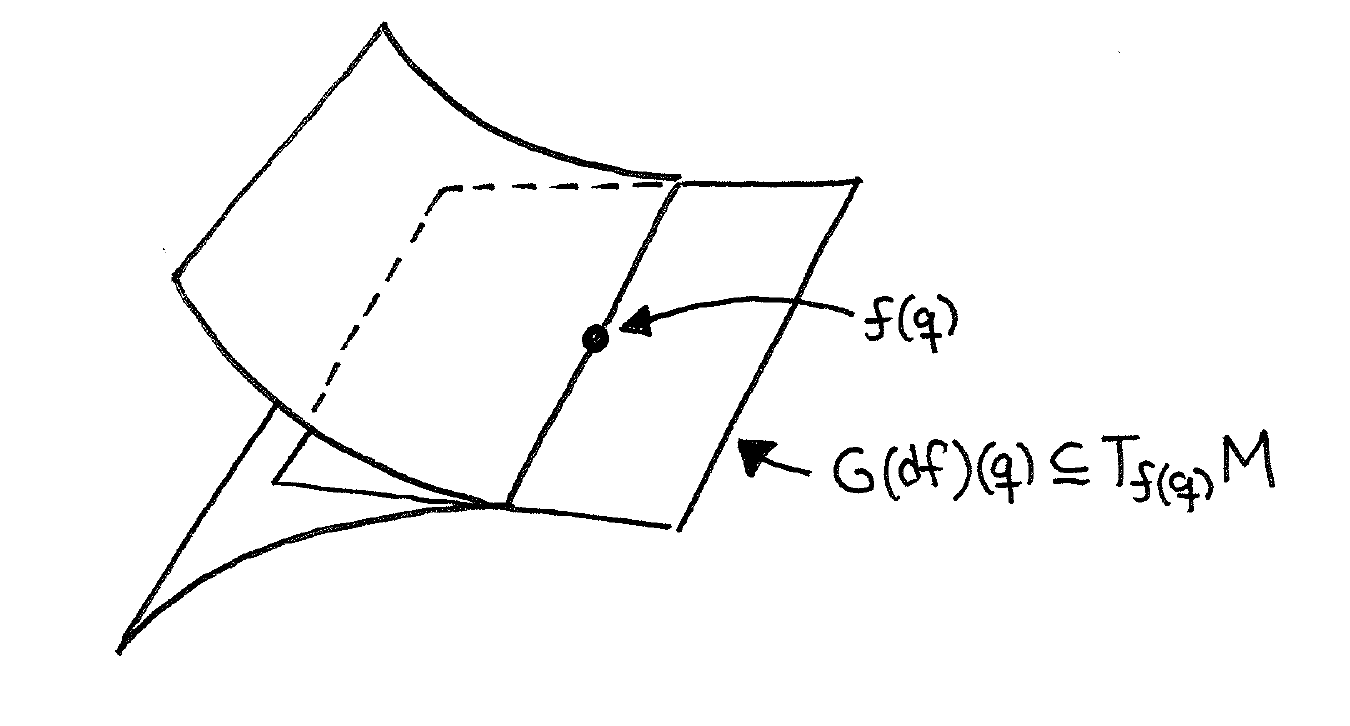}
\caption{A wrinkled embedding has a well-defined Gauss map everywhere, including points in the wrinkling locus.}
\label{gaussmap}
\end{figure}

\subsection{Wrinkled Lagrangian and Legendrian embeddings}\label{Lagrangian and Legendrian wrinkles} 
Let $(M, \omega)$ be a symplectic manifold.
\begin{definition}
A wrinkled Lagrangian embedding is a topological embedding $f:L^n \rightarrow (M^{2n},\omega)$ which is a smooth Lagrangian embedding away from a disjoint union of finitely many topologically trivial embedded $(n-1)$-spheres $S \subset L$, with $f$ equivalent (up to symplectomorphism) on $Op(S)\subset L$ to the local model $\cL_n: Op(S^{n-1}) \subset \bR^n \rightarrow (T^*\bR^n,dp \wedge dq)$ given by
\[ (q_1, \ldots , q_n) \mapsto \Big(q_1\, , \,  \ldots  \, , \,  q_{n-1}\, ,\,  \eta \, , \,  \frac{\partial H}{\partial q_{1}} - h\frac{\partial \eta }{ \partial q_{1}}  \, ,\,  \ldots  \, ,\,  \frac{\partial H}{\partial q_{n-1}} - h\frac{\partial \eta}{ \partial q_{n-1}} \, ,\,  h \Big)\] 
\[\text{where}\, \, \,  \, \eta(q)= q_n^3 + 3(|| \hat{q} ||^2-1)q_n , \quad h(q)= \int_0^{q_n}(|| \hat{q} ||^2+ u^2 -1)^2 du \]
\[ \text{and} \, \,  \, \, H(q)=\int_{0}^{q_n} h(\hat{q} , u) \frac{\partial \eta}{\partial q_n} (\hat{q} , u) du.  \]
\end{definition}
The wrinkled Lagrangian embedding $\cL_n$ is obtained from the wrinkled embedding $\cW_{n,n}$ in the following way. Let $(q,p)$ be the standard coordinates on $T^*\bR^n=\bR^n(q_1, \ldots, q_n) \times \bR^n(p_1, \ldots , p_n)$. Keeping $p_n \circ \cW_{n,n}=h$ fixed, for $j<n$ we replace the zero functions $p_j \circ \cW_{n,n}=0$ with the only possible functions (up to initial conditions) which will make the embedding Lagrangian. Informally, integrate $h$ in the direction $\partial / \partial q_n$ and differentiate the resulting function in the directions $\partial / \partial q_j$, $j<n$. The corresponding definition for Legendrians is entirely analogous. Let $(M,\xi)$ be a contact manifold. 
\begin{definition}
A wrinkled Legendrian embedding is a topological embedding $f:L^n \rightarrow (M^{2n+1},\xi)$ which is a smooth Legendrian embedding away from a disjoint union of finitely many topologically trivial embedded $(n-1)$-spheres $S \subset L$, with $f$ equivalent (up to contactomorphism) on $Op(S)\subset L$ to the local model $\widehat{\cL}_n: Op(S^{n-1}) \subset \bR^n \rightarrow \big(J^1(\bR^n,\bR),\xi_{std}\big)$ given by
\[ (q_1, \ldots , q_n) \mapsto \Big(q_1\, , \,  \ldots  \, , \,  q_{n-1}\, ,\,  \eta \, , \,  \frac{\partial H}{\partial q_{1}} - h\frac{\partial \eta }{ \partial q_{1}}  \, ,\,  \ldots  \, ,\,  \frac{\partial H}{\partial q_{n-1}} - h\frac{\partial \eta }{ \partial q_{n-1}} \, ,\,  h \,, \, H \Big)\] 
\[\text{where}\, \, \,  \, \eta(q)= q_n^3 + 3(|| \hat{q} ||^2-1)q_n , \quad h(q)= \int_0^{q_n}(|| \hat{q} ||^2+ u^2 -1)^2 du, \]
\[ \text{and} \, \,  \, \, H(q)=\int_{0}^{q_n} h(\hat{q} , u) \frac{\partial \eta}{\partial q_n} (\hat{q} , u) du.  \]
\end{definition}
We recall that $J^1(\bR^n,\bR)=T^*\bR^n(q,p) \times \bR(z)$ with the standard contact structure $\xi_{std}=\ker(dz-pdq)$. The Legendrian model $\widehat{\cL}_n$ is the Legendrian lift of the Lagrangian model $\cL_n$ under the Lagrangian projection $J^1(\bR^n,\bR) \to T^*\bR^n$, $(q,p,z) \mapsto (q,p)$. Consider also the front projection $J^1(\bR^n,\bR) \to J^0(\bR^n,\bR)=\bR^n \times \bR$, $(q,p,z) \mapsto (q,z)$. It is conceptually useful to understand the Legendrian front of the model $\widehat{\cL}_n$, which is the map $Op(S^{n-1}) \subset \bR^n \to \bR^n \times \bR$ given by $q \mapsto \big( (\hat{q}, \eta), H\big)$. On each of the hemispheres in $S^{n-1} \setminus S^{n-2}$, the front has semi-quintic cusps. On the equator $S^{n-2} \subset S^{n-1}$, the front has semi-quintic swallowtail singularities. See Figure \ref{quinticswallowtail} for an illustration.

\begin{figure}[h]
\includegraphics[scale=0.6]{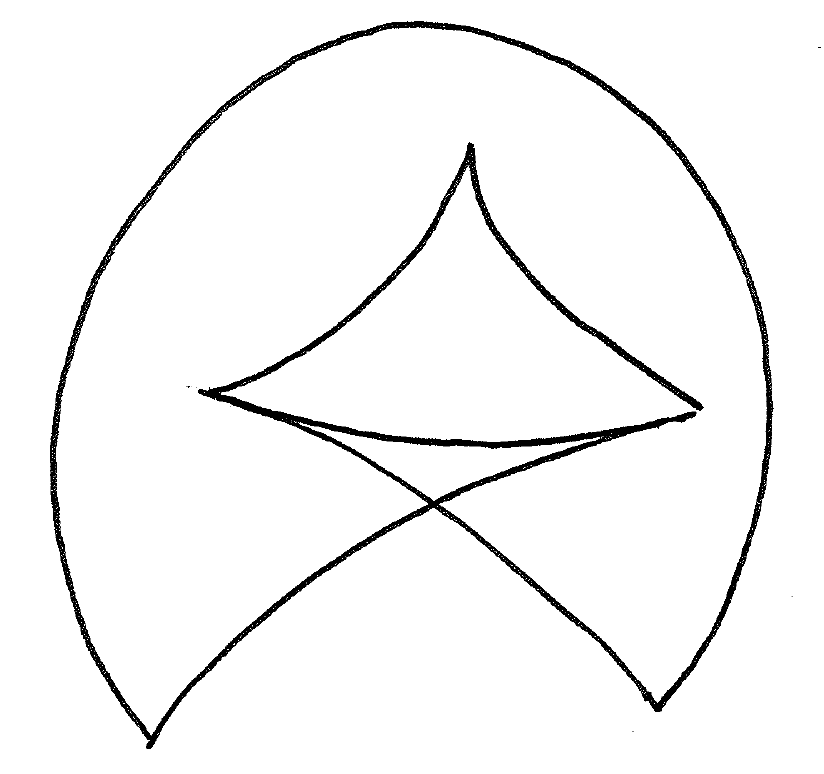}
\caption{The Legendrian front which generates one half of a Legendrian wrinkle. The cusps and swallowtail have a higher order of tangency than the standard cusps or swallowtails which one finds in the front projection of a regular Legendrian. To be more precise, the cusps which appear in the front projection of a Legendrian wrinkle are locally equivalent to $y^2=x^5$, whereas the standard cusps are locally quivalent to $y^2=x^3$.}
\label{quinticswallowtail}
\end{figure}

When we need to specify that a Lagrangian or Legendrian embedding $f:L \to M$ is not wrinkled, we will call $f$ regular. Observe that the Gauss map $G(df)$ of a wrinkled Lagrangian or Legendrian embedding $f:L \to M$ lands in the Lagrangian Grassmannian $\Lambda_n(M)$, just like a regular Lagrangian or Legendrian embedding.

\begin{warning}
A wrinkled Legendrian embedding has zig-zags but is not loose nor is it a wrinkled Legendrian in the sense of Murphy \cite{M12}. Indeed, the zig-zags of Murphy's wrinkled Legendrians occur in the font projection, not in the Legendrian submanifold itself. 
\end{warning}

\subsection{Parametric families of wrinkles}\label{Parametric families of wrinkles}
We will also consider families $f^z$ parametrized by a smooth compact manifold $Z$, possibly with boundary. A family of regular Lagrangian or Legendrian embeddings $f^z:L \to M$ parametrized by $Z$ is simply a smooth map $Z \times L \to M$, $(z,q) \mapsto f^z(q)$, such that for each $z \in Z$ the map $f^z$ is a regular Lagrangian or Legendrian embedding. If we allow the embeddings $f^z$ to be wrinkled, then we must allow the wrinkles to appear and disappear as the parameter $z$ varies. Indeed, in the smooth case considered in \cite{EM09}, Eliashberg and Mishachev allow  wrinkled embeddings to have the following local model $\cE_{n,r}:Op(0) \subset \mathbb{R}^n \rightarrow \mathbb{R}^{n+r}$ near finitely many points. These are embryos of wrinkles, instances of birth/death.
$$ (q_1 , \ldots , q_n) \mapsto \Big(q_1, \ldots, q_{n-1},\mu, 0, \ldots , 0,e \Big), $$
\[ \text{where} \quad \mu(q)= q_n^3 + 3||\hat{q}||^2q_n \quad \text{and} \quad   e=  \int_{0}^{q_n}(||\hat{q}||^2 + u^2)^2 du. \]

In the symplectic or contact case, we can deduce corresponding local forms for Lagrangian or Legendrian embryos by integrating the function $e$ in the direction $\partial / \partial q_n$ and then differentiating in the directions $\partial / \partial q_j$, $j<n$, just like we did in the definition of Lagrangian and Legendrian wrinkles. However, we wish to be slightly more precise in the way in which we allow wrinkles to be born or die and so we give the following definition of a family of wrinkled Lagrangian or Legendrian embeddings. We use the fibered terminology, which is a convenient language and is largely self-explanatory (the reader who wishes to see further details may consult for example \cite{EM09}).
 
 \begin{definition} A fibered wrinkled Lagrangian embedding $f^z:L^n \to (M^{2n}, \omega)$ parametrized by an $m$-dimensional manifold $ Z$ is a topological embedding $f:Z \times L \to Z \times M$, $(z,q) \mapsto (z,f^z(q))$ such that $f$ is a fibered smooth Lagrangian embedding away from a disjoint union of finitely many topologically trivial embedded $(m+n-1)$-spheres $S \subset Z \times L$, with $f$ equivalent (up to fibered symplectomorphism) on $Op(S) \subset Z \times L$ to the local fibered model $\cL^{\cF}_{n,m}:Op(S^{m+n-1})  \subset \bR^m \times \bR^n \to \bR^m \times (T^*\bR^n, dp \wedge dq)$ given by 
\[ ( z_1, \ldots , z_m , q_1 , \ldots , q_n )  \mapsto \Big(z_1, \ldots , z_m,\, q_1, \ldots , q_{n-1} , \, \eta \,  ,\,  \frac{\partial H}{\partial q_1} - h\frac{\partial \eta}{\partial q_1} \, , \, \ldots  \, , \,  \frac{\partial H}{\partial q_{n-1}} - h\frac{ \partial \eta }{\partial q_{n-1}} \,,\,  h\Big), \]
\[ \text{ where } \quad \eta(z,q)= q_n^2 + 3( ||z||^2 + ||\hat{q}||^2 -1)q_n, \quad  h(z,q)=\int_0^{q_n}( ||z||^2 + || \hat{q} ||^2 + u^2 -1)^2 du \]
\[ \text{ and} \qquad H(z,q)= \int_{0}^{q_n} h(z, \hat{q}, u) \frac{ \partial \eta}{\partial q_n}(z, \hat{q}, u) du . \]
 \end{definition}
 If we restrict $\cL^{\cF}_{n,m}$ to the half space $\{ z_1 \geq 0 \}$ we get the local model for the fibered half-wrinkles near the boundary $\partial Z$ of the parameter space. We can define fibered wrinkled Legendrian embeddings in the exact same way, with the local model $\widehat{\cL}^{\cF}_{n,m}=(\cL^{\cF}_{n,m},H): Op(S^{m+n-1}) \subset \bR^{m} \times \bR^n \to \bR^m \times \big(J^1(\bR^n, \bR), \xi_{std}\big)$. When we talk about families of wrinkled Lagrangian or Legendrian embeddings parametrized by a compact manifold, we will always assume that the family is fibered in the sense just described.

 \subsection{Exact homotopies}\label{Exact homotopies} Taking $Z=[0,1]$ in the definition of fibered wrinkled Lagrangian or Legendrian embeddings, we obtain the notion of a homotopy of wrinkled Lagrangian or Legendrian embeddings $f_t:L \to M$, $t \in [0,1]$, in which wrinkles are allowed to be born and to die as time goes by. The notion of exactness for homotopies of regular Lagrangian embeddings can be extended to the wrinkled case in a straightforward way.

\begin{definition}\label{exactwrinkled} Let $f_t:L \rightarrow M$ be a homotopy of (possibly wrinkled) Lagrangian embeddings. We say that $f_t$ is exact if the following condition holds. For the mapping $F:L \times [0,1] \rightarrow M$ defined by $(q,t) \mapsto f_t(q)$, consider the closed form $i_{\partial / \partial t} F^* \omega$ on $L \times [0,1]$. We demand that this form is exact when pulled back to $L$ by each of the inclusions $L \hookrightarrow L \times [0,1]$, $q \mapsto (q,t)$.\end{definition}

\begin{remark} Recall that if $f_t:L \to M$ is a homotopy of regular Lagrangian embeddings, then for small time $t>0$ one can interpret $f_t$ as a closed $1$-form  $\alpha_t$ on $L$ by identifying a neighborhood of the zero section in $T^*L$ with a Weinstein neighborhood of $f_0(L)$ in $M$. In this case exactness of $f_t$ amounts to asking that $\alpha_t$ is exact for every $t \in [0,1]$. 
\end{remark}

The importance of this definition stems from the following fact. If $f_t:L \to M$ is a compactly supported exact homotopy of regular Lagrangian embeddings, then there exists a (compactly supported) ambient Hamiltonian isotopy $\varphi_t:M \to M$ such that $f_t= \varphi_t \circ f_0$. We will always want to ensure that all homotopies of Lagrangian embeddings, regular or wrinkled, are exact. In the contact case, exactness is automatic. Indeed, every homotopy of regular Legendrian embeddings is induced by an ambient Hamiltonian isotopy. For convenience, we shall therefore refer to all homotopies of Legendrian embeddings, regular or wrinkled, as exact.

When a homotopy $f_t$ is fixed on a closed subset $A\subset L$ (usually $A=L \setminus U$ is the complement of some open set $U$ where we are performing some geometric manipulation), the notions of exactness will be understood relative to $Op(A)$. In this way, the ambient Hamiltonian isotopy inducing $f_t$ can be taken to be the identity on $Op\big(f(A) \big) \subset M$.

\begin{figure}[h]
\includegraphics[scale=0.7]{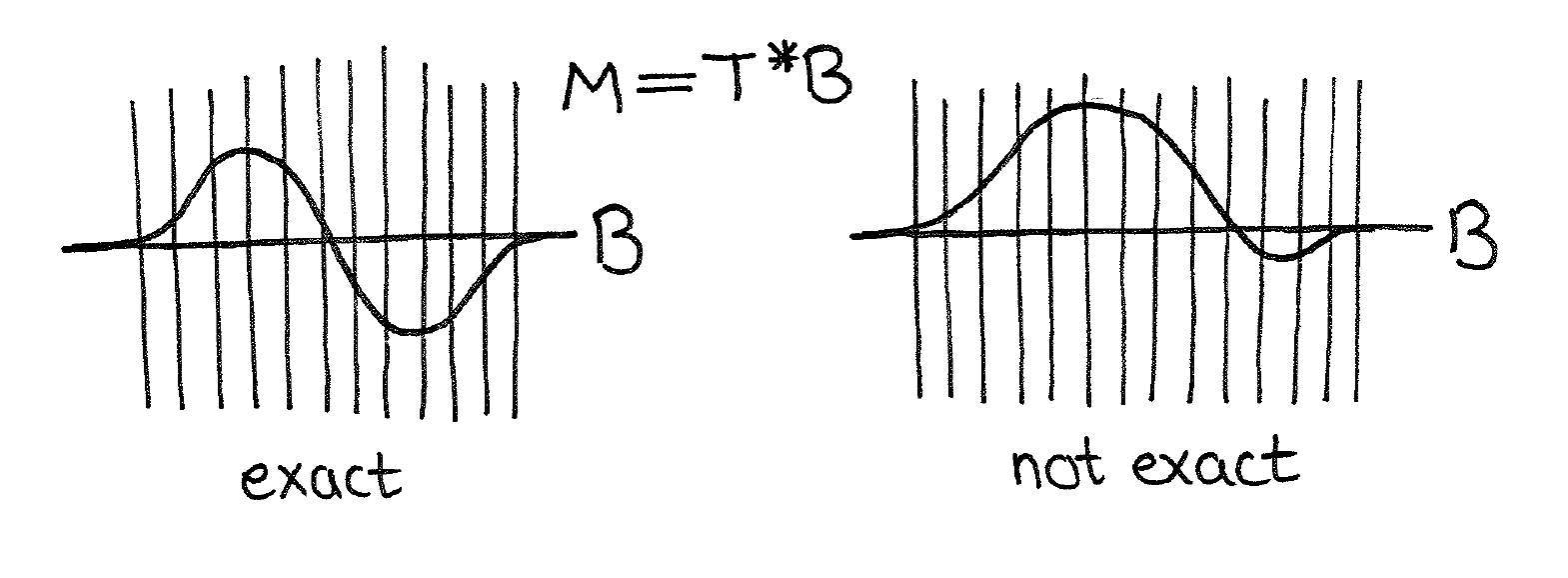}
\caption{The difference between an exact and a non-exact deformation of the zero section $B \hookrightarrow T^*B$. On the left, the areas cancel out, whereas on the right they do not. Exactness can be thought of as an area condition.}
\label{exactness}
\end{figure}

\subsection{Regularization of wrinkles}\label{Regularization of wrinkles}
Wrinkles can be regularized as follows. Consider the local model $\cW_{n,n}(q)=(\hat{q},\eta,0, \ldots, 0,h)$ introduced in Section \ref{Wrinkled embeddings}. Let $\phi: \bR^n \to \bR$ be a $C^\infty$-small function such that $\partial \phi / \partial q_n > 0$ on $S^{n-1}\subset \bR^n$ and such that $\text{supp}({\phi}) \subset Op(S^{n-1})$. Let $\widetilde{h}=h+ \phi$ and observe that $\widetilde{\cW}_{n,n}(q)=(\hat{q},\eta,0,\ldots, 0,\widetilde{h})$ is a smooth regular embedding such that $\widetilde{\cW}_{n,n}= {\cW}_{n,n}$ outside of $Op(S^{n-1})$, see Figure \ref{regularize}. 

 \begin{figure}[h]
\includegraphics[scale=0.6]{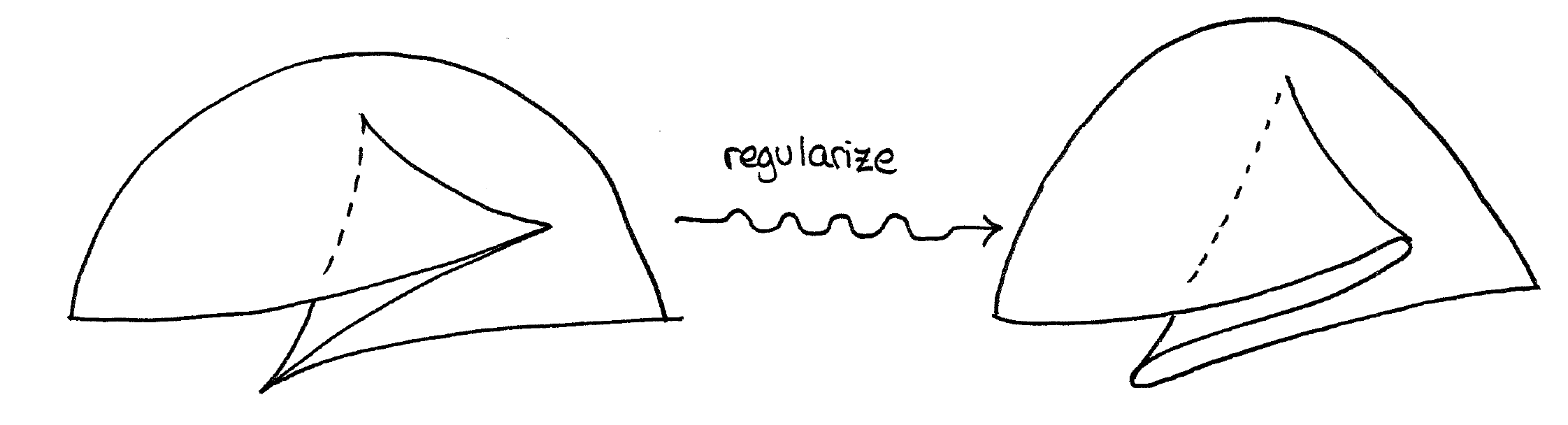}
\caption{Regularization of the standard wrinkle.}
\label{regularize}
\end{figure}

Next, require further that
\[ \int_{0}^{q_n} \phi(\hat{q}, u) \frac{\partial \eta}{\partial q_n}(\hat{q},u)=0 \]
whenever $q=(\hat{q},q_n) \notin \text{supp}({\phi})$, and consider the modified integral
\[ \widetilde{H}(q)= \int_{0}^{q_n} \widetilde{h}(\hat{q},u)\frac{\partial \eta}{\partial q_n}(\hat{q},u)du . \]

We obtain a regular Lagrangian embedding $\widetilde{\cL}_n:Op(S^{n-1}) \to (T^*\bR^n,dp \wedge dq)$ such that $\widetilde{\cL}_n=\cL_n$ outside of $Op(S^{n-1})$ by the formula
\[ (q_1, \ldots, q_n) \mapsto \Big(q_1, \ldots , q_{n-1}, \, \eta \, , \,  \frac{\partial \widetilde{H}}{ \partial q_1 }- \widetilde{h}\frac{\partial \eta }{ \partial q_1} \, , \, \ldots \, , \, \frac{\partial \widetilde{H}}{\partial q_{n-1} }- \widetilde{h}\frac{\partial \eta }{ \partial q_{n-1}} \, , \, \widetilde{h} \Big).\]

The Legendrian counterpart of the regularization is the local model $(\widetilde{\cL}_n,\widetilde{H})$. See Figure \ref{frontregularize} for an illustration of the regularization process in the front projection. Given a wrinkled Lagrangian or Legendrian embedding $f:L \to M$, we can apply this local procedure to every wrinkle and obtain a regular Lagrangian or Legendrian embedding $\widetilde{f}$. Similarly, a fibered wrinkled Lagrangian or Legendrian embedding $f^z$ can be regularized to a fibered regular Lagrangian or Legendrian embedding $\widetilde{f}^z$. If $f_t:L \to M$ is an exact homotopy of wrinkled Lagrangian embeddings, then $\widetilde{f}_t:L \to M$ is an exact homotopy of regular Lagrangian embeddings.

\begin{figure}[h]
\includegraphics[scale=0.5]{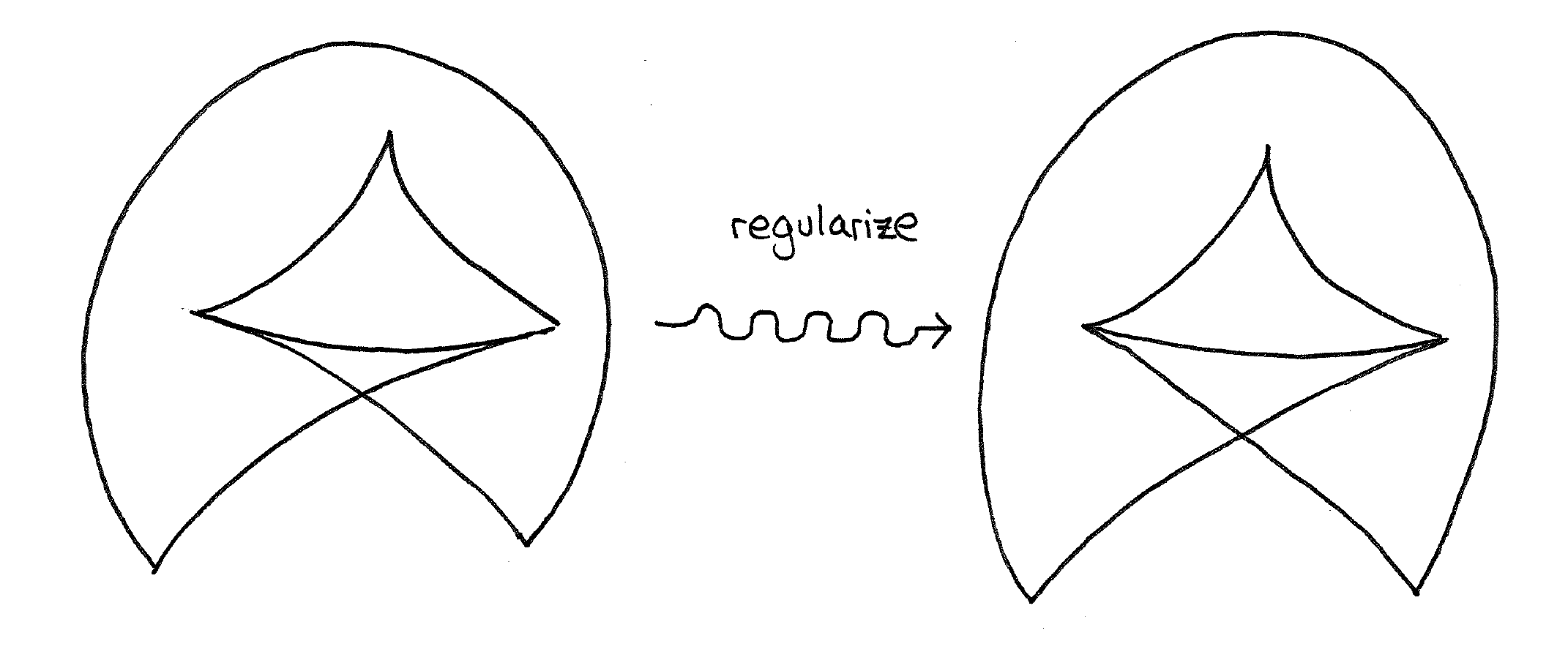}
\caption{The regularization can be also understood in terms of the front projection. The effect is to replace the semi-quintic cusps and swallowtails with semi-cubic cusps and swallowtails.}
\label{frontregularize}
\end{figure}

The change in the order of tangency as well as the geometric meaning of the condition  $\int_{0}^{q_n} \phi(\hat{q}, u) \frac{\partial \eta}{\partial q_n}(\hat{q},u)du=0$  can be better appreciated if we focus on the complement of the equator. See Figure \ref{modelforregularization} for an illustration of the regularization process near a cusp point.

\begin{figure}[h]
\includegraphics[scale=0.6]{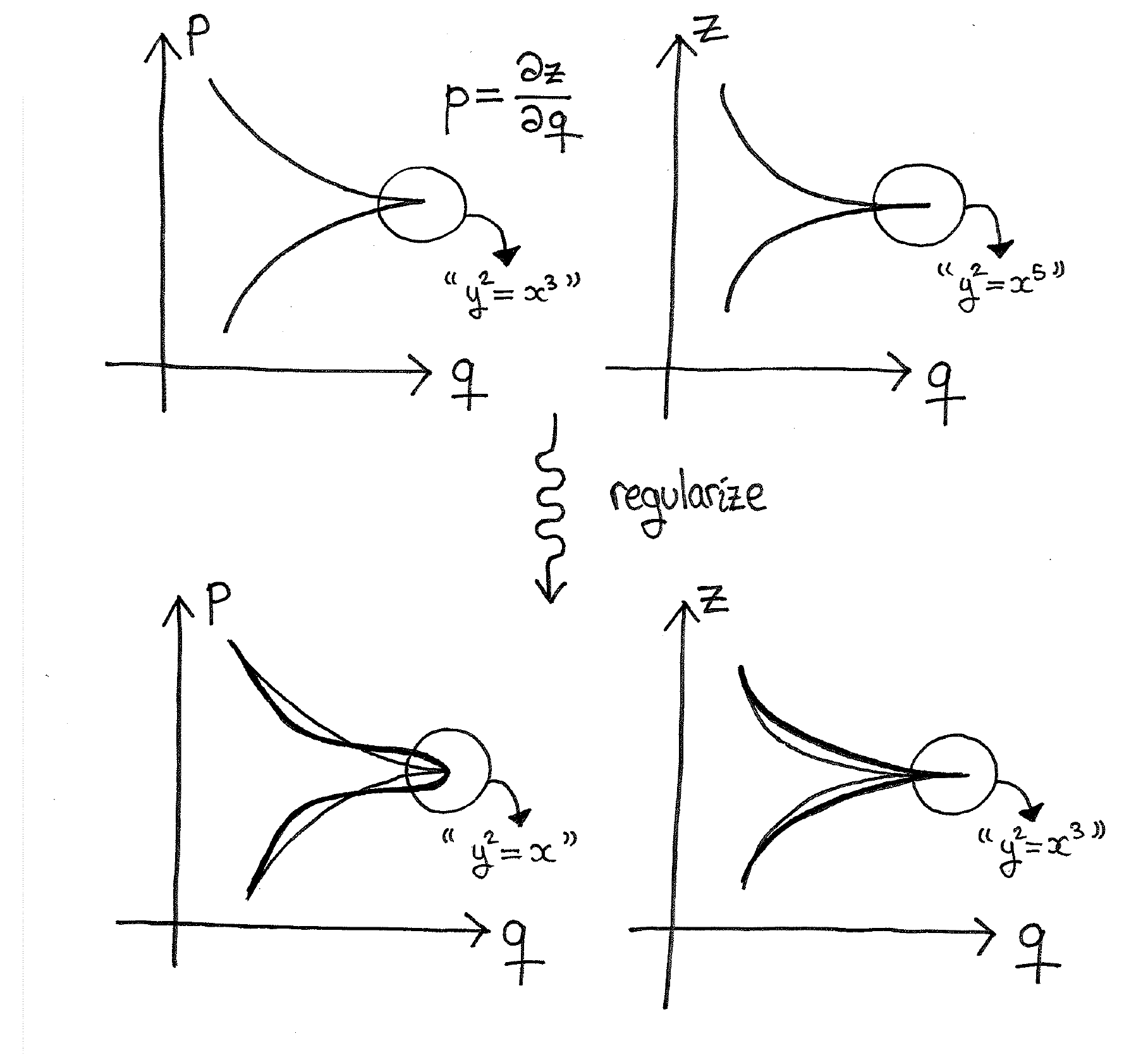}
\caption{Effect of the regularization process away from the equator in both the Lagrangian and front projections. The equation $\int_{0}^{q_n} \phi(\hat{q}, u) \frac{\partial \eta}{\partial q_n}(\hat{q},u)du=0$ manifests itself as an area condition in the bottom left.}
\label{modelforregularization}
\end{figure}
\begin{remark} Observe that the regularization process $f \mapsto \widetilde{f}$ depends on the choice of $\phi$. However, the space of possible $\phi$ is convex and therefore $\widetilde{f}$ is well defined up to a contractible choice. Different choices alter $\widetilde{f}$ by an ambient Hamiltonian isotopy supported on a neighborhood of the image of the wrinkling locus.
\end{remark}

\begin{remark}\label{sigma type singularities}
In the Lagrangian case, let $T^*\bR^n$ be foliated by the fibres of the standard projection $\pi:T^*\bR^n \to \bR^n$ and in the contact case, let $J^1(\bR^n, \bR) = T^*\bR^n \times \bR$ be foliated by the fibres of the front projection $\pi \times id:T^*\bR^n \times \bR \to \bR^n \times \bR$. Observe that the standard Lagrangian and Legendrian wrinkles are transverse to these foliations. Moreover, when we regularize the Lagrangian or Legendrian wrinkle we obtain a regular Lagrangian or Legendrian embedding whose singularities of tangency with respect to the corresponding foliation consist of $\Sigma^{10}$ folds away from the equator and of $\Sigma^{110}$ pleats on the equator.
\end{remark}

\subsection{Sharpening the wrinkles}\label{sharpening the wrinkles} 

Let $D^{\pm}=\{ q \in S^{n-1} | \pm q_n \geq 0 \}$ be the north and south hemispheres of the unit sphere $S^{n-1} \subset \bR^n$ and let $D^{n-1}$ be the closed unit disk in $\bR^{n-1}$, which we think of as sitting in $\bR^n$ via the inclusion $\bR^{n-1} =\bR^{n-1} \times 0 \subset \bR^n$. The standard Lagrangian wrinkle $\cL_n:Op(S^{n-1}) \subset \bR^n \to T^*\bR^n$ is equivalent on $Op(D^\pm) \setminus Op(\partial D^\pm)$ to the following local model $\cC_n: \bR^n \to T^*\bR^n$ on $Op(D^{n-1}) \setminus Op(\partial D^{n-1})$.
\[ \cC_n(q_1, \ldots, q_n) =\left( q_1, \ldots , q_{n-1} , q_n^2, 0, \ldots , q_n^3\right) .\]

Note that $\cC_n$ is the product of $\cC_1:\bR \to T^*\bR$ and the zero section $\bR^{n-1} \hookrightarrow T^*\bR^{n-1}$. By scaling the model $\cC_n$ by any small number $\varepsilon>0$ in the direction of the cotangent fibres we get a sharpened Lagrangian cusp $\varepsilon \, \cC_n:\bR^n \to T^*\bR^n$. Explicitly, we set
\[ \varepsilon \,  \cC_n(q_1, \ldots , q_n) = \left( q_1 , \ldots , q_{n-1} , q_n^2, 0, \ldots , \varepsilon q_n^3 \right). \]

It will be useful for us later on to sharpen the cusps of a Lagrangian wrinkle. This sharpening can be achieved by interpolating between the two models $\cC_n$ and $\varepsilon \, \cC_n$. The key property of the sharpening construction is that the interpolation can be achieved by a $C^1$-small perturbation. The precise result that we will need is the following, where we recall the notation $q=(\hat{q},q_n)$, $\hat{q}=(q_1,\ldots, q_{n-1})$.
\begin{lemma}\label{sharpencusp}
For any $\delta, \varepsilon>0$ there exists an exact homotopy $\cC_{n,t}:\bR^n \to T^*\bR^n$ such that the following properties hold.
\begin{itemize}
\item $\cC_{n,0}=\cC_n$.
\item $\cC_{n,t}=\cC_n$ when $|q_n|>2\delta$ or $||\hat{q}||>1-\delta$.
\item $\cC_{n,1}= \varepsilon \, \cC_n$ when $|q_n|<\delta$ and $|| \hat{q} ||<1-2\delta$.
\item $\text{dist}_{C^1}(\cC_n, \cC_{n,t} ) \leq A \delta$ for some constant $A>0$ independent of $\delta$ and $\varepsilon$.
\end{itemize}
\end{lemma}

The same Lemma also holds for the Legendrian cusp $\widehat{\cC}_n=(\cC_n,C):\bR^n \to T^*\bR^n \times \bR = J^1(\bR^n,\bR)$, where  $C(q)=\frac{2}{5} q_n^5$. We prove the Lagrangian and Legendrian versions simultaneously.

\begin{proof}
 Let $\psi:\bR \times \bR \to [0,1]$ be a function satisfying the following properties.
\begin{itemize}
\item $\psi(x,y) = \varepsilon$ for $(x,y)  \in \left[-\delta,\delta \right] \times [-1+2\delta,1-2\delta]$, 	
\item $\varepsilon \leq \psi(x,y) \leq 1$ for $(x,y) \in  [-2\delta,2\delta] \times [-1+\delta,1-\delta] \setminus  \left[-\delta,\delta \right] \times [-1+2\delta,1-2\delta]$,
\item $\psi(x,y)=1$ for $(x,y) \notin [-2\delta,2\delta] \times [-1+\delta,1-\delta]$.
\end{itemize}
We also demand, as we may, that the following bounds hold for some constant $A>0$ independent of $\delta$ and $\varepsilon$. 
\begin{itemize}
\item $| \partial \psi/\partial x |, \, | \partial \psi / \partial y | \leq A/\delta$.
\item $  |\partial ^2 \psi / \partial x^2 | , \, | \partial^2 \psi / \partial x \partial y |, \, | \partial^2 \psi / \partial y^2 | \leq A / \delta^2$.
\end{itemize}
Finally, we may also choose $\psi$ such that $\partial \psi /\partial y=0$ when $|y|<1-2 \delta$. 

Set $\psi_t=(1-t) + t \psi$ and $C_t(q)=\frac{2}{5} \psi_t(q_n, ||\hat{q}||)q_n^5$. The front $q \mapsto \big( (\hat{q},q_n^2),C_t\big) \in \bR^n \times \bR$ generates the Lagrangian and Legendrian cusps $\cC_{n,t}$ and $\widehat{\cC}_{n,t}=(\cC_{n,t},C_t)$ respectively. To be explicit, we have
\[ \cC_{n,t}(q)=\Big(\hat{q}, q_n^2, \frac{2}{5}\frac{\partial \psi_t}{\partial y}(q_n,||\hat{q}||)\frac{q_1q_n^5}{||\hat{q}||}, \ldots, \frac{2}{5}\frac{\partial \psi_t}{\partial y}(q_n,||\hat{q}||)\frac{q_{n-1}q_n^5}{||\hat{q}||},\frac{1}{5} \frac{\partial \psi_t}{\partial x}(q_n,||\hat{q}||)q_n^4 + \psi_t(q_n, || \hat{q} ||))q_n^3\Big). \]
The first three properties stated in the Lemma are clearly satisfied. The fourth property follows from the uniform bounds on the first and second partial derivatives of $\psi$.
\end{proof}

\begin{figure}[h]
\includegraphics[scale=0.6]{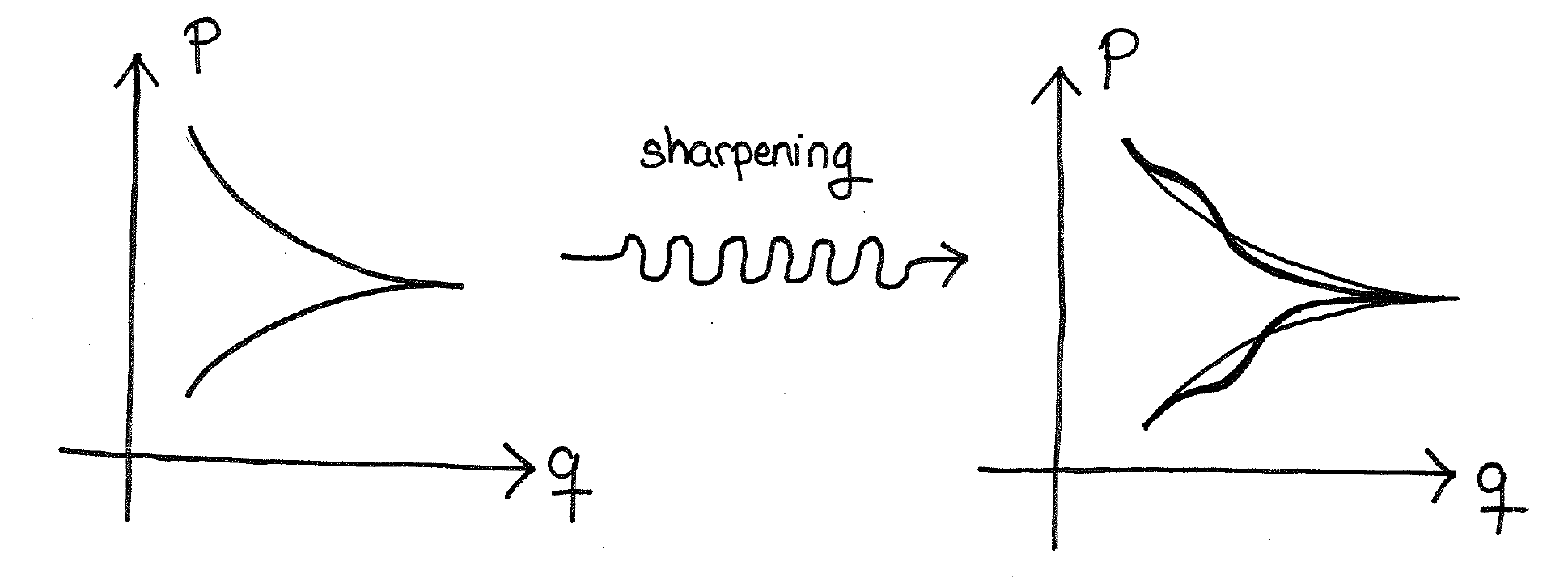}
\caption{Sharpening the Lagrangian cusp. Since we define the sharpening at the level of the front, the area condition which is necessary for exactness is automatically satisfied, as shown on the picture.}
\label{sharpeningcusp}
\end{figure}

Next we explain how to sharpen the birth/deaths of zig-zags on the equator of each wrinkle. The standard Lagrangian wrinkle $\cL_n:Op(S^{n-1}) \subset \bR^n \to T^*\bR^n$ is equivalent on $Op(S^{n-2}) \subset \bR^n$ to the following local model $\cG_n: S^{n-2} \times \bR^2 \to T^*(S^{n-2} \times \bR^2)$ on $Op(S^{n-2} \times 0) \subset S^{n-2} \times \bR^2$. 
\[ \cG_n(\widetilde{q},q_{n-1},q_n) = \Big(\widetilde{q}, q_{n-1} , \tau ,0, \frac{\partial G}{\partial q_{n-1}} - g \frac{\partial \tau}{\partial q_{n-1}} , g\Big) , \qquad q=(\widetilde{q},q_{n-1},q_n) \in S^{n-2} \times \bR \times \bR, \]
\[ \text{where}  \quad \tau(q_{n-1},q_n)= q_n^3-3q_{n-1}q_n , \qquad g(q_{n-1},q_n)= \int_0^{q_n} (u^2-q_{n-1})^2 du \]
\[ \text{and} \qquad G(q_{n-1},q_n)=\int_0^{q_n} g(q_{n-1}, u) \frac{\partial \tau}{\partial q_n}(q_{n-1} , u) \, du. \]

We remark that $\cG_n$ is the product of $\cG_2:\bR^2 \to T^*\bR^2$ with the zero section $S^{n-2} \hookrightarrow T^*S^{n-2}$. For any $\varepsilon>0$, the sharpened model $\varepsilon \, \cG_n: S^{n-2} \times \bR^2 \to T^*(S^{n-2} \times \bR^2)$ is given by
\[ \varepsilon \, \cG_n(\widetilde{q},q_{n-1},q_n) = \left( \widetilde{q},q_{n-1},\tau,0, \varepsilon\left(\frac{\partial G}{\partial q_{n-1}}-g \frac{\partial \tau}{\partial q_{n-1}} \right), \varepsilon  g \, \right).\]

The following result allows us to interpolate between $\cG_n$ and $\varepsilon \, \cG_n$ while maintaining $C^1-$control throughout the perturbation.
\begin{lemma}\label{sharpenbirth}
For any $\delta, \varepsilon>0$ there exists an exact homotopy $\cG_{n,t}:S^{n-2} \times \bR^2 \to T^*(S^{n-2} \times \bR^2)$ such that the following properties hold.
\begin{itemize}
\item $\cG_{n,0}=\cG_n$.
\item $\cG_{n,t}=\cG_n$ when $|q_{n-1}|>2\delta$ or $|q_n|>2\delta$.
\item $\cG_{n,1}= \varepsilon \, \cG_n$ when $|q_{n-1}|<\delta$ and $|q_n|<\delta$.
\item $\text{dist}_{C^1}(\cG_n, \cG_{n,t} ) \leq A \delta$ for some constant $A>0$ independent of $\delta$ and $\varepsilon$.
\end{itemize}
\end{lemma}

As before, the same Lemma also holds for the Legendrian counterpart $\widehat{\cG}_n=(\cG_n,G):S^{n-2} \times \bR^2 \to T^*(S^{n-2} \times \bR^2) \times \bR=J^1(S^{n-2} \times \bR^2,\bR)$ and we prove both versions simultaneously.

\begin{proof}
Let $\phi:\bR^2 \to [0,1]$ be a function satisfying the following properties.
\begin{itemize}
\item $\phi(x,y) = \varepsilon$ for $(x,y) \in \left[-\delta,\delta \right]^2 $, 
\item $\varepsilon \leq \phi(x,y) \leq 1$ for $(x,y) \in \left[-2\delta, 2\delta \right]^2  \setminus \left[-\delta, \delta\right]^2$,
\item $\phi(x,y)=1$ for $(x,y) \notin [-2\delta, 2\delta]^2 $.
\end{itemize}

We again demand that the following bounds hold for a constant $A>0$ independent of $\delta$ and $\varepsilon$. 
\begin{itemize}
\item $| \partial \phi/\partial x |, \, | \partial \phi / \partial y | \leq A/\delta$.
\item $  |\partial ^2 \phi / \partial x^2 | , \, | \partial^2 \phi / \partial x \partial y |, \, | \partial^2 \phi / \partial y^2 | \leq A / \delta^2$.
\end{itemize}
Moreover, we may also choose $\phi$ such that $\partial \phi /\partial y=0$ if $|y|<\delta$.

Set $\phi_t=(1-t) + t \phi_t$ and $G_t(q)= \phi_t(q_{n-1},q_n)G(q)$. The front $q \mapsto \big( (\widetilde{q},q_{n-1},\tau),G_t\big)$ generates the Lagrangian and Legendrian birth/deaths of zig-zags $\cG_{n,t}$ and $\widehat{\cG}_{n,t}=( \cG_{n,t},G_t)$ respectively. To be explicit, we have
\[ \cG_{n,t}(\widetilde{q},q_{n-1},q_n)= \left(\widetilde{q}, q_{n-1} ,  \, \tau \, , 0 \, ,\,  \frac{\partial G_t }{\partial q_{n-1}} - \big(\frac{\partial \phi_t}{\partial y} \frac{G}{(\frac{\partial \tau }{ \partial q_n})} + \phi_t \,  g \big) \frac{\partial \tau}{\partial q_{n-1}} \,   , \,\frac{\partial \phi_t}{\partial y} \frac{G}{(\frac{\partial \tau }{ \partial q_n})} + \phi_t\,  g \right). \]
The first three properties stated in the Lemma are clearly satisfied. The fourth property follows from the uniform bounds on the first and second partial derivatives of $\phi$.
\end{proof}

\begin{figure}[h]
\includegraphics[scale=0.6]{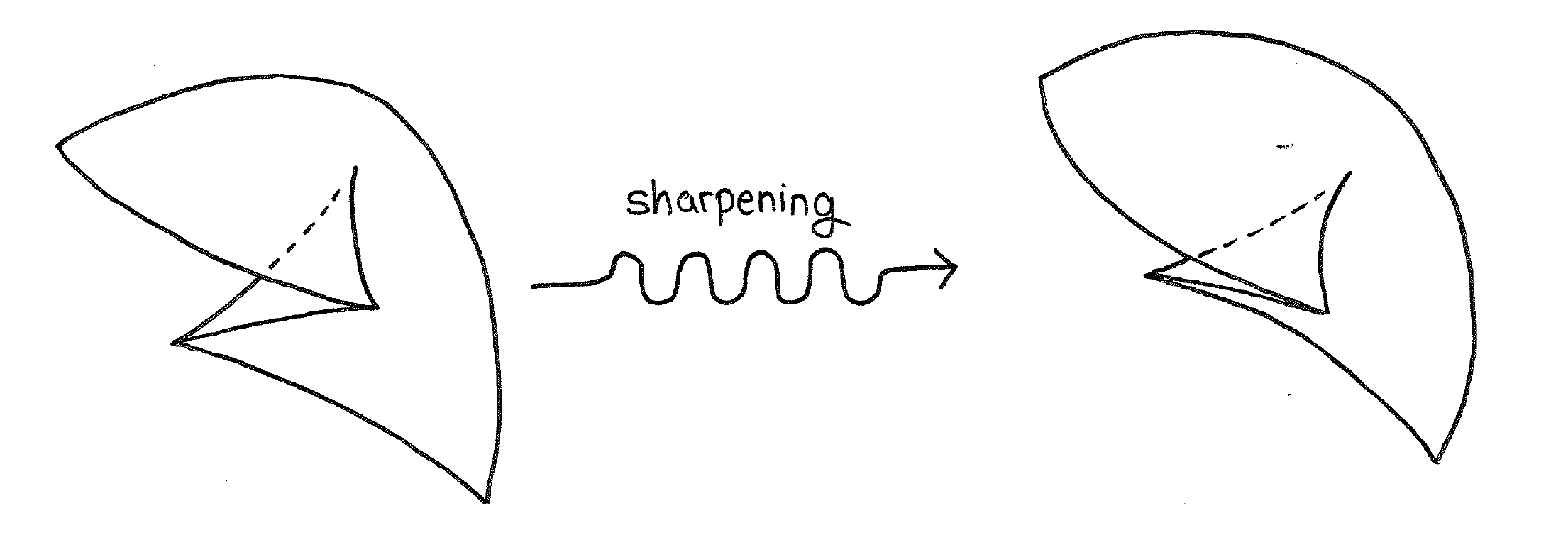}
\caption{Sharpening the Lagrangian birth/death of zig-zags..}
\label{sharpeningbirth}
\end{figure}

\begin{remark} The sharpening construction can also be applied to a family of wrinkled Lagrangian or Legendrian embeddings. To do this, one needs to work instead with the local model for the fibered wrinkle and repeat the above construction in the fibered setting. The proofs only differ in notation.
\end{remark}

\section{Lagrangian and Legendrian rotations}\label{Rotations of Lagrangian and Legendrian planes}

\subsection{Tangential rotations}
In Section \ref{Tangential homotopies} we introduced the notion of a tangential rotation, which decouples a Gauss map $G(df):L \to \Lambda_n(M)$ from its underlying Lagrangian or Legendrian embedding $f:L \to M$. We repeat the definition below for convenience. Recall that $\Pi:\Lambda_n(M) \to M$ denotes the Lagrangian Grassmannian of a symplectic or contact manifold $M$. 

\begin{definition}
A tangential rotation of a regular Lagrangian or Legendrian embeddings $f:L \to M$ is a compactly supported deformation $G_t:L \to \Lambda_n(M)$, $t \in [0,1]$, of $G_0=G(df)$ such that $\Pi \circ G_t = f$. 
\end{definition}

%

We will also need to consider tangential rotations of wrinkled Lagrangian and Legendrian embeddings. As in the unwrinkled case, a tangential rotation of a wrinkled Lagrangian or Legendrian embedding $f:L \to M$ is a compactly supported deformation $G_t:L \to \Lambda_n(M)$, $t \in [0,1]$, of $G_0=G(df)$ such that $\Pi  \circ G_t=f$. 

\subsection{Simple tangential rotations}

Let $G_t:L \to \Lambda_n(M)$ be a tangential rotation of a possibly wrinkled Lagrangian or Legendrian embedding $f:L \to M$. A priori, the one-parameter family of Lagrangian planes $G_t(q)$ could rotate around wildly inside $T_{f(q)}M$. It will be useful for us to restrict these rotations to be of a particularly simple type. See Figure \ref{simplerotation} for an illustration of the desired simplicity.

\begin{definition} A tangential rotation  $G_t:L \to \Lambda_n(M)$ of a possibly wrinkled Lagrangian or Legendrian embedding $f:L \to M$ is simple if there exists a field of $(n-1)$-dimensional planes $H^{n-1} \subset TM$ defined along some open subset $\cO \subset M$ such that 
\begin{itemize}
\item on $f^{-1}(\cO)$ we have $H \subset \text{im}(G_t)$ for all $t \in [0,1]$. 
\item on $L \setminus f^{-1}(\cO)$ the rotation $G_t$ is constant. 
\end{itemize}
We say that $G_t$ is simple with respect to $H$. \end{definition}

\begin{figure}[h]
\includegraphics[scale=0.65]{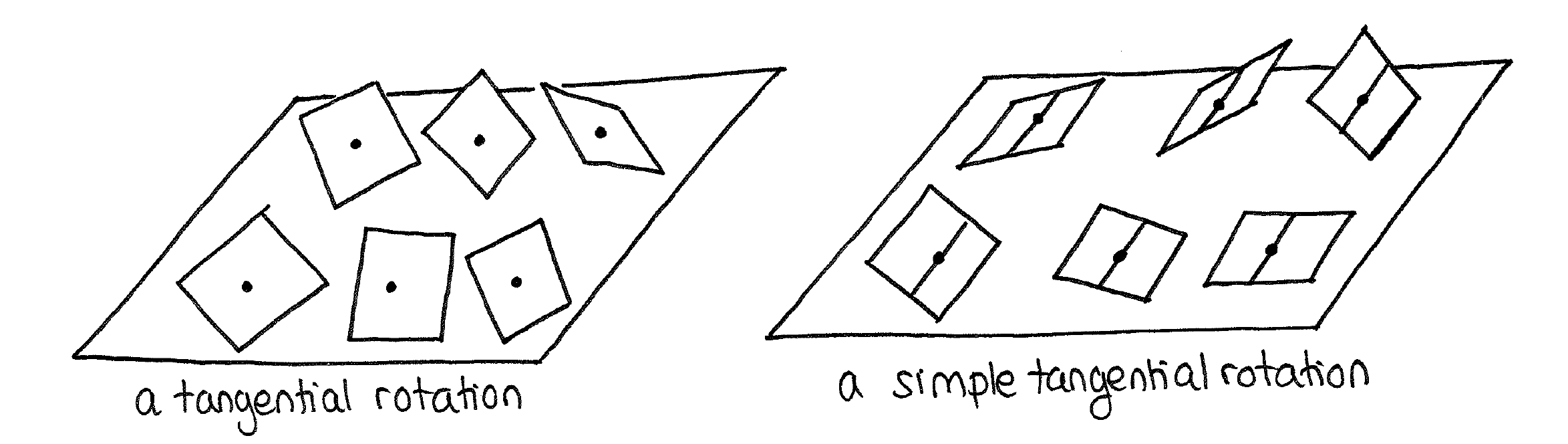}
\caption{The difference between a non-simple tangential rotation and a simple tangential rotation. Observe that in the simple case, the rotating planes $G_t$ are constrained so that the $(n-1)$ directions contained in $H$ are kept fixed, leaving only one degree of freedom. }
\label{simplerotation}
\end{figure}

If $f$ is regular, then we can think of $H \subset df(TL)$ as a hyperplane field in $TL$. When $f$ is wrinkled we need to be a little bit careful near the wrinkling locus so it will be best to think of $H$ as an ambient $(n-1)$-plane field in $TM$.

\begin{remark} Our definition of simple tangential rotations is slightly more restrictive than the definition given by Eliashberg and Mishachev in \cite{EM09} for the smooth analogue of this notion. This is the case because the Lagrangian or Legendrian wrinkling model that we are able to construct below is less general than the model used in their proof. \end{remark}  

We will also need the notion of piecewise simplicity. A tangential rotation $G_t$ of a regular Lagrangian or Legendrian embedding $f$ is piecewise simple if we can subdivide the time interval ${0=t_0< \cdots <t_k=1}$ so that the following property holds. We demand that there exist $(n-1)$-plane fields $H^j \subset \text{im}(G_{t_j})$ defined along open subsets $\cO_j \subset M$ such that $G_t=G_{t_j}$ outside of $f^{-1}(\cO_j)$ and $H^j \subset \text{im}(G_t)$ on $f^{-1}(\cO_j)$ for all $t \in  [t_j, t_{j+1} ]$. We will prove below that any tangential rotation of a regular Lagrangian or Legendrian embedding can be $C^0$-approximated as accurately as desired by a piecewise simple tangential rotation. In order to do this we first translate the notion of a tangential rotation into the language of jet spaces.

\subsection{Rotations of $2$-jets}\label{Rotations of $2$-jets} 
Let $f:L \to M$ be a regular Lagrangian embedding. Fix once and for all a Riemannian metric on $L$. For $\delta>0$ small enough, the Weinstein theorem guarantees the existence of a symplectomorphism $\Phi$ between a neighborhood $\cN$ of $f(L)$ in $(M,\omega)$ and $(T^*_\delta L, dp \wedge dq)$, where $T^*_{\delta}L=\{ (q,p) \in T_q^*L : \, \, || p ||<\delta \}$. We call $\Phi$ the Weinstein parametrization. The zero section $L \hookrightarrow T^*_\delta L$ corresponds under $\Phi$ to the embedding $f:L \to M$. More generally, for any open subset $U \subset L$ and any function $h:U  \to \bR$ such that $|| dh ||< \delta$, the section $dh: U\to T^*_\delta L$ corresponds under $\Phi$ to a regular Lagrangian embedding $f_h: U \to M$ which is graphical over $f|_U$.  

\begin{figure}[h]
\includegraphics[scale=0.65]{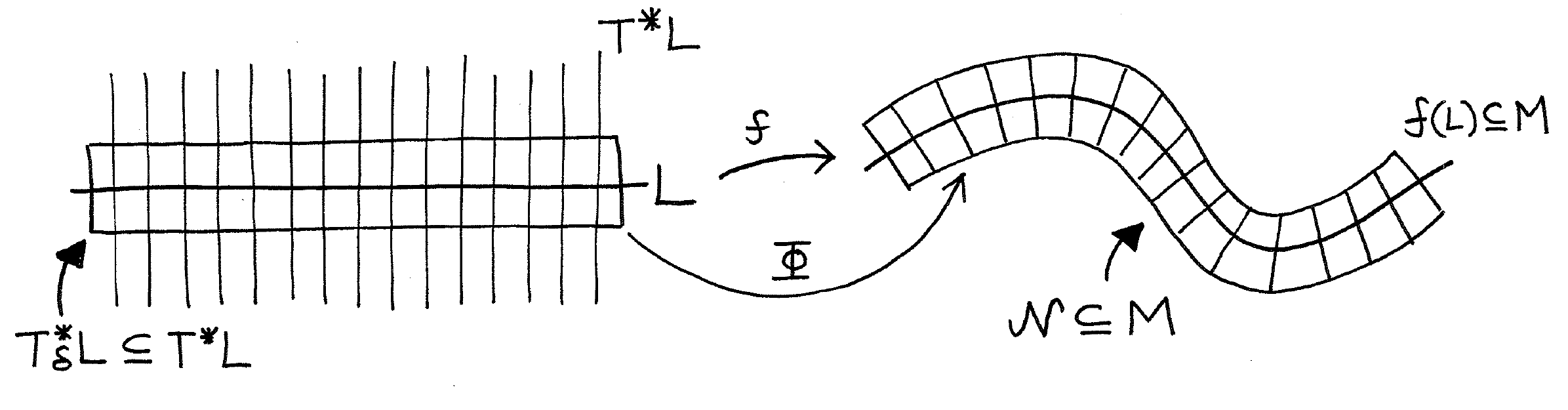}
\caption{A Weinstein neighborhood $\cN$ of $f(L)$ in $M$.}
\label{weinstein}
\end{figure}

Similarly, if $f:L \to M$ is a regular Legendrian embedding, then for some $\delta>0$ small enough there exists a contactomorphism $\Phi$ between a neighborhood $\cN$ of $f(L)$ in $(M, \xi)$ and $J^1_{\delta}(L, \bR)=T^*_{\delta}L \times (-\delta, \delta)$, which is equipped with the standard contact structure. We still call $\Phi$ the Weinstein parametrization. For any open subset $U \subset L$ and any function $h:U \to \bR$ such that $|h|<\delta$ and $|| dh || < \delta$, we obtain a regular Legendrian embedding $f_h:U \to M$ which is graphical over $f|_U$. The embedding $f_h$ corresponds under $\Phi$ to the section $j^1(h):U \to J^1_{\delta}(L,\bR)$. 

In order to capture the tangential information contained in $1$-jets we must  consider $2$-jets. The Riemannian metric fixed on $L$ induces the following trivialization of the $2$-jet space $J^2(L, \bR)$.
\[ J^2(L, \bR) = \{(q,z,p,Q), \quad q \in L, \, \, \,  z\in \bR,\, \, \, p:T_qL \to \bR,\, \, \, Q: T_qL \to \bR\}, \]
where $p$ is a linear form and $Q$ is a quadratic form. Explicitly, given a germ of a function $h:Op(q) \subset L \to \bR$, we set $j^2(h)(q)=\big(q,h(q),dh_q,\text{Hess}(h)_q\big) \in J^2(L, \bR)$. We obtain a vector bundle $J^2(L, \bR) \to L$, where the linear structure is induced by the above trivialization.

\begin{example}
When $L=\bR^n$ with the standard Euclidean metric and standard coordinates $q=(q_1, \ldots, q_n)$, we have a canonical identification $T_q\bR^n \simeq \bR^n$ for each $q \in \bR^n$. Under this identification, $dh(v)=\sum_{i=1}^n (\partial h / \partial q_i)v_i$ and $\text{Hess}(h)(v)= \sum_{i,j=1}^n ( \partial^2 h / \partial q_i \partial q_j) v_i v_j$ for all $v=(v_1, \ldots , v_n) \in \bR^n$. 
\end{example}
\begin{definition} A $2$-jet rotation of $L$ is a compactly supported deformation $s_t : L \to J^2(L, \bR)$, $ t \in [0,1]$, of the zero section $s_0=0$ which is of the form $s_t(q)=\big(q,0,0,Q_t(q) \big)$ for some family of quadratic forms $Q_t:TL \to \bR$.
\end{definition}
In other words, a $2$-jet rotation is a deformation of the zero section whose $1$-jet component is  zero at all times. The corresponding notion of simplicity for $2$-jet rotations is the following.
\begin{definition} A $2$-jet rotation $s_t:L \to J^2(L, \bR)$  is simple if there exists a hyperplane field $H \subset TL$ defined along an open subset $U \subset L$ containing $\text{supp}(s_t)$ such that $H \subset \text{ker}(Q_t)$ for all $t \in [0,1]$. We say that $s_t$ is simple with respect to $H$.
\end{definition}

\begin{remark}
Observe in particular that $Q_t$ has rank $\leq 1$. However, the condition of simplicity is stronger, we demand that the kernel always contains a fixed $(n-1)$-dimensional distribution. See Figure \ref{simple} for an illustration of $2$-jet simplicity.
\end{remark}

\begin{figure}[h]
\includegraphics[scale=0.65]{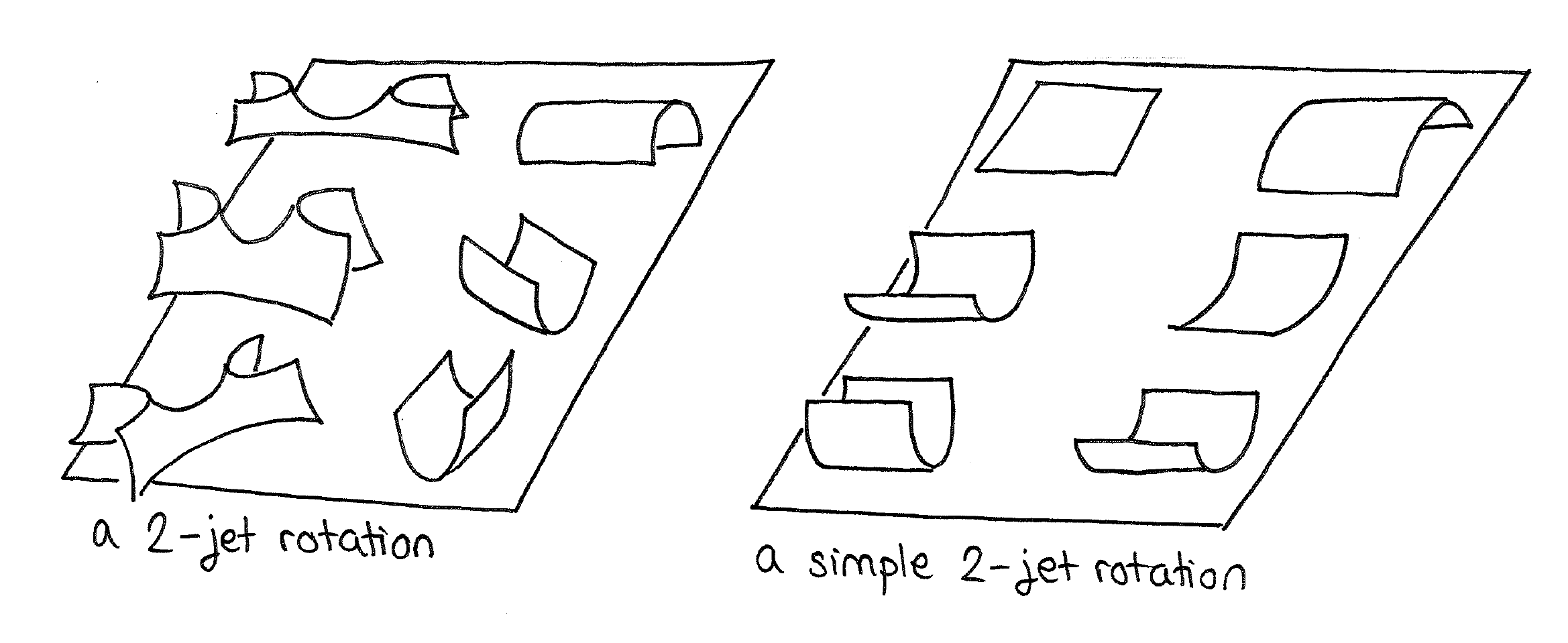}
\caption{The difference between a non-simple $2$-jet rotation and a simple $2$-jet rotation.}
\label{simple}
\end{figure}

In the same vein, we say that a $2$-jet rotation $s_t:L \to J^2(L, \bR)$ is piecewise simple if there exists a subdivision $0=t_0<\cdots<t_k=1$ of the time interval $[0,1]$ such that on each subinterval  $[t_j,t_{j+1}]$ we have $s_t=s_{t_j}+r^j_t$ for some simple $2$-jet rotation $r^j_t: L \to J^2(L, \bR)$. 
\begin{remark}
The proper language for this discussion would naturally extend our definitions  to include the concepts of $l$- and $\perp$-holonomic sections of the $r$-jet bundle associated to any fibre bundle. These ideas were introduced by Gromov in \cite{G86} in the context of convex integration. We explore these notions further in the context of holonomic approximation in our paper \cite{AG15}, the results of which will be crucially used below.
\end{remark}

Given a regular Lagrangian or Legendrian embedding $f:L \to M$, a Weinstein parametrization $\Phi$ of a neighborhood $\cN$ of $f(L)$ in $M$ and a $2$-jet rotation $s_t:L \to J^2(L, \bR)$, we can define a tangential rotation $G(\Phi,s_t) :L \to \Lambda_n(M)$ of $f$ associated to $\Phi$ and $s_t$. Explicitly, we set $G(\Phi,s_t)(q)=G(df_{h_t})(q)$ at each point $q \in L$, where $h_t:Op(q) \subset L \to \bR$ is any function germ such that $j^2(h_t)(q)=s_t(q)$ and $f_{h_t}:Op(q) \subset L \to M$ is the Lagrangian or Legendrian embedding corresponding to $h_t$ under $\Phi$. Observe that if $s_t$ is simple, then $G(\Phi,s_t)$ is also simple. 

Conversely, given a regular Lagrangian or Legendrian embedding $f:L \to M$, a Weinstein parametrization $\Phi$ and a tangential rotation $G_t: L \to \Lambda_n(M)$ of $f$, there exists a unique $2$-jet rotation $s_t:L \to J^2(L, \bR)$ such that $G(\Phi,s_t)=G_t$. To be more precise, $s_t$ might only be defined in a small time interval $[0, \varepsilon] \subset [0,1]$, since the Lagrangian planes $G_t(q)$ could at some point stop being graphical over $df(T_qL)$ with respect to $\Phi$, see Figure \ref{graphical}.
\begin{definition}\label{simple definition} When $s_t$ is defined for all $t \in [0,1]$, we say that $G_t$ is graphical.
\end{definition}
 The Weinstein parametrization $\Phi$ is implicit in the definition. Observe again that if $G_t$ is simple, then $s_t$ is also simple. The notions of piecewise simplicity also coincide under this correspondence.

\begin{figure}[h]
\includegraphics[scale=0.65]{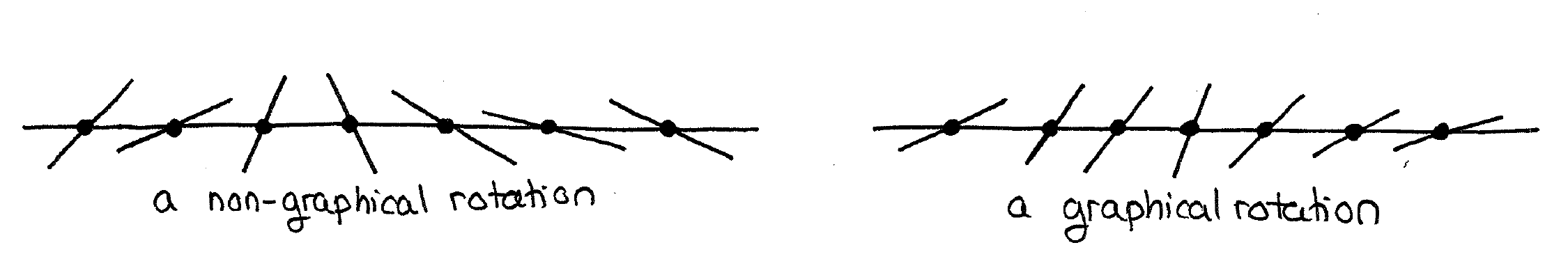}
\caption{The difference between a graphical and a non-graphical tangential rotation.}
\label{graphical}
\end{figure}

\subsection{Approximation by simple rotations}
Let $I^n=[-1,1]^n$ denote the unit $n$-dimensional cube. The following lemma will allow us to replace any tangential rotation of a Lagrangian or Legendrian embedding by a piecewise simple tangential rotation.

\begin{lemma}\label{approximation lemma} Let $s_t:I^n \to J^2(\bR^n,\bR)$ be a $2$-jet rotation such that $s_t=0$ on $Op(\partial I^n)$. Then there exists a piecewise simple $2$-jet rotation $r_t:I^n \to J^2(\bR^n,\bR)$ which is $C^0$-close to $s_t$ and such that $r_t=0$ on $Op(\partial I^n)$.
\end{lemma}

Lemma \ref{approximation lemma} is an immediate consequence of a more general approximation result which we prove in \cite{AG15}. For completeness we present below the outline of the argument in our concrete setting. The idea goes back to Gromov's iterated convex hull extensions in \cite{G86}, which used similar decompositions into so-called principal subspaces. Indeed, in convex integration one is also forced to work one pure partial derivative at a time. These decompositions are studied carefully in Spring's book \cite{S98}. 

For our purposes, we only need to remark that any homogeneous degree $2$ polynomial can be written as a sum of squares of linear polynomials. Explicitly, we have the polynomial identity ${X_iX_j=\frac{1}{2}\big((X_i+X_j)^2-X_i^2-X_j^2\big)}$. We can think of a $2$-jet rotation as a parametric family of Taylor polynomials which are homogeneous of degree $2$. By applying the above identity we obtain a decomposition $s_t=\sum r^{i,j}_t$, where the $2$-jet rotation $r^{i,j}_t$ is simple with respect to the hyperplane field $\tau_{i,j}=\ker(dq_i+dq_j)$ and the sum is taken over all $1 \leq i \leq j \leq n$. Moreover, it follows that if $s_t=0$ on $Op(\partial I^n)$, then $r^{i,j}_t=0$ on $Op(\partial I^n)$ for all $i,j$.  

\begin{figure}[h]
\includegraphics[scale=0.5]{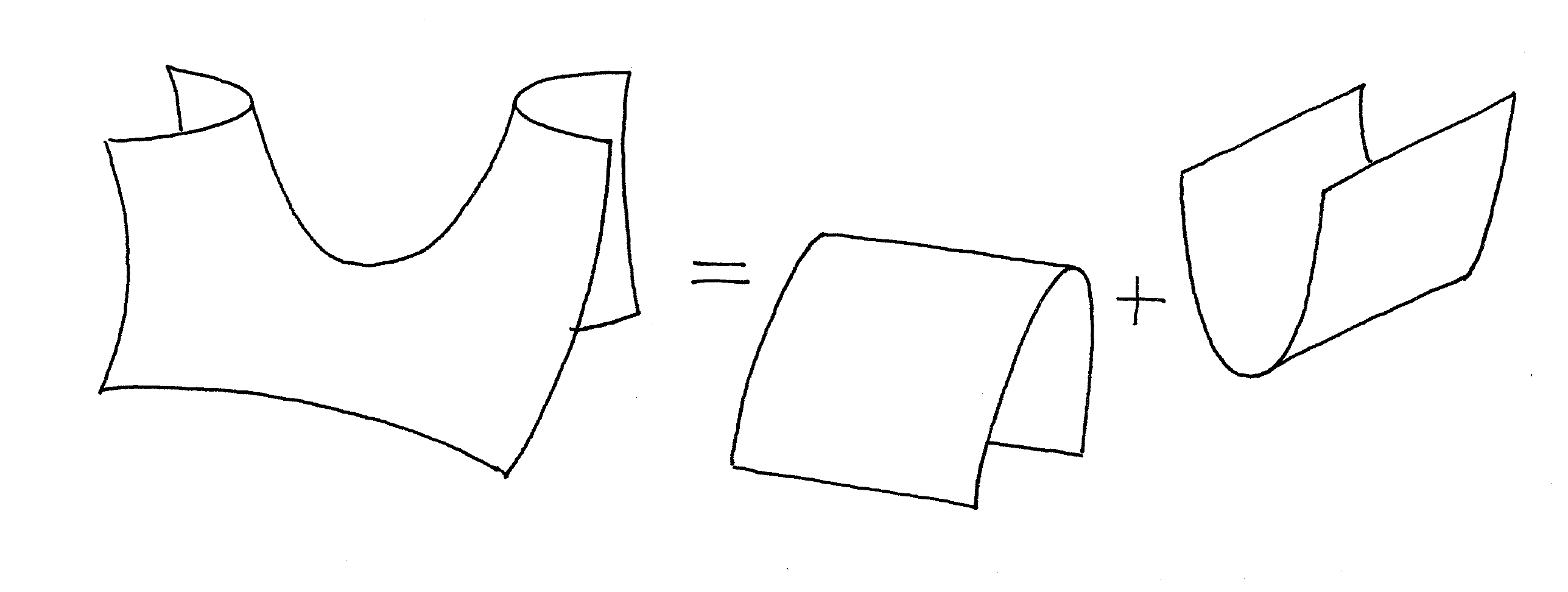}
\caption{Decomposing a homogeneous degree $2$ polynomial into a sum of squares of linear polynomials.}
\label{decomposition}
\end{figure}

Once we have this decomposition, we can subdivide the interval $[0,1]$ very finely and add a fraction of each $r^{i,j}_t$ at a time to obtain the desired piecewise simple approximation of $s_t$. The parametric version is proved in the exact same way. The statement reads as follows. 
\begin{lemma}\label{parametric approximation lemma} Let $s^z_t:I^n \to J^2(\bR^n,\bR)$ be a family of $2$-jet rotations parametrized by a compact manifold $Z$ such that $s^z_t=0$ on $Op(\partial I^n)$ and such that $s^z_t=0$ for $z \in Op(\partial Z)$. Then there exists a family of piecewise simple $2$-jet rotations $r^z_t:I^n \to J^2(\bR^n,\bR)$ which is $C^0$-close to $s^z_t$, such that $r^z_t=0$ on $Op(\partial I^n)$ and such that $r^z_t=0$ for $z \in Op(\partial Z)$.
\end{lemma}

To be more precise, for the piecewise simple family we demand that there exists a single subdivision $0=t_0< \cdots < t_k=1$ of the time interval $[0,1]$ such that  every $r^z_t$ is simple on each piece $[t_j,t_{j+1}]$. We can translate Lemmas \ref{approximation lemma} and \ref{parametric approximation lemma} from the world of jet spaces back into the world of symplectic and contact topology. The precise consequence that we wish to extract is the following.

\begin{proposition}\label{decomposition proposition} Let $G_t:L \to \Lambda_n(M)$ be a  tangential rotation of a regular Lagrangian or Legendrian embedding $f:L \to M$. Then we can $C^0$-approximate $G_t$ as much as desired by a piecewise simple tangential rotation $R_t:L \to \Lambda_n(M)$.
\end{proposition}

\begin{proof} By using a partition of unity and a fine enough subdivision $0=t_0<\cdots < t_k=1$ of the interval $[0,1]$, we can localize in space and time to obtain a tangential rotation $\widetilde{G}_t:L \to \Lambda_n(M)$ which is  $C^0$-close to $G_t$ and such that on each subinterval $[t_j,t_{j+1}]$ the rotation $\widetilde{G}_t$ is constant outside of some ball $B_j \subset L$. In the Lagrangian case, let $\Phi_j$ be a symplectic isomorphism of the symplectic vector bundle $(TM|_{f(B_j)} ,\omega) \to B_j$ such that $\Phi_j \cdot G(df)   = \widetilde{G}_{t_j}$. In the Legendrian case, we ask that $\Phi_j$ satisfies the same property but is instead a symplectic isomorphism of the symplectic vector bundle $(\xi|_{f(B_j)}, d \alpha) \to B_j$, where $ \xi = \ker(\alpha)$ on the ball $B_j$. 

Consider the tangential rotation $S^j_t= (\Phi_j)^{-1} \cdot \widetilde{G}_t$, $t \in [t_j, t_{j+1}]$. Observe that $S^j_t=G(df)$ on $Op(\partial B_j)$. By further subdividing the time interval if necessary, we may assume that $S^j_t$ is graphical. In other words, $S^j_t$ corresponds to a $2$-jet rotation $s^j_t:B_j \to J^2(B_j, \bR)$ such that $s^j_t=0$ on $Op(\partial B_j)$. Lemma \ref{approximation lemma} asserts the existence of a piecewise simple $2$-jet rotation $r^j_t:B_j \to J^2(B_j, \bR)$ which is $C^0$-close to $s^j_t$ and such that $r^j_t=0$ on $Op(\partial B_j)$.  We obtain a corresponding piecewise simple tangential rotation $R^j_t:B_j \to \Lambda_n(M)$ which is $C^0$-close to $S^j_t$ and such that $R^j_t=G(df)$ on $Op(\partial B_j)$. Set $R_t=\Phi_j \cdot R^j_t $, $t \in [t_j, t_{j+1}]$ on $B_j$. Outside of $B_j$ we extend by setting $R_t=\widetilde{G}_t$, which is constant for $t \in [t_j,t_{j+1}]$. This piecewise definition yields a tangential rotation $R_t:L \to \Lambda_n(M)$, $t \in [0,1]$, where each piece $R_t|_{[t_j,t_{j+1}]}$ is itself a piecewise simple tangential rotation. Hence $R_t$ is also a piecewise simple tangential rotation. Moreover, $R_t$ is  everywhere $C^0$-close to $\widetilde{G}_t$, hence also to $G_t$. \end{proof}

\begin{remark}
From the proof we can also deduce the relative version of Proposition \ref{decomposition proposition}. If $G_t=G(df)$ on $Op(A)$ for some closed subset $A \subset L$, then we can arrange it so that $R_t=G(df)$ on $Op(A)$.

\end{remark}

The parametric version is proved in the same way. The corresponding relative version also holds. As in the case of $2$-jet rotations, by a family of piecewise simple tangential rotations we mean a family of tangential rotations such that for some subdivision $0=t_0<\cdots < t_k=1$ of the time interval $[0,1]$, every tangential rotation of the family is simple on each subinterval $[t_j,t_{j+1}]$. The precise statement that we will need reads as follows.

\begin{proposition}\label{parametric decomposition proposition} Let $G^z_t:L \to \Lambda_n(M)$ be a family of tangential rotations of regular Lagrangian or Legendrian embeddings $f^z:L \to M$ parametrized by a compact manifold $Z$ such that $G^z_t=G(df^z)$ for $z \in Op(\partial Z)$. Then we can $C^0$-approximate the family $G^z_t$ as much as desired by a family of piecewise simple tangential rotations $R^z_t:L \to \Lambda_n(M)$ such that $R^z_t=G(df^z)$ for $z \in Op(\partial Z)$.
\end{proposition}

\section{Wiggling embeddings}\label{Holonomic approximation with controlled cutoff}

\subsection{Regular approximation near the $(n-1)$-skeleton} Let $G_t:L \to \Lambda_n(M)$ be a tangential rotation of a regular Lagrangian or Legendrian embedding $f:L \to M$. It is in general impossible to globally $C^0$-approximate $G_t$ by the Gauss maps $G(df_t)$ of an exact homotopy of regular Lagrangian or Legendrian embeddings $f_t:L \to M$, $f_0=f$. However, it is always possible to achieve this approximation in a wiggled neighborhood of any reasonable subset of $L$ which has positive codimension, see Figure \ref{approximationnearK}. For simplicity, we will restrict ourselves to the following class of stratified subsets.
\begin{definition} A closed subset $K \subset L$ is called a polyhedron if it is a subcomplex of some smooth triangulation of $L$. 
 \end{definition}
In \cite{AG15} we prove several refinements of the holonomic approximation lemma. The following result is a straightforward application of our holonomic approximation lemma for $l$-holonomic sections. 

\begin{theorem}\label{1-holonomic} Let $K \subset L$ be a polyhedron of positive codimension and let $G_t: L \to \Lambda_n(M)$ be a tangential rotation of a regular Lagrangian or Legendrian embedding $f:L \to M$. Then there exists an exact homotopy of regular Lagrangian or Legendrian embeddings $f_t:L \to M$, $f_0=f$, such that $G( df_t )$ is $C^0$-close to $G_t$ on $Op(K) \subset L$.

\end{theorem}

\begin{figure}[h]
\includegraphics[scale=0.7]{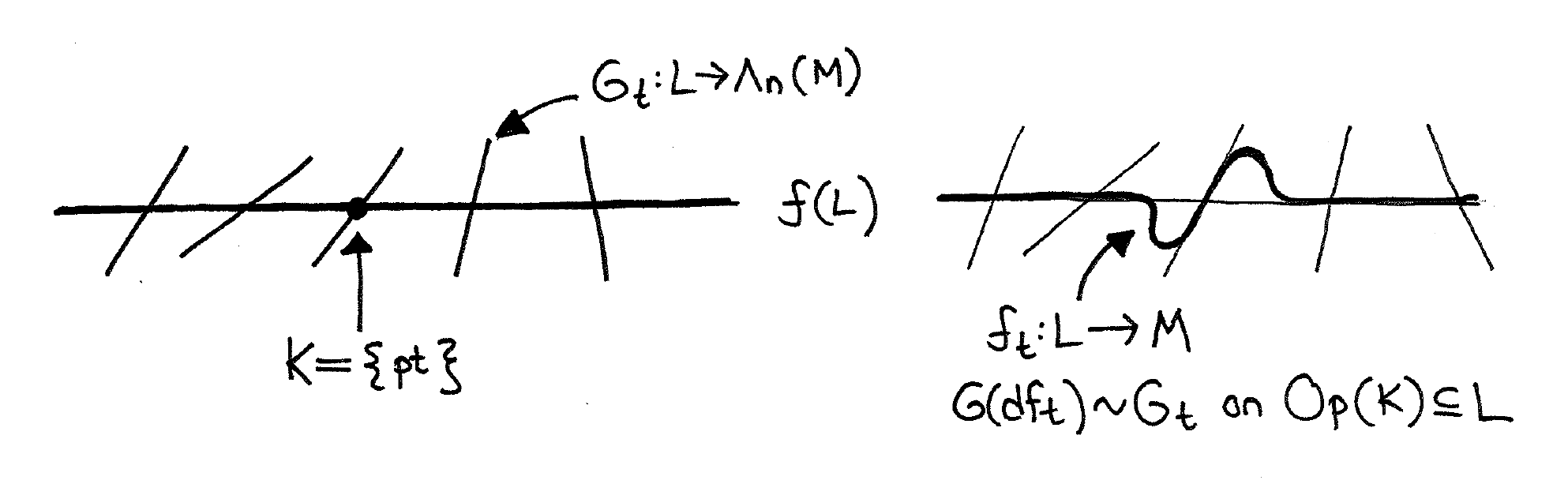}
\caption{We can always approximate $G_t$ by Gauss maps $G(df_t)$ in a neighborhood of any reasonable subset $K \subset L$ of positive codimension.}
\label{approximationnearK}
\end{figure}

 \begin{remark} We can arrange it so that $f_t$ is $C^0$-close to $f$ on all of $L$ and so that $f_t=f$ outside of a slightly bigger neighborhood of $K$ in $L$. Moreover, the result also holds in relative and parametric forms. 
 \end{remark}
\begin{remark} As far as the author can tell, Theorem \ref{1-holonomic} does not follow formally from Eliashberg and Mishachev's holonomic approximation lemma \cite{EM02} or from any of the other standard $h$-principle techniques. However it does follow immediately from the holonomic approximation lemma for $l$-holonomic sections which we established in \cite{AG15}. The added difficulty stems from the pervasive danger of cutoffs in symplectic topology.  \end{remark}
 
In Section \ref{Wrinkling embeddings} we will prove that \emph{any tangential rotation $G_t$ can be globally $C^0$-approximated by the Gauss maps $G(df_t)$ of an exact homotopy of wrinkled Lagrangian or Legendrian embeddings $f_t$}. This is the main technical ingredient in the proof of the $h$-principle for the simplification of singularities in Section \ref{applications to the simplification of singularities} below. In the course of the proof of this global $C^0$-approximation theorem we will need to use a result of the same flavour as Theorem \ref{1-holonomic}, taking $K$ to be the $(n-1)$-skeleton of a smooth triangulation of $L$. The idea is to construct the homotopy $f_t$ by first wiggling $f$ near the $(n-1)$-skeleton. Then one can apply a wrinkling construction on each of the top dimensional simplices to complete the approximation. 

However, on the nose Theorem \ref{1-holonomic} is not quite sufficient for our purposes. The issue is that the local wrinkling model which we construct in Section \ref{Wrinkling embeddings} can only be applied if the tangential rotation is simple. Initially this is not a problem because we can use Proposition \ref{decomposition proposition} to first approximate any given rotation by a piecewise simple rotation. We can then attempt to deal with each simple piece in the decomposition separately, working step by step. Unfortunately the following additional and subtler difficulty arises. When at each step we apply Theorem \ref{1-holonomic} near the $(n-1)$-skeleton of $L$ we might find that our fixed decomposition is  no longer piecewise simple with respect to the wiggled embedding $f_1$. If this is the case, then we cannot continue on to the next step. To fix this issue we need a stronger version of Theorem \ref{1-holonomic} which allows us to control the wiggles with respect to a given simple tangential rotation. We phrase and prove this stronger version in the next section, stated as Theorem \ref{perp-holonomic}.

\subsection{Keeping things simple} The precise result that we need is the following application of our holonomic approximation lemma for $\perp$-holonomic sections from \cite{AG15}. The choice of a Riemannian metric on $L$ and a Weinstein parametrization of a neighborhood of $f(L) $ in $M$ is implicit throughout. We use the language of $2$-jet rotations introduced in Section \ref{Rotations of $2$-jets}.

\begin{theorem}\label{perp-holonomic} Let $K \subset L$ be a polyhedron of positive codimension and let $G_t: L \to \Lambda_n(M)$ be a graphical simple tangential rotation of a regular Lagrangian or Legendrian embedding $f:L \to M$. Then there exists a graphical simple tangential rotation $R_t:L \to \Lambda_n(M)$ of $f$  and an exact homotopy of regular Lagrangian or Legendrian embeddings $f_t:L \to M$, $f_0=f$, such that the following properties hold.
\begin{itemize} 
\item $G(df_t)$ is $C^0$-close to $G_t$ on $Op(K) \subset L$
\item $G(df_t)$ is $C^0$-close to $R_t$ on all of $L$.
\item $R_t$ is simple with respect to the same hyperplane field as $G_t$.
\item $f_t=f$ and $R_t=G(df)$ outside of a slightly bigger neighborhood of $K$ in $L$.
\end{itemize}

\end{theorem}  
\begin{remark} Observe that the second property implies that $f_t$ is everywhere $C^0$-close to $f$. 
\end{remark}
\begin{remark}\label{relative perp-holonomic} The relative form of Theorem \ref{perp-holonomic} also holds. If $G_t=G(df)$ on $Op(A)$ for some closed subset $A \subset L$, then we can arrange it so that $f_t=f$ and $R_t=G(df)$ on $Op(A) \subset L$. 
\end{remark}

\begin{proof}
By definition of graphicality, we can think of $G_t$ as a $2$-jet rotation $s_t:L \to J^2(L, \bR)$ which is simple with respect to some hyperplane field $H \subset TL$. We can therefore apply the ($1$-parametric) holonomic approximation lemma for $\perp$-holonomic sections from \cite{AG15} to $s_t$. The output is a family of functions $h_t:L \to \bR$, $h_0=0$ and an isotopy $F_t:L \to L$ such that the following properties hold.
\begin{itemize}
\item $j^2(h_t)$ is $C^0$-close to $s_t$ on $Op\big(F_t(K) \big) \subset L$.
\item $j^1(h_t)$ is $C^0$-small on all of $L$.
\item $\text{Hess}(h_t)|_H$ is $C^0$-small on all of $L$.
\item $F_t$ is $C^0$-small.
\item $F_t^*H$ is $C^0$-close to $H$.
\item $h_t=0$ and $F_t=id_L$ outside of a slightly bigger neighborhood of $K$ in $L$.\end{itemize}

 The $C^1$-smallness of $h_t$ allows us to think of $dh_t \circ F_t:L \to T^*L$ (in the Lagrangian case) or of $j^1(h_t) \circ F_t:L \to J^1(L, \bR)$ (in the Legendrian case) as an exact homotopy of regular Lagrangian or Legendrian embeddings $f_t:L \to M$. We define the simple tangential rotation $R_t:L \to \Lambda_n(M)$ by specifying its corresponding simple $2$-jet rotation $r_t:L \to J^2(L, \bR)$ as follows. Write $r_t(q)=\big(q,0,0,Q_t(q) \big) \in J^2(L, \bR)$ for $Q_t:TL \to \bR$ a family of quadratic forms and set $Q_t=\text{Hess}(h_t)\circ p:TL \to \bR$ to obtain the desired $R_t$, where $p:TL \to TL$ is the orthogonal projection with kernel $H$. All the properties listed in Theorem \ref{perp-holonomic} hold.
\end{proof}

The condition that $F_t^*H$ is $C^0$-close to $H$ was not used in this proof but will be used below so we include it for future reference. The above argument also works for families. In the parametric case, we note that the polyhedron $K$ may also vary with the parameter. To be more precise, we have the following definition.
\begin{definition} A closed subset $K \subset Z \times L$ is called a fibered polyhedron if it is a subcomplex of a smooth triangulation of $Z \times L$ which is in general position with respect to the fibres $z \times L$, $z \in Z$.
\end{definition}

A consequence of this definition is that for every $z \in Z$ the subset $K^z \subset L$ given by $K \cap \left(z \times L\right)= z  \times K^z$ is a polyhedron in $L$. If $K$ has positive codimension in $Z \times L$, then $K^z$ has positive codimension in $L$ for all $z \in Z$. The parametric version of Theorem \ref{perp-holonomic} is proved in the same way, by adding a parameter in the notation everywhere and invoking our parametric holonomic approximation lemma for $\perp$-holonomic sections from \cite{AG15}. The statement reads as follows. We note that the relative version also holds, as in Remark \ref{relative perp-holonomic}.

\begin{theorem}\label{parametric perp-holonomic} Let $K \subset Z \times L$ be a fibered polyhedron of positive codimension and let $G^z_t: L \to \Lambda_n(M)$ be a family of graphical simple tangential rotations of regular Lagrangian or Legendrian embeddings $f^z:L \to M$ parametrized by a compact manifold $Z$ such that $G^z_t=G(df^z)$ for $z \in Op(\partial Z)$. Then there exists a family of graphical simple tangential rotations $R^z_t:L \to \Lambda_n(M)$ of $f^z$ and a family of exact homotopies of regular Lagrangian or Legendrian embeddings $f^z_t:L \to M$, $f^z_0=f^z$, such that the following properties hold.
\begin{itemize} 
\item $G(df^z_t)$ is $C^0$-close to $G^z_t$ on $Op(K^z) \subset L$.
\item $G(df^z_t)$ is $C^0$-close to $R^z_t$ on all of $L$.
\item $R^z_t$ is simple with respect to the same hyperplane field as $G^z_t$.
\item $f^z_t=f^z$ and $R^z_t=G(df^z)$ outside of a slightly bigger neighborhood of $K^z$ in $L$.
\item $f^z_t=f^z$ and $R^z_t=G(df^z)$ for $z \in Op(\partial Z)$.
\end{itemize}
\end{theorem} 
\subsection{Wiggling the wrinkles}\label{wiggling the wrinkles} In this section we extend Theorems \ref{perp-holonomic} and \ref{parametric perp-holonomic}, which were stated for regular Lagrangian or Legendrian embeddings, to the case of wrinkled Lagrangian or Legendrian embeddings. In the wrinkled case, we cannot invoke our holonomic approximation lemma for $\perp$-holonomic sections from \cite{AG15} directly because a wrinkled Lagrangian or Legendrian embedding is not regular near the wrinkles. The sharpening construction described in section \ref{sharpening the wrinkles} will allow us to resolve this issue, since the sharper the wrinkles, the better they can be approximated locally by a regular Lagrangian or Legendrian submanifold..

Given a wrinkled Lagrangian or Legendrian embedding $f:L \to M$, recall that the subset on which $f$ is wrinkled consists of a disjoint union $W=\bigcup_j S_j$ of finitely many $(n-1)$-dimensional embedded spheres $S_j \subset L$. Each sphere $S_j$ has an $(n-2)$-dimensional equator $E_j \subset S_j$ on which $f$ has birth/deaths of zig-zags. The complement $S_j \setminus E_j$ consists of two hemispheres on which $f$ has cusps. 

We say that a polyhedron $K \subset L$ is compatible with the wrinkles of $f$ if the following condition holds. We demand that the wrinkling locus $W=\bigcup_j S_j$ is contained in the $(n-1)$-skeleton of $K$ and that the union of the equators $\bigcup_j E_j$ is contained in the $(n-2)$-skeleton of $K$. In the same way we can define what it means for a fibered polyhedron $K \subset Z \times L$ to be compatible with the fibered wrinkles of a family $f^z:L \to M$ of wrinkled Lagrangian or Legendrian embeddings parametrized by a compact manifold $Z$. 

We now prove the analogue of Theorem \ref{perp-holonomic} for wrinkled Lagrangian and Legendrian embeddings. The precise statement is the following. 

\begin{theorem}\label{perp-holonomic for wrinkles} Let $K \subset L$ be a polyhedron of positive codimension which is compatible with the wrinkles of a wrinkled Lagrangian or Legendrian embedding $f:L \to M$. Let $G_t:L \to \Lambda_n(M)$ be a graphical simple tangential rotation of $f$. Then there exists a graphical simple tangential rotation $R_t:L \to \Lambda_n(M)$ of $f$ and a family of exact homotopies $f_t:L \to M$, $f_0=f$, of wrinkled Lagrangian or Legendrian embeddings such that all of the properties listed in Theorem \ref{perp-holonomic} hold. 
\end{theorem}

\begin{proof} Consider first a single wrinkle $S$ in the wrinkling locus $W \subset L$ of $f$. The wiggling on $S$ is performed in two steps. First we will wiggle $f$ near the equator $E \subset S$ and then we will wiggle $f$ near the remaining part of $S$. In both cases this wiggling  is achieved by replacing the singular Lagrangian or Legendrian submanifold $f(L)$ with a regular approximation to which holonomic approximation can be applied. We then use the resulting ambient Hamiltonian isotopy to induce a wiggling of $f$. We will restrict our attention to the Lagrangian case for the sake of concreteness, but the Legendrian case is no different.

In Section \ref{sharpening the wrinkles} we introduced the Lagrangian local model $\cG_n$ for the birth/death of zig-zags. Recall that $\cG_n: S^{n-2} \times \bR^2 \to T^*(S^{n-2} \times \bR^2)$ is given by
\[ \cG_n(\widetilde{q},q_{n-1} , q_n) = \Big(\widetilde{q}, q_{n-1} , \tau ,0, \frac{\partial G}{\partial q_{n-1}} - g \frac{\partial \tau}{\partial q_{n-1}} , g\Big) , \quad q=(\widetilde{q},q_{n-1},q_{n}) \in S^{n-2} \times \bR \times \bR.\]
\[ \text{where} \quad \tau(q_{n-1},q_n)= q_n^3-3q_{n-1}q_n ,\qquad  \qquad g(q_{n-1},q_n)= \int_0^{q_n} (u^2-q_{n-1})^2 du \]
\[ \text{and} \qquad G(q_{n-1},q_n)=\int_0^{q_n} g(q_{n-1}, u) \frac{\partial \tau}{\partial q_n}(q_{n-1}, u) \, du. \]

Near the equator $E \subset S$, our wrinkled Lagrangian embedding $f:L \to M$ is locally equivalent to $\cG_n$ near $S^{n-2} \times 0 \subset S^{n-2} \times \bR^2$. Working in this local model, we can think of $G_t$ as a tangential rotation of $\cG_n$ which is simple with respect to an $(n-1)$-plane field $H \subset T\big(T^*(S^{n-2}\times \bR^2)\big)$. Consider the zero section $\cZ:S^{n-2} \times \bR^2 \to T^*(S^{n-2} \times \bR^2)$, which is a Lagrangian cylinder. Observe that $\cZ|_{S^{n-2} \times 0}=\cG_n|_{S^{n-2} \times 0}$, and moreover that $G(d\cZ)|_{S^{n-2} \times 0}=G(d\cG_n )|_{S^{n-2} \times 0 }$. Extend $G_t|_{S^{n-2} \times 0}$ to $S_t:S^{n-2} \times \bR^2 \to \Lambda_n\big(T^*(S^{n-2} \times \bR^2) \big)$, a tangential rotation of $\cZ$ which is simple with respect to $H$. 

Let $\delta>0$ and set $\cN=S^{n-2} \times (-\delta,\delta)^2 \subset S^{n-2} \times \bR^2$. Apply Theorem \ref{perp-holonomic} to the regular Legendrian embedding $\cZ$, the simple tangential rotation $S_t$ and the stratified subset $S^{n-2} \times 0 \subset S^{n-2} \times \bR^2$. We obtain an exact homotopy of regular Legendrian embeddings $\cZ_t:S^{n-2} \times \bR^2 \to T^*(S^{n-2} \times \bR^2)$ which we may assume is constant outside of $\cN$. Recall that $G(d\cZ_t)$ is $C^0$-close to $G_t$ near $S^{n-2} \times 0$. Recall also that $G(d\cZ_t)$ is everywhere $C^0$-close to a tangential rotation $R_t$ which is also simple with respect to $H$ and which is supported on $\cN$.

Write $\cZ_t=\varphi_t \circ \cZ$ for an ambient Hamiltonian isotopy $\varphi_t$ which we may assume constant outside of $Op(\cN) \subset T^*(S^{n-2} \times \bR^2)$. Let $\varepsilon>0$ and consider the sharpening $\cG_{n,t}$ of $\cG_n$ described in Section \ref{sharpening the wrinkles} with respect to the parameters $\delta$ and $\varepsilon$. Recall that $\text{dist}_{C^1}(\cG_n,\cG_{n,t}) \leq A \delta$ for some constant $A>0$ independent of $\delta$ and $\varepsilon$, so by taking $\delta>0$ small enough we may replace $\cG_n$ with $\cG_{n,1}$ from the onset up to an error which is proportional to $\delta$. Recall also that sharpening is supported on $S^{n-2} \times (-2\delta, 2\delta)^2$ and is $\varepsilon$-sharp on $\cN= S^{n-2} \times (\delta, \delta)$. For details see Section \ref{sharpening the wrinkles}.
\begin{figure}[h]
\includegraphics[scale=0.6]{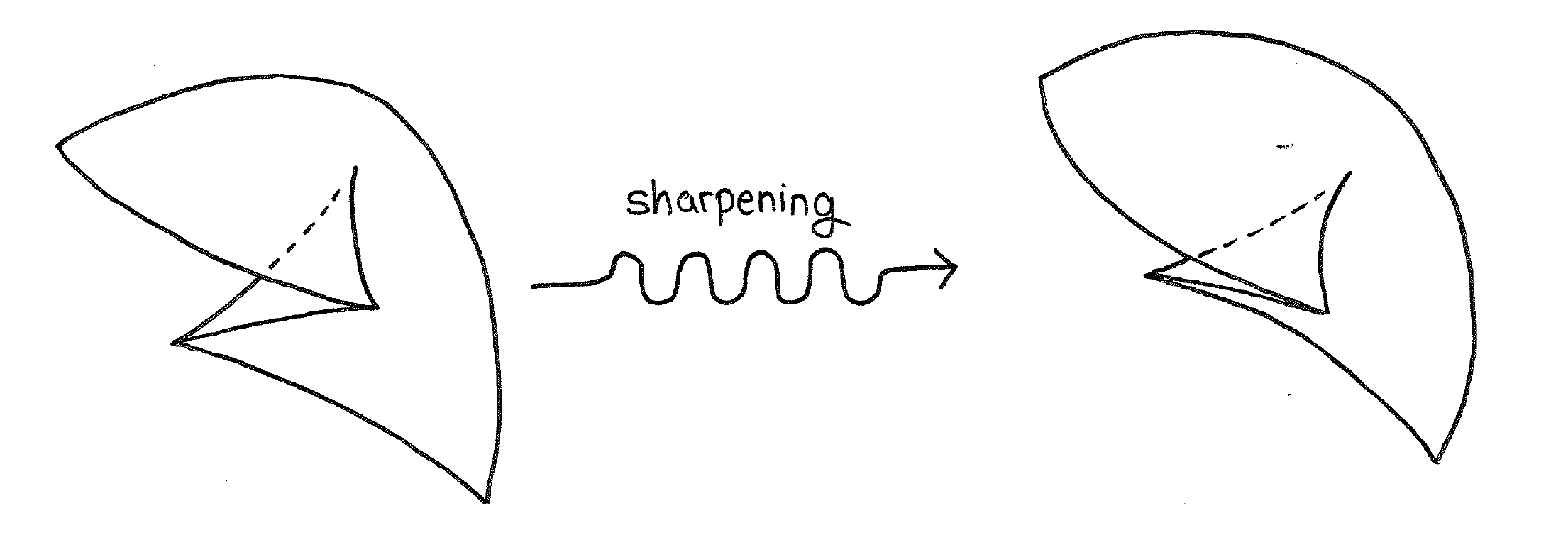}
\caption{The sharpening construction applied to the equator.}
\label{sharpeningbirthagain}
\end{figure}

Consider now $\varphi_t \circ \cG_{n,1}$. Note that on $Op( S^{n-2} \times 0)$ the Gauss map of this composition is $C^0$-close to $G_t$. Indeed, $\cG_{n,1}$ and $\cZ$ are tangent along $S^{n-2} \times 0$ and when we invoke Theorem \ref{perp-holonomic} to construct $\cZ_t$ we can demand as much accuracy in the approximation as we want. Next, observe that $\cZ_t$ is supported on $\cN$ and on that neighborhood $G(d\cZ_t)$ is $C^0$-close to a tangential rotation which is simple with respect to $H$. Let $\pi:T^*(S^{n-2} \times \bR^2) \to S^{n-2} \times \bR^2$ denote the standard projection. As $\varepsilon \to 0$ in the sharpening $\cG_{n,1}$, for each $q \in \cN$ the tangent plane $G(d\cG_{n,1})(q)$ converges to the horizontal plane tangent to the zero section at the point $\pi(q)$, and hence $G\big(d( \varphi_t \circ \cG_{n,1}) \big)(q)$ converges to $G( d \cZ_t)\big( \pi(q) \big)$. It follows that by taking $\varepsilon>0$ as small as is necessary, we can use $R_t$ to exhibit a tangential rotation of $\cG_{n,1}$ which is simple with respect to $H$ and which is arbitrarily $C^0$-close to $G\big( d ( \varphi_t \circ \cG_{n,1}) \big)$ on all of $L$. We have therefore achieved the required global approximation up to an error which is proportional to $\delta$. Since we can take $\delta>0$ to be arbitrarily small, this completes the wiggling near the equator. 

Once we have wiggled $f$ near the equator $E$ we proceed to wiggle $f$ on the two hemispheres  $D^{\pm}$ of the complement $S\setminus E$. Near the interior of each of the two disks $D^{+}$ and $D^{-}$ the map $f$ is equivalent to the local model $\cC_n: \bR^n \to T^*\bR^n$ on $Op(D^{n-1}) \setminus Op(\partial D^{n-1})$, where we recall from Section \ref{sharpening the wrinkles} that

\[   \cC_n(q_1, \ldots , q_n) = \left( q_1 , \ldots , q_{n-1} , q_n^2, 0, \ldots , q_n^3 \right). \]

Our input this time is a simple tangential rotation $G_t$ of the local model $\cC_n|_{D^{n-1}}$ which we assume to be constant on $Op(\partial D^{n-1})$. The strategy is the same as before. Consider the zero section $\cZ:\bR^n \to T^*\bR^n$ and extend $G_t|_{ D^{n-1}}$ to a simple tangential rotation $S_t$ of $\cZ$. Then apply (the relative version of) Theorem \ref{perp-holonomic} to $\cZ$, $S_t$ and $D^{n-1}$ to obtain an exact homotopy of regular Lagrangian embeddings $\cZ_t$ which is induced by an ambient Hamiltonian isotopy $\varphi_t$ fixing the boundary. For a suitable choice of parameters $\delta$ and $\varepsilon$, the concatenation of the sharpening homotopy $\cC_{n,t}$ described in Section \ref{sharpening the wrinkles} followed by the isotopy $\varphi_t \circ \cC_{n,1}$ gives the required wiggling on $S \setminus E.$

This process can now be repeated on all wrinkles $S$ until we have achieved the desired wiggling on the locus $W$ where $f$ fails to be a regular Lagrangian embedding. The proof of Theorem \ref{perp-holonomic for wrinkles} is completed by applying (the relative version) of Theorem \ref{perp-holonomic} on the regular locus. \end{proof}

The analogue of the parametric Theorem \ref{parametric perp-holonomic} for families of wrinkled Lagrangian or Legendrian embeddings also holds, where we demand that the fibered polyhedron $K \subset Z \times L$ is compatible with the wrinkles. The proof only differs in notation and the precise statement reads as follows.

\begin{theorem}\label{parametric perp-holonomic for wrinkles} Let $K \subset Z \times L$ be a fibered polyhedron of positive codimension which is compatible with the wrinkles of a family of wrinkled Lagrangian or Legendrian embeddings $f^z:L \to M$ parametrized by a compact manifold $Z$. Let $G^z_t:L \to \Lambda_n(M)$ be a family of graphical simple tangential rotations of $f^z$ such that $G^z_t=G(df^z)$ for $z \in Op(\partial Z)$. Then there exists a family of graphical simple tangential rotations $R^z_t:L \to \Lambda_n(M)$ of $f^z$ and an exact homotopy of wrinkled Lagrangian or Legendrian embeddings $f^z_t:L \to M$, $f^z_0=f^z$, such that all of the properties listed in Theorem \ref{parametric perp-holonomic} hold. 
\end{theorem}

\begin{remark}\label{isotopyvhomotopy} Observe that no wrinkles appear or disappear in the homotopies of wrinkled Lagrangian or Legendrian embeddings produced by Theorems \ref{perp-holonomic for wrinkles} and \ref{parametric perp-holonomic for wrinkles}. In \cite{EM09}, Eliashberg and Mishachev refer to such a homotopy as an isotopy of wrinkled embeddings. We like to think of this process as `wiggling'.
\end{remark}

\section{Wrinkling embeddings}\label{Wrinkling embeddings}
\subsection{Wrinkled approximation on the whole manifold}\label{Wrinkled approximation on the whole manifold}
As we already mentioned, we cannot in general hope to globally $C^0$-approximate a tangential rotation $G_t:L\to \Lambda_n(M)$ of a regular Lagrangian or Legendrian embedding $f:L \to M$ by the Gauss maps $G(df_t)$ of a homotopy $f_t$ of regular Lagrangian or Legendrian embeddings. In the previous section, we showed that the approximation can nevertheless be achieved by such a regular homotopy in a small neighborhood of any polyhedron $K \subset L$ of positive codimension. In this section we show that the approximation can be globally achieved on the whole manifold $L$ \emph{if we allow the homotopy $f_t$ to be wrinkled}. See Figure \ref{wrinklingtheorem} for an illustration. More, precisely we have the following theorem, which is the main result of this section.
\begin{theorem}\label{wrinkling lagrangians}
Let $G_t:L \to \Lambda_n(M)$ be a tangential rotation of a regular Lagrangian or Legendrian embedding $f:L \to M$. Then there exists a compactly supported exact homotopy of wrinkled Lagrangian or Legendrian embeddings $f_t:L \to M$, $f_0=f$ such that $G(df_t)$ is $C^0$-close to $G_t$. 
\end{theorem}

\begin{figure}[h]
\includegraphics[scale=0.7]{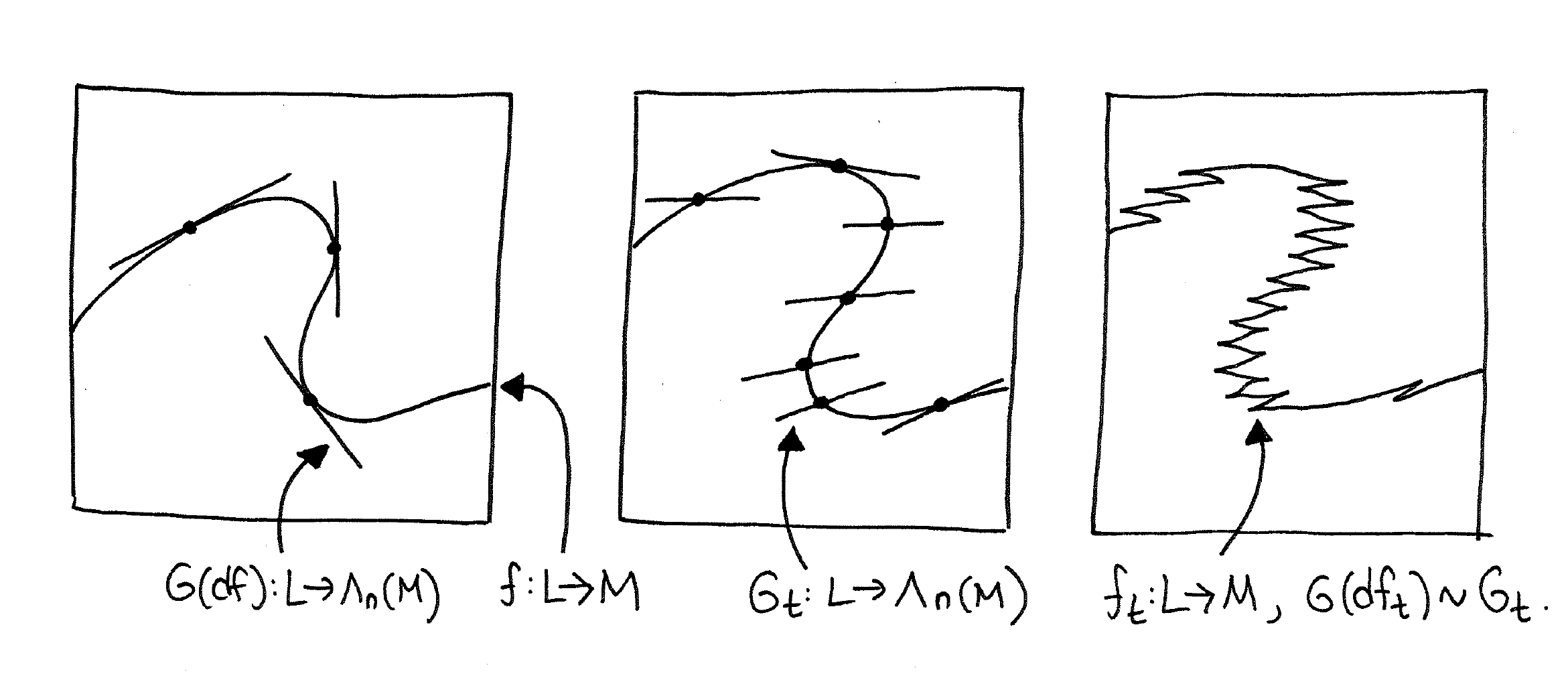}
\caption{The wrinkling theorem in action.}
\label{wrinklingtheorem}
\end{figure}

By Proposition \ref{decomposition proposition} we can reduce Theorem \ref{wrinkling lagrangians} to the following statement.
\begin{theorem}\label{simple wrinkling lagrangians}
Let $G_t:L \to \Lambda_n(M)$ be a graphical simple rotation of a wrinkled Lagrangian or Legendrian embedding $f:L \to M$. Then there exists a compactly supported exact homotopy of wrinkled Lagrangian or Legendrian embeddings $f_t:L \to M$, $f_0=f$ such that $G(df_t)$ is $C^0$-close to $G_t$.
\end{theorem}
The parametric version of Theorem \ref{wrinkling lagrangians} reads as follows.
\begin{theorem}\label{parametric wrinkling lagrangians}
Let $G^z_t:L \to \Lambda_n(M)$ be a family of tangential rotations of regular Lagrangian or Legendrian embeddings $f^z:L \to M$ parametrized by a compact manifold $Z$ such that $G^z_t=G(df^z)$ for $z \in Op(\partial Z)$. Then there exists a family of compactly supported exact homotopies of wrinkled Lagrangian or Legendrian embeddings $f^z_t:L \to M$, $f^z_0=f^z$ such that $G(df^z_t)$ is $C^0$-close to $G^z_t$ and such that $f^z_t=f^z$ for $z \in Op(\partial Z)$.
\end{theorem}

As in the non-parametric case, by Proposition \ref{parametric decomposition proposition} we can reduce Theorem \ref{parametric wrinkling lagrangians} to the following statement.
\begin{theorem}\label{parametric simple wrinkling lagrangians}
Let $G^z_t:L \to \Lambda_n(M)$ be a family of graphical simple rotations of wrinkled Lagrangian or Legendrian embeddings $f^z:L \to M$ parametrized by a compact manifold $Z$ such that $G^z_t=G(df^z)$ for $z \in Op(\partial Z)$. Then there exists a family of compactly supported exact homotopies of wrinkled Lagrangian or Legendrian embeddings $f^z_t:L \to M$, $f^z_0=f^z$ such that $G(df^z_t)$ is $C^0$-close to $G^z_t$ and such that $f^z_t=f^z$ for $z \in Op(\partial Z)$.
\end{theorem}

The proof of Theorems \ref{simple wrinkling lagrangians} and \ref{parametric simple wrinkling lagrangians} consists of two steps. The first step is the construction of a local wrinkling model, which we carry out in Section \ref{Local wrinkling lemma}. The second step is to combine this local wrinkling model with the wiggling results established in Section \ref{Holonomic approximation with controlled cutoff} to obtain the desired global approximation. We carry out this second step in Section \ref{Wrinkling the wiggles}.

\subsection{Local wrinkling model}\label{Local wrinkling lemma}
We begin by describing the local model for the oscillating function that will generate the wrinkles. This is essentially the same local model used by Eliashberg and Mishachev in \cite{EM09}. In fact, our local wrinkling model for Lagrangians and Legendrians is obtained from theirs by simply integrating and differentiating the formulae, just like we did in Section \ref{wrinkled lagrangian and legendrian embeddings} with the definition of wrinkled Lagrangian and Legendrian embeddings. 

The basic geometric idea behind the construction is quite straightforward. One wishes to wrinkle the Lagrangian or Legendrian submanifold back and forth so that the wrinkles are parallel to the rotating planes $G_t(q)$. Since we model the wrinkles on a highly oscillating function, the Gauss map of the resulting wrinkled embedding gives an arbitrarily good approximation of $G_t$. There is a delicate part of the construction regarding the embryos of the zig-zags because the oscillating function is forced to have a derivative with the `wrong sign' in some neighboring region. However, we will impose bounds on the size of this bad derivative to ensure that its effect is not significant. 

\begin{construction}[The oscillating function] First, we fix some notation. We will localize our problem from a general $n$-dimensional manifold $L$ to the unit cube $I^n=[-1,1]^n \subset \bR^n$. A point $q=(q_1, \ldots , q_n) \in I^n$ will be written as $q=(\hat{q}, q_n)$,  where $\hat{q}=(q_1, \ldots , q_{n-1})$. We will consider rotations which are simple with respect to the (constant) hyperplane field $H^{n-1} \subset TI^n$ spanned by the vectors $\partial / \partial q_1 , \ldots , \partial / \partial   q_{n-1}$. Hence the last coordinate $q_n$ will play a special role in our discussion. We will also need a time parameter, which will be denoted by $t$. Sometimes it will be convenient to consider time as another spatial parameter, in which case we will think of the domain of our local model as  $[0,1] \times I^n$.  \\

Consider the family of curves $Z_s \subset \bR^2$, $s\in \bR$, given by parametric equations
\[ x_s(u)=\frac{15}{8} \int_0^u (w^2-s)^2 dw, \qquad y_s(u)= \frac{1}{2}(u^3-3su). \]

The curve $Z_s$ is a graph of a continuous function $z_s: \bR \rightarrow \bR$ which is smooth for $s<0$ and smooth on $\bR \setminus \{- s^{5/2} , s^{5/2}  \}$ for $s\geq0$, where we note that $x_s(\pm\sqrt{s})=\pm s^{5/2}$. See Figure \ref{Z} for an illustration. We note that the constants $15/8$ and $1/2$ are chosen for convenience in the calculation but are otherwise immaterial. 
\begin{remark}\label{composition} Observe that the composition $y_s(u)= z_s\big(x_s(u)\big)$ is smooth for all $s \in \bR$. 
\end{remark}

\begin{figure}[h]
\includegraphics[scale=0.6]{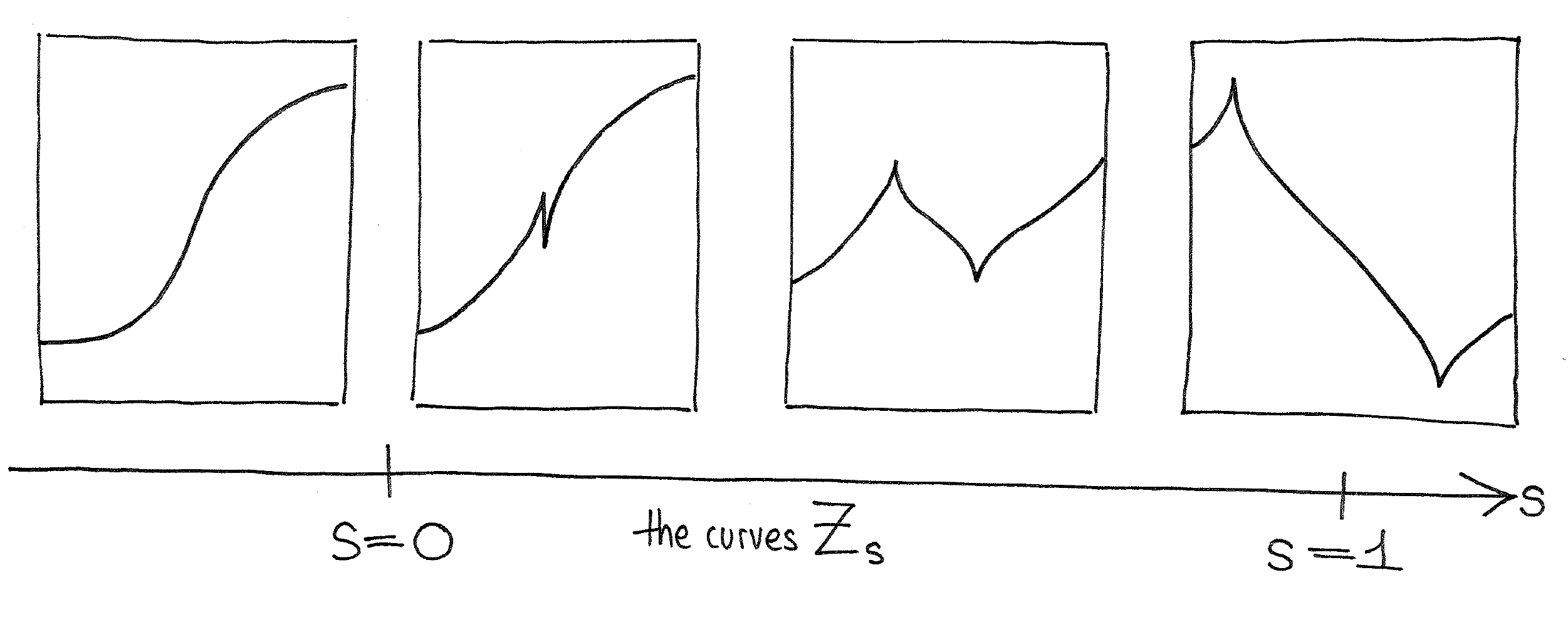}
\caption{The family of curves $Z_s$ gives the local model for the birth/death of semi-cubical zig-zags.}
\label{Z}
\end{figure}

 Let $\sigma,\alpha>0$ be small and choose an odd $1$-periodic family of functions $\zeta_{s} : \bR \rightarrow \bR$, $s \in [-1,1]$, illustrated in Figure \ref{zeta}, which satisfies the following properties. 
 
 \begin{figure}[h]
\includegraphics[scale=0.62]{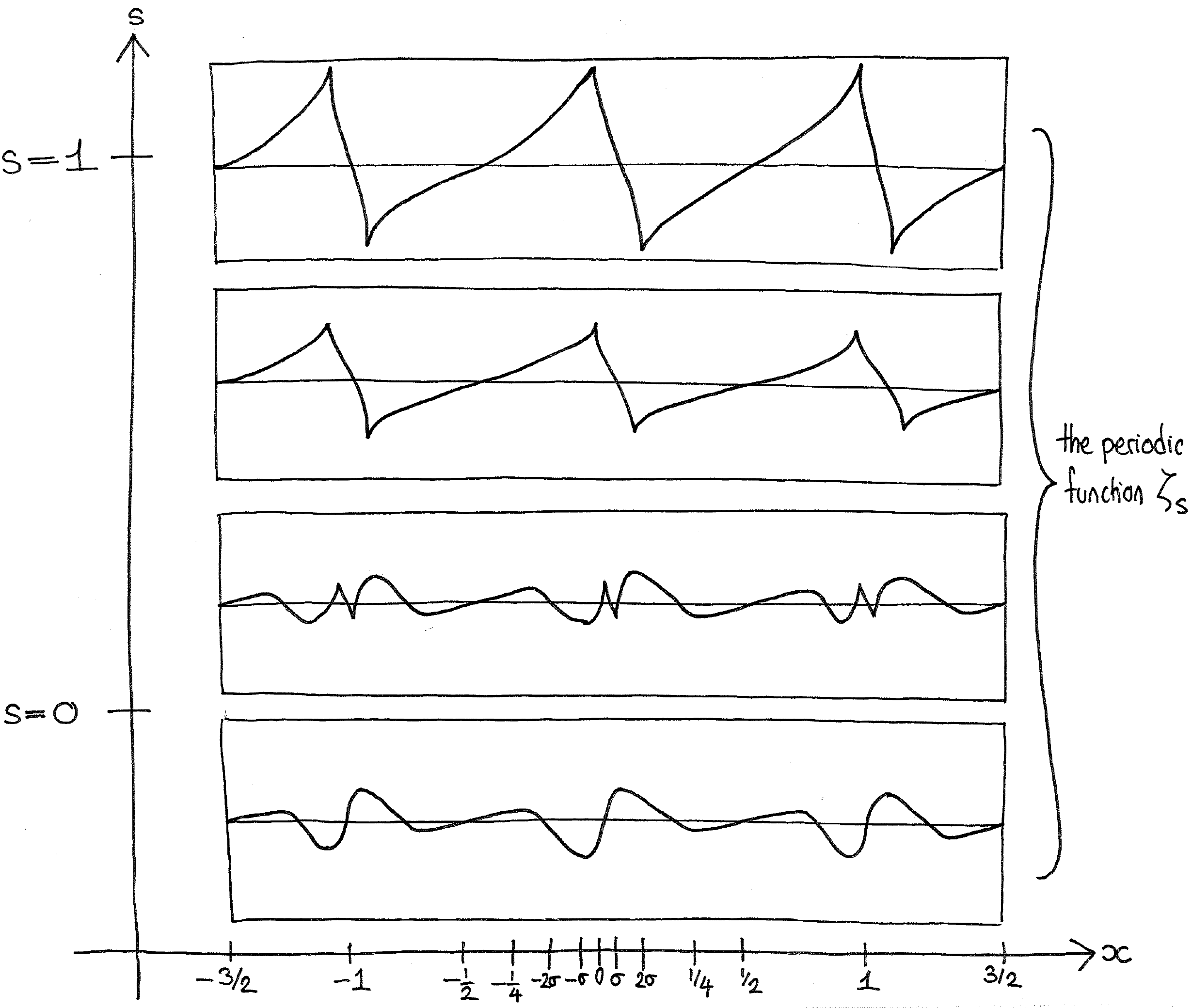}
\caption{The family $\zeta_s$. Observe that for $s=0$ the derivative $d \zeta_0 / dx$ blows up near zero but is everywhere bounded below by $-\alpha$, where the parameter $\alpha$ can be taken to be arbitrarily small. This lower bound also holds everywhere for $s<0$ and outside of $[-\sigma s^{5/2}, \sigma s^{5/2} ]$ for $s>0$.}
\label{zeta}
\end{figure} 
 
\begin{displaymath}  
\zeta_{s}(x) \left\{ \begin{array}{ll} = z_s(\frac{x }{ \sigma})  & \textrm{for $x \in Op(\big[-\sigma s^{5/2}, \sigma s^{5/2} \big])$, $ s  \in [0,1] ,$}\\ 
=z_s(\frac{x}{\sigma})  & \textrm{for $x \in Op(0)$, $ s  \in [-1,0] ,$}\\ 

  \geq 0 & \textrm{for $x \in [-\frac{1}{2},-\frac{1}{4}], \, s \in [-1,1],$} \\
  \leq 0 & \textrm{for $x \in [\frac{1}{4},\frac{1}{2}], \, s \in [-1,1].$} 

  \end{array}\right.
 \end{displaymath} \\
 \begin{displaymath}  
\frac{d\zeta_{1 }}{dx}(x) \left\{ \begin{array}{ll}  
\leq -\frac{4}{\sigma}  & \textrm{for $x \in(-\sigma, \sigma), $}\\ 
\geq 1 & \textrm{for $x \in[-2\sigma, -\sigma) \cup (\sigma , 2\sigma], $}\\ 
 \in [1,2] & \textrm{for $x \in[-\frac{1}{2},-2\sigma] \cup[  2\sigma,\frac{1}{2}]. $ }\end{array}\right.
 \end{displaymath} \\
 \begin{displaymath}  
\frac{d\zeta_{s}}{dx}(x) \left\{ \begin{array}{ll}  
\leq -\frac{4}{\sigma}  & \textrm{for $x \in(-\sigma s^{5/2},\sigma s^{5/2})$,  $s \in (0,1], $}\\ 
\geq -\alpha & \textrm{for $x \in [-2\sigma,-\sigma s^{5/2}) \cup (\sigma s^{5/2}, 2\sigma ]$, $s \in (0,1], $}\\ 
\geq -\alpha& \textrm{for $x \in [-2\sigma, 2\sigma],$ $s \in [-1,0] ,$}\\ 
\in [ -\alpha, 2] & \textrm{for $x \in [-\frac{1}{2}, -2\sigma] \cup [ 2\sigma , \frac{1}{2} ]$, $s \in [-1,1] .$}\\ 

\end{array}\right.
 \end{displaymath}  \\

Let $D^n =\{ x \in \bR^n : \,\, ||x|| \leq 1\}$ denote the closed unit $n-$dimensional disk. We now use the family $\zeta_{s}$ to define a model $\xi=\xi_{\sigma, \alpha, \gamma, \delta, N}: D^n(t,\hat{q})\times [-1,1](q_n) \rightarrow \bR$ which like $\zeta_s$ depends on $\sigma,\alpha>0$ but also depends on three more parameters $\gamma, \delta >0$ and $N \in \bN$. 

Fix a non-increasing function $\eta:[0,1] \rightarrow \bR$ such that 
\begin{itemize}
\item $\eta(x)=1$ for $x \in [0,1-2\delta] $,
\item $\eta(x)= -\delta$ for $x  \in [1-\delta,1] $.
\end{itemize}

Fix a non-increasing cutoff function $\rho: [0,1] \to \bR$ such that
\begin{itemize}
\item $\rho(x)=1 $ for $ x \in [0, 1-\delta] $
\item $\rho(x)=0$ for $x$ near $1$.
\end{itemize}

Fix also another non-increasing cutoff function $\psi: [0,1] \rightarrow [0,1]$ such that 
\begin{itemize}
\item $\psi(x)=1 $ for $ x \in \big[0,1-\frac{1}{4N+2}\big], $
\item $\psi(x)=0$ for $x$ near $1$.
\end{itemize}

%

We define our oscillating model $\xi$ by the following formula, see Figure \ref{xi} for an illustration.

\[ \xi (t,q) = \gamma \, \, \rho \big( ||(t,  \hat{q} )|| \big)  \, \, \psi\big(|q_n|\big) \, \, \zeta_{\eta \left( ||(t, \hat{q})|| \right)} \Big( \frac{2N+1}{ 2}q_n\Big), \qquad (t, \hat{q}) \in D^n , \, \, \, q_n \in [-1,1]. \]

\begin{figure}[h]
\includegraphics[scale=0.7]{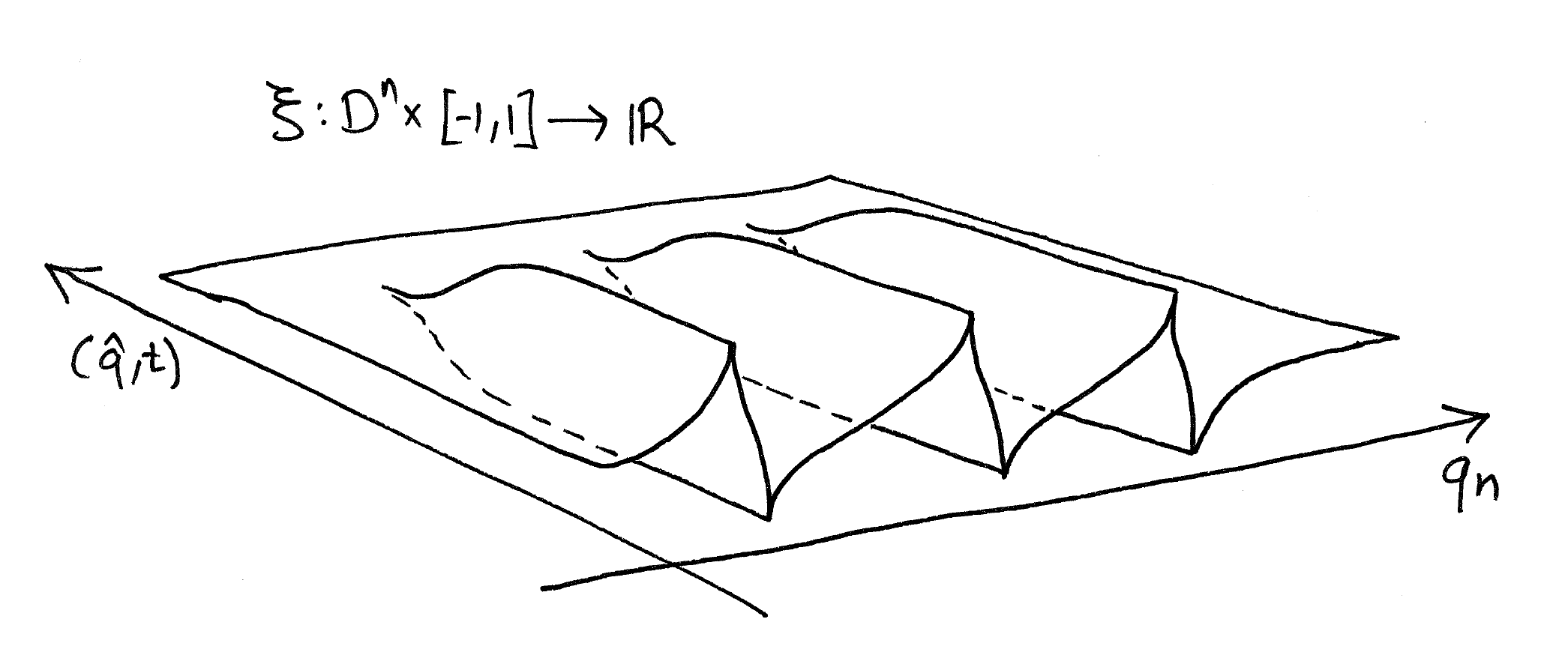}
\caption{One-half of the oscillating function $\xi$.}
\label{xi}
\end{figure}
  
Given $t \in [0,1]$, $q \in I^n$ and $b,c>0$, let $C=C(t,q,b,c)$ denote the box
\[ C= (t,q) + \big( b D^{n} \big) \times [-c,c]  \subset \bR^{n+1}( t, \hat{q}, q_n)\]
which is a copy of $D^n \times [-1,1]$ centered at $(t,q)$ and scaled by $b$ and $c$ in the $(t,\hat{q})$ and $q_n$ directions respectively.
Let $\psi: C \to D^n \times [-1,1] $ be the obvious diffeomorphism obtained by translating and rescaling. Define $\xi_C= \xi \circ \psi :C \to \bR$. The oscillating function $\xi_C$ also depends on the parameters $\sigma, \alpha, \gamma$, $\delta$ and $N$. We will call $\psi^{-1}\big( D^n \times 1 \big)$ and $\psi^{-1}\big( D^n \times -1 \big)$ the top and bottom of the box $C$ respectively. We will also need to consider the slightly smaller boxes 
\[ \widehat{C}=C\left(t,q,(1-\delta)b,\Big(1-\frac{1}{4N+2}\Big)c\right) \subset C
\] \[\text{and} \quad \widetilde{C}=C\left( t,q ,  (1-2\delta)b, \Big(1-\frac{1}{4N+2}\Big)c \right) \subset \widehat{C}.\]

 Observe that $\xi_C$ has wild oscillations on $\widetilde{C}$ which die out on $\widehat{C} \setminus \widetilde{C}$, so that $\xi_C$ is smooth on $C \setminus \widehat{C}$ and $\xi_C=0$ on $Op(\partial C)$. 
 
Finally, we modify our local model $\xi$ to make it Lagrangian. We do this by integrating and differentiating as in the definition of wrinkled Lagrangian embeddings. Define $\ell: D^n(t,\hat{q}) \times [-1,1](q_n) \to   T^*\bR^n(q,p)$ by the formula
 
 \[ \ell(t,q) = \Big( q_1 , \ldots , q_n ,\, \frac{\partial K }{\partial q_1}  \, \,, \ldots , \frac{\partial K}{\partial q_{n-1}}\,\,,  \xi \Big), \quad \text{where} \, \, \, K(t,q)= \int_{-1}^{q_n} \xi(t, \hat{q}, u)du \]
 
Observe that $\ell$ is defined in terms of $\xi$, hence also depends on the parameters $\sigma, \alpha, \gamma , \delta$ and $N$. Observe also that $\xi$ is odd in the $q_n$ variable, hence $K=0$ on $Op\big(\partial ( D^n \times [-1,1] ) \big)$. It follows that $\ell$ has a Legendrian lift $(\ell, K)$ which agrees with the zero section on $Op \big( \partial ( D^n \times [-1,1] ) \big)$. 

Given any box $C$ we can similarly define a translated and scaled version $\ell_C$ of $\ell$ which has support in $C$. This completes the construction of our local wrinkling model.

\begin{remark}\label{not smooth}
The function $\xi$ is not smooth and hence $\ell$ is also not smooth. However, $\xi$ can be smoothly reparametrized and therefore so can $\ell$. We will revisit this nuance later on but it will not cause us any trouble.
 \end{remark}
 
 \end{construction}

We are now ready to state and prove the local wrinkling lemma. Note that a tangential rotation $G_t:I^n \to \Lambda_n(T^*I^n)$ of the inclusion of the zero section $i:I^n \hookrightarrow T^*I^n$ is simple with respect to the hyperplane field  $H=\text{span}( \partial / \partial q_1 , \ldots , \partial / \partial q_{n-1}) \subset TI^n$  if it can be written as 
\[ G_t = \text{span}\big( \frac{\partial }{ \partial q_1} , \,  \ldots \,  , \, \frac{ \partial }{\partial q_{n-1}  }, \, \, \cos  ( \lambda_t )\frac{ \partial }{ \partial q_n} +  \sin( \lambda_t )\frac{ \partial }{ \partial p_n} \,  \big) \]
for some angle function $\lambda_t:I^n \to \bR$.  According to our previous definition we say that $G_t$ is graphical when $\text{im}(\lambda_t) \subset (-\pi/2, \pi/2)$. We will say that $G_t$ is quasi-graphical when $\text{im}(\lambda_t) \subset (-\pi, \pi)$. 

\begin{figure}[h]
\includegraphics[scale=0.7]{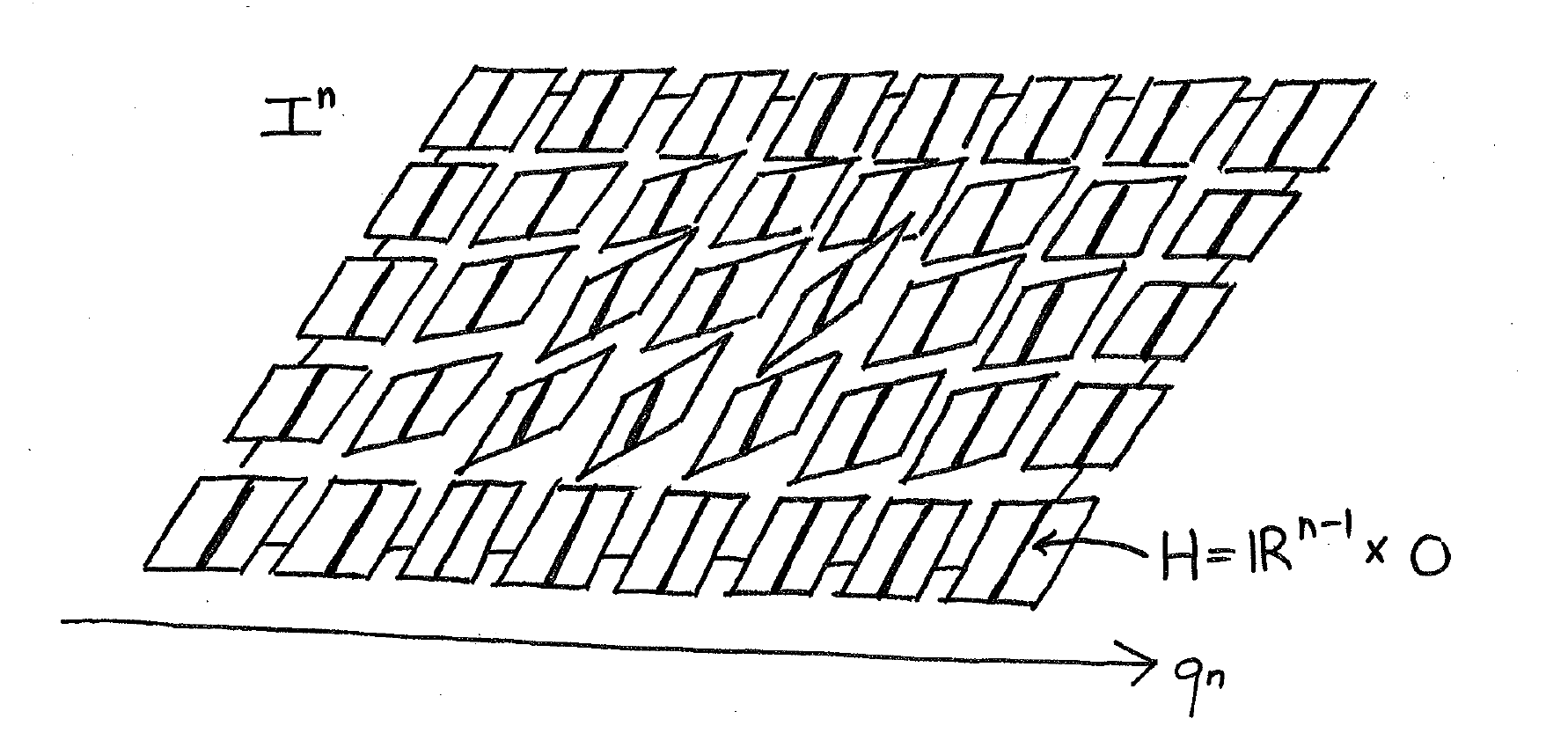}
\caption{A tangential rotation which is quasi-graphical and simple with respect to $H$.}
\label{Hsimplerotation}
\end{figure}

\begin{lemma}[Local wrinkling for Lagrangians]\label{local wrinkling Lagrangians} Let $G_t:I^n \to \Lambda_n(T^*I^n)$ be a tangential rotation of the zero section $i:I^n \hookrightarrow T^*I^n$ which is quasi-graphical and simple with respect to $H$ and such that $G_t=G(di)$ on $Op(\partial I^n)$. Then there exists an exact homotopy of wrinkled Lagrangian embeddings $f_t:I^n \to T^*I^n$, $f_0=i$, such that the following properties hold.
\begin{itemize}
\item $G(df_t)$ is $C^0$-close to $G_t$.
\item $f_t=i$ on $Op(\partial I^n)$. 
\end{itemize}
\end{lemma}

\begin{proof} Let $\tau>0$ be small. We will be precise about exactly how small we need $\tau$ to be later on. Wrinkling is dangerous and unnecessary where $\lambda_t$ is close to zero, so we will first use our oscillating model $\ell$ to define a similar model which does not oscillate on the subset of $[0,1] \times I^n$ in which $|\lambda_t|<\tau$. 

\begin{remark} Although we want to think of time as a spatial parameter, observe that $\lambda_t \neq 0$ on the boundary face $1 \times I^n \subset \partial ( [0,1] \times I^n)$, so we are not quite in the relative setting. To remedy this, we extend the time interval from $[0,1]$ to $[0,2]$ by setting $\lambda_t = \lambda_{2-t}$ for $t \in [1,2]$. We can then work with the box $[0,2] \times I^n$ as our local model, which has the advantage that $\lambda_t=0$ on $Op\big(\partial ([0,2] \times I^n)\big)$. We can later restrict back to only considering times  $t \in [0,1]$ and forget about the rest.
\end{remark}

Let $\Omega_{\tau}=\{ (t,q) \in [0,2] \times I^n : \, \, \, | \lambda_t(q)| > \tau \}$. We call a box $C=C(t,q,b,c) \subset [0,2] \times I^n$ special if $|\lambda_t(q)|<2\tau$ for $(t,q)$ near the top and bottom of $C$. Choose special boxes $C_1, \ldots , C_m \subset [0,2] \times I^n$ which are contained in $\Omega_{\tau}$ and such that the smaller boxes $\widetilde{C}_1, \ldots , \widetilde{C}_m$ are still special and cover $ \Omega_{2 \tau}$. This can be achieved if $\delta$ is sufficiently small and $N$ is sufficiently big. Write $\psi_j$ for the parametrizing diffeomorphisms $\psi_j : C_j \to D^n \times [-1,1]$ as above. We can assume that the sets $\psi_j^{-1} \big( D^n \times ([-1,1] \cap \bQ) \big) \subset [0,2] \times I^n$ are disjoint. Therefore for each integer $N$ there exists a number $\sigma(N)>0$ such that for all $\sigma<\sigma(N)$ the subsets
\[ \psi_j^{-1} \Big( D^n \times \Big[ \frac{2k}{2N+1} -\widetilde{ \sigma}, \frac{2k}{2N+1} + \widetilde{\sigma} \Big] \Big), \qquad \widetilde{\sigma} = \frac{4\sigma }{ 2N+1}, \quad -N \leq k \leq N, \]
are also disjoint. When we let $N \to \infty$ below, we will let $\sigma \to 0$ accordingly so that we always have $\sigma<\sigma(N)$. 

For each box $C_j \subset \Omega_\tau$ we have an oscillating Lagrangian model $\ell_{C_j}$. Let $\text{sign}(j)=\text{sign}(\lambda_t|_{C_j}) \in \{ \pm1\}$. Define the Lagrangian oscillating model $w_t$ adapted to $G_t$ by setting $w_t(q)= \sum_j  \text{sign}(j) \ell_{C_j}(t,q)$. More precisely, we set
\[ w_t(q) = \Big(q_1, \ldots , q_n,\, \frac{\partial H_t}{\partial q_1} \, \, , \ldots , \, \frac{\partial H_t}{\partial q_{n-1} } \, \, , \, \sum_j   \text{sign}(j) \xi_{C_j} \Big), \quad \text{where} \, \, H_t = \sum_j \int_{-1}^{q_n}  \text{sign}(j)  \xi_{C_j}(t,\hat{q}, u) du \]
Observe that $w_t=0$ and $H_t=0$ outside of $\Omega_\tau$. At this point we can restrict back to the time interval $[0,1] \subset [0,2]$, which is all that we really cared about. 

Consider the Hamiltonian function $F_t: T^*\Omega_\tau \to \bR$ given by $F_t(q,p) =\frac{1}{2} \cot\big( \lambda_t(q) \big) p_n^2$. For each $t \in [0,1]$ we get a Hamiltonian isotopy $\varphi^s_t:T^*\Omega_\tau \to T^*\bR^n$ such that the vector field $X_t = \partial_s \varphi^s_t(q)$ is the symplectic dual of $dF_t(q)=\cot\big(\lambda_t(q) \big)p_ndp_n+\frac{1}{2}\text{cosec}^2\big( \lambda_t(q) \big) p_n^2 d\lambda_t(q) $. Hence we have
\[ X_t(q,p)=  \cot\big( \lambda_t(q) \big) p_n \frac{\partial}{\partial q_n} +  \frac{1}{2}\text{cosec}^2\big( \lambda_t(q) \big) p_n^2 \sum_{j=1}^n \frac{\partial \lambda_t}{\partial q_j} \frac{\partial}{\partial p_j}.\]
It follows by explicit computation that
\[ \varphi^s_t(q,p) = \Big(\hat{q}, \, q_n + \cot\big( \lambda_t(q) \big) p_n s \, ,  \, \, p_1 +\frac{1}{2} \text{cosec}^2\big( \lambda_t(q) \big)p_n^2 \frac{\partial \lambda_t}{\partial q_1} s \, \,, \, \, \ldots \, \, ,\, p_n +\frac{1}{2} \text{cosec}^2\big( \lambda_t(q) \big)p_n^2 \frac{\partial \lambda_t}{\partial q_n} s\Big). \]
We set $\varphi_t= \varphi_t^1$. Note that $\varphi_t=id$ on $\Omega_\tau \subset T^*\Omega_\tau$ since $p=0$. Note moreover that on $\Omega_\tau$ we have
\[\frac{\partial \varphi_t}{\partial q_j}= \frac{\partial}{\partial q_j} \quad \text{for} \, j=1,\ldots , n, \quad   \frac{\partial \varphi_t}{\partial p_j} = \frac{\partial }{\partial p_j} \quad \text{for }j<n  \]
\[ \text{and} \quad \frac{\partial \varphi_t}{\partial p_n} = \cot(\lambda_t) \frac{\partial}{\partial q_n} + \frac{\partial }{\partial p_n} .\]

\begin{figure}[h]
\includegraphics[scale=0.7]{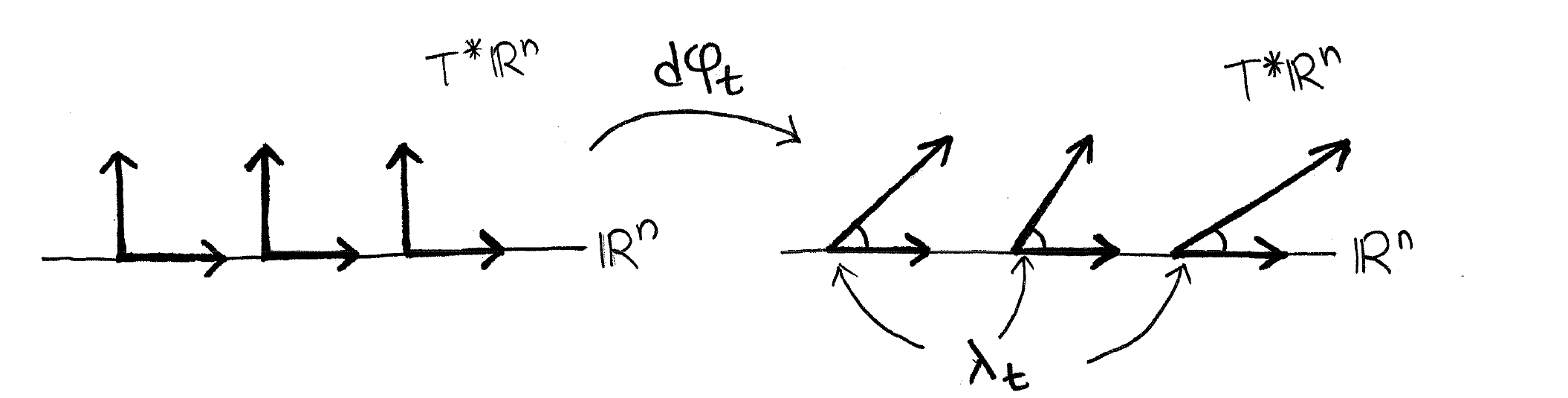}
\caption{Along the zero section $\bR^n \subset T^*\bR^n$ we have $d\varphi_t( \partial /\partial q_n) = \partial / \partial q_n$ and $d\varphi_t( \partial / \partial p_n) = \cot(\lambda_t) \partial / \partial q_n  + \partial / \partial p_n$.}
\label{effectofphi}
\end{figure}

Hence in particular on $\Omega_{\tau}$ we have
\[ d\varphi_t \Big( \text{span} \big( \frac{\partial}{\partial q_1} , \ldots , \, \frac{\partial}{\partial q_{n-1} } , \frac{\partial}{\partial p_n} \big) \Big) = G_t. \]

Set $f_t=\varphi_t \circ w_t$. We recall from Remark \ref{not smooth} that each $\ell_{C_j}$ is not smooth, hence $w_t$ is not smooth, hence the same is true for $f_t$. However, we can precompose $w_t$ with a reparametrization of the domain so that $w_t$ and hence also $f_t$ is smooth. Note moreover that this reparametrization does not change the image of $f_t$ and therefore it also doesn't change the image of the Gauss map $G(df_t)$, which is what we actually care about. By abusing notation we will also use $f_t$ to denote the reparametrized smooth map whenever this is convenient. See Figures \ref{wrinklingpcoordinates} and \ref{wrinklingzcoordinates} for an illustration of $f_t$.

\begin{figure}[h]
\includegraphics[scale=0.7]{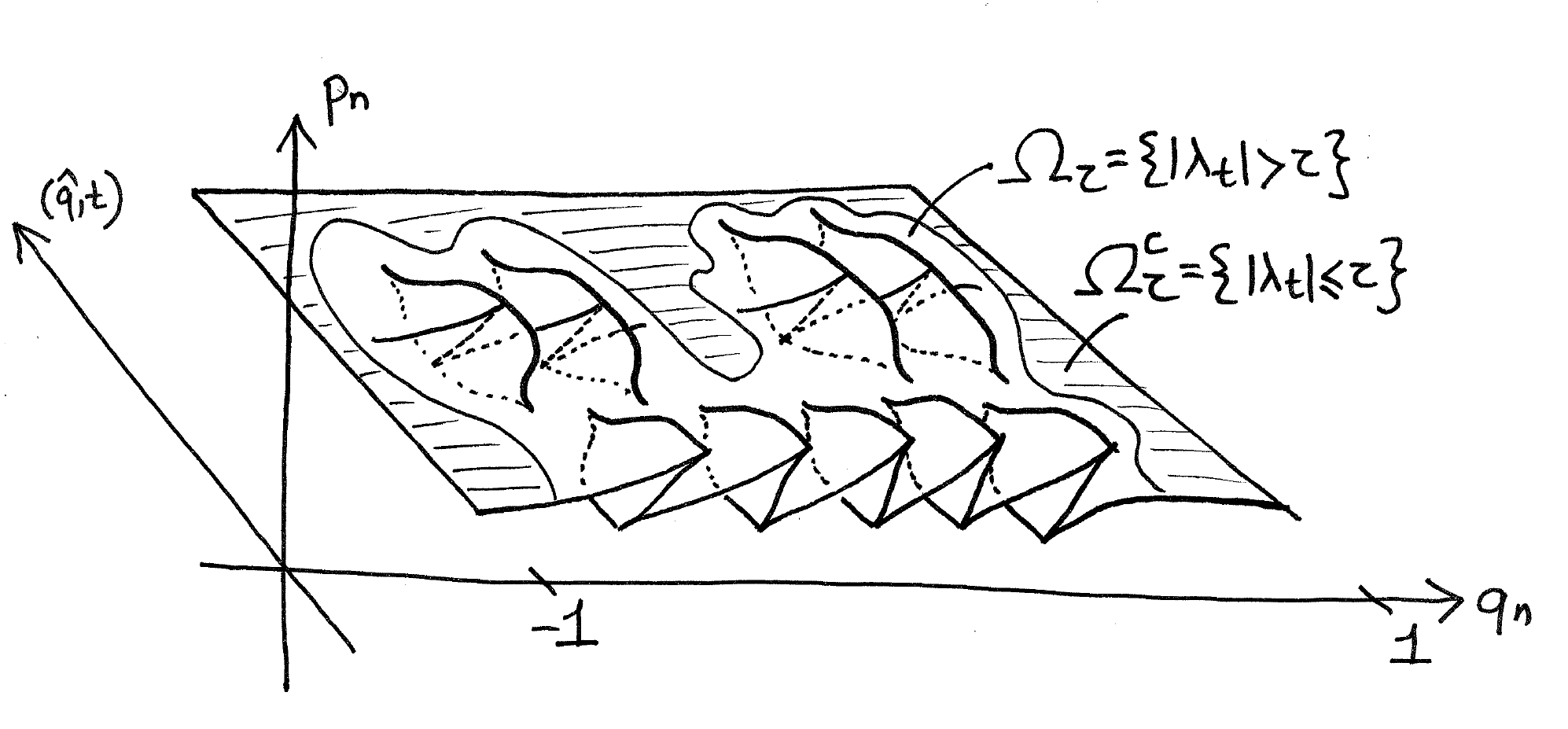}
\caption{The $p_n$-coordinate of the map $f_t$. The cusps are semi-cubic.}
\label{wrinklingpcoordinates}
\end{figure}

\begin{figure}[h]
\includegraphics[scale=0.7]{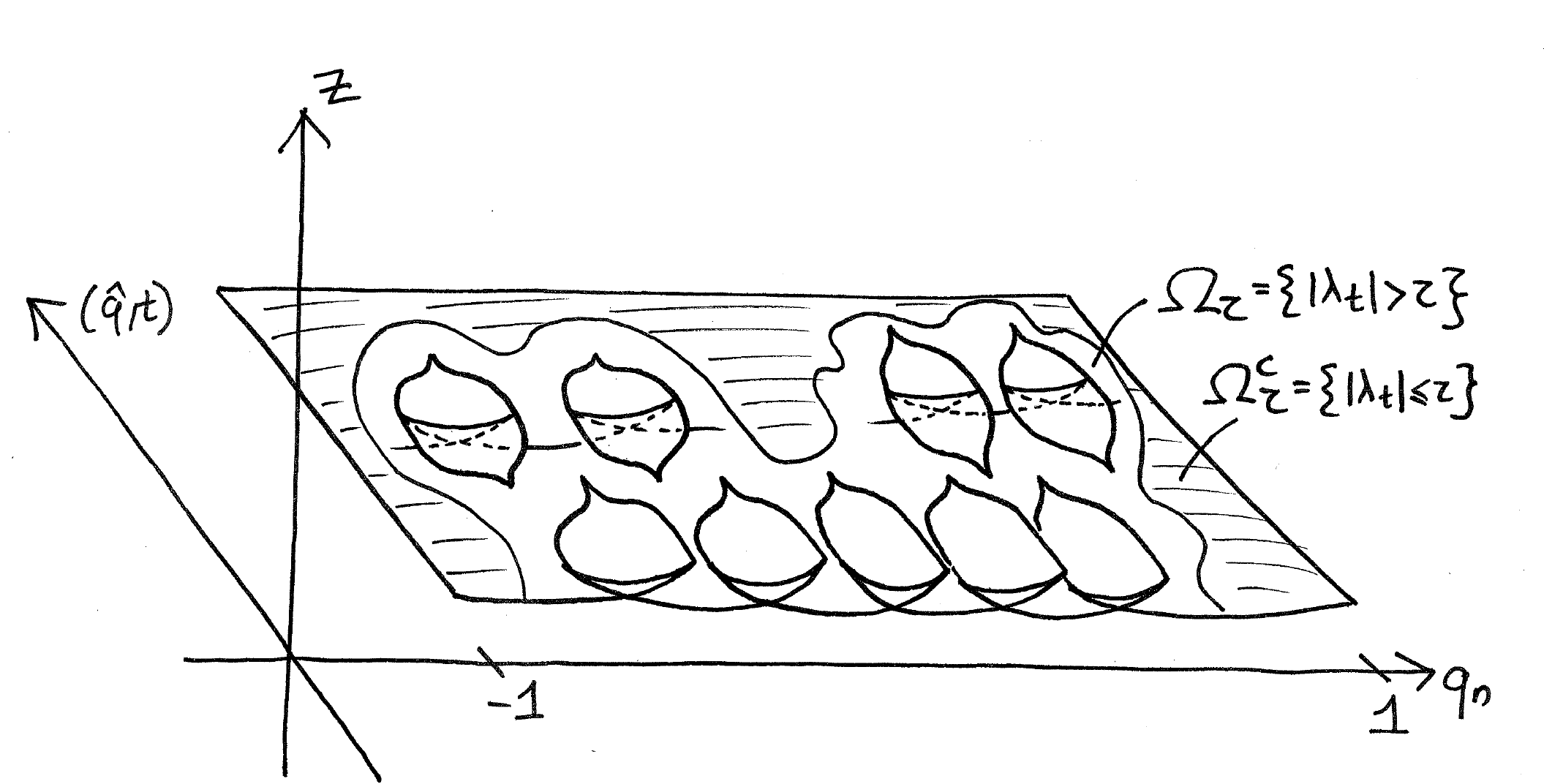}
\caption{The $z$-coordinate of the Legendrian lift of $f_t$. In other words, this is the Legendrian front of $f_t$. The cusps are semi-quintic.}
\label{wrinklingzcoordinates}
\end{figure}

\begin{claim}\label{parameters} For any $\varepsilon>0$ we can choose parameters $\tau, \delta, \sigma, \alpha,  \gamma $ and $N$ so that $\text{dist}_{C^0}(G(df_t),G_t)<\varepsilon$. 
\end{claim}
We recall the parameters at play. The game is all about controlling the different rates at which the parameters tend to zero or infinity, so it will be important to be precise in the interdependence of the parameters and in the order of quantifiers. 
\begin{itemize}
\item $\tau$ is the cutoff angle of $G_t$ under which we will perform no wrinkling.
\item $\delta$ is proportional to the the width of the shell between a box $C$ and the smaller box $\widetilde{C}$.
\item $\sigma$ is the order of magnitude of $\zeta_s'$ on the regions where it is large and negative (inside the wrinkles).
\item $\alpha$ controls the magnitude of the `bad' negative derivative $\zeta_s'$ when the wrinkles die out.
\item $\gamma$ is the height of the oscillating model $\xi$.
\item $N$ is proportional to the number of wrinkles in $\xi$.
\end{itemize}

We begin by fixing $\varepsilon>0$ arbitrarily small. To choose $\tau$, observe that 
\[ d \varphi_t \Big( \frac{\partial}{\partial q_n} + \beta \frac{\partial }{\partial p_n} \Big) = \frac{\partial}{\partial q_n} + \beta \big( \cot(\lambda_t) \frac{\partial}{\partial q_n} + \frac{\partial}{\partial p_n} \big) , \qquad \beta \in \bR.\]
and hence if $\text{sign}(\beta) = \text{sign}(\lambda_t)$, then the scalar product of $\partial / \partial q_n$ and $d\varphi_t( \partial / \partial q_n + \beta \partial / \partial p_n)$ is positive and moreover we have
\[ \measuredangle \Big( \frac{\partial}{\partial q_n}, \, \,  d \varphi_t \big( \frac{\partial}{\partial q_n} + \beta \frac{\partial}{\partial p_n} \big) \Big) \, \, < \,   |  \lambda_t | . \]
Recall that on the subset $\Omega_{\tau} \setminus \Omega_{2 \tau}$ we have $\tau<|\lambda_t| \leq 2\tau$.  Suppose that $\tau < \varepsilon/4$. It follows that if $\text{sign}(\beta)=\text{sign}(\lambda_t)$, then we have 
\[ \measuredangle\Big( \frac{\partial }{ \partial q_n} ,\, \, d\varphi_t\big(\frac{ \partial }{ \partial q_n} + \beta \frac{\partial }{ \partial p_n }\big) \Big) <  2 \tau < \frac{\varepsilon}{2}  \quad \text{on}  \, \,\Omega_{\tau} \setminus \Omega_{2\tau}  .\]

Once $\tau < \varepsilon/4$ is fixed, we choose $\delta$ small enough so that the construction of $w_t= \sum_j  \text{sign}(j) \ell_{C_j}$ (which depends implicitly on $\tau$) is possible. The other parameters must be chosen somewhat more judiciously. Our first task is to understand the geometry of the initial local model $(s,u) \mapsto \big(x_s(u), y_s(u) \big)$ in order to control the error produced when we modify the model to make it Lagrangian. 

Consider the Lagrangian version in $T^*\bR^2=\bR^4(q_1,q_2,p_1,p_2)$ given by the formula
\[ m(s,u)=\Big(s,x_s(u),r_s(u) , y_s(u) \Big) \in T^*\bR^2, \quad r_s(u) = \int_0^u \partial_s\big(y_s(u) \big)\partial_u\big( x_s(u) \big) - \partial_u \big( y_s(u) \big) \partial_s \big( x_s(u) \big) du, \]

where we recall that
\[ x_s(u)=\frac{15}{8} \int_0^u (w^2-s)^2 dw, \qquad y_s(u)= \frac{1}{2}(u^3-3su). \]

We also have the corresponding scaled version

\[ m_{\gamma, N}(s,u)=\Big(s,\frac{1}{N}x_s(u),\frac{\gamma}{N}r_s(u) , \gamma y_s(u) \Big) \in T^*\bR^2.\]

If $\gamma \to 0$ and $N \to \infty$ in such a way that $N \gamma \to \infty$, then the Gauss map $G(dm_{\gamma,N})$ converges (on compact subsets of the $(s,u)$ plane) to the distribution spanned by the vectors $\partial/\partial q_1=(1,0,0,0)$ and $\partial/\partial p_2=(0,0,0,1)$. The proof is the following explicit computation.

It will we convenient to carry out our calculations in terms of the function $F(s,u)=\frac{1}{3}(u^3-3su)$ and its derivative $F_u(s,u)=u^2-s$. Note that the zero set $\{F_u=0\}$ is precisely the wrinkling locus of $m$. We compute:
\[\partial_u\big( x_s(u) \big ) = \frac{15}{8}F_u^2 , \qquad \partial_s\big( x_s(u) \big ) = -\frac{15}{4}F\]
\[ \partial_u\big( y_s(u) \big ) =  \frac{3}{2}F_u , \qquad \partial_s\big( y_s(u) \big ) = -\frac{3}{2}u \]
\[ \partial_u\big(r_s(u) \big ) = \frac{15}{8}F_u \big( -\frac{3}{2}uF_u + 3F\big) =  -\frac{15}{16}F_u(u^3 + 3su).\] 

\[ \frac{\partial m_{\gamma, N} }{\partial s}=\Big(1, \, \, \frac{1}{N} \partial_s\big(x_s(u)\big)  , \,  \frac{\gamma}{N} \partial_s\big( r_s(u) \big), \, \gamma \partial_s \big( y_s(u) \big) \Big) \rightarrow \big(1,0,0,0\big) \, \,\,  \text{as} \, \, \,\gamma \to 0, \, \, N \to \infty,\]

\[ \frac{\partial m_{\gamma,N}}{\partial u} = \Big( 0, \,\frac{1}{N} \frac{15}{8}F_u^2,-\frac{\gamma}{N} \frac{15}{16}(u^3+3su)F_u, \, \gamma\frac{3}{2}F_u \Big) =\,  -\gamma F_u \Big( 0, \, \frac{1}{N \gamma} F_u , \, -\frac{1}{N} \frac{15}{16}(u^3+3su), \,     \frac{3}{2}  \Big) \]

and hence provided that $N \gamma \to \infty$ we have

\[ \text{span} \Big( \frac{\partial m_{\gamma, N} }{\partial s}, \,  \frac{\partial m_{\gamma,N}}{\partial u} \Big) \longrightarrow \text{span}\Big( \frac{\partial}{\partial q_1} , \frac{\partial}{\partial p_2}\Big). \]

With some minor modifications we can extend our computations to the scaled $n-$dimensional model for the Lagrangian wrinkle as it appears in $\ell$.

\[(t , q) \mapsto \Big(q_1, \ldots , q_{n-1}, \frac{\sigma x_{\eta} (q_n)}{2N+1} , \, \,  \frac{\sigma \gamma  r_{\eta}(q_n)}{2N+1}\frac{\partial ||(t, \hat{q}) ||}{\partial q_1}  \eta' ,  \,  \ldots  \, ,  \frac{\sigma \gamma  r_{\eta}(q_n)}{2N+1}\ \frac{\partial ||(t, \hat{q}) ||}{\partial q_{n-1}} \eta',  \, \gamma y_{\eta}(q_n) \Big)\]
where $\eta=\eta \big( ||(t,\hat{q}) || \big)$. Indeed, the only difference comes from the terms $\partial_j \eta = \eta' \partial_j ||( t,\hat{q}) || $ for $j<n$ and their partial derivatives, which give an error that tends to zero as $\gamma \to 0$ and $N \to \infty$. The conclusion is that provided we have $N \gamma \to \infty$, the Gauss map converges to the distribution
\[V= \text{span} \Big( \frac{\partial}{\partial q_1}, \ldots , \frac{\partial}{\partial q_{n-1} }, \frac{\partial }{\partial p_n} \Big).\]

Recall that we must ensure $\sigma < \sigma(N)$ so that the singularity loci $\Sigma( \ell_{C_j}) \subset [0,1] \times I^n$ are disjoint. Hence if we let $N \to \infty$, then we must also allow for $\sigma \to 0$. But this only helps us in the above computation so there is no issue.

Consider next the oscillating model $\ell$ defined above. Let $\Sigma \subset [-1,1] \times I^n$ be the locus on which $\ell$ is not smooth. The set $\Sigma$ consists of a disjoint union of spheres with cuspidal equators. Let $E$ be the compact region bounded by $\Sigma$. If $\gamma \to 0$ and $N \to \infty $ so that $N \gamma \to \infty$, then the above computations show that on $Op(E)$ the Gauss map of $\ell$ converges to the distribution $V$. In the complement of  $Op(E)$, the model $\ell$ is smooth and for $j<n$ we have $\partial \ell / \partial q_j \to \partial / \partial q_j$ as $\gamma \to 0$. On the subset $B=[-1+2 \delta , 1-2\delta]^n\times [-1 + \frac{1}{4N+2} ,1-\frac{1}{4N+2}] $ the Gauss map of $\ell$ converges to $V$, indeed on the remaining part $B \setminus Op(E)$ the derivative $dp_n( \partial \ell / \partial q_n) = \partial \xi / \partial q_n$ is strictly positive and scales by $N \gamma$ while $dp_j(\partial \ell / \partial q_n)$ scales by $\gamma$ for $j<n$. On $I^n \setminus B$ we cannot control $\partial  \ell / \partial q_n$ so precisely but we assert that outside of $Op(E)$ there still holds the following lower bound:
\[dp_n(\partial \ell / \partial q_n ) =  \partial \xi / \partial q_n \geq  -  (N+1) \gamma \alpha,.\]

To confirm this assertion, we compute

\[ \frac{\partial \xi}{\partial q_n} =\gamma \rho(t, \hat{q})  \Bigg( \text{sign}(q_n) \psi'( |q_n|) \zeta_{\eta(t, \hat{q})}\Big(\frac{2N+1}{2}q_n\Big) +\frac{2N+1}{2} \psi(|q_n|) \zeta'_{\eta(t,\hat{q})}\Big(\frac{2N+1}{2}q_n\Big)\Bigg). \]
Since $\psi'\leq 0$ and $\text{sign}( \zeta_{\eta(t,\hat{q})}(\frac{2N+1}{2}q_n)\big)=-\text{sign}(q_n)$ in the region where $\psi' \neq 0$, the first term is always non-negative. For the second term we use our assumption that $\zeta_s' \geq -\alpha$ and the desired inequality follows. 

We deduce from this inequality that if we let $\gamma , \alpha \to 0$ and $N \to \infty $ so that $N \gamma \to \infty$ and $N \gamma \alpha \to 0$, then on the complement of $Op(E)$ we have $\liminf dp_n(\partial \ell / \partial q_n) \geq 0$. Of course we also still have $dq_j(\partial \ell / \partial q_n)=0$ for $j<n$, $dq_n(\partial \ell / \partial q_n)=1$ and  $dp_j( \partial \ell / \partial q_n) \to 0$ as $\gamma \to 0$. 

Next we proceed to study the model $w_t = \sum_j  \text{sign}(j) \ell_{C_j}$ which is adapted to our rotation $G_t$. Assume first for simplicity that $\lambda_t \geq 0$, so that $\text{sign}(j)=1$ for all $j$. Let $\widetilde{\Sigma} \subset [0,1] \times I^n$ be the non-smooth locus of $w_t$. The set $\widetilde{\Sigma}$ is again a disjoint union of spheres which have cuspidal equators. Let $\widetilde{E}$ be the compact subset bounded by $\widetilde{\Sigma}$. Note that $\widetilde{E} \subset \Omega_{\tau}$. On $\Omega_{2 \tau} \setminus \widetilde{E}$ all the derivatives $\partial \xi_{C_j} / \partial q_n$ are bounded below by a positive constant times $-N \gamma \alpha$ and at each point there is at least one of them which is bounded below by a constant times $N \gamma$. This last assertion holds because the boxes $\widetilde{C}_j \subset C_j$ cover $\Omega_{2\tau}$. Inside $\widetilde{E}$ all the derivatives $\partial \xi_{C_j}/ \partial q_n$ are bounded above by a positive constant times $N \gamma$ and at each point there is exactly one derivative $\partial \xi_{C_j} / \partial q_n$ for which is bounded above by a constant times $-N \gamma/\sigma$. This last derivative corresponds to the $\ell_{C_j}$ whose non-smooth locus bounds the given component of $\widetilde{E}$ containing the point we're looking at. We recall that we are letting $\sigma \to 0$ with the only requirement that $\sigma<\sigma(N)$. Hence if $N \to \infty$ and $\gamma, \alpha, \sigma \to 0$ in such a way that this condition holds and if additionally we have $N \gamma \to \infty$ and $N \gamma \alpha \to 0$, then on the region $\Omega_{2 \tau} \cup Op(\widetilde{E})$ the Gauss map of $w_t$ converges to the distribution $V$ and on
and on $\Omega_\tau \setminus \big(\Omega_{2\tau} \cup Op(\widetilde{E})\big)$ we know that $\partial w_t/ \partial q_j \to \partial / \partial q_j$ for $j<n$ and that $\partial w_t / \partial q_n $ gets arbitrarily close to the sector
\[ \cC=    \text{span}\{\frac{\partial}{\partial q_n} + \beta \frac{\partial}{\partial p_n}  : \, \, \,  \beta \geq 0 \} \subset T(T^*\bR^n)|_{\bR^n}. \]
Consider next the general case where we don't assume that $\text{sign}(j)=1$ for all $j$. Since $\Omega_\tau = \{ \lambda_t>\tau \} \cup \{ \lambda_t<-\tau\}$ is a disjoint union, we can repeat the above reasoning on each component and reach the same conclusion, provided that we modify that definition of the subset $\cC$ as follows
\[ \cC=    \text{span}\{\frac{\partial}{\partial q_n} + \beta \frac{\partial}{\partial p_n}  : \, \, \,  \text{sign}(\beta)=\text{sign}(\lambda_t) \} \subset T(T^*\bR^n)|_{\Omega_\tau}. \]

We now return to the wrinkled Lagrangian embedding $f_t = \varphi_t \circ w_t$. Recall that along the zero section the linear symplectic isomorphism $d\varphi_t$ is the map which sends 
\[ \frac{\partial }{\partial q_j} \mapsto \frac{\partial }{\partial q_j} , \, \,  j=1,\ldots , n , \qquad \frac{\partial}{\partial p_j} \mapsto \frac{\partial}{\partial p_j} , \, \, j<n\]
\[ \text{and} \quad \frac{\partial}{\partial p_n} \mapsto \cot(\lambda_t) \frac{\partial}{\partial q_n} + \frac{\partial}{\partial p_n}, \]
so that $d\varphi_t(V)=G_t$ along the zero section. Recall also that we chose $\tau=\tau(\varepsilon)$ so that on $\Omega_\tau \setminus \Omega_{2\tau}$ we have $\measuredangle\big(  d\varphi_t(v), \partial/ \partial q_n \big) < \varepsilon/2$ for all $v \in \cC$. Under the above convergence assumptions it follows that we have $
\limsup \measuredangle(  \partial  f_t / \partial q_n, \partial / \partial q_n  )  \leq  \varepsilon/2$ on $\Omega_{\tau} \setminus \big(\Omega_{2\tau}\cup Op(\widetilde{E})\big)$ and hence also $\limsup \text{dist}( G(df_t) , T\bR^n) \leq \varepsilon/2$. Therefore $\limsup \text{dist}\big( G(df_t) , G_t \big) \leq \text{dist}\big( G(df_t) , T\bR^n \big) + \text{dist}\big(T\bR^n, G_t\big) < \varepsilon/2 + 2\tau < \varepsilon$ on $\Omega_{\tau}/\big(\Omega_{2\tau}\cup Op(\widetilde{E})\big)$. Outside of $\Omega_{\tau}$ we have $\text{dist}\big(G(df_t) , G_t) = \text{dist}\big( T\bR^n, G_t ) < \tau < \varepsilon$.  If we assume that on $\Omega_{2\tau}\cup Op(\widetilde{E})$ the Gauss map of $w_t$ converges to the distribution $V$, then for $f_t$ we have
\[ G(df_t) \to d\varphi_t(V) = G_t \qquad \text{on} \, \, \Omega_{2\tau}\cup Op(\widetilde{E}). \]

Therefore to conclude the proof of Claim \ref{parameters}, and hence also of Lemma \ref{local wrinkling Lagrangians}, it suffices to show that we can arrange it so that $\gamma , \alpha \to 0$ and $N \to \infty$ in such a way that $N \gamma \to \infty$ and $N \gamma \alpha \to 0$. This is clearly possible, for instance we can let $\gamma = N^{-1/2}$ and $\alpha = N^{-2/3}$. \end{proof}

The analogous result for Legendrians is stated and proved in the same way. Observe as in the Lagrangian case that a tangential rotation $G_t:I^n \to \Lambda_n\big( J^1(I^n, \bR) \big)$ of the inclusion of the zero section $i:I^n \hookrightarrow J^1(I^n, \bR)$ is simple with respect to the hyperplane field  $H=\text{span}( \partial / \partial q_1 , \ldots , \partial / \partial q_{n-1}) \subset TI^n$  if it can be written as 
\[ G_t = \text{span}\big( \partial / \partial q_1 , \ldots , \partial / \partial q_{n-1} , \cos  ( \lambda_t ) \partial / \partial q_n + \sin( \lambda_t ) \partial / \partial p_n \big) \]
for some function $\lambda_t:I^n \to \bR$.  According to our previous definition we say that $G_t$ is graphical when $\text{im}(\lambda_t) \subset (-\pi/2, \pi/2)$. We will say that $G_t$ is quasi-graphical when $\text{im}(\lambda_t) \subset (-\pi, \pi)$.

\begin{lemma}[Local wrinkling for Legendrians]\label{local wrinkling Legendrians} Let $G_t:I^n \to \Lambda_n\big(J^1(I^n, \bR) \big)$ be a tangential rotation of the zero section $i:I^n \hookrightarrow J^1(I^n, \bR)$ which is quasi-graphical and simple with respect to $H$ and such that $G_t=G(di)$ on $Op(\partial I^n)$. Then there exists an exact homotopy of wrinkled Legendrian embeddings $f_t:I^n \to J^1(I^n, \bR)$, $f_0=i$, such that the following properties hold.
\begin{itemize}
\item $G(df_t)$ is $C^0$-close to $G_t$.
\item $f_t=i$ on $Op(\partial I^n)$. 
\end{itemize}
\end{lemma}

\begin{proof} We proceed exactly like we did in the proof of Lemma \ref{local wrinkling Lagrangians}. The Legendrian model is simply given by the Legendrian lift  $\widehat{\ell}=(\ell,K)$ of the Lagrangian model $\ell$ which exists because of the exactness condition $K=0$ on $[-1,1] \times Op(\partial I^n)$.  
\end{proof}
The parametric versions read as follows. Note that we also localize the problem from a general $m$-dimensional parameter space $Z$ to the unit cube $I^m=[-1,1]^m$.
\begin{lemma}[Parametric local wrinkling for Lagrangians]\label{parametric local wrinkling Lagrangians} Let $G^z_t:I^n \to \Lambda_n(T^*I^n)$ be a family of tangential rotations of the zero section $i:I^n \hookrightarrow T^*I^n$ parametrized by the unit cube $I^m$ which are all quasi-graphical and simple with respect to $H$, such that $G^z_t=G(di)$ on $Op(\partial I^n)$ and such that $G^z_t=G(di)$ for $z \in Op(\partial I^m)$. Then there exists a family of exact homotopies of wrinkled Lagrangian embeddings $f^z_t:I^n \to T^*I^n$, $f^z_0=i$, such that the following properties hold.
\begin{itemize}
\item $G(df^z_t)$ is $C^0$-close to $G^z_t$.
\item $f^z_t=i$ on $Op(\partial I^n)$. 
\item $f^z_t=i$ for $z \in Op(\partial I^m)$.
\end{itemize}
\end{lemma}
\begin{lemma}[Parametric local wrinkling for Legendrians]\label{parametric local wrinkling Legendrians} Let $G^z_t:I^n \to \Lambda_n\big(J^1(I^n, \bR) \big)$ be a family of tangential rotations of the zero section $i:I^n \hookrightarrow J^1(I^n, \bR)$ parametrized by the unit cube $I^m$ which are all quasi-graphical and simple with respect to $H$, such that $G^z_t=G(di)$ on $Op(\partial I^n)$ and such that $G^z_t=G(di)$ for $z \in Op(\partial I^m)$. Then there exists a family of exact homotopies of wrinkled Legendrian embeddings $f^z_t:I^n \to J^1(I^n, \bR)$, $f^z_0=i$, such that the following properties hold.
\begin{itemize}
\item $G(df^z_t)$ is $C^0$-close to $G^z_t$.
\item $f^z_t=i$ on $Op(\partial I^n)$. 
\item $f^z_t=i$ for $z \in Op(\partial I^m)$.
\end{itemize}
\end{lemma}
Lemmas \ref{parametric local wrinkling Lagrangians} and \ref{parametric local wrinkling Legendrians} are proved in the same way as Lemmas \ref{local wrinkling Lagrangians} and \ref{local wrinkling Legendrians}, adapting our construction to the fibered case as in \cite{EM09}. To be more precise, in the local model for the oscillating function $\xi$ we replace the box $D^n(t, \hat{q}) \times [-1,1](q_n)$ by the box $D^n(t, \hat{q}) \times [-1,1](q_n) \times D^m(z)$ and set
\[ \xi (t,q,z) = \gamma \,   \rho( ||z||) \, \rho\big( ||(t, \hat{q}) || \big)  \, \psi\big(|q_n|\big) \,\zeta_{\eta \left( ||(t,\hat{q})||\right)} \Big( \frac{2N+1}{ 2}q_n\Big), \, \, (t,\hat{q}) \in D^n , \, \, q_n \in [-1,1] , \, \, z \in D^m.\]

The rest of the proof can then be repeated carrying the parameter $z \in D^m$ along for the ride.

\subsection{Wrinkling the wiggles}\label{Wrinkling the wiggles} We are now ready to prove that tangential rotations can be globally approximated by Gauss maps of wrinkled embeddings. 

\begin{proof}[Proof of Theorem \ref{simple wrinkling lagrangians}] Let $G_t:L \to \Lambda_n(M)$ be a graphical simple rotation of a wrinkled Lagrangian or Legendrian embedding $f:L \to M$. Let $\Delta$ be a triangulation of $L$ which is compatible with the wrinkles of $f$ as in Section \ref{wiggling the wrinkles}. Set $K=\Delta^{n-1}$, the $(n-1)$-skeleton of $\Delta$. By Theorem \ref{perp-holonomic for wrinkles}, there exists an exact homotopy of wrinkled Lagrangian or Legendrian embeddings $g_t:L \to M$, $g_0=f$, (which is in fact an isotopy in the sense of Remark \ref{isotopyvhomotopy}) and a tangential rotation $R_t:L \to \Lambda_n(M)$ of $f$ such that the following properties hold.
\begin{itemize}
\item $G(dg_t)$ is $C^0$-close to $G_t$ on $Op(K)$.
\item $G(dg_t)$ is $C^0$-close to $R_t$ on all of $L$.
\item $R_t$ is graphical and simple with respect to the same hyperplane field $H$ as $G_t$.
\item $g_t=f$ and $R_t=G(df)$ outside of a slightly bigger neighborhood of $K$ in $L$.
\end{itemize}

Take an open $n-$simplex $D$ in $\Delta^n$, so that $g_t|_D : D \to M$ is an exact homotopy of regular Lagrangian or Legendrian embeddings. Replace the hyperplane field $H|_D$ by a constant hyperplane field, so that we reduce to the model $D \simeq (-1,1)^n$, $H=\text{span}(\partial / \partial q_1 , \ldots , \partial / \partial q_{n-1})$. The error in the approximation can be made arbitrarily small by first subdividing the triangulation $\Delta$ as finely as is necessary. There exists a tangential rotation $S_t:D \to \Lambda_n(M)$ covering the exact homotopy $g_t$ such that the following properties hold.
\begin{itemize}
\item $S_t$ is $C^0$-close to $G_t$.
\item $S_t$ is quasi-graphical and simple with respect to $H$.
\item $S_t=G(dg_t)$ on $Op(\partial D)$.
\end{itemize}

The rotation $S_t$ is nothing other than $G_t$ from the point of view of the deformed embedding $g_t$. Hence we can apply our local wrinkling Lemma \ref{local wrinkling Lagrangians} or \ref{local wrinkling Legendrians} to obtain an exact homotopy of wrinkled Lagrangian or Legendrian embeddings $\widetilde{g}_t:D \to M$ such that the following properties hold.
\begin{itemize}
\item $G(d \widetilde{g}_t)$ is $C^0$-close to $G_t$
\item $\widetilde{g}_t=g_t$ on $Op(\partial D)$.
\end{itemize}

The local wrinklings $\widetilde{g}_t$ glue up to an exact homotopy of wrinkled Lagrangian or Legendrian embeddings $f_t: L \to M$ such that $G(df_t)$ is $C^0$-close to $G_t$. The proof is complete. For an illustration of the argument, see Figure \ref{wigglingwrinkles}. \end{proof}

\begin{figure}[h]
\includegraphics[scale=0.6]{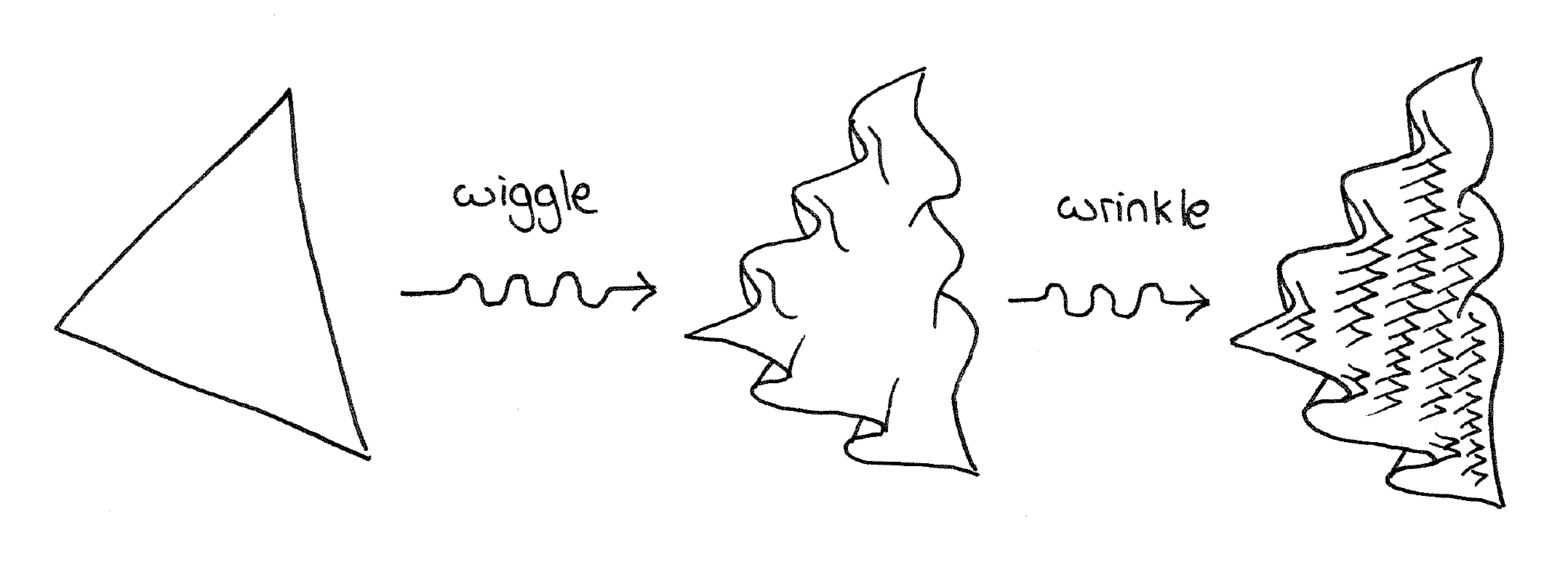}
\caption{The two-step process applied to a given simplex $D$. First we wiggle, then we wrinkle.}
\label{wigglingwrinkles}
\end{figure}

The proof of the parametric Theorem \ref{parametric simple wrinkling lagrangians} follows the same outline, using the parametric Theorem \ref{parametric perp-holonomic} instead of Theorem \ref{perp-holonomic} and using the parametric Lemmas \ref{parametric local wrinkling Lagrangians} and  \ref{parametric local wrinkling Legendrians} instead of Lemmas \ref{local wrinkling Lagrangians} and \ref{local wrinkling Legendrians}. The only essential difference is that in order to localize the parameter space from an arbitrary $m$-dimensional manifold $Z$ to the unit cube $Z=I^m$ we need to choose a triangulation $\Delta$ of $Z \times L$ which is compatible with the wrinkles of $f$ and such that every simplex in the triangulation is in general position with respect to the fibres $z\times L  \subset Z \times L$, $z \in Z$. The existence of such a triangulation was proved by Thurston in \cite{T74}. Once we know that such a triangulation exists, we can take the fibered polyhedron $K=\Delta^{n+m-1} \subset Z \times L$ and work simplex by simplex.

\section{The simplification of singularities}\label{applications to the simplification of singularities}

\subsection{Wrinkles, swallowtails and double folds}
We now return to the setting described in \mbox{Section \ref{introduction and statement of results}}. Let $M$ be a symplectic or contact manifold and let $\cF$ be a foliation of $M$ by Lagrangian or Legendrian leaves. Suppose that $f:L \to M$ is a winkled Lagrangian or Legendrian embedding which is transverse to $\cF$. We can apply the regularization procedure described in Section \ref{Regularization of wrinkles} to $f$ and obtain a regular Lagrangian or Legendrian embedding $\widetilde{f}:L \to M$. We already observed in Remark \ref{sigma type singularities} that $\widetilde{f}$ only has $\Sigma^1$-type singularities with respect to $\cF$, see Figure \ref{regularizeagain} for an illustration. More precisely, $\Sigma(\widetilde{f}, \cF)$ consists of a disjoint union of regularized wrinkles, which are defined as follows.

\begin{definition}
A regularized wrinkle of a regular Lagrangian or Legendrian embedding $g:L \to M$ is a connected component of the singularity locus $\Sigma(g,\cF)$ which consists of a topologically trivial codimension $1$ sphere $S \subset L$ such that we can decompose $S=D_1 \cup E \cup D_2$ into two hemispheres $D_1$ and $D_2$ and an equator $E$ satisfying the following. 
\begin{itemize} 
\item the equator $E$ consists of $\Sigma^{110}$ pleats.
\item the disks $D_1$ and $D_2$ consist of $\Sigma^{10}$ folds.
\end{itemize}
\end{definition} 

\begin{figure}[h]
\includegraphics[scale=0.6]{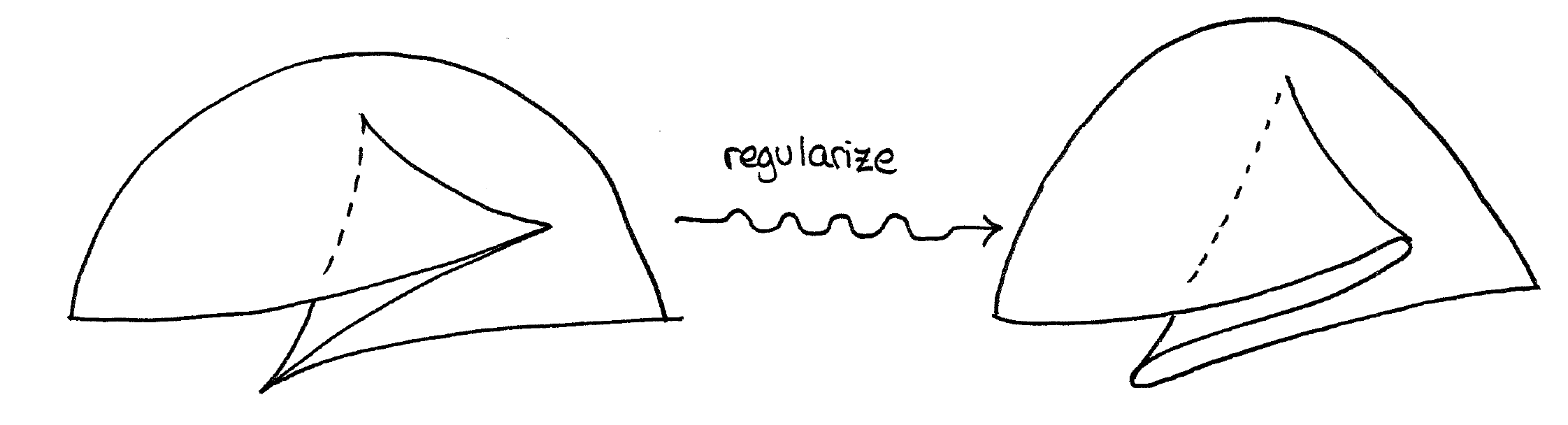}
\caption{One-half of a regularized wrinkle. In this picture, the ambient foliation should be thought of as being vertical.}
\label{regularizeagain}
\end{figure}

For a concrete local model, one can take the standard Lagrangian or Legendrian wrinkle defined in Section \ref{Lagrangian and Legendrian wrinkles}, after regularizing as described in Section \ref{Regularization of wrinkles}. In the Lagrangian case, the foliation $\cF$ of the cotangent bundle is given by the fibres of the standard projection $\pi:T^*\bR^n \to \bR^n$. In the Legendrian case, the foliation $\cF$ of the $1$-jet space $J^1(\bR^n, \bR) = T^*\bR^n \times \bR$ is given by the fibres of the projection $\pi \times id:T^*\bR^n \times \bR \to \bR^n \times \bR$.

\begin{remark}
If the foliation $\cF$ is induced by a Lagrangian fibration $\pi:M^{2n} \to B^n$, then for any regular Lagrangian embedding $f:L^n \to M^{2n}$ the following two conditions are equivalent. 
\begin{itemize}
\item the singularities of tangency of $g$ with respect to $\cF$ consist of a union of regularized wrinkles.
\item the front $\pi \circ g : L^n \to B^n$ is a generalized wrinkled mapping in the sense of \cite{EM09}.
\end{itemize}
\end{remark}

In the contact case where $\pi:M^{2n+1} \to B^{n+1}$ is a Legendrian fibration (which we think of as the front projection), we can think of regularized wrinkles in the following way. The singularities of tangency of a regular Legendrian embedding consist of a union $W=\bigcup_j S_j$ of regularized wrinkles if and only if the front  of the embedding has cusps on each sphere $S_j$ together with swallowtails on the equator $E_j$ of each $S_j$. 

Regularized wrinkles are also close relatives of the double folds introduced in Section \ref{Hierarchy of singularities}. We recall the definition for convenience.

\begin{definition} A double fold is a pair of topologically trivial $(n-1)$-spheres $S_1$ and $S_2$ in the fold locus $\Sigma^{10} \subset L$ which have opposite Maslov co-orientations and such that $S_1 \cup S_2$ is the boundary of an embedded annulus $A \subset L$.
\end{definition} 

Indeed, the Entov surgery of \cite{E97} can be used to open up a regularized wrinkle along its equator, producing a double fold. This is achieved by taking one of the two hemispheres of a regularized wrinkle $S \subset L$ and pushing it slightly away from $S$ while keeping it fixed on the equator $E$. We obtain an embedded disk $D \subset L$ contained in an arbitrarily small neighborhood of $S$ in $L$ such that $\partial D = E$ and $\text{int}(D) \cap S = \varnothing$. In fact, we require that $\text{int}(D)$ is outside of the $n$-ball $B \subset L$ bounded by $S$. The surgery construction removes the $\Sigma^{110}$ pleats from $E$ and trades them for $\Sigma^{10}$ folds on two parallel copies of $D$. One of these two parallel copies of $D$ is surgered onto one of the hemispheres of $S$ and the other parallel copy is surgered onto the other hemisphere, so that the end result consists of a disjoint union of two parallel spheres on which the embedding has $\Sigma^{10}$ folds. The Maslov co-orientations on the two resulting spheres are opposite of each other. Hence we end up with the desired double fold. See \cite{E97} for the details of the surgery construction and see Figure \ref{wrinklestofolds} for an illustration.

\begin{figure}[h]
\includegraphics[scale=0.6]{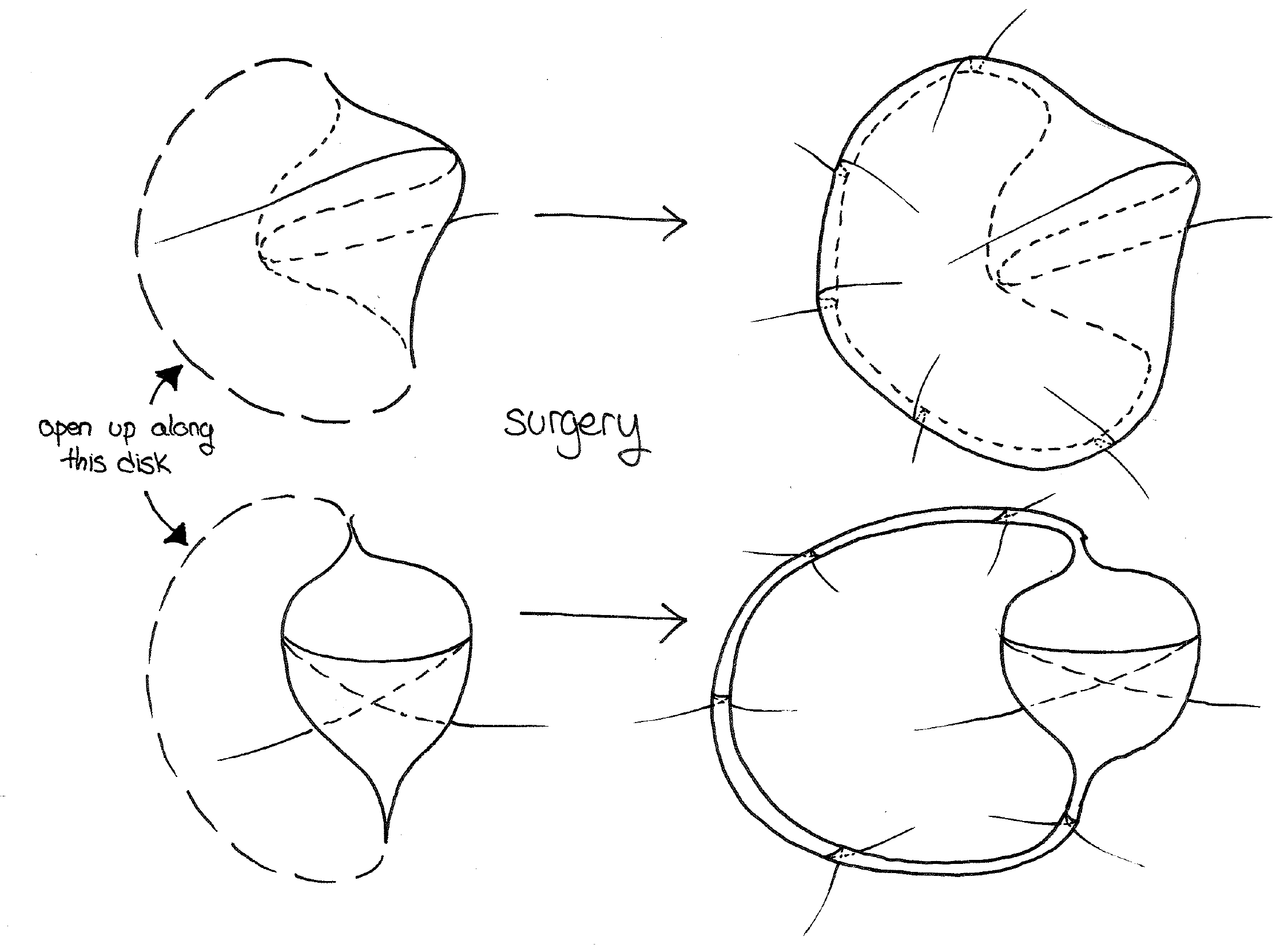}
\caption{Opening up a wrinkle into a double fold. The upper picture corresponds to the Lagrangian projection and the lower picture to the front projection.}
\label{wrinklestofolds}
\end{figure}

The precise statement that we will need is the following. Given a regular Lagrangian or Legendrian embedding $g:L \to M$ and given $S \subset \Sigma(g, \cF)$ a regularized wrinkle, there exists a $C^0$-small ambient Hamiltonian isotopy $\varphi_t:M \to M$ such that $\varphi_t = id_M$ outside of an arbitrarily small neighborhood of $g(S)$ in $M$ and such that inside this neighborhood the regularized wrinkle of $g$ is replaced by a double fold of $\varphi_1 \circ g$. If $g= \widetilde{f}$ is the regularization of a wrinkled Lagrangian or Legendrian embedding $f:L \to M$, then the wrinkles $S$ of $f$ will typically be nested. By this we mean that the ball $B \subset L$ bounded by any wrinkle $S$ of $f$ may contain other wrinkles of $f$. Hence when we apply the surgery construction on each regularized wrinkle of $g=\widetilde{f}$, we obtain a regular Lagrangian or Legendrian embedding $\varphi_1 \circ g$ whose singularity locus consists of a disjoint union of  double folds which are nested in the sense of Section \ref{Hierarchy of singularities}. 

\begin{remark}
We could of course have worked with double folds all along without ever mentioning wrinkles. Instead of defining wrinkled Lagrangian and Legendrian embeddings as we did, we could have defined `doubly cusped' Lagrangian and Legendrian embeddings to be smooth Lagrangian or Legendrian embeddings away from a finite union of pairs of parallel spheres where the embedding has cusps of opposite Maslov co-orientation (the cusps are semi-quintic in the ambient symplectic or contact manifold and semi-cubic in the front projection). 

Our $C^0$-approximation result for a tangential rotation $G_t$ would also hold for the class of doubly cusped Lagrangian and Legendrian embeddings. Moreover, the regularization of a doubly cusped Lagrangian or Legendrian embedding which is transverse to a foliation $\cF$ is a regular Lagrangian or Legendrian embedding whose singularities of tangency with respect to $\cF$ consist of double folds. The $h$-principle for the simplification of singularities proved below then follows with the same proof.

We have chosen to work with wrinkles for two reasons. One is historical, to draw the parallel with the smooth wrinkled embeddings theorem \cite{EM09}. The second is pedagogic, since wrinkling is not only a central notion in flexible geometric topology but also the main idea in our proof, so hiding it in the background would be dishonest.
\end{remark}

Suppose next that $f^z:L \to M$ is a family of wrinkled Lagrangian or Legendrian embeddings parametrized by a compact manifold $Z$. We can also in this case regularize and obtain a family of regular Lagrangian or Legendrian embeddings $\widetilde{f}^z:L \to M$. If $f^z$ is transverse to $\cF$, then the singularities of tangency of the family $\widetilde{f}^z$ with respect to $\cF$ consist of fibered regularized wrinkles. In particular, for some values of the parameter $z \in Z$ the regular Lagrangian or Legendrian embedding $\widetilde{f}^z$ will have regularized embryos in addition to regularized wrinkles. Regularized embryos are non-generic $\Sigma^1$-type singularities of tangency which occur at the instance of birth/death of a regularized wrinkle. One can of course give a concrete local model for the regularized embryo, however it is simpler to think about families as a single object using the fibered terminology. For a concrete local model, one can take the standard fibered Lagrangian or Legendrian wrinkle defined in Section \ref{Parametric families of wrinkles}, after regularizing as described in Section \ref{Regularization of wrinkles}. The foliation $\cF$ is given as in the non-parametric case.

\begin{remark}
If the foliation $\cF$ is induced by a Lagrangian fibration $\pi:M^{2n} \to B^n$, then for any family of regular Lagrangian embeddings $g:Z^m \times L^n \to M^{2n}$ the following two conditions are equivalent. 
\begin{itemize}
\item the singularities of tangency of $g$ with respect to $\cF$ consist of a union of fibered regularized wrinkles.
\item the fibered front $p \circ g :Z^m \times L^n \to Z^m \times B^n$ is a fibered generalized wrinkled mapping in the sense of \cite{EM09}.
\end{itemize}
\end{remark}

In the Legendrian case one can of course reinterpret what fibered regularized wrinkles mean in the front projection in terms of cusps and swallowtails. Note that one can also use the Entov surgery in families to replace fibered regularized wrinkles with fibered double fold singularities. The embryos of regularized wrinkles will become embryos of double folds. An embryo of a double fold is a non-generic locus of $\Sigma^1$-type singularities of tangency consisting of a single codimension $1$ sphere from which the two parallel spheres of folds can either be born or die, see Figure \ref{embryosofdoublefolds}. 

\begin{figure}[h]
\includegraphics[scale=0.6]{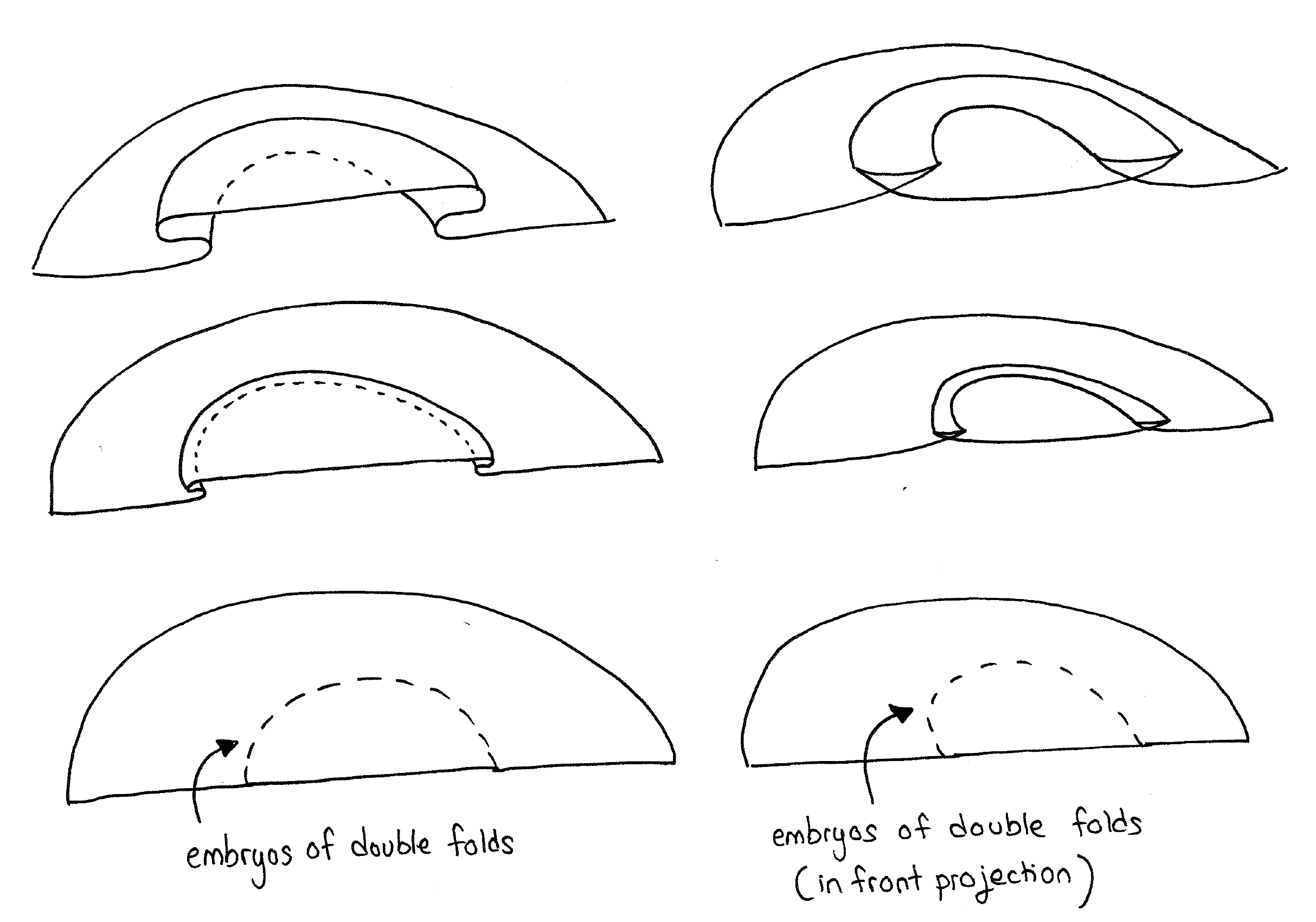}
\caption{One-half of the birth/death of a double fold. The picture on the left corresponds to the Lagrangian projection and the picture on the right corresponds to the front projection.}
\label{embryosofdoublefolds}
\end{figure}

\subsection{The $h$-principle for the simplification of singularities} We are now ready to establish the flexibility of singularities of Lagrangian and Legendrian fronts. As above, $\cF$ denotes a foliation by Lagrangian or Legendrian leaves of a symplectic or contact manifold $M$.

\begin{theorem}\label{main result again} Suppose that there exists a tangential rotation $G_t:L \to \Lambda_n(M)$ of a regular Lagrangian or Legendrian embedding $f:L \to M$ such that $G_1 \pitchfork \cF$. Then there exists a compactly supported ambient Hamiltonian isotopy $\varphi_t:M \to M$ such that the singularities of $\varphi_1 \circ f$ consist of a union of nested regularized wrinkles.
\end{theorem}

\begin{proof}
Apply the wrinkling Theorem \ref{wrinkling lagrangians} to $G_t$ and $f$. We obtain a compactly supported exact homotopy of wrinkled Lagrangian or Legendrian embeddings $f_t:L \to M$ such that $G(df_1) \pitchfork \cF$. Next, apply the regularization process described in Section \ref{Regularization of wrinkles} to the homotopy $f_t$. We obtain a compactly supported exact homotopy of regular Lagrangian or Legendrian embeddings $\widetilde{f}_t:L \to M$ such that the singularity locus $\Sigma(\widetilde{f}_1, \cF) \subset L$ consists of a disjoint union of regularized wrinkles.  Finally, since the homotopy $\widetilde{f}_t$ is exact and compactly supported, we can write $\widetilde{f}_t=\varphi_t \circ f$ for some compactly supported ambient Hamiltonian isotopy $\varphi_t:M \to M$.
\end{proof}

To deduce the version with double folds stated in Theorem \ref{main result}, we simply apply the Entov surgery construction of \cite{E97} to open up each of the wrinkles as described in the previous section. 

\begin{remark}
At each stage of the proof, when we apply Theorem \ref{wrinkling lagrangians}, the regularization of Section \ref{Regularization of wrinkles} and the Entov surgery, we can always ensure that the resulting homotopy of embeddings is $C^0$-close to $f$. Hence Theorem \ref{main result again} also holds in $C^0$-close form, where we demand that the Hamiltonian isotopy $\varphi_t$ is $C^0$-close to the identity $id_M$. Moreover, we can also ensure that $\varphi_t=id_M$ outside of a neighborhood of $f(L)$ in $M$.
\end{remark}

\begin{remark}
Suppose that $G_t=G(df)$ on $Op(A)$ for some closed subset $A \subset L$. At each stage of the proof, when we apply Theorem \ref{wrinkling lagrangians}, the regularization of Section \ref{Regularization of wrinkles} and the Entov surgery, we can always ensure that the resulting homotopy of embeddings agrees with $f$ on $Op(A)$. Hence Theorem \ref{main result again} also holds in relative form. More precisely, we can demand that $\varphi_t=id_M$ on $Op\big(f(A)\big)\subset M$. 
\end{remark}

The parametric version reads as follows, and is proved in exactly the same way. At each stage we just need to invoke the parametric versions of each of the ingredients of the proof. The corresponding $C^0$-close and relative versions also hold, for the same reasons as in the non-parametric case.

\begin{theorem}\label{parametric main result again} Suppose that there exists a family of tangential rotations $G^z_t:L \to \Lambda_n(M)$ of regular Lagrangian or Legendrian embeddings $f^z:L \to M$ parametrized by a compact manifold $Z$ such that $G^z_1 \pitchfork \cF$ for all $z \in Z$ and such that $G^z_t=G(df^z)$ for $z \in Op(\partial Z)$. Then there exists a family of compactly supported ambient Hamiltonian isotopies $\varphi^z_t:M \to M$ such that the singularities of $\varphi^z_1 \circ f^z$ consist of a union of fibered nested regularized wrinkles and such that $\varphi^z_t=id_M$ for $z \in Op(\partial Z)$.
\end{theorem}

As in the non-parametric case we can open up the fibered regularized wrinkles into fibered double folds using the Entov surgery construction \cite{E97}.

\begin{remark}
Observe that in the case $n=1$ there is no need to resolve a wrinkle into a double fold. Indeed a $1$-dimensional regularized wrinkle consists of nothing more than a pair of points where the embedding has folds of opposite Maslov co-orientation. For fibered regularized wrinkles the two folds die as in the Legendrian Reidemeister I move. We explore the case $n=1$ further in Section \ref{families of $1$-dimensional Legendrians} below.

\end{remark}

\subsection{The $h$-principle for the prescription of singularities}\label{prescription of singularities}
We next prove a strengthened version of Entov's Theorem \ref{Entov theorem}. More precisely, we apply our $h$-principle Theorem \ref{main result again} to drop the $\Sigma^2$-nonsingularity restriction from his result. As an application we establish some concrete results for the simplification of the caustics of spheres in Section \ref{the caustics of spheres} below.

Consider $f:L  \to M$ a Lagrangian or Legendrian embedding and let $D$ be a Lagrangian distribution  in $TM$ defined along $f(L)$. In the symplectic case, $D$ consists of linear Lagrangian subspaces of $(TM,\omega)$ and in the contact case $D$ consists of linear Lagrangian subspaces of $(\xi, d \alpha)$, where locally $\xi = \ker(\alpha) \subset TM$.

 When $\dim(df(T_qL) \cap D_{f(q)} )<2$ for all $q \in L$ we say that $D$ is $\Sigma^2$-nonsingular. In this case, the structure of the singularity locus $\Sigma = \{ q \in L : \, \, df(T_qL) \cap D_{f(q) } \neq 0 \}$ is quite simple. Indeed, for generic $\Sigma^2-$nonsingular $D$ the locus $\Sigma$ is a codimension $1$ submanifold which is naturally stratified as a flag $\Sigma=\Sigma^1 \supset \Sigma^{11} \supset \cdots  \supset \Sigma^{1^n}$ as described in Section \ref{Hierarchy of singularities}. Moreover, the flag comes equipped with certain co-orientation data which we hinted about in Section \ref{Surgery of singularities} and which more precisely consist of the following.
\begin{itemize}
\item Unit vector fields $v_k$, $k>1$, where each $v_k$ is defined on $\Sigma^{1^k} \setminus \Sigma^{1^{k+1}}$ , is normal to $\Sigma^{1^{k-1}}$ in $\Sigma^{1^{k-2}}$ and cannot be extended (as such a unit normal vector field) to any subset $C \subset \Sigma^{1^k}$ which has a nontrivial intersection with $\Sigma^{1^{k+1}}$.
\item An additional unit vector field $v_1$ defined on the whole of $\Sigma$ which is normal to $\Sigma$ in $L$. This vector field is called the Maslov co-orientation.
\end{itemize}

Adapting Eliashberg's terminology from \cite{E72}, Entov defined in \cite{E97} the chain of singularities associated to $f$ and $D$ to consist of the flag $\Sigma^1 \supset \Sigma^{11} \supset \cdots \supset\Sigma^{1^n}$ together with vector fields $v_k$ as above. The $v_k$ are uniquely determined by the geometry of the singularity. See Figure \ref{chain} for an illustration. Two chains of singularities are said to be equivalent if there exists an isotopy of $L$ that transforms one into the other, including the co-orientation data. We can now state and prove an $h$-principle which allows for the prescription of any homotopically allowable chain of singularities. The result also holds in $C^0$-close and relative forms.

\begin{figure}[h]
\includegraphics[scale=0.8]{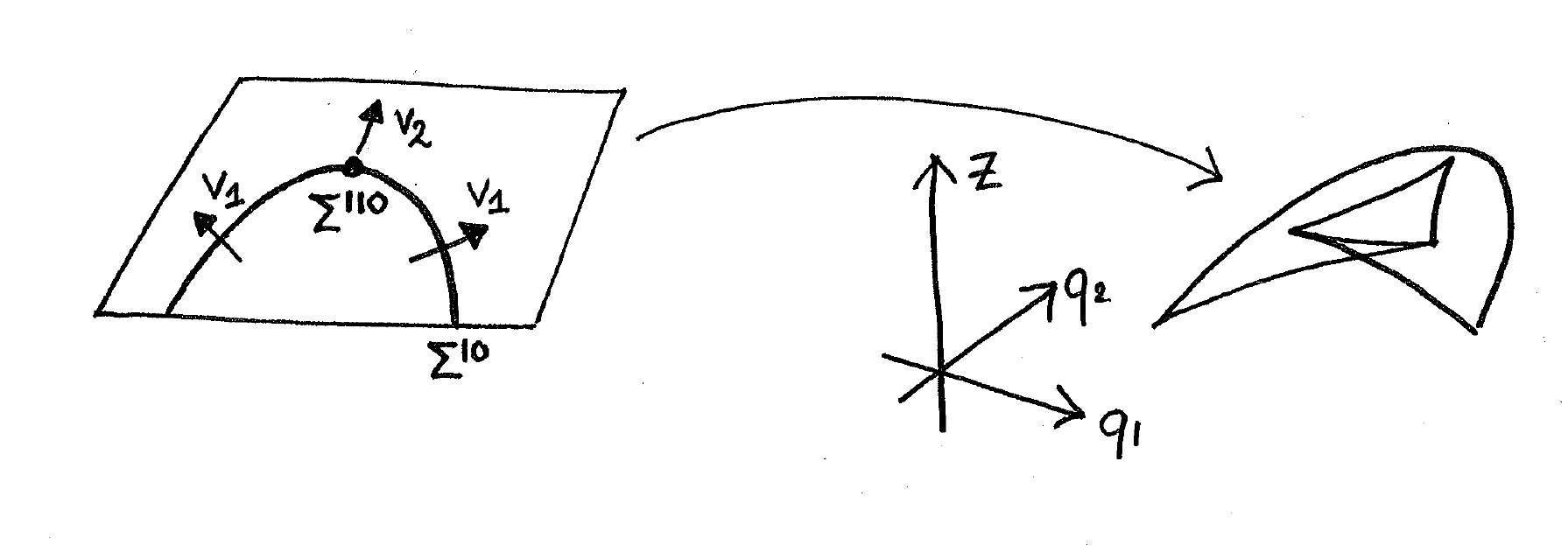}
\caption{The chain of singularities associated to the $\Sigma^{110}$ pleat, which is a swallowtail in the front projection. A flip of the Legendrian front in the $z$ direction would reverse the Maslov co-orientation $v_1$ and fix $v_2$.}
\label{chain}
\end{figure}

\begin{theorem}\label{Entov theorem redux} Let $f:L \to M$ be a regular Lagrangian or Legendrian embedding into a symplectic or contact manifold $M$ equipped with a foliation $\cF$ by Lagrangian or Legendrian leaves. Let $D_t$ be a homotopy of Lagrangian distributions defined along $f(L)$, fixed outside of a compact subset, such that $D_0=T\cF|_{f(L)}$ and such that $f$ is $\Sigma^2$-nonsingular with respect to the distribution $D_1$. We moreover assume that $f \pitchfork \cF$ outside of that compact subset. Then there exists a $C^0$-small compactly supported Hamiltonian isotopy $\varphi_t:M \to M$ such that $\varphi_1 \circ f$ is $\Sigma^2$-nonsingular with respect to $\cF$ and moreover such that the chain of singularities of $\varphi_1 \circ f$ with respect to $\cF$ is equivalent to the chain of singularities of $f$ with respect to $D_1$, together with a union of nested double folds.
\end{theorem}

\begin{proof}

We restrict our attention to the Lagrangian case for concreteness, the Legendrian analogue is no different. Let $\Sigma \subset L$ be the singularity locus of $f$ with respect to $D_1$. By abusing notation, we will also denote by $\Sigma$ the chain of singularities which encodes the flag $\Sigma=\Sigma^1 \supset \Sigma^{11} \supset \cdots \supset \Sigma^{1^n}$ and the corresponding co-orientation data.  Let $\Phi_t$ be a homotopy of linear symplectic isomorphisms of $TM$ defined along $f(L)$ such that $\Phi_0=id$ and $\Phi_t \cdot D_0 = D_t$. Set $G_t=(\Phi_t)^{-1} \cdot G(df)$, a tangential rotation of $f$.  

Our plan will be the following. We will first apply our holonomic approximation lemma for $1-$holonomic sections to $G_t$ to make $f$ transverse to $\cF$ near a parallel copy $\Sigma_{1/2}$ of $\Sigma$. Then we will introduce by hand a cancelling pair of singularity loci $\Sigma_1$ and $\Sigma_2$ in $Op(\Sigma_{1/2})$ such that $\Sigma_2$ is equivalent to $\Sigma$ and such that $\Sigma \cup \Sigma_1$ bounds an embedded annulus which is disjoint from $\Sigma_2$. Formally, $\Sigma$ and $\Sigma_1$ can be cancelled via a rotation $R_t$ which is fixed on $\Sigma_2$ and hence by our relative $h-$principle for the simplification of singularities we are able to keep the singularity locus $\Sigma_2$ and fill in the rest of the Lagrangian submanifold with double folds. See Figure \ref{prescriptionofsingularities} for an illustration of the strategy.

\begin{figure}[h]
\includegraphics[scale=0.7]{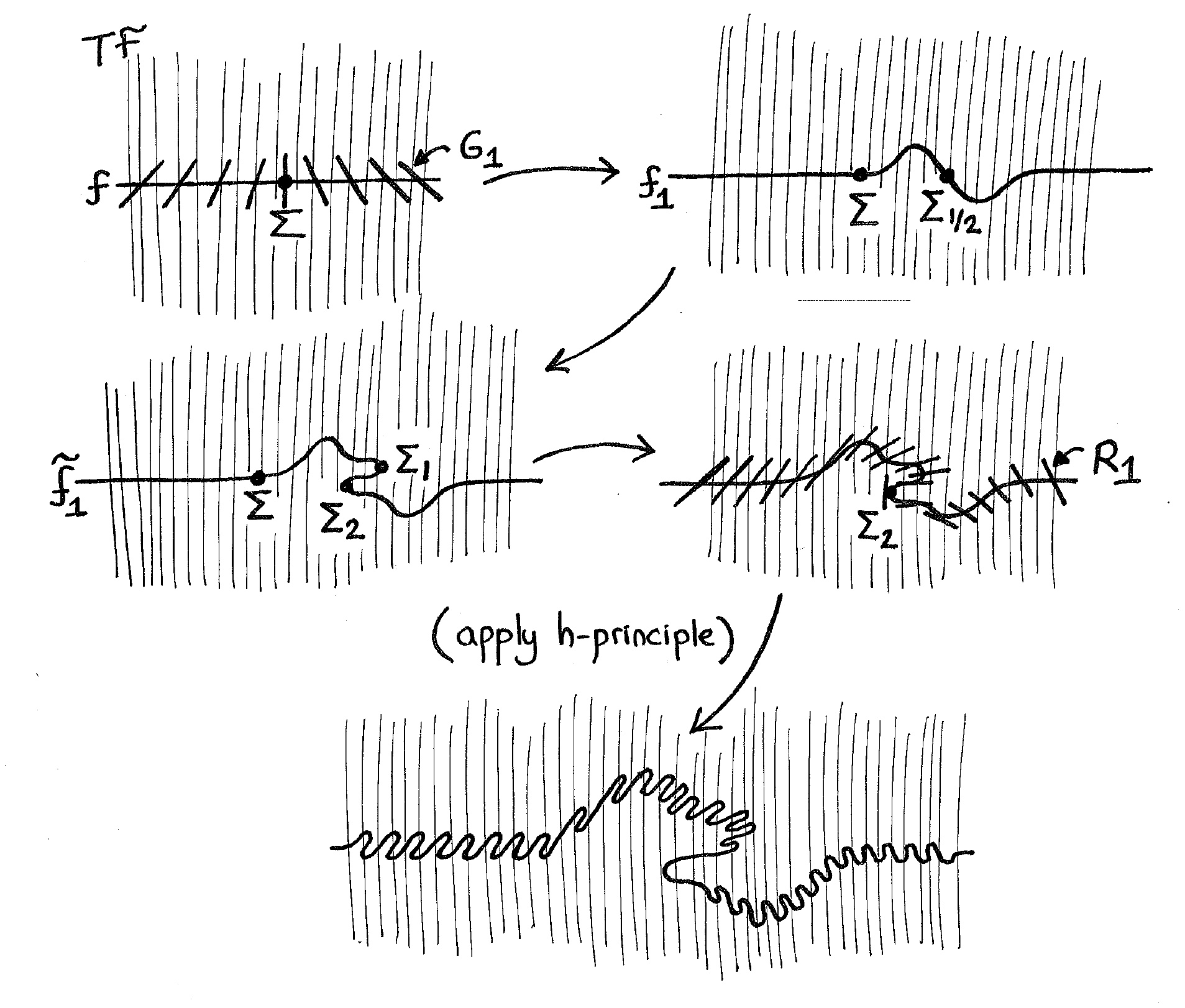}
\caption{The plan for our proof of Theorem \ref{Entov theorem redux}.}
\label{prescriptionofsingularities}
\end{figure}

Let $l=(df)^{-1}(D_1)$, which is a line field on $TL$ defined along $\Sigma$. Extend $l$ to a tubular neighborhood $\cN\simeq \Sigma \times (-1,1)$ of $\Sigma$ in $L$. Denote by $\Sigma_{1/2}$ the parallel copy $\Sigma \times  \frac{1}{2} $ of $\Sigma$ in $\cN$.  Apply Theorem \ref{1-holonomic} to the tangential rotation $G_t$ and the stratified subset $K= \Sigma_{1/2}$. We obtain an exact homotopy of regular Lagrangian embeddings $f_t:L \to M$ such that $G(df_t)$ is $C^0$-close to $G_t$ on $Op(\Sigma_{1/2})$. In particular, $f_1 \pitchfork \cF$ on a neighborhood $U=\Sigma \times (1/2-\varepsilon,\,  1/2 + \varepsilon)$ of $\Sigma_{1/2}$.

Along $f_1(U)$, the Lagrangian distributions $df_1(TU)$ and $T\cF|_{f_1(U)}$ are transverse. We can therefore choose a symplectic isomorphism $ T(T^*U)|_U \simeq TM|_{f_1(U) }$ such that the horizontal distribution $TU$ (which is tangent to the zero section) is mapped to $df_1(TU)$ using $df_1$ and such that the vertical distribution $VU$ (which is tangent to the cotangent fibres) is mapped to $T \cF|_{f_1(U)}$. Choose an $(n-1)$-dimensional complement $P$ for $l$ in $TU$. Set $l^* = P^{\perp} \cap VU$ and $P^* = l^{\perp} \cap VU$, where $\perp$ denotes orthogonality with respect to the symplectic form $dp \wedge dq$. Let $\phi: [1/2 - \varepsilon, 1/2 + \varepsilon] \to \bR$ be a function satisfying the following properties.
\begin{itemize}
\item $\phi(s)=0$ for $s$ near  $1/2 \pm \varepsilon$.
\item $\phi(s) = \pi$  for $s$ near $1/2$.
\item $\phi'(s) \geq 0$ for $s \in (1/2-\varepsilon ,  1/2]$ and $\phi'(s) \leq 0$ for $s \in [1/2, 1/2+ \varepsilon)$.
\end{itemize}

Fix nonzero vector fields $v \in l$ and $w \in l^*$. Define a homotopy of Lagrangian distributions $V_t \subset T(T^*L)$ defined along $U= \Sigma \times (1/2 - \varepsilon, \, 1/2 + \varepsilon)$ by the formula
\[ V_t ( e, s) = \text{span}\big(\sin\big(t\phi(s)\big)v + \cos\big(t\phi(s)\big) w \big)\oplus P^*,  \quad (e,s) \in \Sigma \times (1/2- \varepsilon , \, 1/2 + \varepsilon) . \]

Note that $V_0=VU$, that $V_t= VU$ on $\partial U$ and that $\dim(V_t \cap TU) \leq 1$ for all $t \in [0,1]$. The singularities of tangency of $V_1$ with respect to the zero section $U \hookrightarrow T^*U$ consist of two parallel copies $\Sigma'$ and $\Sigma''$ of $\Sigma$, for concreteness say $\Sigma'$ is between $\Sigma$ and $\Sigma''$. Along these singularitiy loci we have $V_1 \cap TU = l$. The two corresponding chains of singularities, which we also denote by $\Sigma'$ and $\Sigma''$, have opposite Maslov co-orientations but are otherwise equivalent. Replacing the function $\phi$ by the function $-\phi$ if necessary, we may assume that the chain of singularities $\Sigma''$ is equivalent to the chain $\Sigma$.

At this point we wish to use $V_t$ to insert by hand a cancelling pair of singularities modelled on $\Sigma$. The explicit formulas that we need are already written down in Entov's paper \cite{E97}. We could use these formulas to write down a concrete model for the creation of the cancelling pair, but we can make our life even easier by directly applying Entov's Theorem \ref{Entov theorem} to $V_t$. The output of Entov's theorem is an exact homotopy of regular Lagrangian embeddings $g_t:U \to T^*U$ such that $g_0$ is the inclusion of the zero section $U \hookrightarrow T^*U$, such that $g_t$ is fixed on $Op(\partial U)$ and such that the singularities of tangency of $g_1$ with respect to $VU$ are equivalent to those of $g_0$ with respect to $V_1$, together with a union of nested double folds. Furthermore, the homotopy $g_t$ can be assumed to be $C^0$-small, so by taking an appropriate Weinstein neighborhood we can think of this homotopy as happening inside $M$. The result is an exact regular homotopy $\widetilde{f}_t:L \to M$ of $\widetilde{f}_0=f_1$ such that along $U \subset L$ the singularities of tangency of $\widetilde{f}_1$ with respect to $\cF$ consist of a union $\Sigma_1 \cup \Sigma_2 \cup F$, where the chain $\Sigma_1$ is equivalent to $\Sigma'$, the chain $\Sigma_2$ is equivalent to $\Sigma''$ and $F$ is a union of nested double folds. Moreoever, $\Sigma \cup \Sigma_1$ bounds an annulus $A \subset L$ which is disjoint from $\Sigma_2$.

\begin{claim}\label{homotopic claim}
There exists a tangential rotation $R_t:L \to \Lambda_n(M)$ of $\widetilde{f}_1$ which is fixed on $Op(\Sigma_2)$ and such that $R_1 \pitchfork \cF$ away from $\Sigma_2$.
\end{claim}

Once this claim is established we are done, since we can apply the relative version of Theorem \ref{main result again} to construct an exact homotopy of regular Lagrangian embeddings which is fixed on $Op(\Sigma_2)$ and such that at the end of the homotopy the singularities of tangency away from $\Sigma_2$ consist of a union of nested double folds, which is exactly what we wanted to prove. 

To justify the claim, we first observe that there exists a tangential rotation $S_t:L \to \Lambda_n(M)$ of $\widetilde{f}_1$ such that $S_t$ is fixed on $Op(\Sigma_1 \cup \Sigma_2)$, such that $S_1=G_1$ outside of $U$ and such that $S_1 \pitchfork \cF$ away from $\Sigma_1 \cup \Sigma_2 \cup \Sigma$. To define $S_t$, choose $\delta_1 < \delta_2 < \varepsilon$ such that the annuli $U_i=\Sigma \times (1/2 -\delta_i , 1/2 + \delta_i) \subset U$ contain $\Sigma_1 \cup \Sigma_2 \cup F$. Inside of $U_1$, we let $S_t$ kill the double folds of $F$ so that the only remaining singularities are $\Sigma_1 \cup \Sigma_2$. On the rest of $L$ (where we may assume that $\widetilde{f}_1=f_1$ provided that $\delta_1$ and $\delta_2$ are close enough to $\varepsilon$), we construct $S_t$ in three steps.
\begin{itemize}
\item First, rotate $G(d\widetilde{f}_1)=G(df_1)$ to a distribution $W$ which equals $G(df_0)$ away from $U$ and which interpolates between $G(df_0)$ and $G(df_1)$ on $U \setminus U_2$ by means of $G(df_t)$.
\item Since $G(df_t)$ is $C^0$-close to $G_t$ on $U$, we can then rotate $W$ to a distribution $W'$ which equals $G_0=G(df_0)$ away from $U$, which interpolates between $G_0$ and $G_1$ on $U \setminus U_2$ by means of $G_t$ and which then interpolates between $G_1$ and $G(df_1)$ on $U_2 \setminus U_1$. 
\item  We can then rotate $W'$ to a distribution $W''$ which equals $G_1$ outside of $U_2$ and which interpolates between $G_1$ and $G(df_1)$ on $U_2\setminus U_1$. The distribution $W''=S_1$ satisfies the required properties and the rotation $S_t$ is the concatenation of the three steps.
\end{itemize}

Consider now the annulus $A\subset L$ with boundary $\partial A = \Sigma \cup \Sigma_1$. The intersection $\lambda =\text{im}(S_1) \cap T\cF \subset TM$ consists of two line fields defined over the images of $\Sigma$ and $\Sigma_1$. We claim that they extend to a line field $\lambda \subset \text{im}(S_1)$ defined over the image of the whole annulus $A$.

 Indeed, the chain of singularities of $\Sigma_1$ is equivalent to that of $\Sigma$ up to Maslov co-orientation. But the isotopy class of the line field which arises from a $\Sigma^1$-type singularity locus is completely dictated by the flag $\Sigma^1 \supset \Sigma^{11} \supset \cdots \supset \Sigma^{1^n}$ together with the non-Maslov co-orientation data. Hence the line fields are isotopic in $TL$. It follows that we can find a line field $\widetilde{l} \subset TL$ defined along $A$ such that $\widetilde{l}|_{\Sigma_1}=d\widetilde{f}_1^{-1}(\lambda)$ and such that $\widetilde{l}|_\Sigma=df^{-1} \circ \Phi_1(\lambda)$.
 
  Suppose that there exists a family of symplectic isomorphisms $\Psi_t$ of $TM$ such that $\Psi_0=id$, such that $\Psi_t \cdot G(df) = S_t$, such that $\Psi_1 \circ df=d\widetilde{f}_1$ near $\Sigma_1$ and such that $\Psi_1=\Phi_1^{-1}$ near $\Sigma$. Then the line field $\lambda = \Psi_1  \circ df(\widetilde{l})$ is the required extension. It remains to confirm that the family $\Psi_t$ exists. We need to define $\Psi_t$ over $A \times [0,1]$, where $t \in [0,1]$ and we have prescribed $\Psi_t$ over $A \times 0  \cup ( \partial A \times [0,1])$.  Furthermore, we also have prescribed the image of $\Psi_t$ under the map $\Psi_t \mapsto \Psi_t \cdot G(df)$ over all of $A \times [0,1]$. Since this map is a Serre fibration, it follows that we can find a lift to all of $A \times [0,1]$. This completes the proof of the existence of the line field $\lambda \subset \text{im}(S^1)$.
 
Next we observe that the distribution $S_1:A \to \Lambda_n(M)$ satisfies $S_1(\partial A) \subset \Sigma^1(M, \cF) = \bigcup_{x \in M}\{ W \in \Lambda_n(M)_x : \, \dim(W \cap T_x \cF )=1\}$ and $S_1\big( \text{int}(A) \big) \subset \Lambda_n^{\pitchfork}(M, \cF) = \bigcup_{x \in M} \{ W \in \Lambda_n(M)_x : \, W \cap T_x \cF = 0 \}$. Pick a complement $Q \subset \text{im}(S^1)$ to $\lambda$. Set $\lambda^* = Q^{\perp} \cap T\cF$ and $Q^* = \lambda ^{\perp} \cap T\cF $. Pick nonzero vector fields $v \in \lambda$ and $w \in \lambda^*$ such that $\omega(v,w)>0$ on $\text{int}(A)$ and define a rotation $R_t:A \to \Lambda_n(M)$ starting at $R_0=S_1$ by the formula
\[ R_t = \text{span} \big( \cos(\pi t/2) v + \sin ( \pi t /2) w \big) \oplus Q  , \qquad t \in [0,1]. \]

Observe that on $\partial A$ we have $\lambda^* = \lambda$ and hence $R_t=S_1$ for all $t \in [0,1]$. Hence we can extend $R_t$ outside of $A$ by letting it equal $S_1$ elsewhere. Observe also that $R_1 \cap T \cF = \lambda^*$ along $A$ and hence $\text{im}(R_1)|_A \subset \Sigma^1(M, \cF)$. Recall that $\Sigma^1(M, \cF)$ is a two-sided hypersurface of $\Lambda_n(M)$, so that if $\cO \subset \Lambda_n(M, \cF)$ is a small enough neighborhood of $\text{im}(R_1)|_A$, then $\cO \setminus \Sigma^1(M, \cF)$ has exactly two connected components. The fact that the Maslov co-orientations of $\Sigma_1$ and $\Sigma$ are opposite means precisely that $\text{im}(S_1)|_{ Op(A) \setminus A}$ lies in the same connected component of  $\cO \setminus \Sigma^1(M, \cF)$. Hence we can push the image of $R_1$ entirely off of $\Sigma^1(M, \cF)$ by a small deformation which is fixed outside of $Op(A)$. The result is a rotation $\widetilde{R}_t : L \to \Lambda_n(M)$ starting at $\widetilde{R}_0=S_1$ such that $\widetilde{R}_1 = S_1 $ on $Op( \Sigma_2)$ and such that $\widetilde{R}_1 \pitchfork \cF$ away from $\Sigma_2$. This completes the proof of Claim \ref{homotopic claim}, hence also of Theorem \ref{Entov theorem redux}.   \end{proof}

\subsection{Application: the caustics of spheres}\label{the caustics of spheres}

We now return to the first example considered in Section \ref{overview}. Our goal is to study the extent to which it is possible to simply the caustic of an embedded Lagrangian or Legendrian sphere $S \subset M$, where $M$ is a symplectic or contact manifold equipped with a foliation $\cF$ by Lagrangian or Legendrian leaves. For greater clarity of the exposition we will restrict our discussion to the Lagrangian version of the problem, but the Legendrian analogue is no different. 

First we observe that by the Weinstein neighborhood theorem we can immediately reduce to the case where $M=T^*S^n$ and $S$ is the image of the zero section $S^n \hookrightarrow T^*S^n$, which we will also denote by $S^n$. Note that for $n=1$  the problem is uninteresting because the generic caustic consists only of folds, so the simplification of singularities can be trivially achieved. We assume $n>1$ in what follows. Let $V$ be the restriction to $S^n$ of the distribution $T\cF$ of Lagrangian planes tangent to $\cF$. We begin with the following topological obstruction to the simplification of singularities.

\begin{proposition}\label{obstruction}
If $S^n$ is Hamiltonian isotopic to a Lagrangian sphere whose singularities of tangency with respect to $\cF$ consist only of folds, then $V$ is stably trivial as a real vector bundle over the sphere.
\end{proposition}

We precede the proof with some notation. Let $\Sigma \subset S^n$ be any compact hypersurface. Following \cite{EM12(2)}, it is conceptually useful to introduce a real $n$-dimensional vector bundle $T_\Sigma S^n$ which is obtained from $TS^n$ by regluing along $\Sigma$ with a fold. More precisely, write $S^n=X \cup Y$ for $X,Y\subset S^n$ two compact $n$-dimensional submanifolds whose common boundary $\partial X = X \cap Y = \partial Y$ is the hypersurface $\Sigma$. Fix also an identification $TS^n|_\Sigma \simeq T\Sigma \oplus \varepsilon$, where $\varepsilon$ denotes the trivial line bundle. Define $T_\Sigma S^n$ to be the real $n$-dimensional vector bundle over $S^n$ given by gluing the disjoint union $TX \coprod TY$ over the intersection $X \cap Y = \Sigma$ via the isomorphism
\[ \mu= id \oplus (-1) : T\Sigma \oplus \varepsilon \to T \Sigma \oplus \varepsilon.\]

\begin{figure}[h]
\includegraphics[scale=0.7]{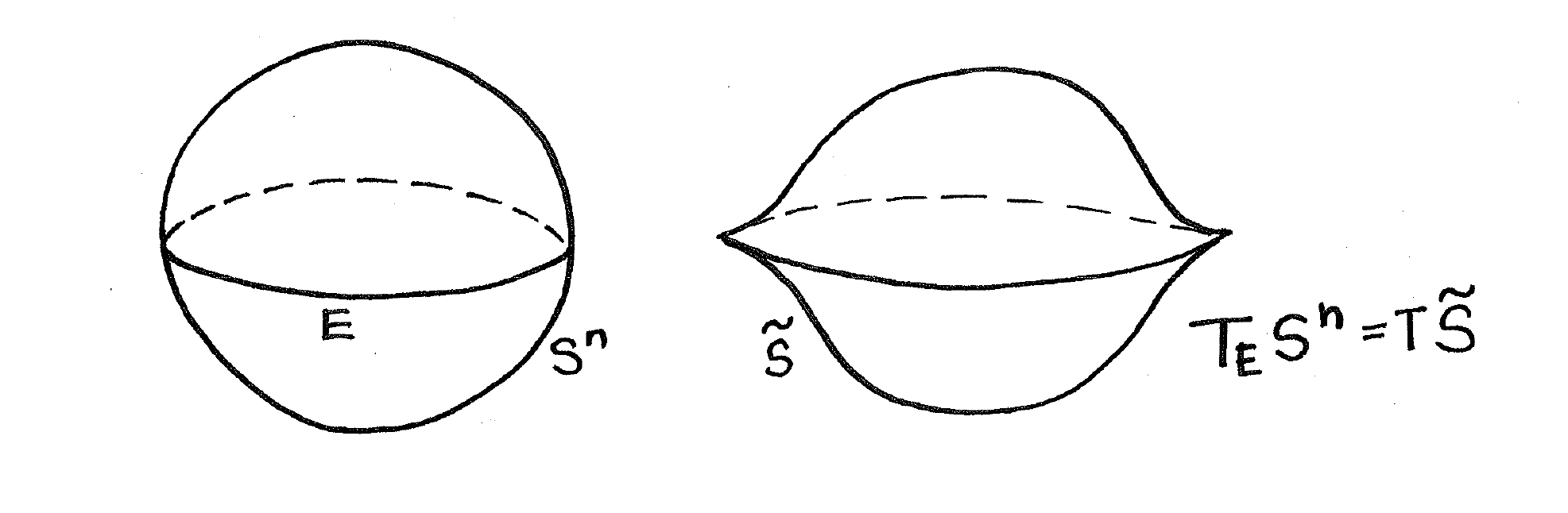}
\caption{The vector bundle $T_ES^n$ can be visualized as the tangent bundle of a singular surface $\widetilde{S}$ as illustrated above.}
\label{foldedtangentspace}
\end{figure}

The bundle $T_\Sigma S^n$ can be realized as a distribution of Lagrangian planes $V_\Sigma$ in $T^*S^n$ defined along the zero section $S^n \hookrightarrow T^*S^n$ whose singularities of tangency with respect to the zero section $S^n$ consist of folds along $\Sigma$. In order to do this, we fix a co-orientation of $\Sigma$, which will agree with the Maslov co-orientation induced by $V_\Sigma$. Let $ \Sigma \times (-1,1) \simeq \cN \subset S^n$ be a tubular neighborhood of $\Sigma$ such that the canonical orientation of the interval $(-1,1)$ induces the chosen co-orientation of $\Sigma$. The Lagrangian Grassmannian $\Lambda_1\big(T^*(-1,1)\big)$ is the trivial circle bundle $T^*(-1,1) \times S^1$. We use the canonical coordinates $(q,p) \in (-1,1) \times \bR = T^*(-1,1)$. Let $l:(-1,1) \to \Lambda_1\big( T^*(-1,1) \big)$ be the rotating line field defined over the zero section $(-1,1) \hookrightarrow T^*(-1.1)$ by the formula 
\[ l_q=  \text{span} \big(  \cos\left( \frac{\pi i q }{ 2}\right)  \frac{\partial}{\partial q} + \sin \left(\frac{ \pi iq }{2} \right)\frac{\partial}{\partial p} \, \big)\subset T_{(q.0)}\big(T^*(-1,1) \big). \]

\begin{figure}[h]
\includegraphics[scale=0.7]{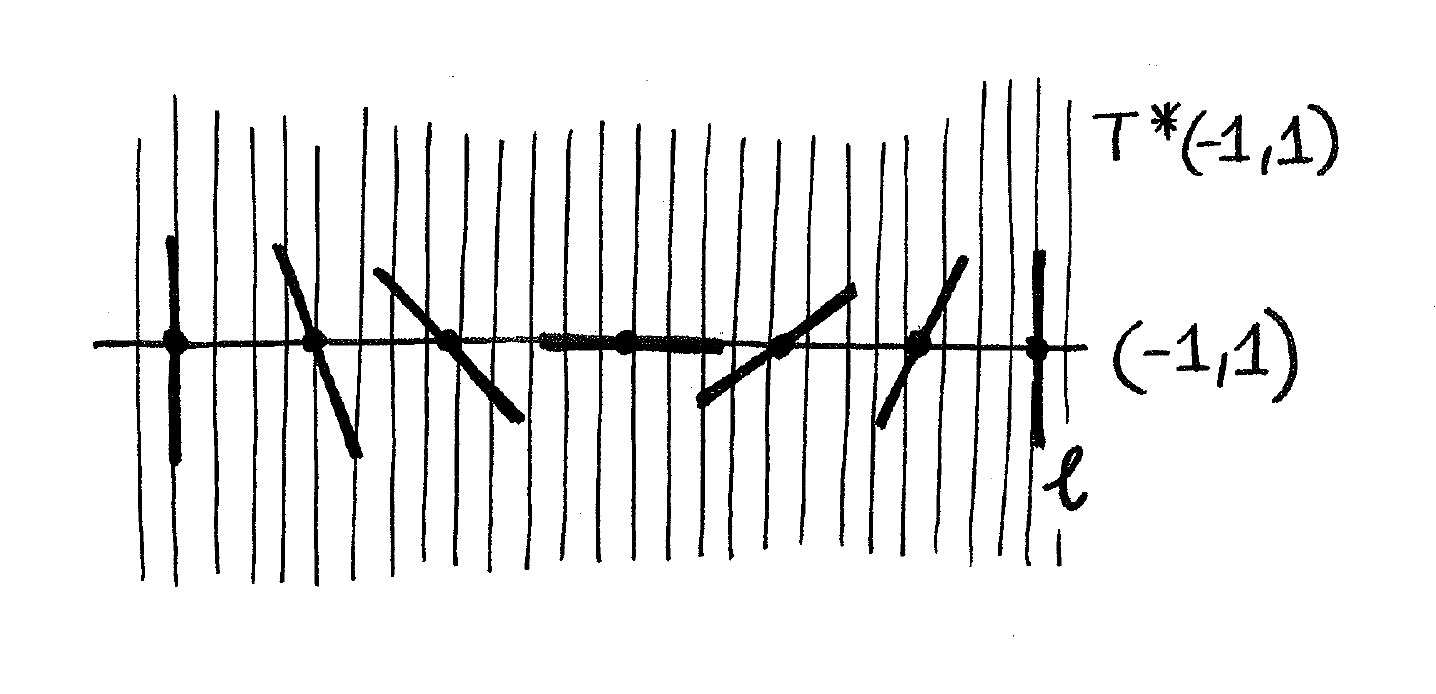}
\caption{The rotating line field $l \subset T\big(T^*(-1,1) \big)$.}
\label{rotatingline}
\end{figure}

Define $V_{\Sigma}:\cN \to \Lambda_n(T^*\cN)$ to be the distribution of Lagrangian planes defined over the zero section $\cN \hookrightarrow T^*\cN$  which corresponds to the product of the cotangent fibres of $T^*\Sigma$ and the line field $l$ under the isomorphism $T^*\cN \simeq T^*\Sigma \times T^*(-1,1)$. The distribution $V_{\Sigma}$ extends to the complement of $\cN$ in $S^n$ by letting it consist of the cotagent fibres of $T^*S^n$ on $S^n \setminus \cN$. The real vector bundle underlying $V_\Sigma$ is isomorphic to $T_\Sigma S^n$.

\begin{proof}[Proof of Proposition \ref{obstruction}]

We first consider the special case where $S^n$ itself has only fold singularities with respect to $\cF$. Then the caustic $\Sigma=\Sigma(S^n,\cF)$ is an embedded hypersurface in $S^n$ co-oriented by the Maslov co-orientation.  A direct consequence of the local model for the $\Sigma^{10}$ fold is that $V$ and $V_\Sigma$ are homotopic in the space of Lagrangian distributions. Since the real vector bundle underlying $V_\Sigma$ is isomorphic to $T_\Sigma S^n$, it remains to show that $T_\Sigma S^n$ is stably trivial. To see this, observe that $T_\Sigma S^n \oplus \varepsilon$ is obtained from $TX \oplus \varepsilon$ and $TY \oplus \varepsilon$ by using the gluing $\mu \oplus(1)=id \oplus (-1) \oplus (1)$ along $X \cap Y = \Sigma$, where we still think of $TS^n|_\Sigma$ as $T\Sigma \oplus \varepsilon $. Nothing changes if instead we use the gluing $\eta=id \oplus (1) \oplus (-1)$, since the two linear isomorphims of $\bR^2$ given by $(x,y) \mapsto (-x,y)$ and $(x,y) \mapsto (x,-y)$ are in the same connected component of $GL(2,\bR)$. We can therefore define a bundle map $T_\Sigma S^n \oplus \varepsilon \to TS^n \oplus \varepsilon$ by sending $TX \oplus \varepsilon \to TS^n \oplus \varepsilon$ via the inclusion $id \oplus (1)$, by sending $TY \oplus \varepsilon \to TS^n \oplus \varepsilon$ via the map $id \oplus (-1)$ and by gluing the two pieces into a global map $T_\Sigma S^n \oplus \varepsilon \to TS^n \oplus \varepsilon$ using $\eta$. This glued up map is an isomorphism, hence $T_\Sigma S^n \oplus \varepsilon \simeq TS^n \oplus \varepsilon \simeq \varepsilon^{n+1}$, as claimed.

Consider now the general case where $\varphi_t:T^*S^n \to T^*S^n$ is a Hamiltonian isotopy such that $\varphi_1(S^n)$ only has fold singularities with respect to $\cF$. Equivalently, $S^n$ only has fold singularities with respect to the pullback foliation $\varphi_1^*\cF$. From the special case already considered it follows that the restriction $V'$ to $S^n$ of the distribution $T(\varphi^*_1\cF)$ must be stably trivial as a real vector bundle over the sphere. But $V$ and $V'$ are homotopic as distributions of Lagrangian planes and therefore isomorphic as real vector bundles. Hence $V$ is also stably trivial.
\end{proof}

We now use our $h$-principle for the prescription of singularities to show that for $n$ even, the necessary condition for the simplification of singularities provided by Proposition \ref{obstruction} is also sufficient. 

\begin{corollary}\label{caustics of spheres} Assume that $n$ is even and that $V=T\cF|_S$ is stably trivial as a real vector bundle over the sphere. Then there exists a compactly supported Hamiltonian isotopy $\varphi_t: T^*S^n \to T^*S^n$ such that the singularities of tangency of $\varphi_1(S^n)$ with respect to $\cF$ consist only of folds. Moreover, we can take $\varphi_t$ to be $C^0$-close to the identity and supported on an arbitrarily small neighborhood of the zero section. 
\end{corollary}

\begin{remark} From the proof we can also extract a precise description of the permissible fold loci $\Sigma=\Sigma(\varphi_1(S^n),\cF)$ as hypersurfaces of $S^n$ in terms of the Euler number $e(V)$ of $V$. The locus $\Sigma$ can be arranged to consist of the boundary $\partial Y$ of any $n$-dimensional compact submanifold $Y \subset S^n$ of Euler characteristic $\chi(Y)=1 \pm \frac{1}{2}e(V)$, together with a disjoint union of nested double folds.
\end{remark}

\begin{proof}

If $B \subset S^n$ is a closed embedded $n-$ball, it is readily seen that $T_{\partial B} S^n$ is the trivial bundle. Fix a trivialization $V_{\partial B} \simeq S^n \times \bR^n$. We obtain a trivialization $T(T^*S^n)|_{S^n} \simeq S^n \times \bC^n$ by identifying both bundles with $V_{\partial B} \otimes \bC$. Suppose that $B$ is chosen so that $\cF$ is transverse to $S^n$ along $Op(B)$. Then with respect to this trivialization the distribution $V$ determines a class $\alpha \in \pi_n(\Lambda_n )$, where $\Lambda_n=U_n/O_n$ is the Grassmannian of linear Lagrangian subspaces of $\bC^n$ and we choose any $b \in \text{int}(B)$ as a basepoint. Let $\beta \in \pi_{n-1}(O_n)$ be the image of $\alpha$ under the map $\pi_n(\Lambda_n) \to \pi_{n-1}(O_n)$ given by long exact sequence in homotopy groups associated to the fibration $O_n \to U_n \to \Lambda_n$. Observe that $\beta$ is the clutching function corresponding to the real vector bundle underlying the distribution $V$. Note that the choice of ball $B$ induces a choice of orientation on $V$, which is encoded in the class $\beta$.

The stable triviality of $V$ means that $\beta$ is in the kernel of the map $\pi_{n-1} ( O_n ) \to \pi_{n-1} ( O )$, where $O=\lim_k O_k$ is the stable orthogonal group. However, $\pi_{n-1} ( O_k ) \to \pi_{n-1} ( O_{k+1} )$ is an isomorphism as soon as $k>n$, and therefore $\beta \in \ker\big( \pi_{n-1}( O_n) \to \pi_{n-1} ( O_{n+1}) \big) =  \text{im}\big( \pi_n(S^n) \to \pi_{n-1} ( O_n ) \big)$, where the map is given by the long exact sequence in homotopy groups associated to the fibration $O_n \to O_{n+1} \to S^n$. Recall that under this map the fundamental class $1 \in \bZ  \simeq \pi_n(S^n)$ is sent to the clutching function $\gamma \in \pi_{n-1}(O_n)$ corresponding to the tangent bundle $TS^n$. We can therefore write $\beta = k \gamma$ for some $k \in \bZ$.

Let $E \subset S^n$ by any compact hypersurface disjoint from $B$. Let $X$ and $Y$ be as in the construction of $T_ES^n$, so that $S^n=X \cup Y$ and $X \cap Y=E$. We choose the labels so that $B \subset X$, and then we agree to orient $T_ES^n$ so that the inclusion $TX \hookrightarrow T_E S^n$ is orientation preserving. It is straightforward to compute the Euler class $e(T_ES^n)=2-2\chi(Y)$ using for example the Poincar\'e-Hopf index theorem. Since $e(V)=2k$, if we choose the hypersurface $E$ so that $\chi(Y)=1-k$, then it follows that $T_E S^n$ and $V$ are isomorphic as oriented real vector bundles. 

Using the same construction as above, we can exhibit $T_ES^n$ as a distribution $V_E$ of Lagrangian planes in $T^*S^n$ defined along the zero section $S^n \hookrightarrow T^*S^n$. Observe that the singularities of tangency of the zero section $S^n$ with respect to the distribution $V_E$ consist of $\Sigma^{10}$ folds along $E$. 

Since $n$ is even, $\pi_n(U_n)=0$ and hence we have an injection $\pi_n(\Lambda_n) \hookrightarrow \pi_{n-1}(O_n)$. Observe that the homotopy classes in $\pi_n(\Lambda_n)$ determined by the distributions $V_E$ and $V$ have the same image $\beta $ under this map. It follows that $V_E$ and $V$ are homotopic in the space of Lagrangian distributions. The $h$-principle for the prescription of singularities Theorem \ref{Entov theorem redux} applies to produce a $C^0$-small Hamiltonian isotopy $\varphi_t:T^*S^n \to T^*S^n$ supported in a neighborhood of the zero section such that the singularities of tangency of $\varphi_1(S^n)$ with respect to $\cF$ are equivalent to those of $S^n$ with $V$  together with a union of nested double folds, which completes the proof. \end{proof}

In fact, the assumption that $V$ is stably trivial is automatically satisfied for all even $n$ such that $n \not\equiv 2$ mod $8$. One can argue in the following way. Choose a class $\beta \in \pi_{n-1}(O_n)$, which we think of as the clutching function of a real vector bundle. By exactness of the long exact sequence in homotopy groups associated to the fibration $O_n \to U_n \to \Lambda_n$, it is equivalent to ask that $\beta$ is in the image of the map $\pi_{n}( \Lambda_n ) \to \pi_{n-1} ( O_n )$ or to ask that it is in the kernel of the map $\pi_{n-1}(O_n) \to \pi_{n-1}(U_n)$. The first condition says that the vector bundle can be realized as a distribution of Lagrangian planes in $T^*S^n$ defined along the zero section $S^n \hookrightarrow T^*S^n$, while the second condition says that the complexification of the vector bundle is trivial. Suppose that $\beta$ is such a class and let $S(\beta) \in \pi_{n-1}( O_{n+1} )$ be the image of $\beta$ under the stabilization map $S$ induced by the inclusion $O_n \subset O_{n+1}$. By commutativity of the diagram below, observe that $S(\beta)$ lies in the kernel of the map $\pi_{n-1} (O_{n+1}) \to \pi_{n-1}( U_{n+1} )$.

\begin{center}
$ \begin{CD}
\pi_{n-1}\big(O_n \big) @>>>  \pi_{n-1} \big( U_n \big) \\
@VVV @VVV \\
\pi_{n-1}\big(O_{n+1} \big) @>>> \pi_{n-1} \big( U_{n+1} \big)
\end{CD} $
\end{center}

However, $\ker\big( \pi_{n-1}( O_{n+1} ) \to \pi_{n-1}( U_{n+1} ) \big) \simeq \ker\big( \pi_{n-1}(O) \to \pi_{n-1}(U)  \big)$, since both homotopy groups lie in the stable range. This kernel can be computed from Bott periodicity. Indeed, $\Omega(U/O) \simeq \bZ \times BO$ implies that $\pi_k(U/O) \simeq \pi_{k-2}(O)$ and therefore the groups appearing in the exact sequence $\pi_n(U/O) \to \pi_{n-1}(O) \to \pi_{n-1}(U)$ depend on the residue class of $n$ mod $8$ as follows.
\begin{center}
\begin{tabular}{ |c|c|c|c| } 
\hline
$n$ mod $8$ & $\pi_n(U/O)$ & $\pi_{n-1}(O)$ & $\pi_{n-1}(U)$ \\
\hline
$0$ & $0 $& $\bZ$ & $\bZ$  \\ 
$1$ & $\bZ$ & $\bZ/2$ &  $0$ \\ 
$2$ & $\bZ/2$ & $\bZ/2$ & $\bZ$ \\ 
$3$ & $\bZ/2$ & $0$ & $0 $\\ 
$4$ & $0$ & $\bZ$ & $\bZ$ \\ 
$5$ & $\bZ$ & $0$ & $0 $\\ 
$6$ & $0$ & $0$ & $\bZ$ \\ 
$7$ & $0$ & $0$  & $0 $\\ 
\hline
\end{tabular}
\end{center} 
\[ \]

From the table we deduce that $\ker\big( \pi_{n-1}(O) \to \pi_{n-1}(U) \big)=0$ except if $n \equiv 1$ or $ 2$ mod $8$ (in which case the kernel is isomorphic to $\bZ/2$). It follows that if $n$ is even and $n \not\equiv 2$ mod $8$, then we necessarily have $S(\beta)=0$, as claimed.
\begin{remark}
The simplest example of a caustic that cannot be simplified to consist only of folds occurs when $n=2$ and $V$ is the Hopf bundle on $S^2$. It is easy to check that in this case a $\Sigma^{110}$ pleat is unavoidable, in addition to the $\Sigma^{10}$ folds.
\end{remark}

When $n$ is odd, the same reasoning still shows that a necessary and sufficient condition for the simplification of singularities to be possible is that $V$ is homotopic to one of the standard models $V_\Sigma$ in the space of Lagrangian distributions. However, stable triviality of the underlying real vector bundle is not sufficient to guarantee that this condition is satisfied because $\pi_n(U_n) \neq 0$ and hence the map $\pi_n(\Lambda_n) \to \pi_{n-1}(O_n)$ need not be an injection.

We have only touched the surface of the homotopy theoretic calculations which are necessary to understand the formal condition obstructing the simplification of caustics. In the very concrete example of spheres considered above we were able to reason in a fairly hands-on manner. We believe that it should be possible to carry out a more systematic approach in the spirit of obstruction theory to study the general case.

\subsection{Application: families of $1$-dimensional Legendrians}\label{families of $1$-dimensional Legendrians}

We now turn to the second application discussed in Section \ref{overview}. Our goal is to establish that higher singularities are unnecessary for the homotopy theoretic study of the space of Legendrian knots in the standard contact Euclidean $\bR^3$. Recall that we think of $\bR^3$ as the jet space $J^1(\bR,\bR)=\bR q \times \bR p \times \bR z$ which comes equipped with the contact form $dz-pdq$. The Lagrangian projection is the map $\bR^3 \to \bR^2$, $(q,p,z) \mapsto (q,p)$ which corresponds to the forgetful map $J^1(\bR,\bR) \to T^*\bR$. The front projection is the map $\bR^3 \to \bR^2$, $(q,p,z) \mapsto (q,z)$ which corresponds to the forgetful map $J^1(\bR,\bR) \to J^0(\bR,\bR)$. The Reeb direction is $\partial / \partial z$ and it will also be useful to think of the projection along the Reeb direction $\bR^2 \to \bR$ which is the map $(q,z) \mapsto q$. 

The fibres of the front projection form a Legendrian foliation $\cF$ of $\bR^3$. Recall that a Legendrian knot $f:S^1 \to \bR^3$ is said to have mild singularities when the only singularities tangency of $f$ with respect to $\cF$ are folds and embryos. Folds are the generic $\Sigma^{10}$ singularities of a single Legendrian knot and in the front projection correspond to cusps, see Figure \ref{legendriancusp}. Embryos are the generic $\Sigma^{110}$ singularities of a $1$-parametric family of Legendrian knots and in the front projection correspond to Type I Reidemeister moves, namely the instances of birth/death of two cusps. See Figure \ref{reidemesitermove}.

\begin{figure}[h]
\includegraphics[scale=0.6]{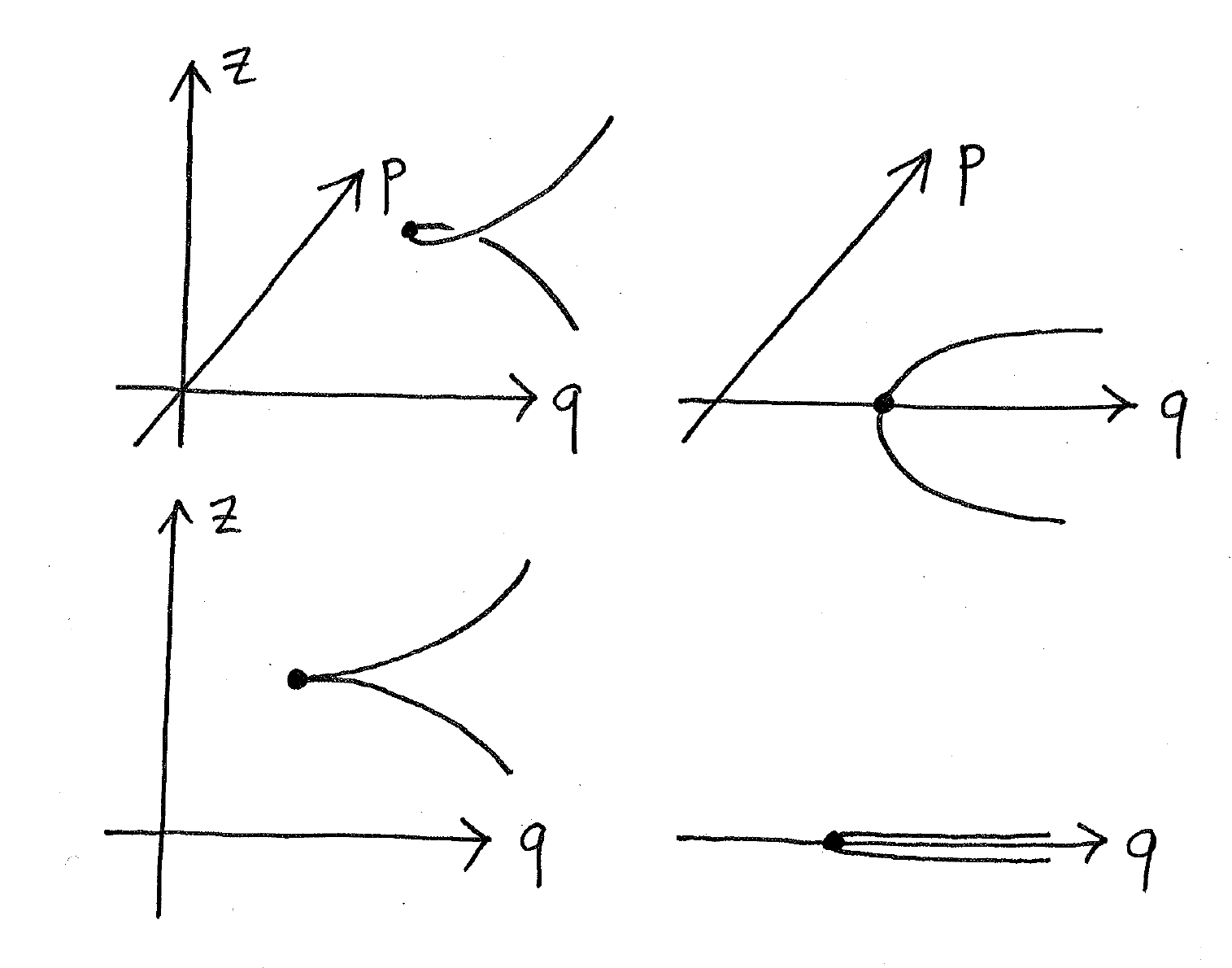}
\caption{The standard fold as seen from the Lagrangian and front projections (top right and bottom left respectively). If we project all the way down to $\bR=\bR(q)$ (bottom right), the germ of the resulting map is equivalent to that of $x \mapsto x^2$.}
\label{legendriancusp}
\end{figure}

\begin{figure}[h]
\includegraphics[scale=0.6]{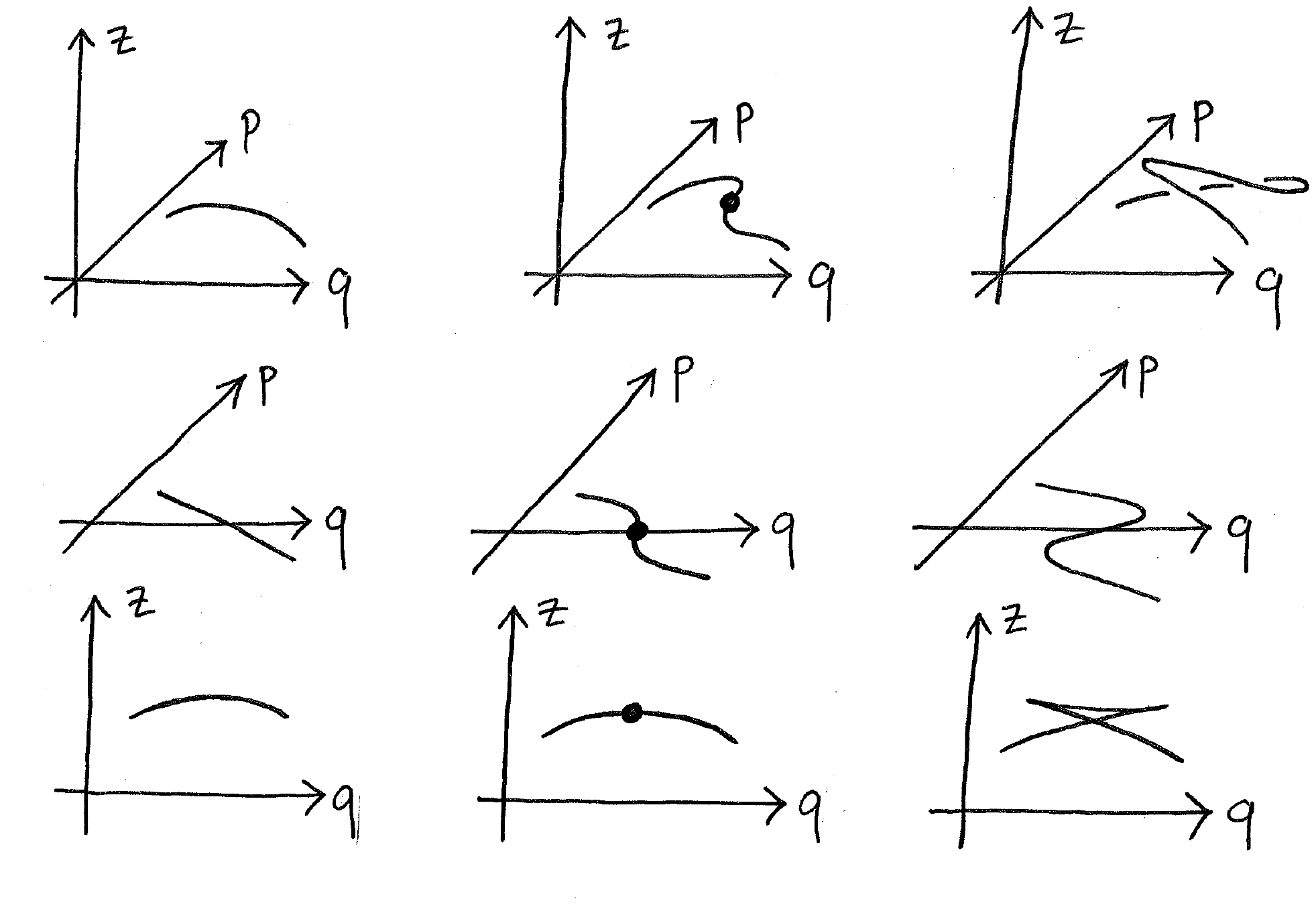}
\caption{The embryo singularity is illustrated in the middle column. We can picture it in the ambient contact $\bR^3$ (top), in the Lagrangian projection (middle) and in the front projection (bottom). An embryo is a generically isolated singularity of a $1-$parametric family of Legendrian knots, which we exhibit from left to right. The bottom row (which takes place in the front projection) gives us the familiar Reidemeister I move for Legendrian fronts.} 
\label{reidemesitermove}
\end{figure}

Generically, a Legendrian knot only has folds and a $1$-parametric family of Legendrian knots only has folds and embryos. However, the caustic of a family of Legendrian knots parametrized by a space of high dimension will generically be very complicated. It is therefore not a priori clear how the topology of the space of Legendrian knots $\cL$ is related to that of the subspace $\cM \subset \cL$ consisting of those Legendrian knots whose singularities are mild. In Section \ref{overview} we defined a space of decorations $\widetilde{\text{C}}(S^1)$ and a space $\cD$ of pairs $(f,D)$ consisting of a Legendrian knot with mild singularities $f \in \cM$ together with a decoration $D \in \widetilde{\text{C}}(S^1)$ of the singularities of $f$. See Figure \ref{figureeightknot} for an example of a decoration $D$ compatible with the standard front projection of the figure eight knot.

\begin{figure}[h]
\includegraphics[scale=0.7]{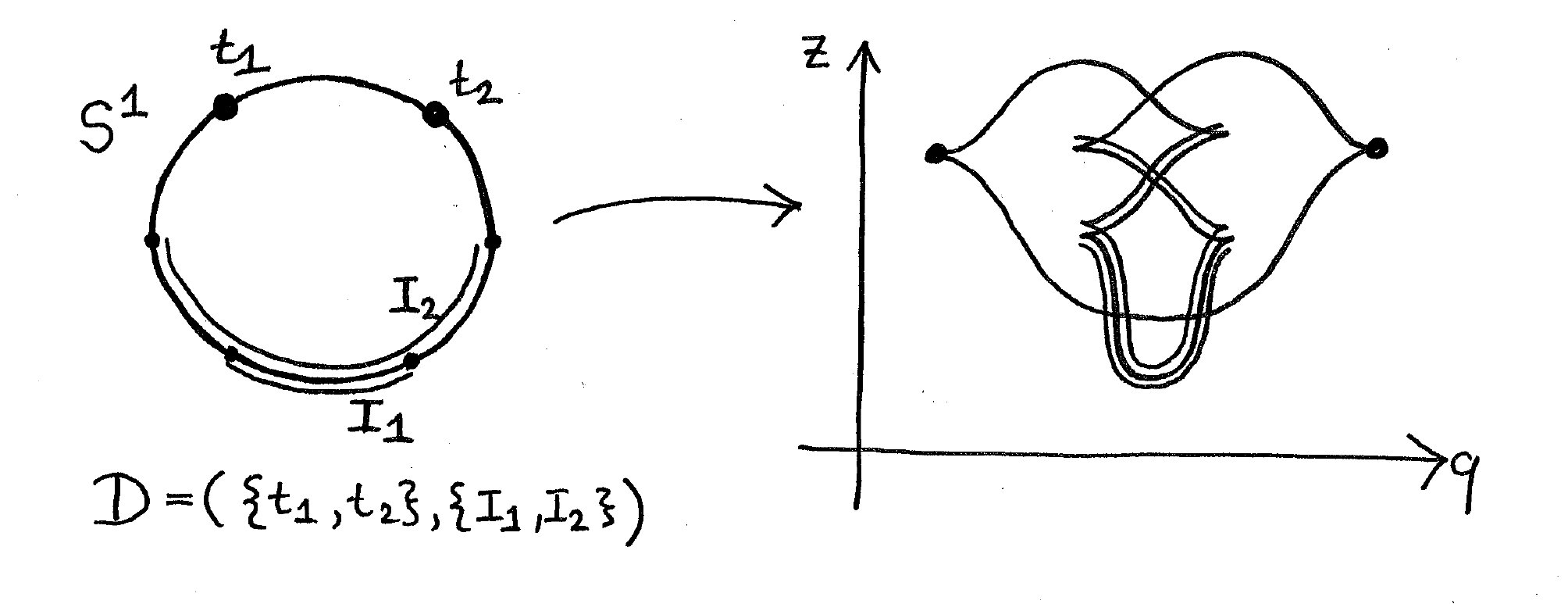}
\caption{An example of a decoration which consists of two points $t_1,t_2$ and two nested intervals $I_1 \subset I_2$.}
\label{figureeightknot}
\end{figure}

By composing the forgetful map $\cD \to \cM$ given by $(f,D) \mapsto f$ with the inclusion $\cM \hookrightarrow \cL$ we obtain a map $\cD \to \cL$. In this section we will prove the following result, which is a consequence of our parametric $h$-principle for the simplification of caustics.

\begin{corollary}\label{higher Reidemeister}
The map $\cD \to \cL$ is a weak homotopy equivalence on each connected component.
\end{corollary}

\begin{remark}
The decoration $D$ is necessary because the inclusion $\cM \hookrightarrow \cL$ is not a homotopy equivalence, indeed $\pi_2(\cL,\cM) \neq 0$. To see this, let $f^z$ be a family of Legendrian knots parametrized the closed unit $2$-disk $D^2$ which has mild singularities everywhere except for a single $\Sigma^{1110}$ singularity appearing in the interior. Then it is easy to see that the family $\{f^z\}_{z \in \partial D^2}$ represents a nontrivial element of $\pi_2(\cL,\cM)$. The decoration  $D$ is designed to kill this homotopy group.
\end{remark}

\begin{remark}
If $f \in \cL$ is any Legendrian knot, then by a generic perturbation we may assume that the singularities $\Sigma \subset S^1$ of $f$ consist only of a finite number of folds. Then $f$ is compatible with the trivial decoration $D=(\{t_i\},\{I_j\})$ consisting of $\{t_i\}=\Sigma$ and $\{I_j\}=\varnothing$. It follows that $\pi_0(\cD) \to \pi_0(\cL)$ is surjective. However, is it easy to see that $\pi_0(\cD) \to \pi_1(\cL)$ is not injective, since in the space $\cD$ we are keeping track of the decoration $D$.
\end{remark}

To prove Corollary \ref{higher Reidemeister} it suffices to show that $\pi_n(\cL,\cD)=0$ for $n>1$ and that $\pi_1(\cD) \to \pi_1(\cL)$ is surjective.  We deal with each of the statements separately. 

\begin{proof}[Proof that $\pi_n(\cL,\cD)=0$ for $n>1$]

Let $\alpha \in \pi_n(\cL,\cD)$ be any class. We can represent $\alpha$ by a map $F:D^n \to \cL$ such that $F|_{\partial D^n}$ lifts to a map $\widetilde{F}: \partial D^n \to \cD$. To conclude that $\alpha=0$ we must show that there exist a homotopy $F_t: D^n \to \cL$ which is fixed on $Op(\partial D^n)$ and such that $ \widetilde{F}: \partial D^n \to \cD$ extends to a lift $\widetilde{F}_1: D^n \to \cD$ of $F_1$. 

We begin by examining the singularity locus of $F$ on the boundary, which is the subset $\Sigma( F|_{\partial D^n}) \subset \partial D^n \times S^1$ consisting of all pairs $(z,s) \in S^{n-1} \times S^1$ such that the front of the Legendrian knot $F(z): S^1 \to \bR^3$ has a fold or embryo singularity at the point $s \in S^1$. Denote the map  $(f,D) \mapsto D$ which forgets the knot but remembers the decoration by $\text{dec}: \cD \to \widetilde{\text{C}}(S^1)$. The family of decorations $\text{dec} \circ \widetilde{F}: \partial D^n \to \widetilde{\text{C}}(S^1) $ induces a decomposition of the singularity locus $\Sigma( F|_{\partial D^n}) = \cC \cup \cW$, where $\cC$ consists of folds and $\cW$ consists of pairs of folds with opposite Maslov co-orientations together with the embryos that give rise to the birth/death of such pairs. The folds of $\cC$ correspond to the points $t_1, \ldots , t_k$ and the pairs of folds or embryos of $\cW$ correspond to the endpoints of the intervals $I_1, \ldots , I_m$. Note that the number $m$ of intervals may vary with the parameter $z$ but the number $k$ of points is fixed since $n>1$. After a generic perturbation we may assume that $\cC$ and $\cW$ are smooth codimension $1$ submanifolds of $S^{n-1} \times S^1$ and moreover that the set of embryos $\cE$ is a smooth codimension $1$ submanifold of $\cW$.

Our strategy is the following. The first step is to construct $F_t$ near the boundary of the parameter space $\partial D^n$. This involves manually killing all the pairs of folds in $\cW$. The next step is to extend the folds in $\cC$ to the interior of the parameter space $\text{int}(D^n)$. After these two preparatory steps we can apply the relative form of our parametric $h$-principle to construct $F_t$ everywhere else so that the only additional singularities of the deformed family $F_1$ are the folds and embryos resulting from the wrinkling process. By construction the resulting map $F_1:D^n \to \cM$ will have an obvious lift to $\cD$, completing the proof.

We now perform the first of these preparatory steps. The key idea, which appears repeatedly throughout the literature of the wrinkling philosophy, is that to kill a zig-zag one may create a very small new zig-zag near one end of the old zig-zag and then slowly let the new zig-zag take over, eventually killing the old zig-zag and replacing it. The newly created zig-zag does not bother us because it will end up completely contained in the interior of the parameter space $D^n$. 

Fix a collar neighborhood $A \subset D^n$ of $S^{n-1}$, which we parametrize radially as $A=[0,1) \times S^{n-1}$ with $0 \times S^{n-1}$ corresponding to $\partial D^n$. It will be convenient to assume that $F$ is radially invariant on $A$, and indeed by means of an initial homotopy of $F$ fixed on the boundary we can arrange it so that $F( \lambda, z)=F(0, z)$ for all $\lambda \in [0,1)$ and all $z \in S^{n-1}$. Note then that  $F(A) \subset \cM$ and moreover $\Sigma(F|_A) =  [0,1) \times \Sigma(F_{\partial D^n}) $. For an $F$ satisfying this condition we establish the following preparatory result.

\begin{lemma}[Preliminary arrangement near the boundary]\label{preparatory} There exists a homotopy $F_t:D^n \to \cL$ of $F=F_0$ such that the following properties hold.

\begin{itemize}
\item $F_t$ is fixed on $Op\big(\partial D^n \cup (D^n \setminus A)\big)$.
\item $F_t(A) \subset \cM$.
\item The folds in $\cC$ are left untouched throughout the homotopy. To be more precise, the subset $ [0,1) \times \cC \subset \Sigma ( F_t |_A )$ does not vary with time.
\item The pairs of folds in $\cW$ are killed at the end of the homotopy. To be more precise, over each cylinder $[0,1) \times  z   \times S^1 \subset A\times S^1$ the singularity locus $\Sigma(F_1|_A)$ contains arcs $a_1 , \ldots , a_m$ whose interiors lie in $(0, 1) \times z \times S^1$ and whose endpoints lie in $0\times z \times S^1$ and in fact consist precisely of the endpoints of the intervals $I_1 , \ldots ,  I_m$. Moreover, each arc $a_j$ consists everywhere of folds except at a single point in its interior, which is an embryo. 
\end{itemize}
\end{lemma} 

\begin{figure}[h]
\includegraphics[scale=0.6]{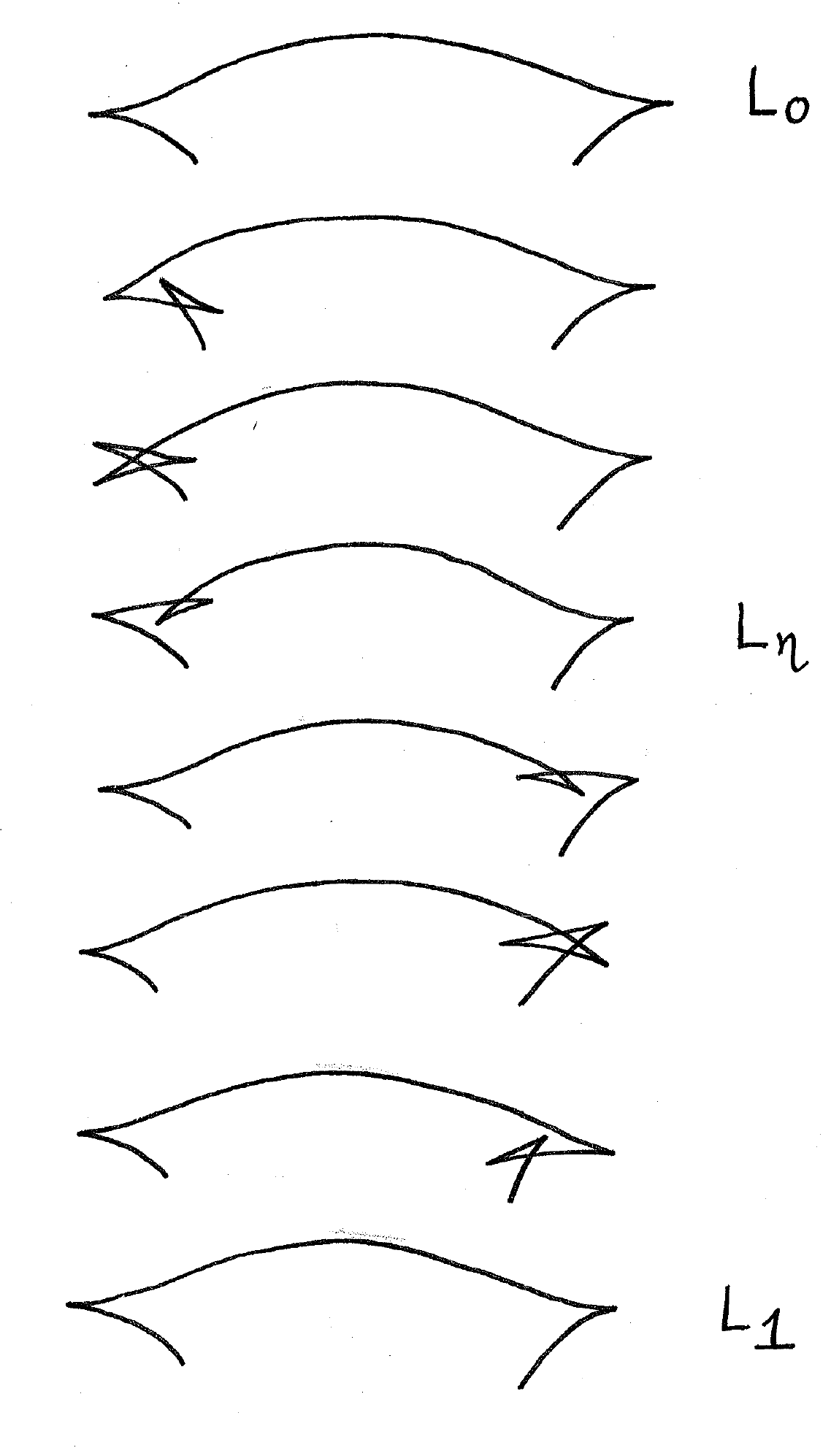}
\caption{The family $L_\eta$. The parameter $\eta$ runs from $0$ to $1$.}
\label{firstfamilyL}
\end{figure}

\begin{proof} To construct the homotopy $F_t$ we will use the $1-$parameter family of Legendrian fronts $L_{\eta}$ exhibited in Figure \ref{firstfamilyL}. Suppose that $I_j$ is a non-degenerate interval appearing in the decoration $D=\text{dec}\big(\widetilde{F}(z) \big)$ for some $z \in \partial D^n$. Assume moreover that $I_j$ is isolated, meaning that there are no other intervals $I_k$ contained inside $I_j$ or containing $I_j$. In a neighborhood of $I_j \subset S^1$ the front of the knot $F(z)$ is equivalent to either the local model $L_{0}$ or to a flip of $L_0$ in the vertical direction, depending on the Maslov co-orientations. By replacing $L_\eta$ by the vertical flip of $L_\eta$ whenever this is needed, we may assume without loss of generality that the former case holds.

Note that the family of fronts $L_\eta$ can be made to be $C^1$-close to the constant family $L_0$, so that the resulting Legendrian isotopy is $C^0$-small. We can therefore think of the $1-$parameter family $L_{ \eta}$ as a Legendrian isotopy of $F(z)$ supported on $Op(I_j)$.

 It is conceptually useful to understand the projection of the family $L_{\eta}$ along the Reeb direction. The front $L_{0}$ projects down to a zig-zag. As the parameter $\eta$ increases from $0$ to $1$, a new zig-zag is created just outside of $I_j$. We then make this new zig-zag bigger and bigger, until it takes over and replaces the old zig-zag, which has died by the time that $\eta$ is close to $1$. This process is illustrated in Figure \ref{zigzagsintarget}

\begin{figure}[h]
\includegraphics[scale=0.5]{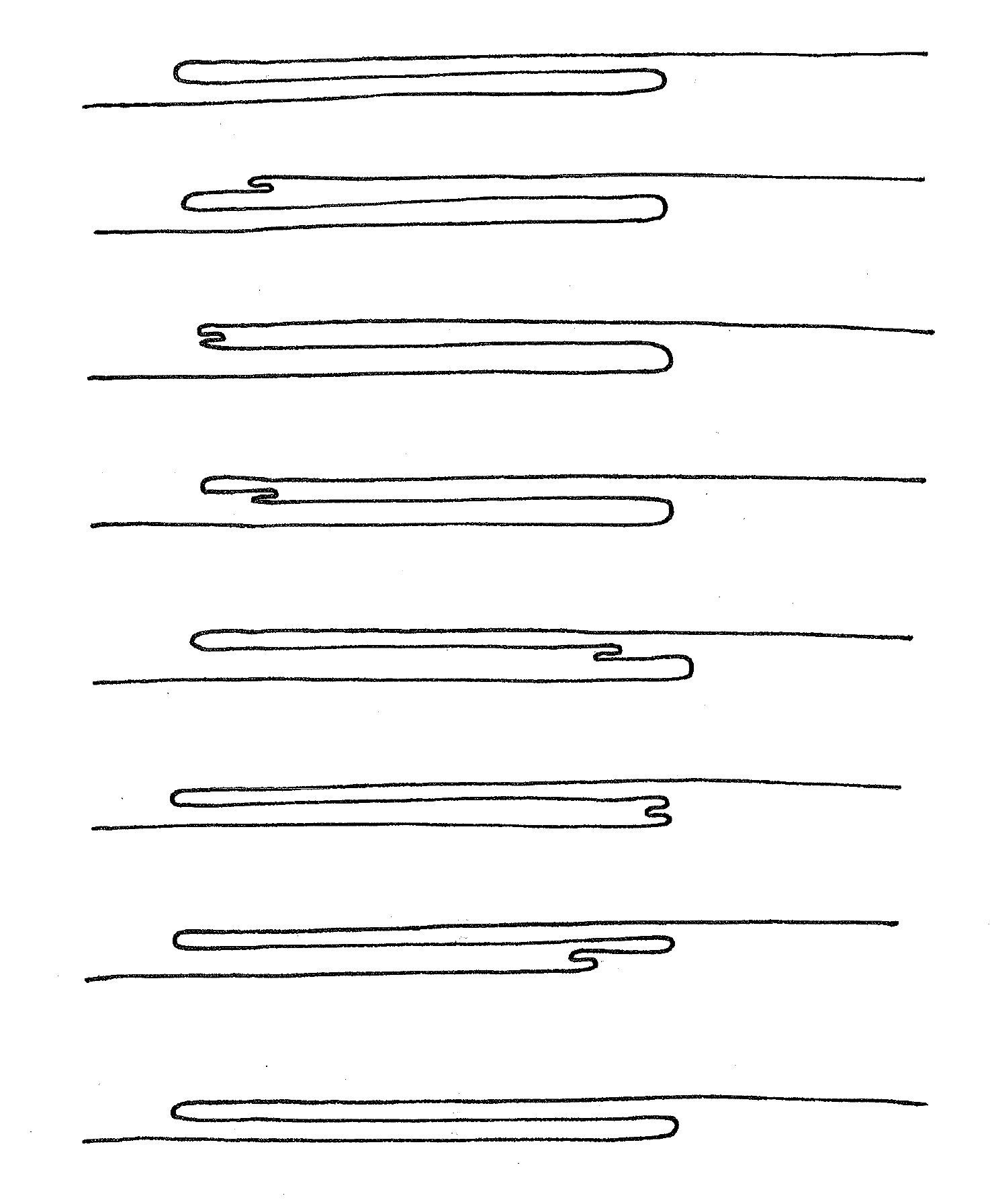}
\caption{The projection of Figure \ref{firstfamilyL} along the Reeb direction.}
\label{zigzagsintarget}
\end{figure}

To define $F_t$ formally, let $\varphi : [0,1] \to [0,1]$ be a function such that the following properties hold.
\begin{itemize}
\item $\varphi=0$ on $Op( \partial [0,1])$.
\item $\varphi=1$ on $Op( \frac{1}{2} )$.
\item $ \varphi$ is non-decreasing on $[0,\frac{1}{2}]$ and non-increasing on $[\frac{1}{2},1]$.
\end{itemize}

We define the homotopy $F_t$ on $[0,1)\times z  \times Op(I_j)$ by the formula $F_t( \lambda , z, s)= L_{t \varphi(\lambda)}(s)$. Suppose next that there are two nested intervals $I_k \subset I_j$ with no other interval either contained or containing $I_k$ or $I_j$. Then we define the homotopy $F_t$ just like we did before, but using a nested version of the family $L_{\eta}$ which we exhibit in Figure \ref{nestedfamilyL}. For more complicated configurations of intervals $I_j$ we repeat this strategy but using the obvious model which is obtained by nesting the $1-$parameter family $L_\eta$ (or its flip in the vertical direction) according to the nesting of the configuration of intervals. 

\begin{figure}[h]
\includegraphics[scale=0.6]{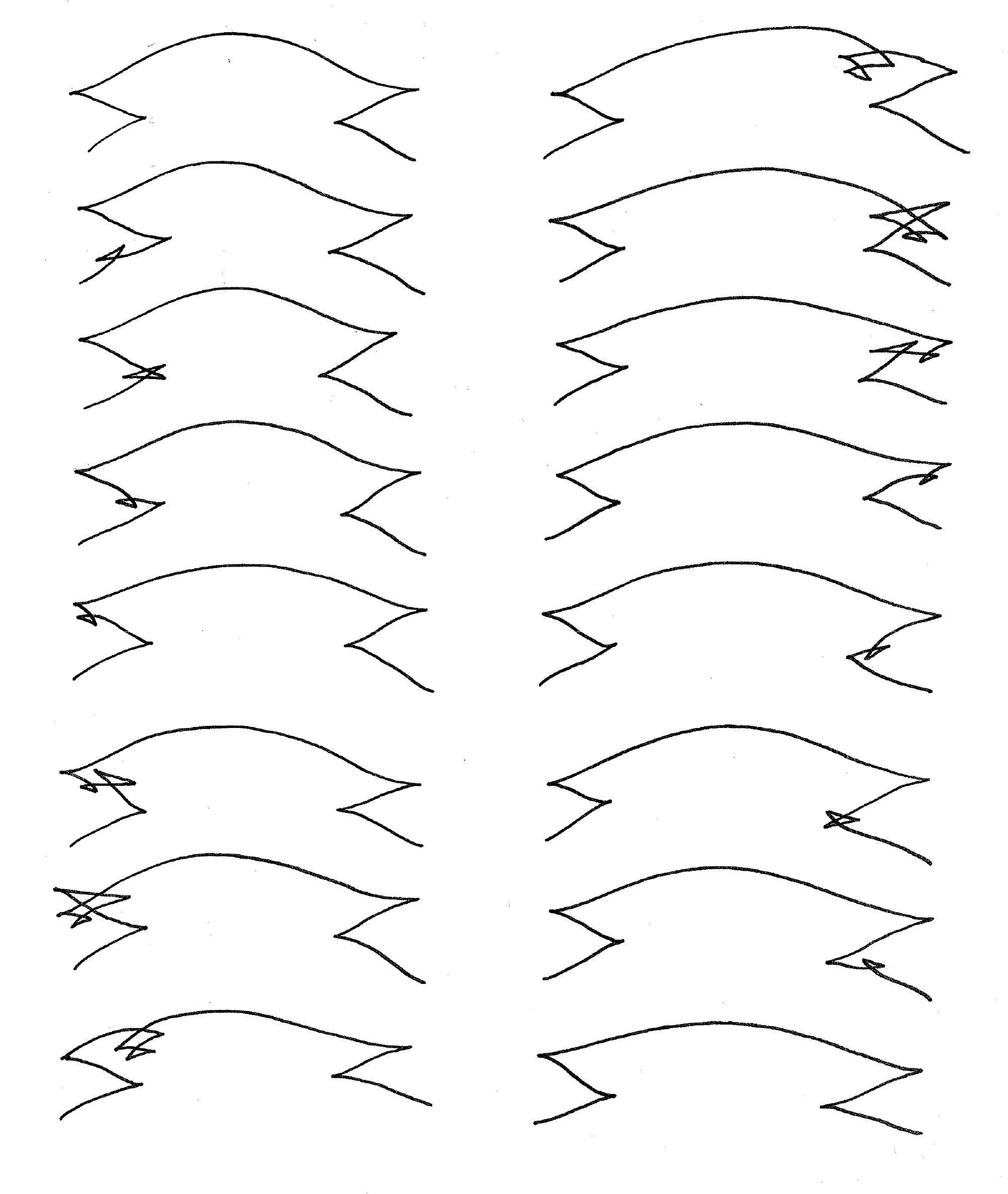}
\caption{A nesting of Figure \ref{firstfamilyL} for two intervals $I_k \subset I_j$.}
\label{nestedfamilyL}
\end{figure}

\begin{figure}[h]
\includegraphics[scale=0.5]{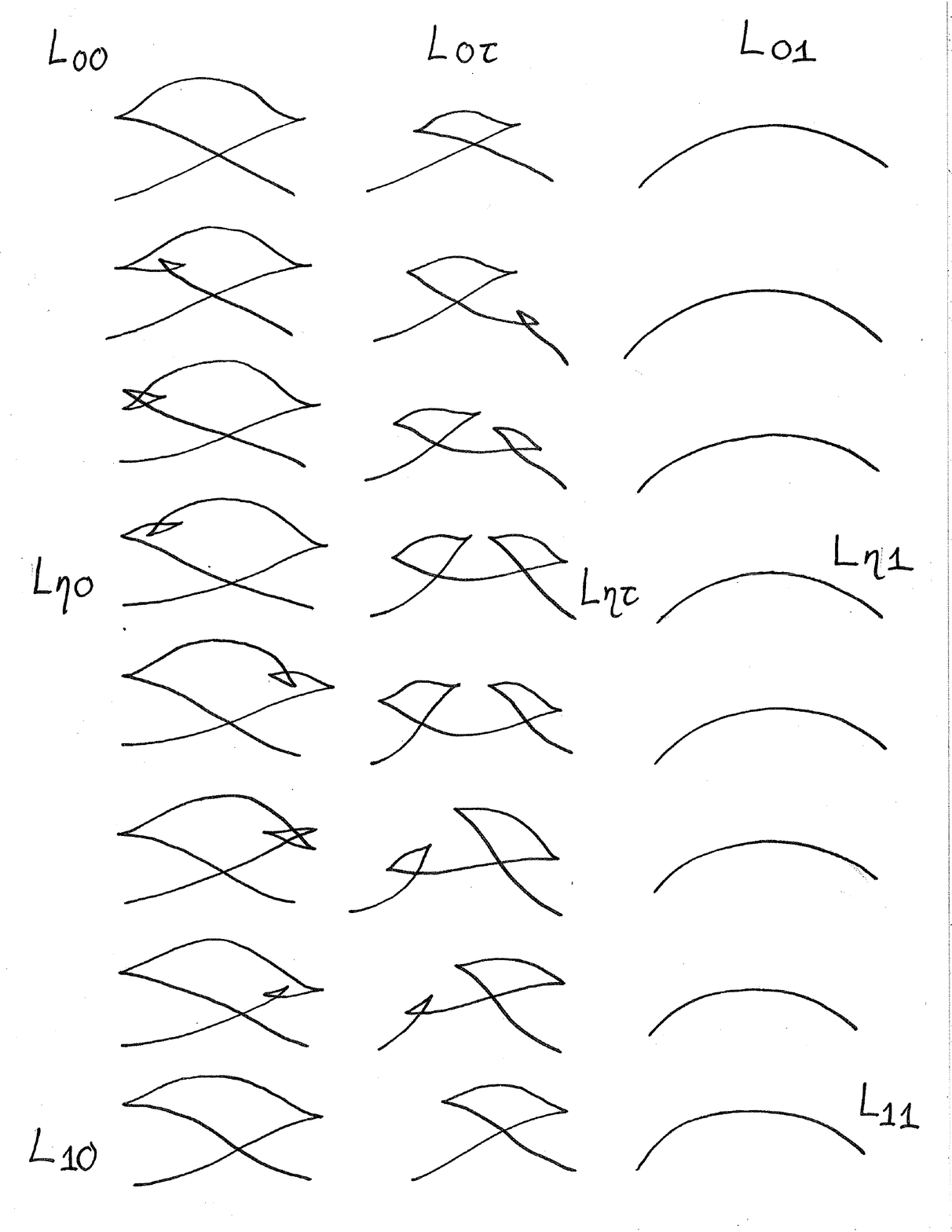}
\caption{The family $L_{\eta, \tau}$. The parameters $\eta$ and $\tau$ both run from $0$ to $1$.}
\label{dyingoutfamilyL}
\end{figure}

The construction described above can be realized parametrically as $z \in S^{n-1}$ varies, as long as no interval $I_j$ degenerates to a point. However, in a neighborhood of the locus $\cE \subset \cW$ of embryos we need a different local model so that the family $L_\eta$ does not degenerate into a higher singularity. The $2$-parametric family $L_{\eta, \tau}$ exhibited in Figure \ref{dyingoutfamilyL} gets the job done. Let us first understand what the locus $\cW$ looks like in a neighborhood of $\cE$. Fix a connected component $\cW_0 \subset \cW$ and set $\cE_0=\cE \cap \cW_0$. Consider the image $\widehat{\cW}_0$ of $\cW_0$ under the projection $S^{n-1} \times S^1 \to S^{n-1}$. Note that $\widehat{\cW}_0 \subset S^{n-1}$ is a smooth codimension $0$ submanifold with boundary, that $\widehat{\cE}_0=\partial \widehat{\cW}_0$  is the image of $\cE_0$, that the map $\cW_0 \to \widehat{\cW}_0$ is a $2$ to $1$ cover away from $\cE_0$ and that along $\cE_0$ the map $\cW_0 \to \widehat{\cW}_0$ has folds. In particular, the restriction $\cE_0 \to \widehat{\cE}_0$ is an embedding, see Figure \ref{singularitylocusinthesource}.

\begin{figure}[h]
\includegraphics[scale=0.6]{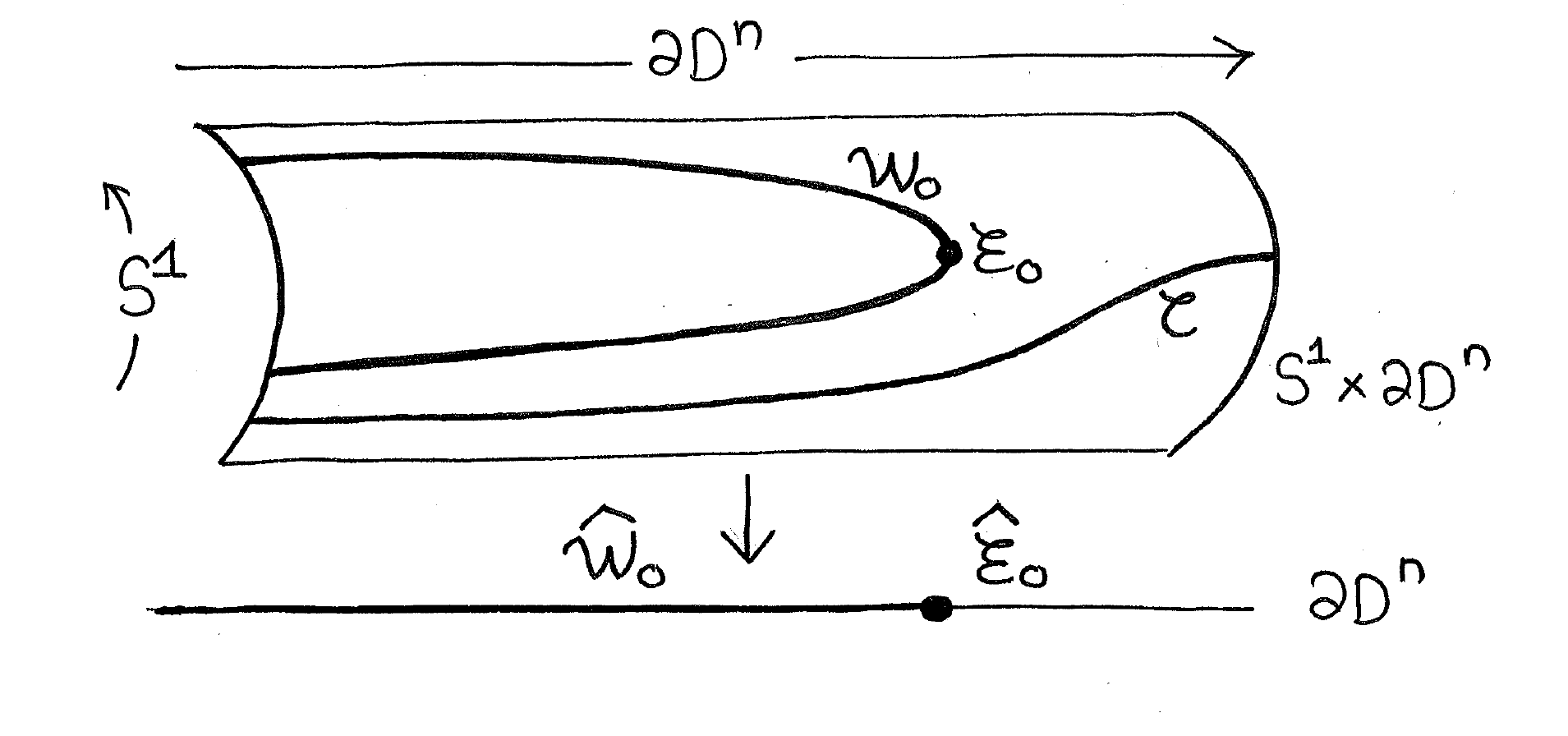}
\caption{The local geometry of the projection $\cW \to \partial D^n$.}
\label{singularitylocusinthesource}
\end{figure}

Let $\widehat{\cE}_0 \times (0,1)$ be a collar neighborhood of $\widehat{\cE}_0=\widehat{\cE}_0 \times \frac{1}{2}$ in $S^{n-1}$ such that $\widehat{\cE}_0 \times (0,\frac{1}{2}] \subset \widehat{\cW}_0$. Given $e \in \cE_0$, let $\widehat{e}$ be its image in $\widehat{\cE}_0$ and let $z_t \in S^{n-1}$, $t \in (0,1)$ correspond to the arc $\widehat{e}  \times (0,1) \subset \widehat{\cE}_0 \times (0,1)$. Then the $1$-parametric family $F(z_t)$ is equivalent in a neighborhood of the embryo point $e$ to the $1$-parametric family $L_{0, \tau}$ exhibited at the top row of Figure \ref{dyingoutfamilyL} (or to its flip in the vertical direction). 

Observe that the family $L_{\eta, \tau}$ can be taken to be $C^1$-close to the family $L_{0, \tau}$ which is constant in $\eta$. We can therefore think of the $2$-parametric family $L_{\eta, \tau}$ as a $C^0$-small Legendrian isotopy of the $1$-parametric family $F(z_t)$ supported in a neighborhood of the embryo point. Notice that $L_{\lambda, 0}=L_\lambda$, so the isotopy is compatible with our previous isotopy. Notice also that $L_{\lambda, 1}$ is constant. We can therefore define the homotopy $F_t$ by the formula $F_t(\lambda, z, s) = L_{t \varphi(\lambda) ,\tau }(s)$, where $z=(\widehat{e},\tau) \in \widehat{\cE}_0 \times (0,1)$. The construction can be realized parametrically in $z$, see Figure \ref{zigzagsinthesource} for an illustration. The construction can also be realized with any configuration of intervals, by nesting the families shown in Figures \ref{nestedfamilyL} and \ref{dyingoutfamilyL} according to the nesting of the intervals. This completes the proof of Lemma \ref{preparatory}.
\end{proof}

\begin{figure}[h]
\includegraphics[scale=0.6]{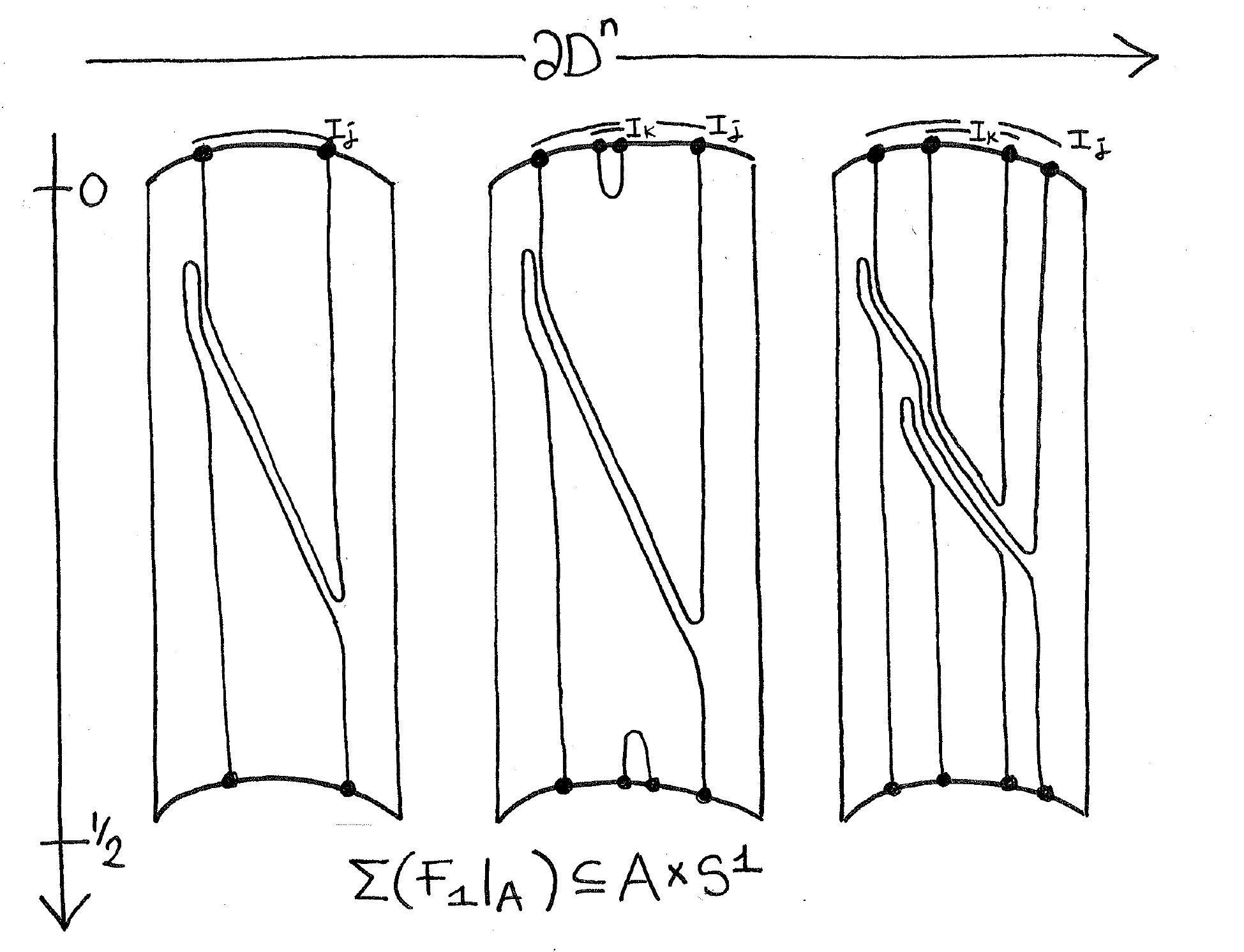}
\caption{The singularity locus of $F_1$ on $\partial D^n \times [0,\frac{1}{2}) \times S^1 \subset A \times S^1$, which is one-half of the full locus $\Sigma(F_1|_A)$.}
\label{zigzagsinthesource}
\end{figure}

The next step is to extend the cusp locus $\cC$ to the interior of the parameter space $D^n$. This is achieved by a second preparatory lemma. For notational convenience, we now forget about our old family and use the letter $F$ to denote the new family $F_1$ produced by Lemma \ref{preparatory}. In particular, all of the properties listed in the conclusion of Lemma \ref{preparatory} are satisfied by $F$.

\begin{lemma}[Preliminary arrangement in the interior]\label{prep2} There exists a homotopy $F_t: D^n \to \cL$ of $F=F_0$ such that the following properties hold.

\begin{itemize}
\item $F_t$ is fixed on $A$.
\item The singularity locus $\Sigma(F_1) \subset D^n \times S^1$ contains a properly embedded submanifold with boundary $\cI$ of codimension $1$ in $D^n \times S^1$ which consists entirely of folds and such that $ \cI \cap( A \times S^1) = \cC \times [0,1)$.
\end{itemize}
\end{lemma} 

\begin{remark}
Since $F_t$ is fixed on $A$, $\Sigma(F_1)$ also contains the properly embedded codimension $1$ submanifold with boundary $\cK$ formed by the arcs $a_j$ which kill $\cW$. In addition to $\cI$ and $\cK$, the singularity locus $\Sigma(F_1)$ may have other components, but we will not care about them because they are all homotopically trivial and contained in $\text{int}(D^n) \times S^1$.
\end{remark}

\begin{proof}
We assume that $\cC \neq \varnothing$, otherwise the Lemma is trivial. Recall that the space of decorations $\widetilde{\text{C} }(S^1)$ is fibered over the (unordered) configuration space of points on the circle $\text{C}(S^1)= \bigsqcup_k \text{C}_k(S^1)$. The map is $\big( \{ t_j\} , \{ I_i\} \big) \mapsto \{t_j \}$ and its fibers are contractible. Denote by $\text{conf}:\cD \to \text{C}(S^1)$ the composition of $\text{dec} : \cD \to \widetilde{\text{C}}(S^1)$ with the fibration $\widetilde{\text{C}}(S^1) \to \text{C}(S^1)$. We claim that the map $\text{conf} \circ \widetilde{F} : \partial D^n \to \text{C}(S^1)$ extends to a map $c: D^n \to \text{C}(S^1)$.

First observe that each component $\text{C}_k(S^1)$ of $\text{C}(S^1)$ is homotopy equivalent to $S^1$. Hence for $n>2$ there is nothing to prove because $\pi_{n-1}(S^1)=0$. If $n=2$, then we need to justify the claim. Write $H^*(\partial D^2 \times S^1 ; \bR) =\bR[x,y]/(x^2,y^2)$, where $x$ is Poincar\'e dual to  $\partial D^2 \times  pt $ and $y$ is Poincar\'e dual to $ pt \times S^1$. Consider the Gauss map $G(dF):  D^2 \times S^1 \to S^1$, $(z,s) \mapsto G\big(dF(z)\big)(s)$, where $\Lambda_1(\bR^3)=\bR^3 \times S^1$ and we project away the $\bR^3$ factor. Explicitly, an angle $\theta$ corresponds to the line field spanned by $\cos( \theta) \partial / \partial p + \sin(\theta) ( \partial / \partial z + p \partial / \partial q) $. Observe that $ \big(\partial D^2 \times S^1  \big) \cap G(dF)^{-1}\big( \text{span}( \partial / \partial p) \big) = \cC \cup \cW$. Observe also that the fundamental class of $\cC$ is Poincar\'e dual to $k x + ly$ for some $l \in \bZ$, where we recall that $k$ is the number of points $t_1, \ldots , t_k$ in the decorations $\text{dec} \circ F(z)$.  If we write $i : \partial D^2 \times S^1 \hookrightarrow D^2 \times S^1$ for the inclusion and denote by $u \in H^1(S^1; \bR)$ the class which is Poincar\'e dual (PD)  to a point, then we have
\[ kx + ly = \text{PD} [ \cC ] = \text{PD}[ \cC \cup \cW]  = \text{PD}[ \big( G(dF) \circ i \big)^{-1} \big( \text{span}( \partial / \partial p )\big)] = ( G(dF) \circ i )^*u = i^* \big( G(dF)^*u. \]
However, $i^* : H^*(D^2 \times S^1 ; \bR) \to H^*( \partial D^2 \times S^1 ; \bR)$ has image generated by $x$. It follows that $l=0$ and hence that $\cC$ is an embedded curve in $\partial D^2 \times S^1$ which is homologous to $k[ \partial D^2 \times  pt ]$. Note then that $\cC$ has necessarily $k$ components, each of which is homologous to $[ \partial D^2 \times  pt ]$. It is now a triviality to check that $\text{conf} \circ \widetilde{F}: \partial D^2 \to \text{C}(S^1)$ extends to a map $c:D^2 \to \text{C}(S^1)$, as claimed.

Choose then such an extension $c$ and assume without loss of generality that $c$ is radially constant in the annulus $A \subset D^n$. Choose also a tangential rotation $G_t: D^n \times S^1 \to \Lambda_1(\bR^3)$ of the family $F$ such that the following properties hold.
\begin{itemize}
\item $G_t$ is fixed on $A$.
\item $G_1 = \partial / \partial p $ on the subset $\cI= \{ (z,t) : \, \, t \in c(z) \} \subset D^n \times S^1$.
\end{itemize}

Using the parametric version of theorem Theorem \ref{1-holonomic} (which in the $1$-dimensional case is the same as Theorem \ref{parametric perp-holonomic} since all rotations are simple) we obtain a homotopy $F_t$ of the family $F$ which is fixed on $A$ and such that $G(dF_t)$ is $C^0$-close to $G_t$ on $Op(\cI)$. The family $F_1$ does not quite have folds along $\cI$, but $G(dF_1)$ is almost parallel to $\partial / \partial p$ on $Op(\cI)$ and $F_1|_A$ does have folds along $\cI \cap A$. By implanting the local model for the creation of a pair of folds exhibited in Figure \ref{creationoffolds} into $F_1$ we can arrange it so that the new family does have folds precisely along $\cI$. Moreover, we can arrange it so that the new family agrees with the old family inside $A$. Away from $Op(A \cup \cI)$ the singularities of $F$ might be a mess but we don't care. The proof of Lemma \ref{prep2} is complete.\end{proof}

\begin{figure}[h]
\includegraphics[scale=0.6]{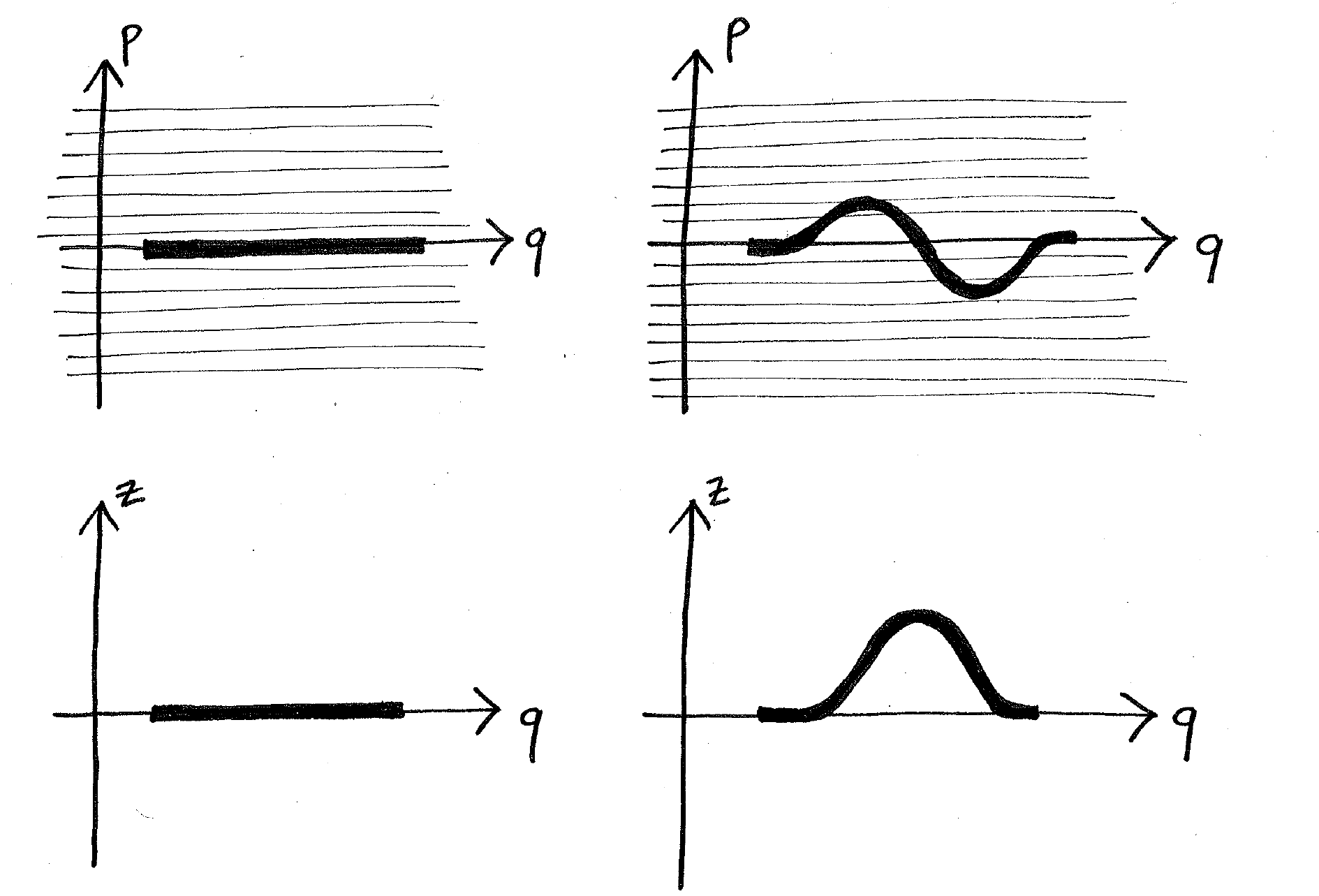}
\caption{The local model for a creation of double folds when the Legendrian is almost tangent to the foliation. For $\cF=\text{span}( \partial/\partial q + p \partial / \partial z)$ the model obviously creates a double fold, hence by stability it also does so for nearby $\cF$.}
\label{creationoffolds}
\end{figure}

We can now conclude the proof that $\pi_{n}(\cL, \cD)=0$ for $n>2$. Given $\alpha \in \pi_{n-1}(\cL, \cD)$ represented by a family $F$, we can apply Lemmas \ref{preparatory} and \ref{prep2} and replace $F$ with the family obtained after performing the two preliminary arrangements, in that order. For the new $F$, we claim the existence of a family of tangential rotations $G_t: D^n \times S^1 \to S^1$ of the family $F$ such that the following properties hold.
\begin{itemize}
\item $G_t$ is fixed on $Op\big( (\partial D^n \times S^1) \cup \cK \cup \cI ) \big).$
\item $G_1 \pitchfork \cF$ away from $\cK \cup \cI$.
\end{itemize}

 To verify the claim, we begin by considering the restriction of the Gauss map $G(dF):D^n \times S^1 \to S^1$ to the annulus $A$. Note that by construction the lift $\widetilde{F} : \partial D^n \times S^1 \to \cD$ extends to a lift $\widetilde{F}: A \times S^1 \to \cD$, where we assign intervals to the new pair of folds created by the family $L_\eta$. The intervals $I_1, \ldots , I_m$ of the decoration $\text{dec} \circ \widetilde{F} : A \to \widetilde{\text{C}}(S^1)$ which do not correspond to pairs of folds in $\cK$ give us a homotopically canonical deformation $G_t:A \times S^1 \to S^1$ of $G(dF)|_A$ such that $G_t$ is fixed on $Op\big( (\partial D^n \times S^1) \cup \cK \cup ( \cI \cap A) \big)$ and such that $G_1 \pitchfork \cF$ away from $\cK \cup ( \cI \cap A)$. Together with the requirement that $G_t$ is fixed near $\cI$, this defines the map $(z,s,t) \mapsto G_t(z,s)$ on $(A \times S^1 \times [0,1] ) \cup (D^n \times S^1 \times 0)\cup (Op(\cI) \times [0,1])$. Each connected component of the complement of $Op(\cI) $ in $ (D^n \setminus A)  \times S^1 $  is diffeomorphic to $(D^n \setminus A) \times J $, where $J$ is a closed interval and the diffeomorphism is of the form $(z,s) \mapsto \big(z, \psi(z,s) \big)$. Consider the cube
 \[ Q =  \partial( D^n \setminus A) \times J \times [0,1]  \, \, \cup \, \, ( D^n \setminus A) \times J \times 0 \,\, \cup  \, \, (D^n \setminus A) \times \partial J \times [0,1] \]
 which we think of as a subset of $(D^n \setminus A) \times S^1 \times [0,1]$ via the above diffeomorphism. See Figure \ref{cubeA}. Note that $Q$ has boundary 
 \[ \partial Q = \partial ( D^n \setminus A) \times J \times 1    \, \, \cup \, \,(D^n \setminus A) \times \partial J  \times 1.  \]
 
 \begin{figure}[h]
\includegraphics[scale=0.6]{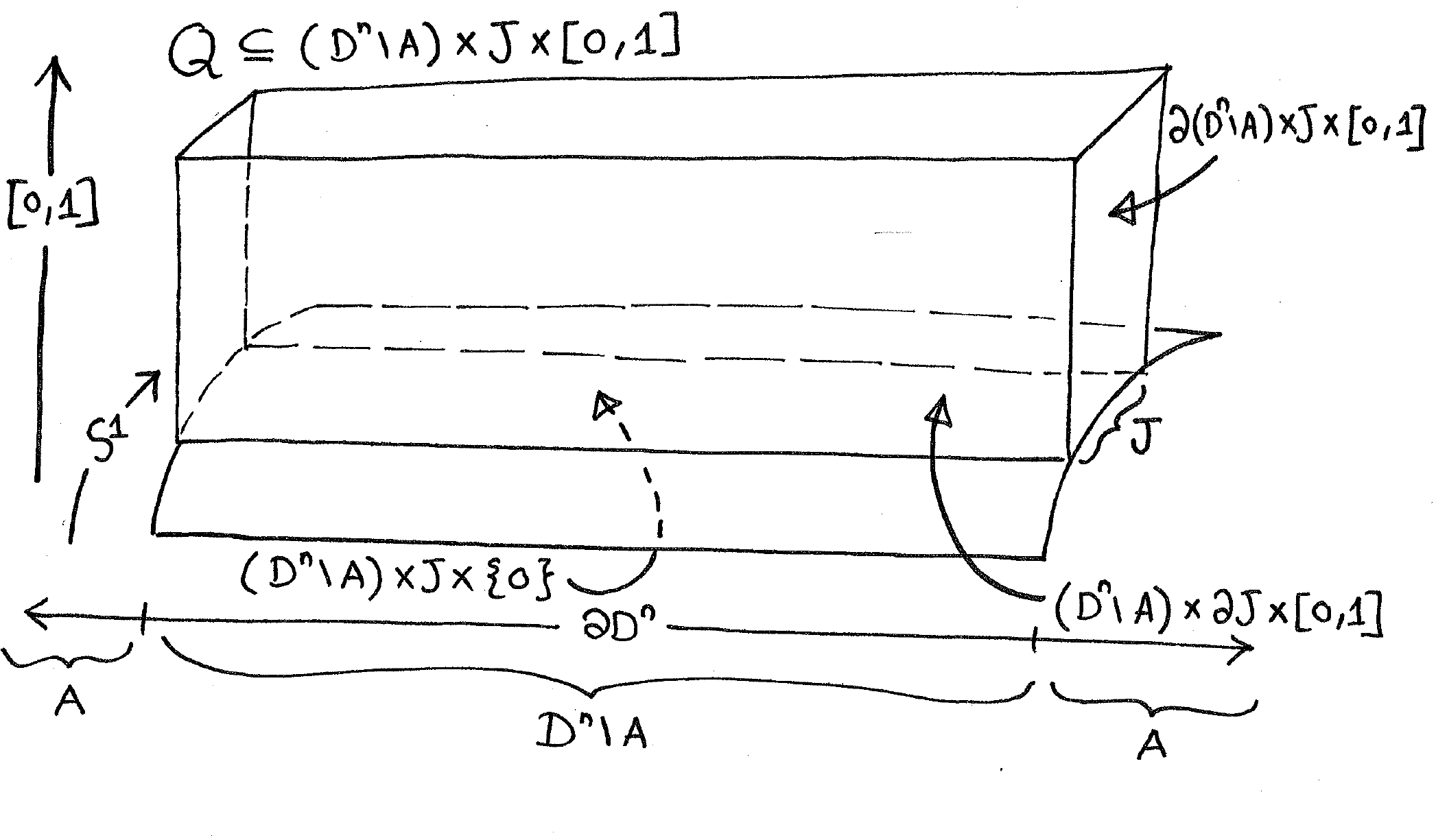}
\caption{The cube $Q$ sits like an open box inside $(D^n \setminus A) \times J \times [0,1]$.}
\label{cubeA}
\end{figure}
 
 The homotopy $G_t$ defined thus far gives a map of pairs $(Q, \partial Q) \to (S^1 , S^1 \setminus pt)$, where $pt= \text{span}( \partial / \partial p)$. Since $\pi_j(S^1, S^1 \setminus pt)=0$ for $j>1$, there exists a homotopy of pairs relative to the boundary so that at the end of the homotopy the image is disjoint from $\text{span}( \partial / \partial p)$. This is precisely what we needed to define $G_t$ on the remaining part of $D^n \times S^1 \times [0,1]$ so that the required conditions are satisfied.
 
 Now that we have established the existence of such a tangential rotation $G_t$, we can invoke Theorem \ref{parametric main result again} to construct a homotopy $F_t : D^n \to \cL$ of $F$ which is fixed on $Op \big( ( \partial D^n \times S^1) \cup \cK \cup \cI) \big)$ and such that away from $\cK \cup \cI$ the singularities of the family $F_1$ consist of a finite union of fibered nested regularized wrinkles. It only remains to show that $ \widetilde{F} : \partial D^n \to \cD$ extends to a lift of $F_1$ to $\cD$. However, this is clear because to the folds of $\cI$ and to the pairs of folds of $\cK$ we can assign points and intervals in the obvious way, while away from $\cK \cup \cI$ the singularities of $F_1$ consist only of the pairs of points in the fibered regularized wrinkles, to which intervals can be canonically assigned. This completes the proof that $\pi_{n}(\cL, \cD) = 0$ for $n>1$. \end{proof}

\begin{proof}[Proof that $\pi_1(\cD) \to \pi_1(\cL)$ is surjective.] Let $\alpha \in \pi_1(\cL)$ be any class. We can represent $\alpha$ by a map $F: [0,1] \to \cL$ such that $F(0)=F(1)=f_0$. Choose any decoration $D_0$ which is compatible with $f_0$. We must show that there exists a homotopy $F_t: [0,1] \to \cL$ of $F=F_0$ such that $F_t(0)=F_t(1)=f_0$ for all $t \in [0,1]$ and such that $F_1: [0,1] \to \cL$ lifts to a map $\widetilde{F}_1:[0,1] \to \cD$ with $\widetilde{F}_1(0)=\widetilde{F}_1(1)=(f_0,D_0)$.

Write $D_0=( \{t_i\} , \{I_j\})$ for points $t_1, \ldots , t_k \in S^1$ and non-degenerate intervals $I_1, \ldots , I_m \subset S^1$. Let $K=\{t_1, \ldots , t_k\}  \cup \partial I_1 \cup \cdots \cup \partial I_m \subset S^1$. Observe that the Gauss map $G(dF):[0,1] \times S^1 \to \Lambda_1(\bR^3)$ of the family $F$ satisfies $G(dF)=\text{span}( \partial / \partial p)$ on $ \partial [0,1] \times K$. Let $G_t:[0,1] \times S^1 \to \Lambda_1(\bR^3)$ be a tangential rotation of the family $F$ such that $G_t$ is fixed on $Op( \partial [0,1] \times S^1)$ and such that $G_1=\text{span}(\partial/\partial p)$ on $[0,1] \times K$. Using Theorem \ref{parametric perp-holonomic} as above, we can construct a homotopy $F_t:[0,1] \to \cL$ which is fixed near $\partial [0,1]$ and such that $G(dF_1)$ is $C^0$-close to $\text{span}(\partial / \partial p)$ on $[0,1] \times K$. 

By the insertion of the local model in Figure \ref{creationoffolds} we can assume that $F_1$ actually has folds along $[0,1] \times K$. Theorem \ref{parametric main result again} can then be used to further homotope $F_1$ rel $Op\big( ( \partial [0,1] \times S^1) \cup [0,1] \times K) \big)$ so that on the complement of $[0,1] \times K$ the only singularities are fibered nested regularized wrinkles.  This new $F_1:[0,1] \to \cL$ admits a canonical lift $\widetilde{F}_1:[0,1] \to \cD$ by assigning intervals to the pairs of points in the fibered regularized wrinkles. This completes the proof that $\pi_1(M) \to \pi_1(\cL)$ is surjective. Hence Corollary \ref{higher Reidemeister} is also proved.  \end{proof}

We conclude this section with a remark. Proving that $\pi_n(\cL, \cD)=0$ for $n>1$ amounts to solving the following lifting problem. Given a diagram of the form \\

\begin{center}
$ \begin{CD}
\cD @>>>  \cL \\
@AA A @AAA \\
S^{n-1} @> >> D^n
\end{CD} $
\end{center}
\[ \]
we must show that there exists a map $D^n \to \cD$ such that when added to the above diagram all compositions commute up to a homotopy fixed on $S^{n-1}$. The proof of Corollary \ref{higher Reidemeister} achieves this, but in fact proves slightly more. Because all of the theorems invoked hold in $C^0$-close form and because all of the local models used are $C^0$-small perturbations, it follows that the composition of the lift $D^n \to \cD$ with the map $\cD \to \cL$ can be taken to be $C^0$-close to the original map $D^n \to \cL$. The analogous $C^0$-approximation result holds for the corresponding lifting property for proving that $\pi_1(\cD) \to \pi_1(\cL)$ is surjective. 

\subsection{Final remarks}

We conclude our discussion with a couple of remarks.

\begin{remark}
All of the results proved in this paper also hold for immersed rather than embedded Lagrangians or Legendrians $f:L\to M$. The reason is that from the onset one can replace $M$ with $T^*L$ or $J^1(L,\bR)$ by choosing a Weinstein neighborhood of the immersion, thereby reducing to the embedded case. The only difference in the conclusion is that the resulting exact homotopy of regular Lagrangian or Legendrian immersions will not be induced by an ambient Hamiltonian isotopy in the original manifold $M$.
\end{remark}

\begin{remark} It is worth giving the following warning. If the singularities of a regular Lagrangian or Legendrian embedding $g:L \to M$ with respect to $\cF$ only consist of a disjoint union of regularized wrinkles (or double folds), then the singularity locus is quite simple in the source. However, in the target the image of the singularity locus is likely to be very complicated. Moreover, this complication is likely to encode most, if not all, of the rigidity of the Lagrangian or Legendrian submanifold.
\end{remark}

\begin{remark}
From Theorems \ref{wrinkling lagrangians} and \ref{parametric wrinkling lagrangians} we can also deduce a full $h$-principle for directed embeddings of wrinkled Lagrangian or Legendrian embeddings analogous to the one deduced by Eliashberg and Mishachev from their wrinkled embeddings theorem \cite{EM09}. Before we can state it, we need a definition.
\begin{definition}
For any Lagrangian or Legendrian embedding $f:L \to M$ and for any subset $A \subset \Lambda_n(M)$, we say that $f$ is $A$-directed if $\text{im}\big(G(df) \big) \subset A$. 
\end{definition}
The result is then the following.
\begin{theorem} Let $f:L \to M$ be a Lagrangian or Legendrian embedding, let $A \subset \Lambda_n(M)$ be any open subset and assume that there exists a tangential rotation $G_t$ of $f$ such that $\text{im} \big( G_1 \big ) \subset A$. Then there exists an exact homotopy of wrinkled Lagrangian or Legendrian embeddings $f_t:L \to M$ such that $f_1$ is $A$-directed. 
\end{theorem}
This theorem holds in $C^0$-close, relative and parametric forms and follows immediately from Theorems \ref{wrinkling lagrangians} and \ref{parametric wrinkling lagrangians}  since $A$ is assumed to be open. 
\end{remark}

\end{document}